\documentclass[12pt, oneside]{article}   	% use "amsart" instead of "article" for AMSLaTeX format
\usepackage{geometry}                		% See geometry.pdf to learn the layout options. There are lots.
\usepackage{amsthm}
\usepackage{amsmath}
\usepackage{graphicx}				% Use pdf, png, jpg, or eps? with pdflatex; use eps in DVI mode
\usepackage{amssymb}
\usepackage[all]{xy}
\usepackage{makeidx}
\usepackage{nameref}
\usepackage{xcolor}
\usepackage{tikz-cd}
\usetikzlibrary{arrows,decorations.pathreplacing,angles,quotes,bending}
\usepackage{mathscinet}
\usepackage{enumerate}
\usepackage{float}
\usepackage{enumitem}
\usepackage[font=scriptsize]{caption}
\usepackage{url}
\makeindex
\theoremstyle{plain}
\newtheorem{theorem}{Theorem}[section]
\newtheorem{lemma}[theorem]{Lemma}
\newtheorem{corollary}[theorem]{Corollary}
\newtheorem{proposition}[theorem]{Proposition}
\theoremstyle{definition}
\newtheorem{definition}[theorem]{Definition}
\newtheorem{example}[theorem]{Example}

\newtheorem{notation}[theorem]{Notation}

\newtheorem{remark}[theorem]{Remark}

\title{A strong characterization of the entries of the Burau matrices of $4$-braids: The Burau representation of the braid group $B_4$ is faithful almost everywhere} 
\author{Amitesh Datta}
\begin{document}
\maketitle
\abstract{We establish strong constraints on the kernel of the (reduced) Burau representation $\beta_4:B_4\to \text{GL}_3\left(\mathbb{Z}\left[q^{\pm 1}\right]\right)$ of the braid group $B_4$. We develop a theory to explicitly determine the entries of the Burau matrices of braids in $B_4$, and this is an important step toward demonstrating that $\beta_4$ is faithful (a longstanding question posed in the 1930s). The theory is based on a novel combinatorial interpretation of $\beta_4\left(g\right)$, in terms of the Garside normal form of $g\in B_4$ and a new product decomposition of positive braids. We develop cancellation results for words in matrix groups to show that if $\sigma$ is a generic positive braid in $B_4$ and if $t\neq 2$ is a prime number, then the leading coefficients in at least one row of the matrix $\beta_4\left(\sigma\right)$ are non-zero modulo $t$. We exploit these cancellation results to deduce that the Burau representation of $B_4$ is faithful almost everywhere.} 

\tableofcontents
\section{Introduction}
\label{intro}

The (reduced) \textit{Burau representation} $\beta_n:B_n\to \text{GL}_{n-1}\left(\mathbb{Z}\left[q^{\pm 1}\right]\right)$, defined by Burau~\cite{burau1935zopfgruppen} in the 1930s, is a classical representation of the braid group $B_n$. The Burau representation is fundamental in low-dimensional topology. For example, the Alexander polynomial of a link $L\subseteq \mathbb{S}^3$ can be defined completely in terms of $\beta_n\left(g\right)$, if $L$ is the braid closure of $g\in B_n$. 

If $B_n$ is defined as the mapping class group of the $n$-times punctured disk $\mathbb{D}_n$, then Burau defined $\beta_n$ as the action of $B_n$ on the first homology of a certain infinite cyclic covering space $\widetilde{\mathbb{D}}_n$ of $\mathbb{D}_n$ viewed as a $\mathbb{Z}\left[q^{\pm 1}\right]$-module. The first homology group $H_1\left(\widetilde{\mathbb{D}}_n\right)$ is a $\mathbb{Z}\left[q^{\pm 1}\right]$-module of rank $n-1$ where the action of $q$ is induced by a generator of the deck transformation group of the covering map $\widetilde{\mathbb{D}}_n\to \mathbb{D}_n$. 

We can also define $\beta_n$ algebraically in terms of invertible $\left(n-1\right)\times \left(n-1\right)$ matrices with entries in $\mathbb{Z}\left[q^{\pm 1}\right]$ assigned to the generators in the Artin presentation of $B_n$, satisfying the relations in this presentation. In this paper, we will use the algebraic interpretation of the Burau representation. 

A representation is \textit{faithful} if it is one-to-one, or equivalently, its kernel is trivial. The values of $n$ for which $\beta_n$ is faithful is a longstanding open question. Note that $\beta_3$ is known to be faithful~\cite{magnus1969theorem}. However, in 1991, Moody~\cite{moody1991burau} proved that $\beta_n$ is not faithful for $n\geq 9$, and subsequently Long and Paton~\cite{long1993burau} in 1993 and Bigelow~\cite{bigelow1999burau} in 1999 proved that $\beta_n$ is not faithful for $n\geq 6$ and $n=5$, respectively. 

Nonetheless, the question as to whether $\beta_4$ is faithful remains an important open problem. A corollary of the faithfulness of $\beta_4$ would be that $B_4$ is a group of $3\times 3$ matrices over the real numbers (we can embed $\text{GL}_3\left(\mathbb{Z}\left[q^{\pm 1}\right]\right)\to \text{GL}_3\left(\mathbb{R}\right)$ by specializing $q$ to be any transcendental number). Moreover, this would imply that $d=3$ is minimal with respect to the property that $B_4$ is a group of $d\times d$ matrices over the real numbers~\cite{button2016minimal}. The faithfulness question of $\beta_4$ is also closely related to the question of whether or not the Jones polynomial of a knot detects the unknot.

In this paper, we develop a new theory to characterize the entries in the Burau matrix $\beta_4\left(g\right)$ of an element $g\in B_4$. If $g$ is expressed as a word in the Artin generators, then the Burau matrix $\beta_4\left(g\right)$ is the product of the Burau matrices of the Artin generators in the word. The entries in the Burau matrices of the Artin generators are $0$, $1$, $\pm q$. In particular, when multiplying these matrices, there is cancellation of signed monomials in $q$, and understanding this cancellation is crucial to the faithfulness question of $\beta_4$.

We introduce a novel geometric interpretation of this cancellation in terms of pairs of paths in a graph bounding simple geometric shapes such as triangles and trapeziums. We demonstrate the existence of numerous non-cancelling terms in the Burau matrix $\beta_4\left(g\right)$ for a generic braid $g\in B_4$ by geometrically constructing paths in a graph with special properties. We use our theory to establish strong constraints on the kernel of $\beta_4$, and specifically characterize a generic family of braids that are not in the kernel of $\beta_4$. A consequence of our main result is that $\beta_4$ is faithful almost everywhere, in the sense that the percentage of braids of word length $L$ (with respect to the standard Artin generating set) that are not in the kernel of $\beta_4$ rapidly converges to $100\%$ as $L\to \infty$. Furthermore, we establish the same results for the reduction of $\beta_4$ modulo every prime not equal to $2$ and $3$, and for the specialization of $\beta_4$ at every rational number $q$ such that the denominator of $q$ has a prime factor distinct from $2$ and $3$.

A common approach that attempts to establish the faithfulness of $\beta_4$ in the literature originated in~\cite{birmanbraidslinks}. The approach reduces the faithfulness of $\beta_4$ to showing that a pair of $3\times 3$ matrices in $\text{GL}_3\left(\mathbb{Z}\left[q^{\pm 1}\right]\right)$ generates a free group. However, attempts based on this approach do not explicitly characterize the entries of Burau matrices of $4$-braids. For example, only the restriction of $\beta_4$ to a very special subgroup of $B_4$ is considered and approaches toward establishing that $\beta_4$ is faithful on this subgroup are based on simply showing that certain matrices are not the identity (e.g., using the ping-pong lemma) rather than explicitly characterizing entries of Burau matrices. The approach also does not characterize a generic family of braid words that are not in the kernel of $\beta_4$. Finally, the approach is based on the special algebraic structure of the braid group $B_4$ and does not have the potential for generalization to studying representations of the braid groups $B_n$ for $n>4$. 

On the other hand, the theory developed in our paper studies the Burau representation in general (going beyond just the faithfulness question) and gives strong characterizations of the entries of Burau matrices of $4$-braids. The theory can also be generalized to study quantum representations of braid groups $B_n$ for $n>4$ (such as the irreducible summands of the Jones representations of the braid groups). An important application of explicitly characterizing the entries of Burau matrices of $4$-braids (besides a deeper understanding of $\beta_4$ itself) is the explicit characterization of classical knot polynomials (such as the Alexander and Jones polynomials) of (generic) $4$-braid closures (in terms of the braid words).

\subsection{The main result}

We briefly review background, notation and terminology before stating the main result. 

\begin{definition} 
\label{buraudef}
The \textit{Artin presentation} of $B_4$ is \[\left\langle \sigma_1,\sigma_2,\sigma_3:\sigma_1\sigma_2\sigma_1=\sigma_2\sigma_1\sigma_2, \sigma_2\sigma_3\sigma_2=\sigma_3\sigma_2\sigma_3, \sigma_1\sigma_3=\sigma_3\sigma_1\right\rangle.\] The \textit{Burau representation} $\beta_4:B_4\to \text{GL}_3\left(\mathbb{Z}\left[q^{\pm 1}\right]\right)$ is defined in terms of the generators in this presentation as follows: \[\sigma_1\to \left( \begin{array}{ccc}
q & 0 & 0 \\
-q & 1 & 0 \\
0 & 0 & 1 \end{array} \right)\] \[\sigma_2\to \left( \begin{array}{ccc}
1 & 1 & 0 \\
0 & q & 0 \\
0 & -q & 1 \end{array} \right)\] \[\sigma_3\to \left( \begin{array}{ccc}
1 & 0 & 0 \\
0 & 1 & 1 \\
0 & 0 & q \end{array} \right)\] 
\end{definition}

Of course, one can check that the matrices assigned to the generators are invertible and satisfy the corresponding relations in the presentation. Note that other definitions of $\beta_4$ in the literature are equivalent to the one above by a change of basis.

We recall that the \textit{Garside element} of $B_4$ is $\Delta=\left(\sigma_1\sigma_2\sigma_3\right)\left(\sigma_1\sigma_2\right)\sigma_1$, and a \textit{positive braid} in $B_4$ is a braid in the submonoid of $B_4$ generated by $\sigma_1, \sigma_2, \sigma_3$. In~\cite{garside1969braid}, Garside proved that every braid $g\in B_4$ can be expressed in the form $g=\Delta^{k}\sigma$ for $k\in \mathbb{Z}$ and $\sigma$ a positive braid in $B_4$. If $k\in \mathbb{Z}$ is chosen to be as large as possible, then this expression is unique and is referred to as the \textit{Garside normal form} of $g$. 

We introduce a condition for positive braids in $B_4$ (Definition~\ref{normal}) before formally stating one of the main results (Theorem~\ref{main}, of which Theorem~\ref{tfirst} is a corollary), for ease of exposition. Firstly, we recall a notion due to Garside~\cite{garside1969braid}, and introduce terminology for this notion in the following Definition~\ref{introminform}. We refer the reader to Subsection~\ref{subsecgar} of Section~\ref{def} for a detailed discussion.

\begin{definition}
\label{introminform}
Let $\sigma$ be a positive braid in $B_4$. We write $\sigma=\prod_{i=1}^{n} \sigma_1^{a_i}\sigma_3^{b_i}\sigma_2^{c_i}$ for a product expansion of $\sigma$ in the submonoid of positive braids (i.e., $a_i,b_i,c_i\geq 0$ for $1\leq i\leq n$). A product expansion of $\sigma$ is in \textit{minimal form} if the sequence of indices of the generators (read from left to right) is lexicographically minimal among all product expansions of $\sigma$ in the submonoid of positive braids. The \textit{minimal form of $\sigma$} is the (unique) product expansion of $\sigma$ in minimal form.
\end{definition}

For example, the minimal form of $\sigma_2\sigma_1\sigma_2$ is $\sigma_1\sigma_2\sigma_1$. The minimal form of $\sigma_3\sigma_1$ is $\sigma_1\sigma_3$. Finally, for a slightly more complex example, the minimal form of $\sigma_1\sigma_3\sigma_2\sigma_1\sigma_2$ is $\sigma_1^2\sigma_3\sigma_2\sigma_1$. 

\begin{definition}
\label{normal}
Let $\sigma$ be a positive braid in $B_4$ and write $\sigma=\prod_{i=1}^{n} \sigma_1^{a_i}\sigma_3^{b_i}\sigma_2^{c_i}$ for the minimal form of $\sigma$ (see preceding Definition~\ref{introminform}). We write that $\sigma$ is a \textit{normal braid} if the following conditions are satisfied for each $p<n$. 

\begin{description}[style=unboxed,leftmargin=0cm]
\item[(i)] If $a_p=0$ and $b_p = 1$, then $c_p\neq 1,2$. If $a_p = 0$, $b_p = 1$, and $a_{p-1}+1\geq c_p$, then $c_{p-1}\neq 1$. 
\item[(ii)] If $b_p>0$, $a_p\geq b_p+1$ and $c_p = 1$, then $a_{p+1}\neq 2$. If $a_p = 1$, $b_p = 1$, and $c_p = 1$, then $a_{p+1}\neq 2$. 
\end{description}
\end{definition}

We view the condition that a positive braid $\sigma$ is a normal braid as a ``local constraint" on the minimal form of $\sigma$. We observe that generic positive braids are normal braids. For example, $\left(a_p,b_p,c_{p-1},c_p\right)\in \mathbb{N}^4$ can be \textit{any} quadruple of nonnegative integers. The condition \textbf{(i)} only excludes a subset of the union $\{0\}\times \{1\}\times \left(\left(\{1\}\times \mathbb{N}\right)\cup \left(\mathbb{N}\times \{1,2\}\right)\right)$, which is a union of three copies of $\mathbb{N}$, in $\mathbb{N}^4$. Of course, this is a codimension three condition in $\mathbb{N}^4$. Similarly, the condition \textbf{(ii)} is a codimension two condition in $\mathbb{N}^4$. In this sense, the percentage of positive braids of length $L$ that are normal braids rapidly converges to $100\%$ as $L\to \infty$.

Furthermore, a general property satisfied by the minimal form of $\sigma$ (that we establish later in Lemma~\ref{minimalform} \textbf{(iv)}) implies that the set of conditions $b_p = 1$, $c_p = 1$ and $a_{p+1}>0$ can only occur \textit{at most once}. Furthermore, if it occurs, then the minimal form of $\sigma$ entirely consists of $\sigma_1$s and $\sigma_2$s until that point (no $\sigma_3$s). In particular, a part of the condition \textbf{(i)} in Definition~\ref{normal} that $c_p\neq 1$, and the part of the condition \textbf{(ii)} in Definition~\ref{normal} in the case $b_p = 1$, are only required for at most one value of $p$ even in non-generic situations, and are thus very weak conditions.

We are prepared to state one of the main results.

\begin{theorem}
\label{tfirst}
Let $g\in B_4$ be a non-identity braid and write $g = \Delta^{k}\sigma$ for the Garside normal form of $g$. If either $k\geq 0$ or $\sigma$ is a normal braid, then $g\not\in \text{ker}\left(\beta_4\right)$. 
\end{theorem}

If we remove the requirement that $\sigma$ is a normal braid (for $k<0$) in the statement of Theorem~\ref{tfirst}, then the statement implies that $\beta_4$ is faithful.

We can also define a pair of \textit{specializations} of the Burau representation. If $t$ is a prime number, then we define $\left(\beta_4\right)_{t}:B_4\to \text{GL}_3\left(\mathbb{Z}/t\mathbb{Z}\left[q^{\pm 1}\right]\right)$ to be the reduction of the Burau representation modulo $t$. If $\frac{a}{b}\in \mathbb{Q}$ is a rational number, then we define $\left.\left(\beta_4\right)\right|_{q = \frac{a}{b}}:B_4\to \text{GL}_3\left(\mathbb{Q}\right)$ to be the specialization of the Burau representation at $q = \frac{a}{b}$. We establish the following more general statement on the kernels of these specializations of the Burau representation, of which Theorem~\ref{tfirst} is an immediate corollary (since $\text{ker}\left(\beta_4\right)\subseteq \text{ker}\left(\left(\beta_4\right)_{t}\right)$ for each prime $t$).

\begin{theorem}
\label{main}
Let $g\in B_4$ be a non-identity braid and write $g=\Delta^{k}\sigma$ for the Garside normal form of $g$. Let $t\neq 2$ be a prime number and let $a,b$ be coprime integers such that $b\neq 0$ has a prime factor distinct from $2$.

If either $k\geq 0$ or $\sigma$ is a normal braid, then $g\not\in \text{ker}\left(\left(\beta_4\right)_{t}\right)$, and in particular, $g\not\in \text{ker}\left(\beta_4\right)$. If $\sigma$ is a normal braid, then $g\not\in \text{ker}\left(\left.\left(\beta_4\right)\right|_{q=\frac{a}{b}}\right)$.
\end{theorem}

Theorem~\ref{main} establishes that a linear representation of $B_4$ into a group of $3\times 3$ matrices over the rational numbers is faithful almost everywhere. (The linearity of $B_4$ over the rational numbers is unknown.) Furthermore, the strong constraint on the kernel of the specialization of the Burau representation modulo $t$ in Theorem~\ref{main} is important since it is computationally much simpler to multiply matrices modulo a small prime number. We note that interestingly the Burau representation $\beta_4$ is known not to be faithful modulo two~\cite{cooperlongmod2}.

In fact, we prove a significantly stronger version of Theorem~\ref{main}, where ``normal braid" is replaced by ``weakly normal braid" (but where we also require the prime number $t\neq 3$ and the integer $b$ to have a prime factor distinct from $3$). The precise statement is Theorem~\ref{mainstronger} in Section~\ref{smainstronger}, and is a little more notationally involved. We prove Theorem~\ref{main} in Sections \ref{def} - \ref{mains} and Theorem~\ref{mainstronger} in Section~\ref{smainstronger}. We briefly describe a special case of Theorem~\ref{mainstronger} here without using technical language to communicate its idea. 

For example, consider a case where condition~\textbf{(i)} of being a normal braid fails in Theorem~\ref{main} and $a_p = 0$, $b_p = 1$, and $c_p = 2$. The corresponding part of the minimal form of $\sigma$ reads as $\sigma_2^{c_{p-1}}\sigma_3\sigma_2^2$. In this case, it is still true that $\Delta^{k}\sigma\not\in \text{ker}\left(\beta_4\right)$, except possibly if there is a highly constrained string of small exponents (zero, one or two) with length proportional to $c_{p-1} - 2$ directly after $\sigma_2^{c_{p-1}}\sigma_3\sigma_2^2$. For example, if $c_{p-1}  = 100$, and if $\sigma_3^2\sigma_2^2\sigma_3^2\cdots$ is a string with \textit{exactly} $98$ squares directly after $\sigma_2^{c_{p-1}}\sigma_2\sigma_2^{2}$, then $\sigma$ is not a weakly normal braid. However, if even \textit{one} square in this string is a cube or higher power, even \textit{one} $\sigma_3$ in this string is a $\sigma_1$, even \textit{one} $\sigma_3^2$ in this string is a $\sigma_1^{a_{p'}}\sigma_3^{b_{p'}}$ with $a_{p'}, b_{p'}>0$, or there are more than $98$ squares in the string, then $\sigma$ is a weakly normal braid, and Theorem~\ref{mainstronger} implies that $\Delta^{k}\sigma\not\in \text{ker}\left(\beta_4\right)$.

Theorem~\ref{main} is closely related to the open problem of whether or not the Jones polynomial of a knot detects the unknot. Indeed, the Jones polynomial of a knot would \textit{not} detect the unknot if the Jones representation of $B_4$ (defined by Jones in~\cite{jones1987hecke}) were not faithful (\cite{bigelow2002does} and \cite{itojones}). The Burau representation of $B_4$ is an irreducible summand of the Jones representation of $B_4$, and in particular, the kernel of the Jones representation of $B_4$ is contained in the kernel of the Burau representation of $B_4$~\cite{jones1987hecke}. Thus, Theorem~\ref{main} also gives constraints on the kernel of the Jones representation of $B_4$. Moreover, the methods in this paper can be used to study the Jones polynomial invariant of links realized as closures of geometric $4$-braids. 

\subsection{Outline of the proof and organization of the paper}
\label{subsecoutline}

We will simultaneously outline the proof of Theorem~\ref{main} and the organization of the paper. The main purpose is to give the reader a sense of the fundamental objects and techniques in this paper, and how they fit together to establish the main results. We also motivate the main ideas of the proof.

The faithfulness of the Burau representation $\beta_4:B_4\to \text{GL}_3\left(\mathbb{Z}\left[q^{\pm 1}\right]\right)$ is equivalent to the statement that if $e\neq g\in B_4$, then $I\neq \beta_4\left(g\right)\in \text{GL}_3\left(\mathbb{Z}\left[q^{\pm 1}\right]\right)$ (where $e$ is the identity element in $B_4$ and $I$ is the $3\times 3$ identity matrix). If $g\in B_4$, then $\beta_4\left(g\right)$ is the product of the Burau matrices of the Artin generators in a product expansion of $g$. In particular, the two primary aspects of the faithfulness question are the study of words in the Artin generators (the algebraic structure of $B_4$) and the study of products of Burau matrices of the Artin generators (cancellation of entries in matrix products).

In Section~\ref{def}, we will establish statements concerning the algebraic structure of the braid group $B_4$ in connection to Theorem~\ref{main}. We will also review background and terminology that we will use in the rest of the paper. An element $g\in B_4$ has many different product expansions. The first step is to determine a canonical product expansion of each element $g\in B_4$ (a solution to the word problem in the braid group $B_4$), and reduce the faithfulness question to establishing that $\beta_4\left(g\right)\neq I$, where $\beta_4\left(g\right)$ is evaluated in terms of the canonical product expansion of $g\neq e$.

The Garside normal form furnishes a solution to the word problem in the braid groups. If $g = \Delta^{k}\sigma$ is the Garside normal form of $g$, then we determine a canonical product expansion of $g$ by considering the minimal form of $\sigma$ (a canonical product expansion of $\sigma$). In Subsection~\ref{subsecgar}, we will establish a new characterization of the minimal form of a positive braid. We will show that a product expansion is in minimal form only if it satisfies a specific set of constraints (Lemma~\ref{minimalform}). 

If $g = \Delta^k\sigma$, then the statement that $\beta_4\left(g\right)\neq I$ is equivalent to the statement that $\beta_4\left(\sigma\right)\neq \beta_4\left(\Delta^{-k}\right)$. Furthermore, in Lemma~\ref{Garsideaction}, we compute $\beta_4\left(\Delta^{-k}\right)$ and we observe that there is precisely one non-zero entry in each row of $\beta_4\left(\Delta^{-k}\right)$. In Subsection~\ref{subsecreduction}, we use the Garside normal form and these statements to reduce Theorem~\ref{main} (the main result) to proving Theorem~\ref{invariantdetect}. Theorem~\ref{invariantdetect} states that if $\sigma$ is a normal braid and if $t\neq 2$ is a prime number, then the leading coefficients of multiple entries in at least one row of $\beta_4\left(\sigma\right)$ are non-zero modulo $t$. Lemma~\ref{minimalform} restricts our attention to understanding $\beta_4\left(\sigma\right)$ in terms of a product expansion of $\sigma$ satisfying a specific set of constraints. 

In Section~\ref{pathgraph}, in order to establish Theorem~\ref{invariantdetect}, we will give a new interpretation of $\beta_4\left(\sigma\right)$ for a positive braid $\sigma$ in terms of \textit{(admissible) $\sigma$-paths}. A $\sigma$-path is a path in the straight-line graph on the vertex set $\{1,2,3\}$ that is compatible with the minimal form of $\sigma$ in a certain sense, and with length (number of edges) equal to the length of $\sigma$ (as a positive braid) (Definition~\ref{spath}). We will assign a \textit{weight} to each $\sigma$-path (a signed monomial in $q$ of nonnegative degree). An \textit{$\left(r,s\right)$-type $\sigma$-path} is a $\sigma$-path with initial vertex equal to $r$ and final vertex equal to $s$. We will prove that the $\left(r,s\right)$-entry of $\beta_4\left(\sigma\right)$ is the weighted number of $\left(r,s\right)$-type $\sigma$-paths (Proposition~\ref{m=p}). The notion of $\sigma$-paths furnishes a mental model for understanding matrix multiplication of Burau matrices of Artin generators. The mental model allows one to visualize terms in the entries of $\beta_4\left(\sigma\right)$, and explicitly determine the coefficients of high $q$-degree terms, when thinking in terms of matrix multiplication would be unwieldy.

The rest of the proof will be devoted to constructing $\sigma$-paths with weights contributing non-cancelling terms in multiple entries in a row of $\beta_4\left(\sigma\right)$. In Subsection~\ref{subsecadspath}, we will introduce a local obstruction to a $\sigma$-path admitting this property, belonging to a \textit{distinguished pair} of $\sigma$-paths with cancelling weights (Definition~\ref{dispair}). A beautiful application of the theory of $\sigma$-paths is that the local cancellation admits a concrete geometric interpretation. Indeed, a distinguished pair of $\sigma$-paths locally bound a triangle or a trapezium (examples of distinguished pairs and this local cancellation are illustrated in several figures in Subsection~\ref{subsecadspath}). We refer to \textit{admissible $\sigma$-paths} as those $\sigma$-paths which are not locally obstructed in this manner. We will prove that the $\left(r,s\right)$-entry of $\beta_4\left(\sigma\right)$ is the weighted number of admissible $\left(r,s\right)$-type $\sigma$-paths (Proposition~\ref{m=ap}, which is a refinement of Proposition~\ref{m=p}). The proof of Proposition~\ref{m=ap} requires a careful analysis that distinct distingished pairs do not overlap.

In Subsection~\ref{subsecgb}, we introduce the notion of \textit{good and bad $\sigma$-paths} in order to characterize global cancellation of weights of $\sigma$-paths. The basic idea for establishing Theorem~\ref{invariantdetect} is to construct admissible $\sigma$-paths with maximal $q$-degree, and the maximality will ensure that their weights do not cancel with (have opposite sign to) the weights of other admissible $\sigma$-paths. We study a simple example (Example~\ref{exmain}), where we demonstrate that a generic family of braids is not in the kernel of $\beta_4$. Example~\ref{exmain} quickly illustrates the power of the theory of good and bad $\sigma$-paths for the purpose of constructing maximal $q$-degree $\sigma$-paths and establishing strong constraints on the kernel of $\beta_4$. In general, we will establish that good $\sigma$-paths accumulate $q$-degree quickly whereas bad $\sigma$-paths accumulate $q$-degree slowly (Lemma~\ref{lemmagoodext}). In Subsection~\ref{subsecpathextensions}, we will develop a theory of path extensions based on Lemma~\ref{lemmagoodext}.  

Let $\sigma = \prod_{i=1}^{n} \sigma_1^{a_i}\sigma_3^{b_i}\sigma_2^{c_i}$ be the minimal form of $\sigma$. An \textit{$s$-subproduct} of $\sigma$ is a subproduct of $\sigma$ that starts at the beginning of the minimal form of $\sigma$, and either ends with $\sigma_1^{a_p}\sigma_3^{b_p}$ for some $1\leq p\leq n$ if $s\in \{1,3\}$, or ends with $\sigma_2^{c_p}$ for some $1\leq p\leq n$ if $s=2$. In Section~\ref{mains}, we will establish Theorem~\ref{invariantdetect} by defining the property of \textit{strong $r$-regularity} of $s$-subproducts of $\sigma$, where $r\in \{1,2,3\}$ (Definition~\ref{defrregular}). The rough definition of strong $r$-regularity of an $s$-subproduct $P$ is that a maximal $q$-resistance $P$-path is a good $\left(r,s\right)$-type $P$-path and has sufficiently high $q$-resistance compared to other $P$-paths. We will establish two main statements on strong $r$-regularity. Firstly, strong $r$-regularity is an inductive property of $s$-subproducts of a normal braid $\sigma$. Secondly, if the maximal proper $s$-subproduct of $\sigma$ is strongly $r$-regular, then the leading coefficients of multiple entries in the $r$th row of $\beta_4\left(\sigma\right)$ are non-zero modulo each prime $t\neq 2$. (The intuition is that the $q$-degree of a bad $P$-path cannot ``catch-up" to that of a higher $q$-degree good $P$-path and create cancellation of weights, as we extend both $P$-paths to $\sigma$-paths.)

In Subsection~\ref{subsecconstraint}, we will introduce a new product decomposition of positive braids, the \textit{block-road decomposition} of a positive braid into \textit{blocks} and \textit{roads} (Definition~\ref{blockroad}). The rough definition is that the blocks in $\sigma$ are ``small subproducts" in the minimal form of $\sigma$ containing one side of a braid relation (either $\sigma_1\sigma_2\sigma_1$ or $\sigma_2\sigma_3\sigma_2$) and the roads are subproducts in the minimal form of $\sigma$, maximal with respect to the property of not overlapping with a block. We will observe that the condition of normality on a positive braid $\sigma$ can be rephrased naturally as the condition that every block in the block-road decomposition of $\sigma$ is a \textit{normal block} (Proposition~\ref{pnormalblock}). We will also classify blocks in $\sigma$ into \textit{$2$-blocks} (corresponding to the occurrence of $\sigma_1\sigma_2\sigma_1$ in the minimal form of $\sigma$) and \textit{$3$-blocks} (corresponding to the occurrence of $\sigma_2\sigma_3\sigma_2$ in the minimal form of $\sigma$). 

In Subsection~\ref{subsecstrongregularity}, we will establish the two main statements concerning the property of strong $r$-regularity of $s$-subproducts of a normal braid $\sigma$ with respect to the block-road decomposition of $\sigma$. Let $P\subseteq P'$ be such that $P$ is a strongly $r$-regular $s$-subproduct of $\sigma$ and $P'$ is an $s'$-subproduct of $\sigma$ for some $s,s'\in \{1,2,3\}$. If $P'\setminus P$ is a road, then we will show that the leading coefficients of multiple entries in the $r$th row of $\beta_4\left(P'\right)$ are non-zero modulo each prime $t\neq 2$, and if $P'\neq \sigma$, then $P'$ is strongly $r$-regular (Corollary~\ref{lemmaroadregular}). We will use the theory of path extensions developed in Subsection~\ref{subsecpathextensions} to establish Corollary~\ref{lemmaroadregular}. If $P'\setminus P$ is a normal block, then we will show that the leading coefficients of multiple entries in the $r$th row of $\beta_4\left(P'\right)$ are non-zero modulo each prime $t\neq 2$, and if $P'\neq \sigma$, then $P'$ is strongly $r$-regular (Corollary~\ref{lemmablockstronglyregular}). We will split the proof of Corollary~\ref{lemmablockstronglyregular} according to whether $P'\setminus P$ is a $2$-block or a $3$-block. 

In Section~\ref{smainstronger}, we will establish Theorem~\ref{mainstronger}, which is a stronger version of Theorem~\ref{main}. Theorem~\ref{mainstronger} is the generalization of Theorem~\ref{main} where ``normal braid" in the statement of Theorem~\ref{main} is replaced by ``weakly normal braid".
The paper is organized such that the proof of Theorem~\ref{mainstronger} is broken up into two steps. The first step is establishing Theorem~\ref{main} in Sections~\ref{def} -~\ref{mains} based on the theory of path extensions and the block-road decomposition in Sections~\ref{pathgraph} -~\ref{mains}. The second step is developing the theory further in Section~\ref{smainstronger} in order to establish Theorem~\ref{mainstronger}.

We introduce an elementary example (Example~\ref{exglobalcancellation}) at the beginning of Section~\ref{smainstronger} that captures the main ideas of Theorem~\ref{mainstronger}. In Subsection~\ref{subsecweaklynormalbraidbadsubroad}, we precisely define \textit{weakly normal braids} and state Theorem~\ref{mainstronger}. In Subsection~\ref{subsecpathextensionsII}, we extend the theory of path extensions developed in Subsection~\ref{subsecpathextensions} to study the more general class of weakly normal braids. In Subsection~\ref{subsecproofstrongerversion}, we define the property of \textit{$r$-regularity} for $s$-subproducts (Definition~\ref{defweakrregular}), which is more general than the property of strong $r$-regularity. We will establish two main statements on $r$-regularity for weakly normal braids analogous to the ones established on strong $r$-regularity for normal braids. We will establish that $r$-regularity is an inductive property of $s$-subproducts of a weakly normal braid $\sigma$, and use this to show that the leading coefficients of multiple entries in the $r$th row of $\beta_4\left(\sigma\right)$ are non-zero modulo each prime $t\neq 2,3$. Finally, this concludes the proof of Theorem~\ref{mainstronger}.

\subsection{Further work}
The author is working on refining the techniques in this paper and combining these techniques with other methods to establish that the Burau representation $\beta_4$ is faithful. We note that the techniques developed in this paper can be used to establish alternative constraints on the kernel of the Burau representation, and one such set of alternative constraints is established in the author's PhD thesis~\cite{datta2020burauthesis}.

The techniques in this paper are also broadly applicable to studying quantum representations of braid groups in general, such as the irreducible summands of the Jones representation of the braid group $B_n$, for arbitrary $n$. The author is generalizing the theory in this paper to study these representations. The author is also using the new characterization of Burau matrices of generic $4$-braids established in this paper to derive an explicit formula for the Jones polynomial of generic $4$-braid closures. The aforementioned two results will appear in forthcoming papers.

\subsection{Acknowledgements}
The author would like to express his sincere gratitude to Peter Ozsv{\'a}th and Zolt{\'a}n Szab{\'o} for their careful reading and checking of several parts of this paper, numerous helpful suggestions and clarifications of the exposition, and their encouragement and support during the completion of this work. The author would also like to thank Peter Sarnak for his interest in this work, helpful discussions, and encouragement and observation on applying the techniques in this paper to unsolved problems in other areas of mathematics. Finally, the author would like to thank Sophie Morel for her interest in this work and her support. The main ideas of this work were developed toward the end of the author's PhD in 2019 at Princeton University, and the theoretical foundations already appeared in the author's PhD thesis in 2019~\cite{datta2020burauthesis}.

\section{The Burau representation $\beta_4$ and the Garside normal form}
\label{def}

In Subsection~\ref{subsecgar}, we briefly summarize Garside's fundamental theorem for the braid groups in the case of $B_4$, and we also introduce new definitions and establish new results. Definition~\ref{garter} and Theorem~\ref{Garside} are adopted from Garside's foundational paper~\cite{garside1969braid} in the case of $B_4$. In Lemma~\ref{minimalform}, we establish a new characterization of the minimal form of a positive braid (which is referred to as the ``base of the diagram of the positive braid" in~\cite{garside1969braid}). Definition~\ref{isp} introduces a new product decomposition of the minimal form of a positive braid that will be important in the sequel. In Subsection~\ref{subsecreduction}, we reduce proving Theorem~\ref{main} (the main result of this paper) to proving Theorem~\ref{invariantdetect}. The subsequent Sections~\ref{pathgraph} - \ref{mains} of the paper will be devoted to proving Theorem~\ref{invariantdetect}. 
 
In this paper, the convention is that words in the generators are read from left to right. 

\subsection{The minimal form of a positive braid and the Garside normal form}
\label{subsecgar}

We collect the terminology that we adopt from~\cite{garside1969braid} in the following Definition~\ref{garter} for the convenience of the reader. The notions of a positive braid and the minimal form of a positive braid were defined in the Introduction, but here we also indicate why the minimal form of a positive braid is well-defined (a result in~\cite{garside1969braid}).

\begin{definition}
\label{garter}
A \textit{positive braid} is a braid in the submonoid of $B_4$ generated by $\sigma_1,\sigma_2,\sigma_3$. We will denote the submonoid of positive braids by $B_4^{+}$. A \textit{product expansion} of a positive braid is an expression of the braid as a product of positive powers of the generators. The \textit{length} of a positive braid is the number of generators in a product expansion of the braid. (In~\cite{garside1969braid}, it is proven that the length of a positive braid is independent of the choice of product expansion, and thus it is well-defined.) 

The set of product expansions of a positive braid is a finite set, since all have the same length. We can totally order this set using the dictionary order with respect to the index sequence and we define the \textit{minimal form} of the positive braid to be the product expansion minimal with respect to this order.  
\end{definition}

For example, the minimal form of $\sigma_2\sigma_1\sigma_2$ is $\sigma_1\sigma_2\sigma_1$ and the minimal form of $\sigma_3\sigma_1$ is $\sigma_1\sigma_3$ (consistent with our convention that words in the generators are read from left to right). We establish the following new characterization of the minimal form of a positive braid.

\begin{lemma}
\label{minimalform}
Let $\sigma=\prod_{i=1}^{n} \sigma_1^{a_i}\sigma_3^{b_i}\sigma_2^{c_i}$ be a product expansion of the positive braid $\sigma$ where $a_i+b_i$ is a positive integer for $2\leq i\leq n$ and $c_i$ is a positive integer for $1\leq i\leq n-1$. If the product expansion is in minimal form, then the following conditions are satisfied:
\begin{description}
\item[(i)] If $a_p=1$ and $p>1$, then $b_p>0$, except possibly if $p=n$ and $c_p=0$. 
\item[(ii)] If $c_p=1$, then either $b_p=0$ or $b_{p+1}=0$. If, in addition, $p>1$ and $b_p=0$, then $a_{p+1}=0$. 
\item[(iii)] If $a_p=0$, $b_p=1$ and $p>1$, then $b_{p-1}=0$, except possibly if $p=n$ and $c_p=0$. If, in addition, $p>2$ and $c_{p-2}=1$, then $p=3$ and $b_1=0$, except possibly if $p=n$ and $c_p=0$.
\item[(iv)] If $b_p=1=c_p$, then $b_i=0$ for $i<p$, except possibly if $p=n$.
\end{description}
\end{lemma}

We include Remark~\ref{exmform} to informally visualize the statement of Lemma~\ref{minimalform} before the proof.

\begin{remark}
\label{exmform}
Lemma~\ref{minimalform} \textbf{(i)} considers the existence of an isolated $\sigma_1$ and states that $\sigma_2\sigma_1\sigma_2$ does not occur in the minimal form of $\sigma$. Lemma~\ref{minimalform} \textbf{(ii)} considers the existence of an isolated $\sigma_2$ and states that $\sigma_3\sigma_2\sigma_1^{a_{p+1}}\sigma_3^{b_{p+1}}$ does not occur if $b_{p+1}>0$ in any case and $\sigma_1^{a_p}\sigma_2\sigma_1^{a_{p+1}}\sigma_3^{b_{p+1}}$ does not occur if $a_{p+1}>0$ unless it is the beginning of the minimal form of $\sigma$. Lemma~\ref{minimalform} \textbf{(iii)} considers the existence of an isolated $\sigma_3$ and states that neither $\sigma_3\sigma_2^{c_{p-1}}\sigma_3\sigma_2^{c_p}$ nor $\sigma_3\sigma_2\sigma_1^{a_{p-1}}\sigma_2^{c_{p-1}}\sigma_3\sigma_2^{c_p}$ occurs. Finally, Lemma~\ref{minimalform} \textbf{(iv)} states that $\sigma_2^{c_{p-1}}\sigma_1^{a_p}\sigma_3\sigma_2$ with $c_p=1$ does not occur ($a_p=0$ is possible) unless it is either the ending of the minimal form of $\sigma$ or there is no $\sigma_3$ preceding it in the minimal form of $\sigma$.
\end{remark}

\begin{proof}[Proof of Lemma~\ref{minimalform}]
In \textbf{(i)}, an application of the braid relation $\sigma_2\sigma_1\sigma_2=\sigma_1\sigma_2\sigma_1$ shows that the product expansion is not in minimal form if $b_p=0$ and $c_{p-1},c_p>0$.

In \textbf{(ii)}, $a_{p+1}$ applications of the commutativity relation $\sigma_1\sigma_3=\sigma_3\sigma_1$ and an application of the braid relation $\sigma_3\sigma_2\sigma_3=\sigma_2\sigma_3\sigma_2$ shows that the product expansion is not in minimal form if $b_p,b_{p+1}>0$. If $p>1$ and $b_p=0$, then $a_p$ applications of the braid relation $\sigma_1\sigma_2\sigma_1=\sigma_2\sigma_1\sigma_2$ shows that the product expansion is not in minimal form since $c_{p-1}>0$. 

In \textbf{(iii)}, $c_{p-1}$ applications of the braid relation $\sigma_2\sigma_3\sigma_2=\sigma_3\sigma_2\sigma_3$ shows that the product expansion is not in minimal form if $b_{p-1},c_p>0$. If $p>2$ and $c_{p-2}=1$, then $c_{p-1}$ applications of the braid relation $\sigma_2\sigma_3\sigma_2=\sigma_3\sigma_2\sigma_3$, $a_{p-1}$ applications of the commutativity relation $\sigma_1\sigma_3=\sigma_3\sigma_1$, and an application of the braid relation $\sigma_2\sigma_3\sigma_2=\sigma_3\sigma_2\sigma_3$ shows that the product expansion is not in minimal form if $b_{p-2},c_p>0$. Finally, \textbf{(ii)} implies that $p=3$ and $b_1=0$. 

In \textbf{(iv)}, if $p<n$, then $a_{p+1}>0$ by \textbf{(ii)}. Let us consider a general product expansion $\prod_{i=1}^{n'} \sigma_1^{a_i'}\sigma_3^{b_i'}\sigma_2^{c_i'}$ of $\sigma$ with $b_r'=1=c_r'$ and $a_{r+1}'>0$. We consider two cases depending on whether $a_r'>0$ or $a_r' = 0$. If $a_r'>0$, then consider $a_r'$ applications of the commutativity relation $\sigma_1\sigma_3=\sigma_3\sigma_1$ and $a_r'$ applications of the braid relation $\sigma_1\sigma_2\sigma_1=\sigma_2\sigma_1\sigma_2$ to obtain a new product expansion $\prod_{i=1}^{n'} \sigma_1^{a_i''}\sigma_3^{b_i''}\sigma_2^{c_i''}$ of $\sigma$ with $b_r''=1=c_r''$, $a_r''=0$, $a_{r+1}''=1>0$, $a_i''=a_i'$, $b_i''=b_i'$, and $c_i''=c_i'$ for $i\leq r-1$. If $r>1$, $a_r'=0$ and $b_{r-1}'=0$, then consider $c_{r-1}'$ applications of the braid relation $\sigma_2\sigma_3\sigma_2=\sigma_3\sigma_2\sigma_3$ to obtain a new product expansion $\prod_{i=1}^{n'-1} \sigma_1^{a_i''}\sigma_3^{b_i''}\sigma_2^{c_i''}$ of $\sigma$ with $b_{r-1}''=1=c_{r-1}''$, $a_{r-1}''=a_{r-1}'>0$, $a_{r}''=a_{r+1}'>0$, $a_i''=a_i'$, $b_i''=b_i'$, and $c_i''=c_i'$ for $i\leq r-2$. 

Let us now consider the original product expansion of $\sigma$. Let $j<p$ be maximal with respect to the property that $b_j>0$, if it exists. A finite number of iterations of the process in the previous paragraph (each iteration depending on whether $a_r'>0$ or $a_r'=0$) shows that there is a product expansion $\prod_{i=1}^{n'} \sigma_1^{a_i'}\sigma_3^{b_i'}\sigma_2^{c_i'}$ of $\sigma$ with $b_{j+1}'=1=c_{j+1}'$, $a_{j+1}'=0$, $a_i'=a_i$, $b_i'=b_i$, and $c_i'=c_i$ for $i\leq j$. Finally, $c_j$ applications of the braid relation $\sigma_2\sigma_3\sigma_2=\sigma_3\sigma_2\sigma_3$ shows that the original product expansion of $\sigma$ is not in minimal form.
\end{proof}

We will observe later (Lemma~\ref{reduction}) that the faithfulness of $\beta_4:B_4\to \text{GL}_3\left(\mathbb{Z}\left[q^{\pm 1}\right]\right)$ is equivalent to the faithfulness of the restriction of $\beta_4$ to the submonoid of positive braids $B_4^{+}$. In general, we can determine the Burau matrix of a positive braid $\sigma$ in terms of a product expansion of $\sigma$ using Definition~\ref{buraudef} (the Burau matrices of the Artin generators). However, each positive braid has a unique minimal form. Therefore, the characterization of minimal product expansions in Lemma~\ref{minimalform} will allow us simply to determine the Burau matrices of minimal product expansions rather than all product expansions. (Lemma~\ref{minimalform} furnishes a solution to the word problem in the monoid $B_4^{+}$ with respect to the Artin presentation.)

\begin{definition}
\label{isp}
Let $\sigma=\prod_{i=1}^{n} \sigma_1^{a_i}\sigma_3^{b_i}\sigma_2^{c_i}$ be the minimal form of the positive braid $\sigma$ where $a_i+b_i$ is a positive integer for $2\leq i\leq n$ and $c_i$ is a positive integer for $1\leq i\leq n-1$. An \textit{isolated $\sigma_2$ subproduct} is a subproduct $\left(\prod_{i=p}^{p'-1} \sigma_1^{a_i}\sigma_3^{b_i}\sigma_2^{c_i}\right)\sigma_1^{a_{p'}}\sigma_3^{b_{p'}}$ of the minimal form of $\sigma$, maximal with respect to the property that $c_i=1$ for all $p\leq i\leq p'-1$.

\begin{description}
\item[(a)] A \textit{type I isolated $\sigma_2$ subproduct} is an isolated $\sigma_2$ subproduct of the form \begin{equation} \tag{I} \sigma_1^{a_p}\sigma_3^{b_p}\sigma_2\sigma_1^{a_{p+1}}\sigma_2\sigma_3^{b_{p+2}}\sigma_2\cdots \end{equation} with $b_p>0$, and a \textit{type II isolated $\sigma_2$ subproduct} is an isolated $\sigma_2$ subproduct of the form \begin{equation} \tag{II} \sigma_1^{a_p}\sigma_2\sigma_3^{b_{p+1}}\sigma_2\sigma_1^{a_{p+2}}\sigma_2\cdots \end{equation} (the existence or nonexistence of a $\sigma_1$ or a $\sigma_3$ between consecutive $\sigma_2$s alternates after the first $\sigma_2$ in both cases).
\item[(b)] Let $p=1$. A \textit{type III isolated $\sigma_2$ subproduct} is \begin{equation} \tag{III} \sigma_1^{a_1}\sigma_2\sigma_1^{a_2}\sigma_3^{b_2}\sigma_2\sigma_1^{a_3}\sigma_2\sigma_3^{b_4}\sigma_2\cdots \end{equation} (the existence or nonexistence of a $\sigma_1$ or a $\sigma_3$ between consecutive $\sigma_2$s alternates after the \textit{second} $\sigma_2$) with $a_2>0$.
\end{description}
Finally, a \textit{non-isolated $\sigma_2$ subproduct} is a subproduct of the minimal form of $\sigma$, maximal with respect to the property that it does not overlap with an isolated $\sigma_2$ subproduct.
\end{definition}

The product decomposition of a positive braid into isolated $\sigma_2$ subproducts and non-isolated $\sigma_2$ subproducts will be convenient in terms of characterizing the minimal form of a positive braid. Lemma~\ref{minimalform} implies certain constraints on isolated $\sigma_2$ subproducts.

\begin{corollary}
We adopt the notation of Definition~\ref{isp}. 
\begin{description}
\item[(a)] If $p>1$, then an isolated $\sigma_2$ subproduct is either a type I isolated $\sigma_2$ subproduct or a type II isolated $\sigma_2$ subproduct. If $p = 1$, then an isolated $\sigma_2$ subproduct is either a type I isolated $\sigma_2$ subproduct, a type II isolated $\sigma_2$ subproduct, or a type III isolated $\sigma_2$ subproduct.
\item[(b)] If $p<i\leq p'$ and $i-p$ is odd, then $a_i\geq 2$ in a type I isolated $\sigma_2$ subproduct, and if $p\leq i\leq p'$ and $i-p$ is even, then $a_i\geq 2$ in a type II isolated $\sigma_2$ subproduct, except possibly if $i=p'=n$ and $c_i=0$. 
\item[(c)] Let $p<i\leq p'$. If $i-p$ is even, then $b_i\geq 2$ in a type I isolated $\sigma_2$ subproduct, and if $i-p$ is odd, then $b_i\geq 2$ in a type II isolated $\sigma_2$ subproduct, except possibly if $i=p'=n$ and $c_i = 0$ in either case, or $i = p+1$ in the second case. 
\item[(d)] Let $2<i\leq p'$. If $i$ is odd, then $a_i\geq 2$ in a type III isolated $\sigma_2$ subproduct, and if $i$ is even, then $b_i\geq 2$ in a type III isolated $\sigma_2$ subproduct, except possibly if $i = p' = n$ and $c_i = 0$.  
\end{description}
\end{corollary}
\begin{proof}
\begin{description}
\item[(a)] The statement is a consequence of Lemma~\ref{minimalform} \textbf{(ii)}.
\item[(b)] The statement is a consequence of Lemma~\ref{minimalform} \textbf{(i)}
\item[(c)] The statement is a consequence of Lemma~\ref{minimalform} \textbf{(iii)}.
\item[(d)] The statement is a consequence of Lemma~\ref{minimalform} \textbf{(i)} and \textbf{(iii)}.
\end{description}
\end{proof}

We recall \textit{Garside's theorem}~\cite{garside1969braid} in the special case of $B_4$.

\begin{theorem}[{\cite{garside1969braid}}]
\label{Garside}
 The \textit{Garside element} of $B_4$ is $\Delta = \left(\sigma_1\sigma_2\sigma_3\right)\left(\sigma_1\sigma_2\right)\sigma_1$. If $g\in B_4$, then $g$ can be expressed uniquely as a product $\Delta^{k}\sigma$ where $k\in \mathbb{Z}$ and $\sigma$ is a positive braid indivisible by $\Delta$ (i.e., $\sigma\neq \Delta\tau$ for any positive braid $\tau$). We refer to this expression as the \textit{Garside normal form} of $g$.
\end{theorem}

The Garside normal form furnishes a solution to the word problem in the braid groups with respect to the Artin presentation. Indeed, there is an algorithm in~\cite{garside1969braid} to determine the Garside normal form of a braid as the output with a product expansion of the braid in the Artin generators as the input. However, we will not refer to this algorithm in the sequel. The main relevance of Theorem~\ref{Garside} for this paper is that, combined with the notion of the minimal form of a positive braid, there is a canonical product expansion of every braid. We will characterize the Burau matrices of these canonical product expansions (instead of characterizing the Burau matrices of all product expansions, which necessitates doing extra unnecessary work).

We will also require the following statement concerning the images of the powers of the Garside element under the Burau representation. 

\begin{lemma}
\label{Garsideaction}

If $k\in\mathbb{Z}$, then \[ \beta_4\left(\Delta^k\right) =        
\begin{cases}  \left( \begin{array}{ccc}
q^{2k} & 0 & 0 \\
0 & q^{2k} & 0 \\
0 & 0 & q^{2k} \end{array}\right) & \text{if }k\text{ is even} \\
         \left( \begin{array}{ccc}
0 & -q^{2k+1} & 0 \\
-q^{2k-1} & 0 & 0 \\
0 & 0 & -q^{2k} \end{array}\right) & \text{if }k\text{ is odd}.
   \end{cases}
\]
\end{lemma}
\begin{proof}
The statement follows by induction on $k$ and the computation of $\beta_4\left(\Delta\right)$. 
\end{proof}

\subsection{Reduction of the main result to Theorem~\ref{invariantdetect}}
\label{subsecreduction}

We now reduce proving Theorem~\ref{main} to proving Theorem~\ref{invariantdetect} below. The reduction is a straightforward corollary of Garside's theorem.

\begin{lemma}
\label{reduction}
Let $t$ be a prime number. If $\Delta^{-k}\sigma$ is the Garside normal form of a braid in $B_4$, then $\Delta^{-k}\sigma\in\text{ker}\left(\left(\beta_4\right)_{t}\right)$ if and only if $\left(\beta_4\right)_{t}\left(\sigma\right)=\left(\beta_4\right)_{t}\left(\Delta^{k}\right)$ and $k\geq 0$.
\end{lemma}
\begin{proof}
We can assume $k\geq 0$ since the exponents of the polynomials in the entries of $\left(\beta_4\right)_{t}\left(\Delta^{k}\right)$ are negative for $k<0$ by Lemma~\ref{Garsideaction}.
\end{proof}

\begin{theorem}
\label{invariantdetect}
Let $t\neq 2$ be a prime number. If $\sigma$ is a normal braid indivisible by $\Delta$, then the leading coefficients of multiple entries in at least one row of $\beta_4\left(\sigma\right)$ are non-zero modulo $t$. 
\end{theorem}
 
Theorem~\ref{main} of the Introduction is a consequence of Theorem~\ref{invariantdetect}.

\begin{proof}[Proof of Theorem~\ref{main}]
Let $\sigma$ be a normal braid indivisible by $\Delta$. Lemma~\ref{reduction} reduces the first statement in Theorem~\ref{main} to establishing that $\left(\left(\beta_4\right)_{t}\right)\left(\sigma\right)\neq \left(\left(\beta_4\right)_{t}\right)\left(\Delta^k\right)$ for every integer $k\geq 0$. The latter statement is a consequence of Lemma~\ref{Garsideaction} and Theorem~\ref{invariantdetect}. 

We recall the \textit{rational roots theorem}, which states that if $\frac{a}{b}\in \mathbb{Q}$ (in lowest terms) is a root of a polynomial $f\left(q\right)$ with integer coefficients, then $b$ divides the leading coefficient of $f$. Theorem~\ref{invariantdetect} and the rational roots theorem imply the second statement in Theorem~\ref{main}. 
\end{proof} 

The following Sections \ref{pathgraph} - \ref{mains} will be devoted to proving Theorem~\ref{invariantdetect}. The outline of the proof is as follows (see Subsection~\ref{subsecoutline} of the Introduction for a more detailed outline). In Section~\ref{pathgraph}, we will introduce a new interpretation of the entries of $\beta_4\left(\sigma\right)$ for a positive braid $\sigma$ as weighted numbers of admissible $\sigma$-paths (Proposition~\ref{m=ap}). Note that admissible $\sigma$-paths are defined concretely in terms of the minimal form of $\sigma$. In Section~\ref{mains}, we will construct admissible $\sigma$-paths with special properties if $\sigma$ is a normal braid indivisible by $\Delta$. The weights of these $\sigma$-paths will contribute the non-zero leading coefficients of the entries in a row of $\beta_4\left(\sigma\right)$ modulo $t$, and this will establish Theorem~\ref{invariantdetect}.

\section{A new interpretation of $\beta_4\left(\sigma\right)$ in terms of the minimal form of a positive braid $\sigma$}
\label{pathgraph}

In this section, we will introduce a novel interpretation of the entries of the Burau matrix $\beta_4\left(\sigma\right)$ of a positive braid $\sigma$ in terms of weighted counts of \textit{$\sigma$-paths}. In Subsection~\ref{subsecspath}, we will define the notion of a $\sigma$-path, which is a path in the three vertex straight-line graph compatible with the minimal form of $\sigma$ in a suitable sense. We will also define the \textit{weight of a $\sigma$-path} as a signed monomial in $q$. We will establish Proposition~\ref{m=p} which states that the entries in the Burau matrix $\beta_4\left(\sigma\right)$ are weighted counts of $\sigma$-paths. 

In Subsection~\ref{subsecadspath}, we will demonstrate that the weights of certain $\sigma$-paths cancel in pairs, and this cancellation admits a beautiful geometric interpretation. Indeed, in Definition~\ref{dispair}, we define five different types of \textit{distinguished pairs} of $\sigma$-paths. We show that the $\sigma$-paths in a distinguished pair bound local triangles or trapeziums. We define \textit{admissible $\sigma$-paths} to be $\sigma$-paths that do not belong to any distinguished pair. In Proposition~\ref{m=ap}, we refine Proposition~\ref{m=p} and establish that the entries of the Burau matrix $\beta_4\left(\sigma\right)$ are weighted counts of \textit{admissible} $\sigma$-paths. The proof requires a careful verification that distinct distinguished pairs do not overlap. 

We will construct admissible $\sigma$-paths by constructing $\sigma$-paths that avoid local triangles and trapeziums. Although the weights of such admissible $\sigma$-paths do not locally cancel, they may \textit{a priori} globally cancel in a different manner. In Subsection~\ref{subsecgb}, we will develop a theory to construct admissible $\sigma$-paths with weights that do not cancel at all. In Definition~\ref{defgoodbadpath}, we introduce the notion of \textit{good and bad paths}. A $\sigma$-path is constructed as the minimal form of $\sigma$ is read from left to right. If we read the minimal form of a positive braid, then good paths accumulate $q$-degree quickly and bad paths accumulate $q$-degree slowly. In Subsection~\ref{subsecpathextensions}, we will establish this formally in a sequence of statements.

The main utility of the theory of $\sigma$-paths is that one can ``see" entries of the Burau matrix $\beta_4\left(\sigma\right)$ quickly (in linear time) even for long positive braids $\sigma$. In particular, admissible $\sigma$-paths provide a powerful mental model for thinking about the Burau matrix $\beta_4\left(\sigma\right)$, and we will illustrate this in Example~\ref{exmain} in Subsection~\ref{subsecgb}. 

In this paper, $\sigma$ will always denote a positive braid. 

\subsection{$\sigma$-paths}
\label{subsecspath}

In this subsection, we introduce the notion of a $\sigma$-path (Definition~\ref{spath}) and establish a new interpretation of the Burau matrix $\beta_4\left(\sigma\right)$ in terms of $\sigma$-paths (Proposition~\ref{m=p}).

\begin{definition}
\label{spath}
Let $G$ be the three vertex straight-line graph with vertices labelled as $\{1,2,3\}$ and the middle (valence two) vertex labelled as $2$. We will denote a path in $G$ by its sequence of vertices $v_0,\dots,v_{l}$. We define an \textit{$\left(r,s\right)$-type $\sigma$-path} to be a path in $G$, $v_0,\dots,v_l$, with the following properties:
\begin{description}
\item[(i)] $l$ is the length of the positive braid $\sigma$.
\item[(ii)] $v_0=r$ and $v_l=s$ (as vertices of $G$).
\item[(iii)] If the $k$th generator in the minimal form of $\sigma$ is $\sigma_i$, then either $v_k=v_{k-1}$ or $v_k=i$ is adjacent to $v_{k-1}$ in $G$; we refer to the latter as a \textit{vertex change} at $\sigma_i$.
\end{description}
We define the \textit{weight of a $\sigma$-path} to be the signed monomial in $q$ determined as follows:
\begin{description}
\item[(i)] The $q$-degree (which we also refer to as the \textit{$q$-resistance of the $\sigma$-path}) is the number of $k$ for which either $v_k=i=v_{k-1}$ and the $k$th generator (in the minimal form of $\sigma$) is $\sigma_i$, or there is an $i+1\to i$ vertex change at the $k$th generator.
\item[(ii)] The sign is $\left(-1\right)^{e}$ where $e$ is the number of $k$ for which there is an $i+1\to i$ vertex change at the $k$th generator. 
\end{description}
\end{definition}

If context makes $\sigma$ clear, then we will sometimes omit reference to $\sigma$ and refer to $\sigma$-paths as simply paths. The entries of the Burau matrix $\beta_4\left(\sigma\right)$ are polynomials in $q$, and in Proposition~\ref{m=p}, we will show that we can determine $\beta_4\left(\sigma\right)$ by counting $\sigma$-paths (with weights). Figure~\ref{fig1} illustrates Definition~\ref{spath}.

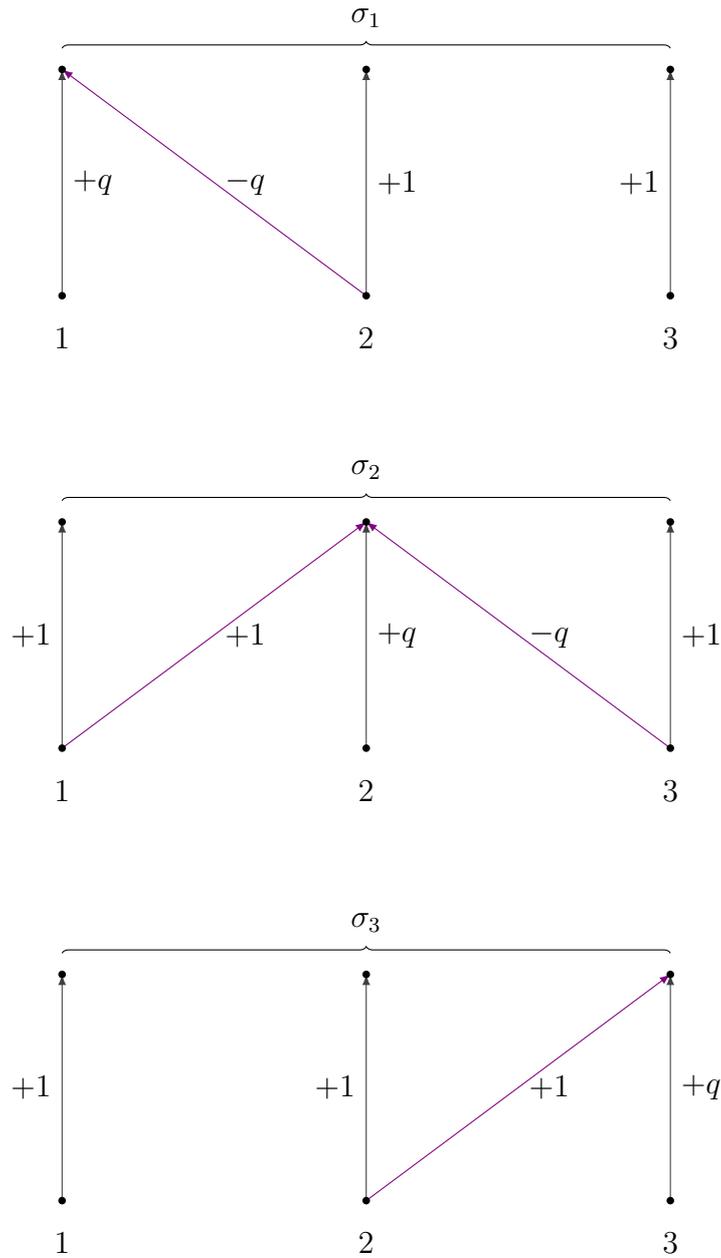
\begin{figure}[H]
\centering
\begin{tikzpicture}[>=latex]
\draw[->][color=darkgray][solid] (0,0) -- (0,3) node[midway, left][color=black]{$+1$};
\draw[->][color=violet][solid] (4,0) -- (8,3) node[midway, above, right][color=black]{$+1$};
\draw[->][color=darkgray][solid] (4,0) -- (4,3) node[midway, left][color=black]{$+1$} ;
\draw[->][color=darkgray][solid] (8,0) -- (8,3) node[midway, right][color=black]{$+q$};
\draw[->][color=darkgray][solid] (0,6) -- (0,9) node[midway, left][color=black]{$+1$};
\draw[->][color=violet][solid] (0,6) -- (4,9) node[midway, below, right][color=black]{$+1$};
\draw[->][color=darkgray][solid] (4,6) -- (4,9) node[midway, right][color=black]{$+q$} ;
\draw[->][color=violet][solid] (8,6) -- (4,9) node[midway, above, right][color=black]{$-q$};
\draw[->][color=darkgray][solid] (8,6) -- (8,9) node[midway, right][color=black]{$+1$};
\draw[->][color=darkgray][solid] (0,12) -- (0,15) node[midway, right][color=black]{$+q$};
\draw[->][color=violet][solid] (4,12) -- (0,15) node[midway, above, right][color=black]{$-q$} ;
\draw[->][color=darkgray][solid] (4,12) -- (4,15) node[midway, right][color=black]{$+1$};
\draw[->][color=darkgray][solid] (8,12) -- (8,15) node[midway, left][color=black]{$+1$} ;
\draw[decoration={brace,raise=8pt},decorate] (0,3) -- (8,3) node[midway,above=12pt][color=black]{$\sigma_3$};
\draw[decoration={brace,raise=8pt},decorate] (0,9) -- (8,9) node[midway,above=12pt][color=black]{$\sigma_2$};
\draw[decoration={brace,raise=8pt},decorate] (0,15) -- (8,15) node[midway,above=12pt][color=black]{$\sigma_1$};
\draw (0,0) node[below=8pt]{$1$};
\draw (4,0) node[below=8pt]{$2$};
\draw (8,0) node[below=8pt]{$3$};
\draw (0,6) node[below=8pt]{$1$};
\draw (4,6) node[below=8pt]{$2$};
\draw (8,6) node[below=8pt]{$3$};
\draw (0,12) node[below=8pt]{$1$};
\draw (4,12) node[below=8pt]{$2$};
\draw (8,12) node[below=8pt]{$3$};
 \foreach \x in {0,4,8} {
        \foreach \y in {0,3} {
            \fill[color=black] (\x,\y) circle (0.05);
        }
    }
     \foreach \x in {0,4,8} {
        \foreach \y in {6,9} {
            \fill[color=black] (\x,\y) circle (0.05);
        }
    }    
     \foreach \x in {0,4,8} {
        \foreach \y in {12,15} {
            \fill[color=black] (\x,\y) circle (0.05);
        }
    }\end{tikzpicture}
\caption{We illustrate Definition~\ref{spath}. We have indicated three pictures corresponding to the three generators $\sigma_1$, $\sigma_2$, $\sigma_3$, labelled at the top of each picture. The vertices of the graph $G$ are labelled at the bottom of each picture. The pictures indicate the possible edges of a $\sigma$-path from $v_{k-1}$ to $v_k$ according to which of $\sigma_1$, $\sigma_2$, $\sigma_3$ is the $k$th generator in the minimal form of $\sigma$. (A specific $\sigma$-path can only follow one edge in the picture corresponding to the $k$th generator in the minimal form of $\sigma$.) The vertical dark gray edges denote the cases $v_{k-1}=v_{k}$ and the diagonal violet edges denote vertex changes ($v_{k-1}\neq v_k$). The labels on the edges denote the multiplicative factor by which the weight of the $\sigma$-path changes, if the $\sigma$-path follows that edge.}
\label{fig1}
\end{figure}

We justify the introduction of $\sigma$-paths in relevance to the Burau representation $\beta_4$.

\begin{proposition}
\label{m=p}
The $\left(r,s\right)$-entry of $\beta_4\left(\sigma\right)$ is the weighted number of the $\left(r,s\right)$-type $\sigma$-paths.\end{proposition}
\begin{proof}
The $\left(r,s\right)$-entry of $\beta_4\left(\sigma_i\right)$ is either equal to the label on the arrow from $r\to s$ in the picture corresponding to $\sigma_i$ in Figure~\ref{fig1} if the arrow exists, or equal to zero if the arrow does not exist. The statement now follows from Definition~\ref{spath} and the definition of matrix multiplication.
\end{proof}
 
If we replace the minimal form of $\sigma$ by any product expansion of $\sigma$ in Definition~\ref{spath}, then Proposition~\ref{m=p} is still true. However, since each positive braid has a unique product expansion in minimal form, we do not require this generality.

\subsection{Admissible $\sigma$-paths}
\label{subsecadspath}

An important point in the statement of Proposition~\ref{m=p} is that some summands in the $\left(r,s\right)$-entry of $\beta_4\left(\sigma\right)$ (corresponding to weights of $\left(r,s\right)$-type $\sigma$-paths) cancel in the weighted sum but what remains is the $\left(r,s\right)$-entry of $\beta_4\left(\sigma\right)$. In other words, there is not a one-to-one correspondence between the (non-cancelling) terms in the $\left(r,s\right)$-entry of $\beta_4\left(\sigma\right)$ and the $\left(r,s\right)$-type $\sigma$-paths. (For example, Lemma~\ref{Garsideaction} implies that $\beta_4\left(\Delta^{2n}\right)$ is a diagonal matrix, although there are a plethora of $\Delta^{2n}$-paths if $n$ is large.)

The beautiful observation is that the weights of certain $\sigma$-paths \textit{cancel in pairs}. Moreover, the cancelling pairs of paths locally bound triangles and trapeziums. In particular, the cancellation of the weights of the paths can be illustrated geometrically.

The next step is to carefully group these $\sigma$-paths with cancelling weights in \textit{distinguished pairs}, such that distinct pairs do not overlap. We will refer to $\sigma$-paths that do not belong to any distinguished pair as \textit{admissible $\sigma$-paths}. Afterwards, in Proposition~\ref{m=ap}, we will refine Proposition~\ref{m=p} by establishing that the entries of the Burau matrix $\beta_4\left(\sigma\right)$ are weighted numbers of admissible $\sigma$-paths.

We now formally state and geometrically illustrate the distinguished pairs of $\sigma$-paths. Note that in the figures below, we will draw paths by composing the pictures for the generators in Figure~\ref{fig1} from bottom to top (also recall that words in the generators are read from left to right). 

\begin{definition}
\label{dispair}
Let $\sigma=\prod_{i=1}^{n} \sigma_1^{a_i}\sigma_3^{b_i}\sigma_2^{c_i}$ be the minimal form of $\sigma$. Let $1\leq p\leq n$ be fixed and let the last $\sigma_2$ in $\sigma_2^{c_p}$ be the $k$th generator in this product expansion (read from left to right). We define the following \textit{distinguished pairs} of $\sigma$-paths and observe that the weights of the $\sigma$-paths in each distinguished pair cancel. See also the figures for accompanying illustrations of special cases of distinguished pairs (the exponents of the generators are specialized to small values to permit an illustration). We will define $\delta$, $\epsilon$, $\zeta$, $\eta$, and $\theta$ distinguished pairs.

\begin{description}
\item[(i)] (\textit{$\delta$ distinguished pairs}) Let $a_{p+1}>0$. If $v_{k-1} = 1 = v_{k+1}$, then the paths $\delta_{k,1}$ and $\delta_{k,2}$ with $v_k = 1$ and $v_k = 2$, respectively, constitute a distinguished pair, unless $c_p=1$ and $v_{k-b_p-2}=2$. The other vertices match for $\delta_{k,1}$ and $\delta_{k,2}$, but are arbitrary, subject only to the conditions of a $\sigma$-path. Note that the weights of $\delta_{k,1}$ and $\delta_{k,2}$ have opposite sign. 

Figure~\ref{fig2} illustrates the local picture of this distinguished pair in the case $c_p\geq 2$. Figure~\ref{fig3} illustrates the local picture of the non-example, where all the conditions of being a distinguished pair in this paragraph are satisfied, except it is the case that $c_p=1$ and $v_{k-b_p-2}=2$. (If $c_p = 1$ and $v_{k-b_p-2} = 2$, then we do not consider $\delta_{k,1}$ and $\delta_{k,2}$ a distinguished pair because the local picture of $\delta_{k,2}$ would have an edge in common with the local picture of an $\eta$ distinguished pair (defined later).) \\

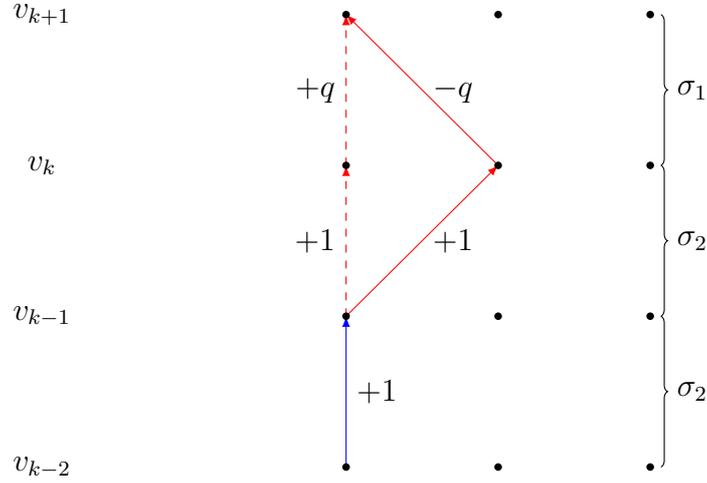
\begin{figure}[H]
\centering
\begin{tikzpicture}[>=latex]
\draw[->][color=blue][solid] (0,0) -- (0,2) node[midway, below, right][color=black]{$+1$};
\draw[->][color=red][solid] (0,2) -- (2,4) node[midway, below, right][color=black]{$+1$};
\draw[->][color=red][solid] (2,4) -- (0,6) node[midway, above, right][color=black]{$-q$} ;
\draw[->][color=red][dashed] (0,2) -- (0,4) node[midway, left][color=black]{$+1$};
\draw[->][color=red][dashed] (0,4) -- (0,6) node[midway, left][color=black]{$+q$} ;
\draw[decoration={brace,mirror,raise=4pt},decorate] (4,0) -- (4,2) node[midway,right=6pt][color=black]{$\sigma_2$};
\draw[decoration={brace,mirror,raise=4pt},decorate] (4,2) -- (4,4) node[midway,right=6pt][color=black]{$\sigma_2$};
\draw[decoration={brace,mirror,raise=4pt},decorate] (4,4) -- (4,6) node[midway,right=6pt][color=black]{$\sigma_1$};
 \foreach \x in {0,2,4} {
        \foreach \y in {0,2,4,6} {
            \fill[color=black] (\x,\y) circle (0.05);
        }
    }
    \node[] at (-4,0) {$v_{k-2}$};
    \node[] at (-4,2) {$v_{k-1}$};
    \node[] at (-4,4) {$v_k$};
    \node[] at (-4,6) {$v_{k+1}$};
\end{tikzpicture}
\caption{The solid and broken red paths (augmented by the blue edge) constitute the local pictures of a $\delta$ distinguished pair $\{\delta_{k,1},\delta_{k,2}\}$. The weights of the local pictures are $\pm q$. In relevance to the text directly above this figure, this is the case $c_p\geq 2$.}
\label{fig2}
\end{figure}

\begin{figure}[H]
\centering
\begin{tikzpicture}[>=latex]
\draw[->][color=blue][solid] (2,0) -- (0,2) node[midway, below, right][color=black]{$-q$};
\draw[->][color=blue][solid] (0,2) -- (0,4) node[midway, below, right][color=black]{$+1$};
\draw[->][color=brown][solid] (0,4) -- (2,6) node[midway, above, right][color=black]{$+1$} ;
\draw[->][color=brown][solid] (2,6) -- (0,8) node[midway, above, right][color=black]{$-q$};
\draw[->][color=brown][dashed] (0,4) -- (0,6) node[midway, left][color=black]{$+1$} ;
\draw[->][color=brown][dashed] (0,6) -- (0,8) node[midway, left][color=black]{$+q$} ;
\draw[decoration={brace,mirror,raise=4pt},decorate] (4,0) -- (4,2) node[midway,right=6pt][color=black]{$\sigma_1$};
\draw[decoration={brace,mirror,raise=4pt},decorate] (4,2) -- (4,4) node[midway,right=6pt][color=black]{$\sigma_3$};
\draw[decoration={brace,mirror,raise=4pt},decorate] (4,4) -- (4,6) node[midway,right=6pt][color=black]{$\sigma_2$};
\draw[decoration={brace,mirror,raise=4pt},decorate] (4,6) -- (4,8) node[midway,right=6pt][color=black]{$\sigma_1$};
 \foreach \x in {0,2,4} {
        \foreach \y in {0,2,4,6,8} {
            \fill[color=black] (\x,\y) circle (0.05);
        }
    }
     \node[] at (-4,0) {$v_{k-3}$};
    \node[] at (-4,2) {$v_{k-2}$};
    \node[] at (-4,4) {$v_{k-1}$};
    \node[] at (-4,6) {$v_{k}$};
    \node[] at (-4,8) {$v_{k+1}$};
    \end{tikzpicture}
\caption{The solid and broken brown paths (augmented by the blue edges) \textit{do not} constitute the local pictures of a distinguished pair, although the weights of the local pictures are $\pm q$. In relevance to the text directly above Figure~\ref{fig2}, this is the case $b_p=1=c_p$. The reason we do not consider this to be a distinguished pair is because the solid brown path (augmented by the blue edges) has an edge in common with an $\eta$-distinguished pair (defined later and illustrated in Figure~\ref{fig6}).}
\label{fig3}
\end{figure}
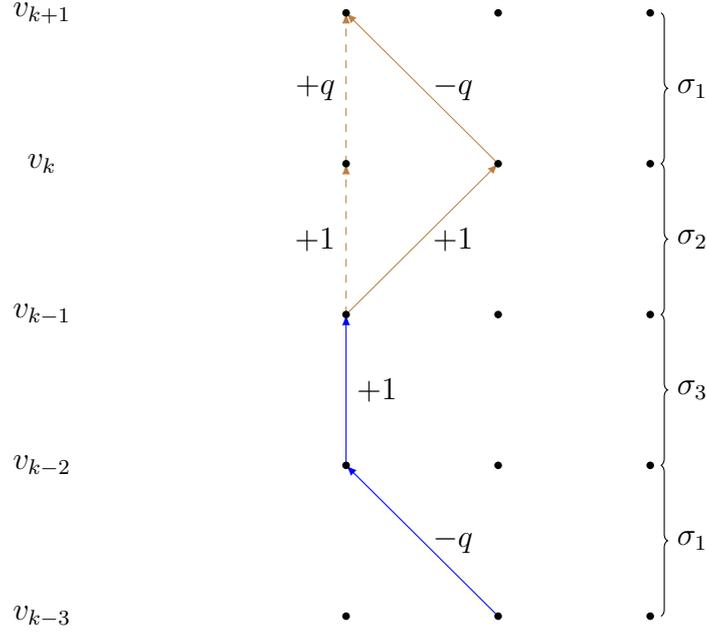

\item[(ii)] (\textit{$\epsilon$ distinguished pairs}) Let $b_{p+1}>0$. If $v_{k-1}=3=v_{k+a_{p+1}+1}$, then the paths $\epsilon_{k,2}$ and $\epsilon_{k,3}$ with $v_{k+j}=2$ and $v_{k+j}=3$ for $0\leq j\leq a_{p+1}$, respectively, constitute a distinguished pair. The other vertices match for $\epsilon_{k,2}$ and $\epsilon_{k,3}$, but are arbitrary, subject only to the conditions of a $\sigma$-path. Note that the weights of $\epsilon_{k,2}$ and $\epsilon_{k,3}$ have opposite sign. Figure~\ref{fig4} illustrates the local picture of this distinguished pair in the case $a_{p+1}=1$. 

\begin{figure}[H]
\centering
\begin{tikzpicture}[>=latex]
\draw[->][color=red][solid] (4,0) -- (2,2) node[midway, below, left][color=black]{$-q$};
\draw[->][color=red][solid] (2,2) -- (2,4) node[midway, above, left][color=black]{$+1$};
\draw[->][color=red][solid] (2,4) -- (4,6) node[midway, above, left][color=black]{$+1$};
\draw[->][color=red][dashed] (4,0) -- (4,2) node[midway, right][color=black]{$+1$};
\draw[->][color=red][dashed] (4,2) -- (4,4) node[midway, right][color=black]{$+1$} ;
\draw[->][color=red][dashed] (4,4) -- (4,6) node[midway, right][color=black]{$+q$} ;
\draw[decoration={brace,raise=4pt},decorate] (0,0) -- (0,2) node[midway,left=6pt][color=black]{$\sigma_2$};
\draw[decoration={brace,raise=4pt},decorate] (0,2) -- (0,4) node[midway,left=6pt][color=black]{$\sigma_1$};
\draw[decoration={brace,raise=4pt},decorate] (0,4) -- (0,6) node[midway,left=6pt][color=black]{$\sigma_3$};

 \foreach \x in {0,2,4} {
        \foreach \y in {0,2,4,6} {
            \fill[color=black] (\x,\y) circle (0.05);
        }
    }
     \node[] at (8,0) {$v_{k-1}$};
    \node[] at (8,2) {$v_{k}$};
    \node[] at (8,4) {$v_{k+1}$};
    \node[] at (8,6) {$v_{k+2}$};
       \end{tikzpicture}
   \caption{The solid and broken red paths constitute the local pictures of an $\epsilon$ distinguished pair $\{\epsilon_{k,2},\epsilon_{k,3}\}$. In relevance to the text directly above this figure, this is the case $a_{p+1}=1$. The weights of the local pictures are $\pm q$ (for all values of $a_{p+1}\geq 0$).}
\label{fig4}
\end{figure}
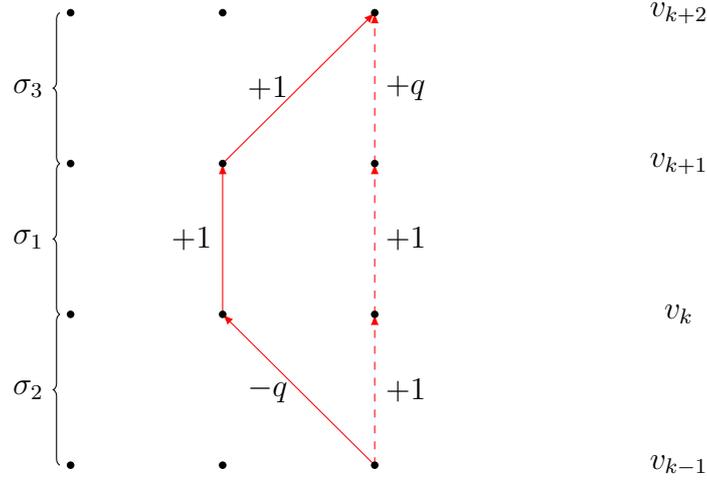

\item[(iii)] (\textit{$\zeta$ distinguished pairs}) Let $c_{p+1}>0$ and $v_{k+j}=2=v_{k+a_{p+1}+b_{p+1}+1}$ for $0\leq j\leq a_{p+1}+b_{p+1}-1$. If $a_{p+1}=1$ and there is a $1\to 2$ vertex change at the $k$th generator, then the paths $\zeta_{k+a_{p+1}+b_{p+1},2}$ and $\zeta_{k+a_{p+1}+b_{p+1},3}$ with $v_{k+a_{p+1}+b_{p+1}}=2$ and $v_{k+a_{p+1}+b_{p+1}}=3$, respectively, constitute a distinguished pair. The other vertices match for $\zeta_{k,1}$ and $\zeta_{k,2}$, but are arbitrary, subject only to the conditions of a $\sigma$-path. Note that the weights of $\zeta_{k+a_{p+1}+b_{p+1},2}$ and $\zeta_{k+a_{p+1}+b_{p+1},3}$ have opposite sign. Figure~\ref{fig5} illustrates the local picture of this distinguished pair in the case $b_{p+1}=2$. (Lemma~\ref{minimalform} \textbf{(i)} implies that $b_{p+1}>0$ and this ensures that a $\zeta$ distinguished pair is well-defined. Furthermore, the condition that there is a $1\to 2$ vertex change at the $k$th generator ensures that the local picture of $\zeta_{k+a_{p+1}+b_{p+1},3}$ does not have an edge in common with the local picture of an $\epsilon$-pair (defined earlier).)

\begin{figure}[H]
\centering
\begin{tikzpicture}[>=latex]
\draw[->][color=blue][solid] (0,0) -- (2,2) node[midway, below, right][color=black]{$+1$};
\draw[->][color=blue][solid] (2,2) -- (2,4) node[midway, above, right][color=black]{$+1$} ;
\draw[->][color=blue][solid] (2,4) -- (2,6) node[midway, above, right][color=black]{$+1$} ;
\draw[->][color=red][dashed] (2,6) -- (2,8) node[midway, left][color=black]{$+1$};
\draw[->][color=red][dashed] (2,8) -- (2,10) node[midway, left][color=black]{$+q$} ;
\draw[->][color=red][solid] (2,6) -- (4,8) node[midway, right][color=black]{$+1$} ;
\draw[->][color=red][solid] (4,8) -- (2,10) node[midway, right][color=black]{$-q$} ;
\draw[decoration={brace,mirror,raise=4pt},decorate] (4,0) -- (4,2) node[midway,right=6pt][color=black]{$\sigma_2$};
\draw[decoration={brace,mirror,raise=4pt},decorate] (4,2) -- (4,4) node[midway,right=6pt][color=black]{$\sigma_1$};
\draw[decoration={brace,mirror,raise=4pt},decorate] (4,4) -- (4,6) node[midway,right=6pt][color=black]{$\sigma_3$};
\draw[decoration={brace,mirror,raise=4pt},decorate] (4,6) -- (4,8) node[midway,right=6pt][color=black]{$\sigma_3$};
\draw[decoration={brace,mirror,raise=4pt},decorate] (4,8) -- (4,10) node[midway,right=6pt][color=black]{$\sigma_2$};
 \foreach \x in {0,2,4} {
        \foreach \y in {0,2,4,6,8,10} {
            \fill[color=black] (\x,\y) circle (0.05);
        }
    }
     \node[] at (-4,0) {$v_{k-1}$};
    \node[] at (-4,2) {$v_{k}$};
    \node[] at (-4,4) {$v_{k+1}$};
    \node[] at (-4,6) {$v_{k+2}$};
    \node[] at (-4,8) {$v_{k+3}$};
    \node[] at (-4,10) {$v_{k+4}$};
    \end{tikzpicture}
\caption{The solid and broken red paths (augmented by the blue edges) constitute the local pictures of a $\zeta$ distinguished pair $\{\zeta_{k+a_{p+1}+b_{p+1},2},\zeta_{k+a_{p+1}+b_{p+1},3}\}$. The weights of the local (red) pictures are $\pm q$. In relevance to the text directly above this figure, this is the case $b_{p+1}=2$.}
\label{fig5}
\end{figure}
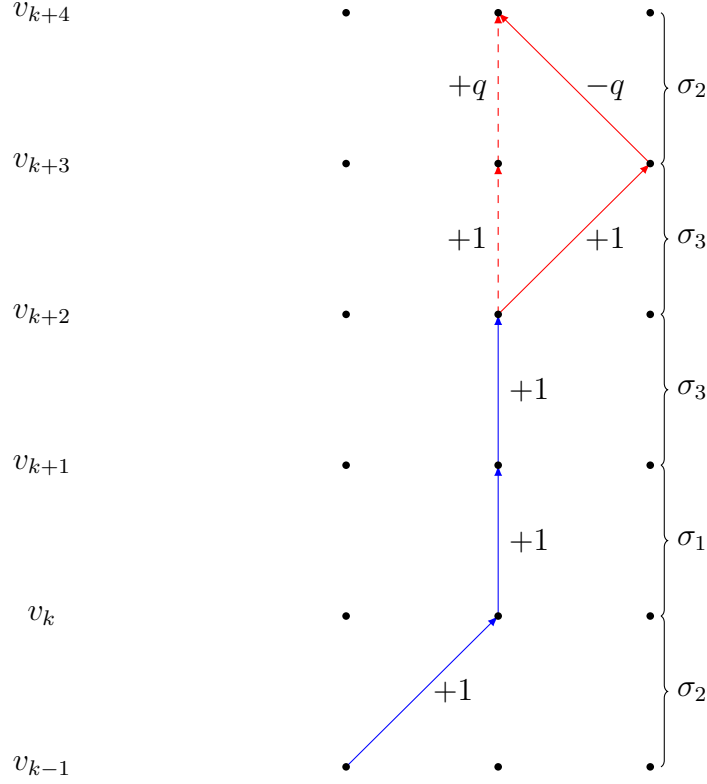

\item[(iv)] (\textit{$\eta$ distinguished pairs}) Let $c_{p+1}>0$ and $v_{k+j}=2=v_{k+a_{p+1}+b_{p+1}+1}$ for $0\leq j\leq a_{p+1}-1$. If either $a_{p+1}=1$ and there is no $1\to 2$ vertex change at the $k$th generator, or $a_{p+1}>1$, then the paths $\eta_{k+a_{p+1},1}$ and $\eta_{k+a_{p+1},2}$ with $v_{k+a_{p+1}+j}=1$ and $v_{k+a_{p+1}+j}=2$ for $0\leq j\leq b_{p+1}$, respectively, constitute a distinguished pair. The other vertices match for $\eta_{k+a_{p+1},1}$ and $\eta_{k+a_{p+1},2}$, but are arbitrary, subject only to the conditions of a $\sigma$-path. Note that the weights of $\eta_{k+a_{p+1},1}$ and $\eta_{k+a_{p+1},2}$ have opposite sign. Figure~\ref{fig6} illustrates the local picture of this distinguished pair in the case $a_{p+1}=1=b_{p+1}$ (and there is no $1\to 2$ vertex change at the $k$th generator). (If $a_{p+1} = 1$ and there is a $1\to 2$ vertex change at the $k$th generator, then we do not consider $\eta_{k+a_{p+1},1}$ and $\eta_{k+a_{p+1},2}$ a distinguished pair because the local picture of $\eta_{k+a_{p+1},1}$ would have an edge in common with the local picture of a $\delta$ distinguished pair (defined earlier).)

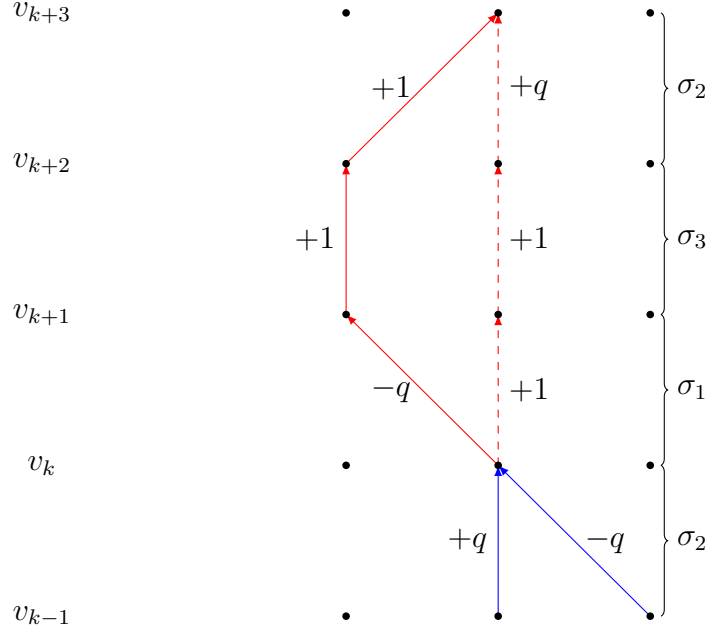
\begin{figure}[H]
\centering
\begin{tikzpicture}[>=latex]
\draw[->][color=blue][solid] (2,0) -- (2,2) node[midway, below, left][color=black]{$+q$};
\draw[->][color=blue][solid] (4,0) -- (2,2) node[midway, below, right][color=black]{$-q$};
\draw[->][color=red][solid] (2,2) -- (0,4) node[midway, above, left][color=black]{$-q$} ;
\draw[->][color=red][dashed] (2,4) -- (2,6) node[midway, right][color=black]{$+1$};
\draw[->][color=red][dashed] (2,6) -- (2,8) node[midway, right][color=black]{$+q$} ;
\draw[->][color=red][dashed] (2,2) -- (2,4) node[midway, right][color=black]{$+1$} ;
\draw[->][color=red][solid] (0,4) -- (0,6) node[midway, left][color=black]{$+1$} ;
\draw[->][color=red][solid] (0,6) -- (2,8) node[midway, left][color=black]{$+1$} ;
\draw[decoration={brace,mirror,raise=4pt},decorate] (4,0) -- (4,2) node[midway,right=6pt][color=black]{$\sigma_2$};
\draw[decoration={brace,mirror,raise=4pt},decorate] (4,2) -- (4,4) node[midway,right=6pt][color=black]{$\sigma_1$};
\draw[decoration={brace,mirror,raise=4pt},decorate] (4,4) -- (4,6) node[midway,right=6pt][color=black]{$\sigma_3$};
\draw[decoration={brace,mirror,raise=4pt},decorate] (4,6) -- (4,8) node[midway,right=6pt][color=black]{$\sigma_2$};
 \foreach \x in {0,2,4} {
        \foreach \y in {0,2,4,6,8} {
            \fill[color=black] (\x,\y) circle (0.05);
        }
    }
     \node[] at (-4,0) {$v_{k-1}$};
    \node[] at (-4,2) {$v_{k}$};
    \node[] at (-4,4) {$v_{k+1}$};
    \node[] at (-4,6) {$v_{k+2}$};
    \node[] at (-4,8) {$v_{k+3}$};
    \end{tikzpicture}
\caption{The solid and broken red paths (augmented by either of the blue edges) constitute the local pictures of a distinguished pair $\{\eta_{k+a_{p+1},1},\eta_{k+a_{p+1},2}\}$. The weights of the local (red) pictures are $\pm q$. In relevance to the text directly above this figure, this is the case $a_{p+1}=1=b_{p+1}$ (and there is no $1\to 2$ vertex change at the $k$th generator).}
\label{fig6}
\end{figure}

\item[(v)] (\textit{$\theta$ distinguished pairs}) Let $a_{p+1} = 0$, $c_{p+1}>0$, and $v_{k+j}=2=v_{k+a_{p+1}+b_{p+1}+1}$ for $0\leq j\leq a_{p+1} + b_{p+1}-1$. If either $b_{p+1}=1$ and there is no $3\to 2$ vertex change at the $k$th generator, or $b_{p+1}>1$, then the paths $\theta_{k+b_{p+1},2}$ and $\theta_{k+b_{p+1},3}$ with $v_{k+a_{p+1}+b_{p+1}}=2$ and $v_{k+a_{p+1}+b_{p+1}}=3$, respectively, constitute a distinguished pair. The other vertices match for $\theta_{k+b_{p+1},2}$ and $\theta_{k+b_{p+1},3}$, but are arbitrary, subject only to the conditions of a $\sigma$-path. Note that the weights of $\theta_{k+b_{p+1},2}$ and $\theta_{k+b_{p+1},3}$ have opposite sign. Figure~\ref{fig7} illustrates the local picture of this distinguished pair in the case $b_{p+1}=1$ (and there is no $3\to 2$ vertex change at the $k$th generator). (The condition that there is no $3\to 2$ vertex change at the $k$th generator is to ensure that the local picture of $\theta_{k+b_{p+1},3}$ does not have an edge in common with the local picture of an $\epsilon$ distinguished pair (defined earlier).)

\begin{figure}[H]
\centering
\begin{tikzpicture}[>=latex]
\draw[->][color=blue][solid] (2,0) -- (2,2) node[midway, right][color=black]{$+q$};
\draw[->][color=blue][solid] (0,0) -- (2,2) node[midway, right][color=black]{$+1$};
\draw[->][color=red][solid] (2,2) -- (4,4) node[midway, above, right][color=black]{$+1$} ;
\draw[->][color=red][solid] (4,4) -- (2,6) node[midway, above, right][color=black]{$-q$} ;
\draw[->][color=red][dashed] (2,2) -- (2,4) node[midway, left][color=black]{$+1$};
\draw[->][color=red][dashed] (2,4) -- (2,6) node[midway, left][color=black]{$+q$} ;
\draw[decoration={brace,mirror,raise=4pt},decorate] (4,0) -- (4,2) node[midway,right=6pt][color=black]{$\sigma_2$};
\draw[decoration={brace,mirror,raise=4pt},decorate] (4,2.) -- (4,4) node[midway,right=6pt][color=black]{$\sigma_3$};
\draw[decoration={brace,mirror,raise=4pt},decorate] (4,4) -- (4,6) node[midway,right=6pt][color=black]{$\sigma_2$};
 \foreach \x in {0,2,4} {
        \foreach \y in {0,2,4,6} {
            \fill[color=black] (\x,\y) circle (0.05);
        }
    }
     \node[] at (-4,0) {$v_{k-1}$};
    \node[] at (-4,2) {$v_{k}$};
    \node[] at (-4,4) {$v_{k+1}$};
    \node[] at (-4,6) {$v_{k+2}$};    
\end{tikzpicture}
\caption{The solid and broken red paths (augmented by either of the blue edges) constitute the local pictures of a distinguished pair $\{\theta_{k+b_{p+1},2},\theta_{k+b_{p+1},3}\}$. The weights of the local (red) pictures are $\pm q$. In relevance to the text directly above this figure, this is the case $b_{p+1}=1$ (and there is no $3\to 2$ vertex change at the $k$th generator).}
\label{fig7}
\end{figure}
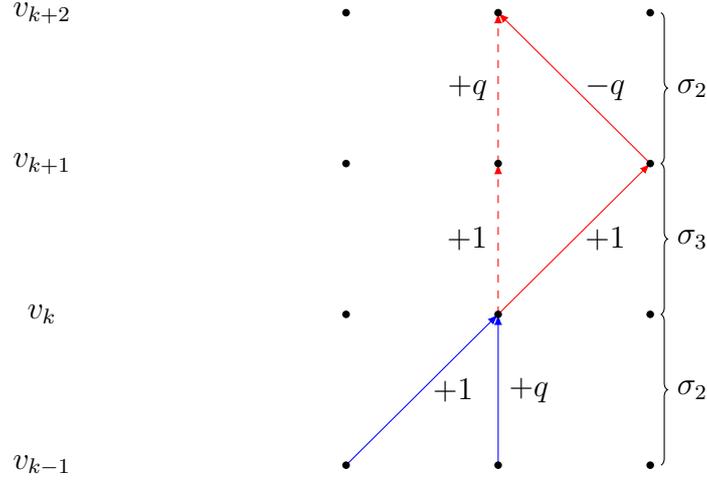

\end{description}

We will refer to $\sigma$-paths that do not belong to a distinguished pair of any of the forms described as \textit{admissible $\sigma$-paths}. 
\end{definition}

The following example elucidates the notion of admissibility. 

\begin{example}
\label{weirdadmissible}
Figure~\ref{fig8} is an illustration of the case where $a_{p+1}=1=b_{p+1}$ and there is no vertex change at the $k$th generator.

\begin{figure}[H]
\centering
\begin{tikzpicture}[>=latex]
\draw[->][color=red][solid] (2,0) -- (0,2) node[midway, left][color=black]{$-q$};
\draw[->][color=red][solid] (0,2) -- (0,4) node[midway, left][color=black]{$+1$};
\draw[->][color=red][solid] (0,4) -- (2,6) node[midway, above, left][color=black]{$+1$} ;
\draw[->][color=green][solid] (4,4) -- (2,6) node[midway, above, right][color=black]{$-q$} ;
\draw[->][color=red][dashed] (2,0) to [bend left] node[midway, above left][color=black]{$+1$} (2,2);
\draw[->][color=red][dashed] (2,2) -- (2,4) node[midway, left][color=black]{$+1$};
\draw[->][color=red][dashed] (2,4) -- (2,6) node[midway, left][color=black]{$+q$} ;
\draw[->][color=green][solid] (2,0) to [bend right] node[midway, above right][color=black]{$+1$} (2,2);
\draw[->][color=green][solid] (2,2) -- (4,4) node[midway, right][color=black]{$+1$};
\draw[decoration={brace,mirror,raise=4pt},decorate] (4,0) -- (4,2) node[midway,right=6pt][color=black]{$\sigma_1$};
\draw[decoration={brace,mirror,raise=4pt},decorate] (4,2.) -- (4,4) node[midway,right=6pt][color=black]{$\sigma_3$};
\draw[decoration={brace,mirror,raise=4pt},decorate] (4,4) -- (4,6) node[midway,right=6pt][color=black]{$\sigma_2$};
 \foreach \x in {0,2,4} {
        \foreach \y in {0,2,4,6} {
            \fill[color=black] (\x,\y) circle (0.05);
        }
    }
     \node[] at (-4,0) {$v_{k}$};
    \node[] at (-4,2) {$v_{k+1}$};
    \node[] at (-4,4) {$v_{k+2}$};
    \node[] at (-4,6) {$v_{k+3}$};
    \end{tikzpicture}
\caption{The solid and broken red paths constitute the local pictures of an $\eta$ distinguished pair (see Figure~\ref{fig6}), provided there is no $1\to 2$ vertex change just beforehand. The local picture of the green path does not belong to a distinguished pair (roughly speaking, it is ``locally admissible" in this case). Of course, if there is a $3\to 2$ vertex change just beforehand, then the extension of the local picture of the green path by one previous step will belong to an $\epsilon$ distinguished pair (see Figure~\ref{fig4}).}
\label{fig8}
\end{figure}
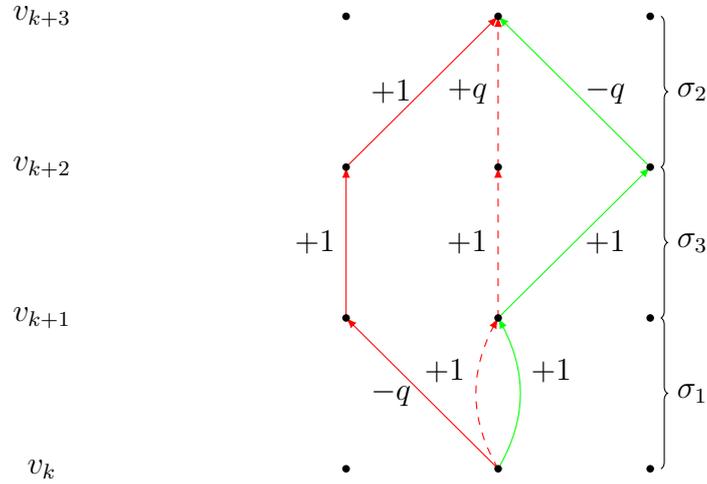

Let us consider $\sigma=\sigma_1\sigma_2\sigma_1$ and Figure~\ref{fig9} below. The $\left(1,1\right)$-type $\sigma$-paths $1,1,1,1$ and $1,1,2,1$ (not depicted in Figure~\ref{fig9}) constitute a $\delta$ distinguished pair (see Figure~\ref{fig2}) and their weights cancel. Furthermore, the $\left(2,1\right)$-type $\sigma$-paths $2,1,2,1$ and $2,2,2,1$ constitute an $\eta$ distinguished pair and their weights cancel (a variation of Figure~\ref{fig6}). On the other hand, $2,1,1,1$ is an admissible $\sigma$-path although its weight has the opposite sign to the weight of $2,1,2,1$ (Figure~\ref{fig2} does not apply since there is a $2\to 1$ vertex change just beforehand; see also Figure~\ref{fig3}). In this case, the weight of $2,1,1,1$ contributes the (unique) non-zero term to the $\left(2,1\right)$ entry of $\beta_4\left(\sigma\right)$. Indeed, it is instructive to verify (the following) Proposition~\ref{m=ap} in this case.

\begin{figure}[H]
\centering
\begin{tikzpicture}[>=latex]
\draw[->][color=green][solid] (2,0) to [bend left] node[midway, left][color=black]{$-q$} (0,2);
\draw[->][color=red][solid] (2,0) to [bend right] node[midway, left][color=black]{$-q$} (0,2);
\draw[->][color=green][solid] (0,2) -- (0,4) node[midway, left][color=black]{$+1$};
\draw[->][color=green][solid] (0,4) -- (0,6) node[midway, above, left][color=black]{$+q$} ;
\draw[->][color=red][dashed] (2,0) to node[midway, above right][color=black]{$+1$} (2,2);
\draw[->][color=red][dashed] (2,2) -- (2,4) node[midway, right][color=black]{$+q$};
\draw[->][color=red][solid] (0,2) -- (2,4) node[midway, right][color=black]{$+1$};
\draw[->][color=red][dashed] (2,4) to node[above left][color=black]{$+1$} (2,6);
\draw[->][color=red][solid] (2,4) to [bend left]  node[midway, left][color=black]{$-q$} (0,6);
\draw[->][color=red][dashed] (2,4) to [bend right] node[above right][color=black]{$+1$} (2,6);
\draw[->][color=red][dashed] (2,4) to [bend right] node[midway, left][color=black]{$-q$} (0,6);
\draw[decoration={brace,mirror,raise=4pt},decorate] (4,0) -- (4,2) node[midway,right=6pt][color=black]{$\sigma_1$};
\draw[decoration={brace,mirror,raise=4pt},decorate] (4,2.) -- (4,4) node[midway,right=6pt][color=black]{$\sigma_2$};
\draw[decoration={brace,mirror,raise=4pt},decorate] (4,4) -- (4,6) node[midway,right=6pt][color=black]{$\sigma_1$};
 \foreach \x in {0,2,4} {
        \foreach \y in {0,2,4,6} {
            \fill[color=black] (\x,\y) circle (0.05);
        }
    }
\end{tikzpicture}
\caption{In this figure, there are two $\eta$ distinguished pairs, each consisting of a solid and a broken red path. The weights in the $\eta$ distinguished pair of $\left(2,1\right)$-type paths are $\pm q^2$ and the weights in the $\eta$ distinguished pair of $\left(2,2\right)$-type paths are $\pm q$. The green $\left(2,1\right)$-type path is admissible and its weight is $-q^2$. However, note that the constant $\left(1,1\right)$-type path (not depicted in this figure) would \textit{not} be admissible since it is part of a $\delta$ distinguished pair (see Figure~\ref{fig2}).}
\label{fig9}
\end{figure}
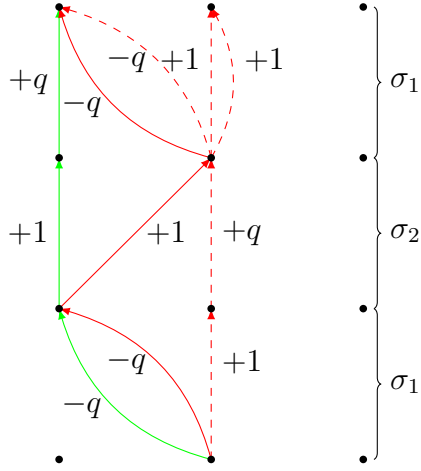

Note that a figure identical to Figure~\ref{fig9}, except translated one unit to the right is applicable for $\sigma_2\sigma_3\sigma_2$. While $\sigma_1\sigma_2\sigma_1$ cannot occur in the minimal form of a positive braid except possibly at the beginning (Lemma~\ref{minimalform} \textbf{(ii)}), $\sigma_2\sigma_3\sigma_2$ could occur in multiple places in the minimal form of a positive braid. We will consider these occurrences in Section~\ref{mains} when we introduce the notion of a $3$-block (see Definition~\ref{blockroad}).
\end{example}

We justify the introduction of admissible $\sigma$-paths in relevance to the Burau representation $\beta_4$.

\begin{proposition}
\label{m=ap}
The $\left(r,s\right)$-entry of $\beta_4\left(\sigma\right)$ is the weighted number of the admissible $\left(r,s\right)$-type $\sigma$-paths.\end{proposition}
\begin{proof}
Let ${\cal P}$ be the set of all $\sigma$-paths belonging to a distinguished pair. We will define an involution $\iota:{\cal P}\to {\cal P}$ with the property that $\sigma$-paths which correspond via $\iota$ have weights that are negatives of each other. The statement follows from the existence of this involution and Proposition~\ref{m=p}. 

Firstly, we define the \textit{local picture} of a distinguished pair as the symmetric difference of the $\sigma$-paths in the distinguished pair (in Figure~\ref{fig2} and each of Figures~\ref{fig4}~-~\ref{fig7} above, the local picture is the union of the solid and dashed red subpaths). Note that the local picture of a distinguished pair is the union of two subpaths with the same initial and final vertex. The geometric shape of the local picture is either a triangle (a possibility for each type of distinguished pair) or a trapezium (only a possibility for $\epsilon$ distinguished pairs and $\eta$ distinguished pairs). An \textit{edge of the local picture} is an edge of its geometric shape in the usual sense. (If the geometric shape of a local picture is a trapezium, then an edge of the local picture can be a sequence of constant edges $v\to v$ of a $\sigma$-path for some $v\in \{1,2,3\}$.) 

We claim that if a $\sigma$-path belongs to two different distinguished pairs, then the corresponding local pictures of the two distinguished pairs do not have an edge in common. If $X\in {\cal P}$, then we define $\iota\left(X\right)$ to be the $\sigma$-path with the property that $\{X,\iota\left(X\right)\}$ is the earliest distinguished pair containing $X$. The fact that $\iota$ is a well-defined involution follows from the claim.
 
Let us prove the claim. In Definition~\ref{dispair}, there are five different types of distinguished pairs, labelled by the Greek letters $\delta,\epsilon,\zeta,\eta$, and $\theta$. Of course, the local pictures of different distinguished pairs in the same type do not have an edge in common, since the corresponding subproducts of generators do not overlap. We will show that the local pictures of distinguished pairs in different types do not have an edge in common.

Firstly, observe that the edges in the local picture of a $\delta$ distinguished pair are $1\to 1$, $1\to 2$ and $2\to 1$, and the edges in the local picture of an $\epsilon$ distinguished pair are $3\to\cdots\to 3$, $3\to 2$, $2\to\cdots\to 2$, and $2\to 3$. In particular, the local pictures of a $\delta$ distinguished pair and an $\epsilon$ distinguished pair do not have an edge in common.

Secondly, observe that the local pictures of the $\zeta$ distinguished pairs, $\eta$ distinguished pairs, and $\theta$ distinguished pairs do not have an edge in common. Indeed, there is an incompatibility in either the corresponding subproducts of generators in the minimal form of $\sigma$, or the conditions on the edge of the $\sigma$-path immediately preceding the local picture. The subproduct of generators corresponding to the local picture of a $\zeta$ distinguished pair is $\sigma_2\sigma_1\sigma_3^{b_{p+1}}\sigma_2$, and in this case there is a $1\to 2$ vertex change at the first $\sigma_2$. The subproduct of generators corresponding to the local picture of an $\eta$ distinguished pair is $\sigma_2\sigma_1^{a_{p+1}}\sigma_3^{b_{p+1}}\sigma_2$, and either $a_{p+1}>1$, or $a_{p+1}=1$ and there is no $1\to 2$ vertex change at the first $\sigma_2$. The subproduct of generators corresponding to the local picture of a $\theta$ distinguished pair is $\sigma_2\sigma_3^{b_{p+1}}\sigma_2$. 

Finally, we will establish that the local picture of either a $\delta$ distinguished pair or an $\epsilon$ distinguished pair does not have an edge in common with the local picture of either a $\zeta$ distinguished pair, $\eta$ distinguished pair, or a $\theta$ distinguished pair. A lot of potential overlaps are ruled out by the conditions on the nature of the vertex change immediately preceding these local pictures in Definition~\ref{dispair}. 

Firstly, the explanation in Definition~\ref{dispair} \textbf{(i)} implies that the bottom edge in the local picture of a $\delta$ distinguished pair and the top edge in the local picture of an $\eta$ distinguished pair are not equal. The explanation in Definition~\ref{dispair} \textbf{(iv)} implies that the top edge in the local picture of a $\delta$ distinguished pair and the bottom edge in the local picture of an $\eta$ distinguished pair are not equal.

Secondly, the explanation in Definition~\ref{dispair} \textbf{(iii)} implies that the bottom edge in the local picture of a $\zeta$ distinguished pair and the top edge in the local picture of an $\epsilon$ distinguished pair are not equal. Lemma~\ref{minimalform} \textbf{(ii)} implies that the top edge in the local picture of a $\zeta$ distinguished pair and the bottom edge in the local picture of an $\epsilon$ distinguished pair are not equal.

Thirdly, the explanation in Definition~\ref{dispair} \textbf{(v)} implies that the bottom edge in the local picture of a $\theta$ distinguished pair and the top edge in the local picture of an $\epsilon$ distinguished pair are not equal. Lemma~\ref{minimalform} \textbf{(ii)} implies that the top edge in the local picture of a $\theta$ distinguished pair and the bottom edge in the local picture of an $\epsilon$ distinguished pair are not equal.

Furthermore, the edges in the local picture of either a $\zeta$ distinguished pair or a $\theta$ distinguished pair are $2\to 2\to 2$, $2\to 3$, and $3\to 2$. Thus the local picture of a $\delta$ distinguished pair and the local picture of either a $\zeta$ distinguished pair or a $\theta$ distinguished pair do not overlap in an edge. However, the edges in the local picture of an $\eta$ distinguished pair are $2\to \cdots\to 2$, $2\to 1$, $1\to\cdots\to 1$, and $1\to 2$, and the local pictures of an $\epsilon$ distinguished pair and an $\eta$ distinguished pair can \textit{partially} overlap in a $2\to \cdots\to 2$ edge. However, they cannot have a $2\to\cdots\to 2$ edge in common since the subproduct of generators corresponding to a $2\to\cdots\to 2$ edge in the local picture of an $\epsilon$ distinguished pair is $\sigma_1^{a_{p+1}}$ and the subproduct of generators corresponding to a $2\to\cdots\to 2$ edge in the local picture of an $\eta$ distinguished pair is $\sigma_1^{a_{p+1}}\sigma_3^{b_{p+1}}\sigma_2$.
\end{proof}

In spite of Proposition~\ref{m=p} and Proposition~\ref{m=ap}, the weight of an admissible $\sigma$-path \textit{does not} necessarily contribute a non-zero term to an entry of $\beta_4\left(\sigma\right)$. Indeed, an admissible $\sigma$-path could still potentially be part of a cancelling pair of a different kind not considered in Definition~\ref{dispair}. The failure of admissibility may be thought of as a \textit{local} obstruction to the weight of a $\sigma$-path contributing a non-zero term to an entry of $\beta_4\left(\sigma\right)$. 

\subsection{Good and bad paths}
\label{subsecgb}
In Section~\ref{mains}, we will construct admissible $\sigma$-paths inductively by extending admissible $P$-paths, where $P$ is a subproduct of the minimal form of $\sigma$. In view of the definition of admissibility and distinguished pairs (Definition~\ref{dispair}), the construction depends on the nature of the final vertex change of the $P$-path. In this subsection, we introduce a categorization of $P$-paths to capture this idea precisely. Firstly, we introduce terminology to characterize the types of subproduct $P$ that we will consider. 

\begin{definition}
\label{defssubproduct}
Let $\sigma=\prod_{i=1}^{n} \sigma_1^{a_i}\sigma_3^{b_i}\sigma_2^{c_i}$ be the minimal form of $\sigma$ where $a_i+b_i$ is a positive integer for $2\leq i\leq n$ and $c_i$ is a positive integer for $1\leq i\leq n-1$. We define an \textit{$s$-subproduct} of $\sigma$ as follows. If $s=2$, then it is of the form $\prod_{i=1}^{p-1} \sigma_1^{a_i}\sigma_3^{b_i}\sigma_2^{c_i}$ for some $1\leq p\leq n+1$, where $p=n+1$ only if $c_n>0$. If $s\in \{1,3\}$, then it is of the form $\left(\prod_{i=1}^{p-1}\sigma_1^{a_i}\sigma_3^{b_i}\sigma_2^{c_i}\right)\sigma_1^{a_{p}}\sigma_3^{b_{p}}$ for some $1\leq p\leq n$. (A $1$-subproduct is the same as a $3$-subproduct.) If $P$ is an $s$-subproduct for $s\in \{1,3\}$, then an $\left(r,s\right)$-type $P$-path refers to either an $\left(r,1\right)$-type $P$-path or an $\left(r,3\right)$-type $P$-path.
\end{definition}

We now define the categorization of $P$-paths into \textit{good}, \textit{bad}, $1$\textit{-switch}, and $3$\textit{-switch} $P$-paths. 

\begin{definition}
\label{defgoodbadpath}
If $P$ is an $s$-subproduct of $\sigma$, then an admissible $\left(r,s\right)$-type $P$-path is \textit{good} if one of the following conditions is satisfied:
\begin{description}
\item[(i)] Let $s\in \{1,3\}$ and $P=\left(\prod_{i=1}^{p-1} \sigma_1^{a_i}\sigma_3^{b_i}\sigma_2^{c_i}\right)\sigma_1^{a_p}\sigma_3^{b_p}$. 

If $s=1$, then the $P$-path does not have a $2\to 1$ vertex change at the last $\sigma_1$ in $\sigma_1^{a_p}$ (if $a_p=0$, then this is vacuously true).

If $s=3$, then either the $P$-path does not have a $2\to 3$ vertex change at the last $\sigma_3$ in $P$ if $a_p=0$, or the $P$-path does not have a $1\to 2$ vertex change at the last $\sigma_2$ in $P$ and a $2\to 3$ vertex change at the last $\sigma_3$ in $\sigma_3^{b_p}$ if $a_p=1$ (if either $b_p=0$ or $a_p>1$, then this is vacuously true).
\item[(ii)] Let $s=2$ and $P=\prod_{i=1}^{p-1} \sigma_1^{a_i}\sigma_3^{b_i}\sigma_2^{c_i}$. 

The $P$-path does not have a $1\to 2$ vertex change at the last $\sigma_2$ in $P$ if $a_p>0$, and the $P$-path does not have a $3\to 2$ vertex change at the last $\sigma_2$ in $P$ if $b_p>0$.
\end{description}
An admissible $P$-path is \textit{bad} if it is not good. (If $P$ is the empty product, then we note that the $P$-path is vacuously good.)

If we wish to be more specific in the case $s=2$, then we will write that a $P$-path is $1\textit{-switch}$ (resp. $3\textit{-switch}$) if it admits a $1\to 2$ (resp. $3\to 2$) vertex change at the last $\sigma_2$ in $P$. 
\end{definition}

The following statement consists of the fundamental properties of extensions of good and bad paths.

\begin{lemma}
\label{lemmagoodext}
Let $P$ be an $s$-subproduct of $\sigma$ and let $P'$ be the $s'$-subproduct of $\sigma$ for some $s'\in \{1,2,3\}$, minimal with respect to the property that $P\subsetneq P'$ (in particular, $\{s,s'\}$ consists of one element in $\{1,3\}$ and one element in $\{2\}$). The following statements concerning extensions of $P$-paths to $P'$-paths are true:

\begin{description}[style=unboxed,leftmargin=0cm]
\item[(i)] An admissible $\left(r,s\right)$-type $P$-path is good if and only if its extension to an $\left(r,s'\right)$-type $P'$-path by any vertex change is admissible. 
\item[(ii)] If $s\in \{1,3\}$, then an admissible $\left(r,s\right)$-type $P$-path is bad if and only if all of its admissible extensions to an $\left(r,2\right)$-type $P'$-path are by a vertex change at the second or later $\sigma_2$ in $P'\setminus P$.
\item[(iii)] An admissible $\left(r,2\right)$-type $P$-path is a $1$-switch $\left(r,2\right)$-type $P$-path if and only if all of its admissible extensions to an $\left(r,s'\right)$-type $P'$-path are by a vertex change at either the second or later $\sigma_1$ in $P'\setminus P$ or every $\sigma_3$ in $P'\setminus P$. 

An admissible $\left(r,2\right)$-type $P$-path is a $3$-switch $\left(r,2\right)$-type $P$-path if and only if all of its admissible extensions to an $\left(r,s'\right)$-type $P'$-path are by a vertex change at either every $\sigma_1$ in $P'\setminus P$ or the second or later $\sigma_3$ in $P'\setminus P$. 
\item[(iv)] Let $P''$ be an $s''$-subproduct for some $s''\in \{1,2,3\}$ such that $P\subseteq P''$. If $t\in \{1,3\}$, then an extension of an admissible $\left(r,t\right)$-type $P$-path to an $\left(r,t\right)$-type $P''$-path with no further vertex change is admissible if and only if there is no $\sigma_t$ in $P''\setminus P$.
\item[(v)] An admissible $\left(r,2\right)$-type $P$-path extends to an admissible $\left(r,2\right)$-type $P'$-path with no further vertex change if and only if $s = 2$.
\end{description}
\end{lemma}
\begin{proof}
Note that if $s\in \{1,3\}$, then $s'=2$, and if $s=2$, then $s'\in \{1,3\}$. We consider cases according to the value of $s\in \{1,2,3\}$. If $s\in \{1,3\}$, then we write $P=\left(\prod_{i=1}^{p-1} \sigma_1^{a_i}\sigma_3^{b_i}\sigma_2^{c_i}\right)\sigma_1^{a_p}\sigma_3^{b_p}$. If $s=2$, then we write $P=\prod_{i=1}^{p-1} \sigma_1^{a_i}\sigma_3^{b_i}\sigma_2^{c_i}$. Let $W$ denote an admissible $\left(r,s\right)$-type $P$-path.

If $s=1$, then an extension of $W$ by a vertex change at the first $\sigma_2$ in $\sigma_2^{c_p}$ is admissible if and only if $W$ does not belong to the relevant $\eta$ distinguished pair. Of course, this is the case if and only if $W$ does not have a vertex change at the last $\sigma_1$ in $\sigma_1^{a_p}$. An extension of $W$ by a vertex change at the second or later $\sigma_2$ in $\sigma_2^{c_p}$ is always admissible since it does not belong to the relevant $\eta$ distinguished pair. We deduce \textbf{(i)} and \textbf{(ii)} in the case $s = 1$.

If $s=3$, then an extension of $W$ by a vertex change at the first $\sigma_2$ in $\sigma_2^{c_p}$ is admissible if and only if $W$ does not belong to the relevant $\zeta$ distinguished pair or $\theta$ distinguished pair. Of course, $W$ does not belong to a $\zeta$ distinguished pair if and only if either (1) $a_p>1$ or (2) $a_p=1$ and $W$ does not simultaneously have a $1\to 2$ vertex change at the last $\sigma_2$ in $P$, and a $2\to 3$ vertex change at the last $\sigma_3$ in $\sigma_3^{b_p}$. Furthermore, $W$ does not belong to a $\theta$ distinguished pair if and only if either (1) $a_p>0$ or (2) $a_p = 0$ and $W$ does not have a $2\to 3$ vertex change at the last $\sigma_3$ in $\sigma_3^{b_p}$. An extension of $W$ by a vertex change at the second or later $\sigma_2$ in $\sigma_2^{c_p}$ is always admissible since it does not belong to the relevant $\zeta$ distinguished pair or $\theta$ distinguished pair. We deduce \textbf{(i)} and \textbf{(ii)} in the case $s = 3$. 

If $s=2$, then an extension of $W$ by a vertex change at the first $\sigma_1$ in $P'$ is admissible if and only if $W$ does not belong to the relevant $\delta$ distinguished pair. Of course, this is the case if and only if $W$ does not have a $1\to 2$ vertex change at the last $\sigma_2$ in $P$. Similarly, an extension of $W$ by a vertex change at the first $\sigma_3$ in $P'$ is admissible if and only if $W$ does not belong to the relevant $\epsilon$ distinguished pair. Of course, this is the case if and only if $W$ does not have a $3\to 2$ vertex change at the last $\sigma_2$ in $P$. Furthermore, an extension of $W$ by a vertex change at either the second or later $\sigma_1$ in $P'$, or the second or later $\sigma_3$ in $P'$, is admissible since it does not belong to the relevant $\delta$ distinguished pair or $\epsilon$ distinguished pair, respectively. We deduce \textbf{(i)} in the case $s = 2$ and \textbf{(iii)}. We have thus far established statements \textbf{(i)} - \textbf{(iii)}.

The statement \textbf{(iv)} follows since the extension of the $\left(r,t\right)$-type $P$-path $W$ with $t\in \{1,3\}$ to an $\left(r,t\right)$-type $P''$-path with no further vertex change is admissible if and only if it does not belong to either a $\delta$ distinguished pair or an $\epsilon$ distinguished pair. Finally, the statement \textbf{(v)} follows since the only distinguished pairs that can contain an extension of an $\left(r,2\right)$-type $P$-path with no further vertex change are a $\zeta$ distinguished pair, $\eta$ distinguished pair, or $\theta$ distinguished pair. Furthermore, none of these distinguished pairs contain the extension of an admissible $\left(r,2\right)$-type $P$-path to an $\left(r,2\right)$-type $P'$-path if and only if there is no $\sigma_2$ in $P'\setminus P$.
\end{proof}

The theory of good and bad paths (Definition~\ref{defgoodbadpath} and Lemma~\ref{lemmagoodext}) is powerful in terms of constructing admissible $\sigma$-paths with weights contributing non-cancelling terms to the entries of $\beta_4\left(\sigma\right)$. We illustrate its power in a concrete example. 

\begin{example}
\label{exmain}
We will consider a generic family of positive braids $\sigma$, and use the theory of good and bad $P$-paths to demonstrate that the leading coefficient of \textit{every} entry in the matrix $\beta_4\left(\sigma\right)$ is non-zero modulo each prime $t\neq 2$. (The statement is a stronger version of Theorem~\ref{invariantdetect}, but in a generic special case.) The strategy is to construct maximal $q$-resistance $\sigma$-paths. The maximality will ensure that the weights of these $\sigma$-paths do not cancel with the weights of other $\sigma$-paths. The statement will then follow from Proposition~\ref{m=ap}.

In Figure~\ref{fig1}, a constant $i\to i$ edge has weight equal to $q$ if and only if it is the $i\to i$ edge corresponding to the picture of $\sigma_i$ (the other constant edges in the picture of $\sigma_i$ have weight equal to $1$). In particular, the na{\"i}ve method for constructing maximal $q$-resistance $\sigma$-paths is to ensure that the $j$th vertex of the $\sigma$-path matches with the index of the $\left(j+1\right)$st generator in the minimal form of $\sigma$ as much as possible.
 
Let $\sigma = \prod_{i=1}^{n} \sigma_1^{a_i}\sigma_3^{b_i}\sigma_2^{c_i}$ be a positive braid, where $a_i+b_i\geq 2$ and $c_i\geq 2$ for $1\leq i\leq n$. We will show that the na{\"i}ve method for constructing maximal $q$-resistance $\sigma$-paths works perfectly in this case. The condition on the exponents is a condition satisfied by generic positive braids, and also ensures that the product expansion is the minimal form of $\sigma$. (Indeed, the only finite sequence of Artin relations we can apply to $\sigma$ is a finite sequence of commutativity relations $\sigma_1\sigma_3 = \sigma_3\sigma_1$.) If $r,s\in \{1,2,3\}$, then we will construct a family of maximal $q$-resistance $\left(r,s\right)$-type $\sigma$-paths such that their weights contribute the highest $q$-degree term in the $\left(r,s\right)$-entry of $\beta_4\left(\sigma\right)$. The cardinality of the family will be a power of $2$, and thus the leading coefficient of this highest $q$-degree term will be non-zero modulo each prime $t\neq 2$. If $1\leq j\leq n$, then we write $P_j = \prod_{i=1}^{j} \sigma_1^{a_i}\sigma_3^{b_i}\sigma_2^{c_i}$ and $P_j' = \left(\prod_{i=1}^{j-1} \sigma_1^{a_i}\sigma_3^{b_i}\sigma_2^{c_i}\right)\sigma_1^{a_j}\sigma_3^{b_j}$, for ease of notation.

Let $X_1'$ be an $\left(r,\cdot\right)$-type $P_1'$-path ($P_1' = \sigma_1^{a_1}\sigma_3^{b_1}$). If $r = 1$, then $X_1'$ is a constant path where every vertex is $1$. Indeed, in the pictures of $\sigma_1$ and $\sigma_3$ in Figure~\ref{fig1}, there is no vertex change from $1$ to another vertex. Furthermore, $X_1'$ is a good $\left(r,1\right)$-type $P_1'$-path. Lemma~\ref{lemmagoodext} \textbf{(i)} implies that we can extend $X_1'$ to a $P_1$-path by a vertex change at any $\sigma_2$ in $\sigma_2^{c_1}$ ($P_1 = \sigma_1^{a_1}\sigma_3^{b_1}\sigma_2^{c_1}$). If the vertex change is immediate, a $1\to 2$ vertex change at the first $\sigma_2$ in $\sigma_2^{c_1}$, then the $q$-resistance of the resulting $P_1$-path is maximized. Let us denote the resulting $P_1$-path by $X_1$. The hypothesis that $c_1\geq 2$ implies that $X_1$ is a good $\left(r,2\right)$-type $P_1$-path. Furthermore, by construction, the maximal $q$-resistance $\left(r,\cdot \right)$-type $P_1$-path is $X_1$. If $r = 3$, then similar reasoning demonstrates that the maximal $q$-resistance $\left(r,\cdot \right)$-type $P_1$-path is a good $\left(r,2\right)$-type $P_1$-path.

We now potentially have two choices to extend the maximal $q$-resistance $\left(r,2\right)$-type $P_1$-path $X_1$ to a maximal $q$-resistance $\left(r,s\right)$-type $P_2'$-path for some $s\in \{1,3\}$ ($P_2' = \sigma_1^{a_1}\sigma_3^{b_1}\sigma_2^{c_1}\sigma_1^{a_2}\sigma_2^{b_2}$). We can extend $X_1$ either by a $2\to 1$ vertex change at the first $\sigma_1$ in $\sigma_1^{a_2}$ or by a $2\to 3$ vertex change at the first $\sigma_3$ in $\sigma_3^{b_2}$. If either $a_2 = 0$ or $b_2 = 0$, then only one of these extensions is possible, and the hypothesis on the exponents in $\sigma$ implies that either $b_2\geq 2$ or $a_2\geq 2$, respectively. If $a_2 = 0$, then the maximal $q$-resistance $\left(r,\cdot\right)$-type $P_2'$-path is a good $\left(r,3\right)$-type $P_2'$-path. If $b_2 = 0$, then the maximal $q$-resistance $\left(r,\cdot\right)$-type $P_2'$-path is a good $\left(r,1\right)$-type $P_2'$-path. If either $a_2 = 0$ or $b_2 = 0$, then let us denote the resulting maximal $q$-resistance good $\left(r,\cdot\right)$-type $P_2'$-path by $X_2'$. Proposition~\ref{m=ap} implies that if $a_2 = 0$, then the weight of the $\left(r,3\right)$-type $P_2'$-path $X_2'$ contributes the highest $q$-degree term to the $\left(r,3\right)$-entry of $\beta_4\left(P_2'\right)$, and if $b_2 = 0$, then the weight of the $\left(r,1\right)$-type $P_2'$-path $X_2'$ contributes the highest $q$-degree term to the $\left(r,1\right)$-entry of $\beta_4\left(P_2'\right)$. 

If $a_2,b_2>0$, then we denote the maximal $q$-resistance extensions of the $\left(r,2\right)$-type $P_1$-path $X_1$ to an $\left(r,1\right)$-type $P_2'$-path and an $\left(r,3\right)$-type $P_2'$-path by $X_{2,1}'$ and $X_{2,2}'$, respectively. If $a_2>1$, then $X_{2,1}$ and $X_{2,2}$ are good $P_2'$-paths. However, if $a_2 = 1$, then only $X_{2,2}$ is a good $P_2'$-path. If $a_2 > 1$, then Lemma~\ref{lemmagoodext} \textbf{(i)} implies that the maximal $q$-resistance extension of $X_{2,1}'$ to an $\left(r,2\right)$-type $P_2$-path is by a $1\to 2$ vertex change at the first $\sigma_2$ in $\sigma_2^{c_2}$ ($P_2 = \sigma_1^{a_1}\sigma_3^{b_1}\sigma_2^{c_1}\sigma_1^{a_2}\sigma_3^{b_2}\sigma_2^{c_2}$). In general, Lemma~\ref{lemmagoodext} \textbf{(i)} implies that the maximal $q$-resistance extension of $X_{2,2}'$ to an $\left(r,2\right)$-type $P_2$-path is by a $3\to 2$ vertex change at the first $\sigma_2$ in $\sigma_2^{c_2}$. The hypothesis that $c_2\geq 2$ implies that the resulting $\left(r,2\right)$-type $P_2$-path in each case is a good $\left(r,2\right)$-type $P_2$-path.

If $a_2 = b_2>1$, then the $q$-resistances of the two $\left(r,2\right)$-type $P_2$-paths are equal. In Figure~\ref{fig1}, the picture of $\sigma_2$ demonstrates that the weights of the two $\left(r,2\right)$-type $P_2$-paths have the same sign, and thus do not cancel even if their $q$-resistances are equal. 

Proposition~\ref{m=ap} implies that the weights of the $\left(r,1\right)$-type $P_2'$-path $X_{2,1}'$ and the $\left(r,3\right)$-type $P_2'$-path $X_{2,2}'$ contribute the highest $q$-degree terms to the $\left(r,1\right)$-entry and the $\left(r,3\right)$-entry of $\beta_4\left(P_2'\right)$, respectively. Lemma~\ref{lemmagoodext} \textbf{(iv)} implies that the $\left(r,1\right)$-entry and the $\left(r,3\right)$-entry of $\beta_4\left(P_2\right)$ are equal to the $\left(r,1\right)$-entry and the $\left(r,3\right)$-entry of $\beta_4\left(P_2'\right)$, respectively. Similarly, Proposition~\ref{m=ap} implies that a maximal $q$-resistance $\left(r,2\right)$-type $P_2$-path (there could be two if $a_2 = b_2>1$) contributes to the highest $q$-degree term in the $\left(r,2\right)$-entry of $\beta_4\left(P_2\right)$. If $a_2>0$ and $b_2>0$, then we deduce that the absolute value of the leading coefficient of every entry in the $r$th row of $\beta_4\left(P_2\right)$ is a power of $2$.

We can continue this process to construct a family of maximal $q$-resistance $P_n$-paths ($P_n = \sigma$). If at least one $\sigma_1$ and at least one $\sigma_3$ appears in $\sigma$, then the weights of the family of maximal $q$-resistance $\left(r,s\right)$-type $\sigma$-paths contribute the highest $q$-degree term in the $\left(r,s\right)$-entry of $\beta_4\left(\sigma\right)$. The family has cardinality a power of $2$ by construction. (If no $\sigma_1$ appears in $\sigma$, then the $\left(2,1\right)$-entry and the $\left(3,1\right)$-entry of $\beta_4\left(\sigma\right)$ are zero since there is no $2\to 1$ or $3\to 1$ vertex change in the pictures of $\sigma_2$ and $\sigma_3$ in Figure~\ref{fig1}. However, the leading coefficient of every other entry in $\beta_4\left(\sigma\right)$ is a power of $2$ by the same reasoning. A similar remark applies to the case where no $\sigma_3$ appears in $\sigma$, where instead the $\left(1,3\right)$-entry and the $\left(2,3\right)$-entry of $\beta_4\left(\sigma\right)$ are zero.)

We have considered the case $r\in \{1,3\}$. If $r = 2$, then a similar argument to extending $P_1$-paths to $P_2'$-paths in the case $r\in \{1,3\}$, shows that there is a maximal $q$-resistance $\left(r,1\right)$-type $P_1'$-path $X_{1,1}'$ if $a_1>0$, and a maximal $q$-resistance $\left(r,3\right)$-type $P_1'$-path $X_{1,2}'$ if $b_1>0$. If either $a_1 = 0$ or $b_1 = 0$, then only one of these $P_1'$-paths is possible and the hypothesis on the exponents in $\sigma$ implies that it is a good $P_1'$-path. If $a_1,b_1>0$ and $a_1>1$, then both $P_1'$-paths are good. If $a_1 = 1$, then only $X_{1,2}'$ is a good $P_1'$-path. The construction of a family of maximal $q$-resistance $\left(2,s\right)$-type $\sigma$-paths for each $s\in \{1,2,3\}$ (where the cardinality of the family is a power of $2$) is similar to the case $r\in \{1,3\}$ (based on Lemma~\ref{lemmagoodext} \textbf{(i)}).

Therefore, if at least one $\sigma_1$ and one $\sigma_3$ appears in $\sigma$, then the absolute value of the leading coefficient of every entry in $\beta_4\left(\sigma\right)$ is a power of $2$. We observe that the argument is constructive in that we can explicitly determine the highest $q$-degree terms in all the entries of the Burau matrix $\beta_4\left(\sigma\right)$. In fact, the theory of good and bad paths is sufficiently powerful to determine the \textit{$q$-polynomials} in all the entries of the Burau matrix $\beta_4\left(\sigma\right)$. However, in this case, we also need to consider bad paths and apply the other parts of Lemma~\ref{lemmagoodext}, so it is more involved.

Unfortunately, the strategy of constructing maximal $q$-resistance $\sigma$-paths by following the indices of the generators in the minimal form of $\sigma$ fails when there are isolated generators in the minimal form of $\sigma$. The case of isolated generators (which we will refer to as \textit{blocks} in $\sigma$ (Definition~\ref{blockroad})) requires different methods that we will study in Section~\ref{mains} and Section~\ref{smainstronger}. However, the main idea of the following Subsection~\ref{subsecpathextensions} and the first half of Subsection~\ref{subsecstrongregularity} is that we can construct maximal $q$-resistance $\sigma$-paths using the technique illustrated in this example for a generic family of positive braids (which constitutes a special class of what we will later refer to as \textit{roads} in $\sigma$ (Definition~\ref{blockroad})). 
\end{example}

\subsection{Notation for $P$-paths}
\label{subsecnotation}

In this subsection, we introduce the following Notation~\ref{notnumgb} for the weighted numbers of various types (good, bad, $1$-switch, $3$-switch etc.) of $P$-paths that we will use throughout the paper. The weighted number of $P$-paths of a certain type is a polynomial in $q$ and we also introduce notation for the coefficients and the degree of this polynomial for each type. 

\begin{notation}
\label{notnumgb}
We fix $r\in \{1,2,3\}$. 
\begin{description}
\item[(a)] Let $P=\prod_{i=1}^{p-1} \sigma_1^{a_i}\sigma_3^{b_i}\sigma_2^{c_i}$ be a $2$-subproduct of $\sigma$.
\begin{description}
\item[(i)] Let us assume $c_{p-1}\geq 2$. We write $w_{\lambda}^{P}$ for the weighted number of $\left(r,2\right)$-type $P$-paths with no vertex change at the last $\sigma_2$. (If either $a_p = 0$ or $b_p = 0$, then the set of $P$-paths with no vertex change at the last $\sigma_2$ is a proper subset of the set of good $P$-paths. If $a_p,b_p>0$, then the two sets coincide.)

Let us assume $c_{p-1} = 1$. We write $w_{\lambda}^{P}$ for the weighted number of good $\left(r,2\right)$-type $P$-paths. (If $p>2$, then Lemma~\ref{minimalform} \textbf{(ii)} implies that either $a_p = 0$ or $b_p = 0$. If $a_p = 0$, then $w_{\lambda}^{P}$ is the weighted number of $1$-switch $\left(r,2\right)$-type $P$-paths. If $b_p = 0$, then $w_{\lambda}^{P}$ is the weighted number of $3$-switch $\left(r,2\right)$-type $P$-paths.)

We write $w_{\lambda}^{P} =\sum_{i=0}^{d_{\lambda}^{P}} \lambda_{i}^{P}q^{d_{\lambda}^{P}-i}$, where $d_{\lambda}^{P}$ is the degree of $w_{\lambda}^{P}$.

\item[(ii)] Let us assume $c_{p-1}\geq 2$. We write $w^{P}_{\mu,1}$ for the weighted number of $1$-switch $\left(r,2\right)$-type $P$-paths. We write $w^{P}_{\mu,3}$ for the weighted number of $3$-switch $\left(r,2\right)$-type $P$-paths.

We write $w_{\mu,1}^{P} = \sum_{i=0}^{d_{\mu,1}^{P}} \mu_{1,i}^{P}q^{d_{\mu,1}^{P}-i}$ and $w_{\mu,3}^{P} = \sum_{i=0}^{d_{\mu,3}^{P}} \mu_{3,i}^{P}q^{d_{\mu,3}^{P}-i}$, where $d_{\mu,1}^{P}$ and $d_{\mu,3}^{P}$ are the degrees of $w_{\mu,1}^{P}$ and $w_{\mu,3}^{P}$, respectively. 

\item[(iii)] If $a_p = 0$, then we write $w_{\mu}^{P}$ for the weighted number of $3$-switch $\left(r,2\right)$-type $P$-paths, and we write $w_{\nu}^{P}$ for the weighted number of $\left(r,1\right)$-type $P$-paths. If $b_p = 0$, then we write $w_{\mu}^{P}$ for the weighted number of $1$-switch $\left(r,2\right)$-type $P$-paths, and we write $w_{\nu}^{P}$ for the weighted number of $\left(r,3\right)$-type $P$-paths.

We write $w_{\mu}^{P} = \sum_{i=0}^{d_{\mu}^{P}} \mu_{i}^{P}q^{d_{\mu}^{P}-i}$ and $w_{\nu}^{P} = \sum_{i=0}^{d_{\nu}^{P}} \nu_{i}^{P}q^{d_{\nu}^{P}-i}$, where $d_{\mu}^{P}$ and $d_{\nu}^{P}$ are the degrees of $w_{\mu}^{P}$ and $w_{\nu}^{P}$, respectively. 

\item[(iv)] Let us assume $c_{p-1}\geq 2$. We write $w_{\lambda,1}^{P} = w_{\lambda}^{P} + w_{\mu,3}^{P}$ and $w_{\lambda,3}^{P} = w_{\lambda}^{P} + w_{\mu,1}^{P}$. 

We write $w_{\lambda,1}^{P} = \sum_{i=0}^{d_{\lambda,1}^{P}} \lambda_{1,i}^{P}q^{d_{\lambda,1}^{P} - i}$ and $w_{\lambda,3}^{P} = \sum_{i=0}^{d_{\lambda,3}^{P}} \lambda_{3,i}^{P}q^{d_{\lambda,3}^{P} - i}$, where $d_{\lambda,1}^{P}$ and $d_{\lambda,3}^{P}$ are the degrees of $w_{\lambda,1}^{P}$ and $w_{\lambda,3}^{P}$, respectively. (We will only use this notation in Subsection~\ref{subsecproofstrongerversion}.)
\end{description}
\item[(b)] Let $P=\left(\prod_{i=1}^{p-1}\sigma_1^{a_i}\sigma_3^{b_i}\sigma_2^{c_i}\right)\sigma_1^{a_p}\sigma_3^{b_p}$ be an $s$-subproduct of $\sigma$ for $s\in \{1,3\}$.
\begin{description}
\item[(i)] Let us assume $a_p,b_p>0$. We write $w^{P}_{\lambda,1}=\sum_{i=0}^{d_{\lambda,1}^{P}} \lambda_{1,i}^{P}q^{d_{\lambda,1}^{P}-i}$ for the weighted number of good $\left(r,1\right)$-type $P$-paths, where $d_{\lambda,1}^{P}$ is the degree of $w_{\lambda,1}^{P}$. We write $w^{P}_{\lambda,3}=\sum_{i=0}^{d_{\lambda,3}^{P}} \lambda_{3,i}^{P}q^{d_{\lambda,3}^{P}-i}$ for the weighted number of good $\left(r,3\right)$-type $P$-paths, where $d_{\lambda,3}^{P}$ is the degree $w_{\lambda,3}^{P}$. 
\item[(ii)] Let us assume $a_p,b_p>0$. We write $w_{\mu,1}^{P}=\sum_{i=0}^{d_{\mu,1}^{P}} \mu_{1,i}^{P}q^{d_{\mu,1}^{P}-i}$ for the weighted number of bad $\left(r,1\right)$-type $P$-paths, where $d_{\mu,1}^{P}$ is the degree of $w_{\mu,1}^{P}$. We write $w_{\mu,3}^{P}=\sum_{i=0}^{d_{\mu,3}^{P}} \mu_{3,i}^{P}q^{d_{\mu,3}^{P}-i}$ for the weighted number of bad $\left(r,3\right)$-type $P$-paths, where $d_{\mu,3}^{P}$ is the degree of $w_{\mu,3}^{P}$. 
\item[(iii)] Let us assume either $a_p = 0$ or $b_p = 0$. If $a_p = 0$, then we write $w_{\lambda}^{P}$ for the weighted number of good $\left(r,3\right)$-type $P$-paths. If $b_p = 0$, then we write $w_{\lambda}^{P}$ for the weighted number of good $\left(r,1\right)$-type $P$-paths. 

We write $w_{\lambda}^{P} = \sum_{i=0}^{d_{\lambda}^{P}} \lambda_i^{P}q^{d_{\lambda}^{P}-i}$, where $d_{\lambda}^{P}$ is the degree of $w_{\lambda}^{P}$.
\item[(iv)] Let us assume either $a_p = 0$ or $b_p = 0$. If $a_p = 0$, then we write $w_{\mu}^{P}$ for the weighted number of bad $\left(r,3\right)$-type $P$-paths. If $b_p = 0$, then we write $w_{\mu}^{P}$ for the weighted number of bad $\left(r,1\right)$-type $P$-paths. 

We write $w_{\mu}^{P} = \sum_{i=0}^{d_{\mu}^{P}} \mu_i^{P}q^{d_{\mu}^{P}-i}$, where $d_{\mu}^{P}$ is the degree of $w_{\mu}^{P}$. 
\item[(v)] Let us assume either $a_p = 0$ or $b_p = 0$. If $a_p = 0$, then we write $w_{\nu}^{P}$ for the weighted number of $\left(r,1\right)$-type $P$-paths. If $b_p = 0$, then we write $w_{\nu}^{P}$ for the weighted number of $\left(r,3\right)$-type $P$-paths. 

We write $w_{\nu}^{P} = \sum_{i=0}^{d_{\nu}^{P}} \nu_i^{P}q^{d_{\nu}^{P}-i}$, where $d_{\nu}^{P}$ is the degree of $w_{\nu}^{P}$. 
\end{description}
\end{description}
\end{notation}

In this paper, we will show that considering only the leading coefficients and degrees of the polynomials in Notation~\ref{notnumgb} is sufficient to deduce very strong constraints on the kernel of the Burau representation $\beta_4$. The following statement combines Proposition~\ref{m=ap} with Notation~\ref{notnumgb}.

\begin{proposition}
\label{pnotationentries}
We fix $r\in \{1,2,3\}$. If $P$ is an $s$-subproduct of $\sigma$, then the following statements are true concerning the $\left(r,s\right)$-entry of $\beta_4\left(P\right)$.
\begin{description}
\item[(a)] Let $P = \prod_{i=1}^{p-1} \sigma_1^{a_i}\sigma_3^{b_i}\sigma_2^{c_i}$ be a $2$-subproduct of $\sigma$. 
\begin{description}
\item[(i)] If $c_{p-1}\geq 2$, then the $\left(r,2\right)$-entry of $\beta_4\left(P\right)$ is $w_{\lambda}^{P} + w_{\mu,1}^{P} + w_{\mu,3}^{P}$. 
\item[(ii)] If $c_{p-1} = 1$, then the $\left(r,2\right)$-entry of $\beta_4\left(P\right)$ is $w_{\lambda}^{P} + w_{\mu}^{P}$.
\item[(iii)] If $a_p = 0$, then the weighted number of good $\left(r,2\right)$-type $P$-paths is $w_{\lambda,3}^{P} = w_{\lambda}^{P} + w_{\mu,1}^{P}$. If $b_p = 0$, then the weighted number of good $\left(r,2\right)$-type $P$-paths is $w_{\lambda,1}^{P} = w_{\lambda}^{P} + w_{\mu,3}^{P}$.
\item[(iv)] Let us assume that $c_{p-1}\geq 2$. If $a_p = 0$, then $w_{\mu,1}^{P} = w_{\nu}^{P}$ and $w_{\mu,3}^{P} = w_{\mu}^{P}$. If $b_p = 0$, then $w_{\mu,1}^{P} = w_{\mu}^{P}$ and $w_{\mu,3}^{P} = -qw_{\nu}^{P}$. 
\end{description}
\item[(b)] Let $P = \left(\prod_{i=1}^{p-1} \sigma_1^{a_i}\sigma_3^{b_i}\sigma_2^{c_i}\right)\sigma_1^{a_p}\sigma_3^{b_p}$ be an $s$-subproduct of $\sigma$ for $s\in \{1,3\}$. 
\begin{description}
\item[(i)] If $a_p = 0$, then the $\left(r,3\right)$-entry of $\beta_4\left(P\right)$ is equal to $w_{\lambda}^{P} + w_{\mu}^{P}$. If $b_p = 0$, then the $\left(r,1\right)$-entry of $\beta_4\left(P\right)$ is equal to $w_{\lambda}^{P} + w_{\mu}^{P}$. 
\item[(ii)] Let us assume that $a_p = 1$ and $b_p >0$. In this case, $w_{\lambda,1}^{P} = 0$ and the $\left(r,1\right)$-entry of $\beta_4\left(P\right)$ is $w_{\mu,1}^{P}$. Furthermore, the $\left(r,3\right)$-entry of $\beta_4\left(P\right)$ is $w_{\lambda,3}^{P} + w_{\mu,3}^{P}$.
\item[(iii)] Let us assume that $a_p>1$ and $b_p>0$. In this case, the $\left(r,1\right)$-entry of $\beta_4\left(P\right)$ is $w_{\lambda,1}^{P} + w_{\mu,1}^{P}$. Furthermore, $w_{\mu,3}^{P} = 0$ and the $\left(r,3\right)$-entry of $\beta_4\left(P\right)$ is $w_{\lambda,3}^{P}$.
\end{description}
\end{description}
\end{proposition}
\begin{proof}
\begin{description}
\item[(a)]
\begin{description}
\item[(i) - (ii)] The statements are a consequence of Proposition~\ref{m=ap}.
\item[(iii)] The statement is a consequence of Definition~\ref{defgoodbadpath} \textbf{(ii)}.
\item[(iv)] Let $P_0$ be the maximal $s$-subproduct contained in $P$ for $s\in \{1,3\}$. Lemma~\ref{lemmagoodext} \textbf{(iv)} implies that there is a one-to-one correspondence between $\left(r,i\right)$-type $P_0$-paths and $\left(r,i\right)$-type $P$-paths such that the weights of corresponding paths are equal. Lemma~\ref{lemmagoodext} \textbf{(i)} - \textbf{(ii)} and the hypothesis that $c_{p-1}\geq 2$ imply that the extension of an admissible $\left(r,i\right)$-type $P_0$-path by an $i\to 2$ vertex change at the last $\sigma_2$ in $\sigma_2^{c_{p-1}}$ is an admissible $i$-switch $\left(r,2\right)$-type $P$-path. The picture of $\sigma_2$ in Figure~\ref{fig1} shows that if $i = 1$, then the weight of the $P$-path is equal to the weight of the $P_0$-path, and if $i = 3$, then the weight of the $P$-path is the product of $-q$ with the weight of the $P_0$-path.
\end{description}
\item[(b)]
\begin{description}
\item[(i)] The statement is a consequence of Proposition~\ref{m=ap}.
\item[(ii)] Let $P_0'$ be the maximal $2$-subproduct contained in $P$. Let $P_0$ be the maximal $s$-subproduct contained in $P_0'$ for $s\in \{1,3\}$. An $\left(r,1\right)$-type $P$-path is either the extension of an $\left(r,2\right)$-type $P_0'$-path by a $2\to 1$ vertex change, or the extension of an $\left(r,1\right)$-type $P_0$-path with no further vertex change. An $\left(r,1\right)$-type $P$-path of the former type is bad since $a_p = 1$. Lemma~\ref{lemmagoodext} \textbf{(iv)} implies that an $\left(r,1\right)$-type $P$-path of the latter type is not admissible. Definition~\ref{defgoodbadpath} \textbf{(i)} implies that $w_{\lambda,1}^{P} = 0$. The remainder of the statement is a consequence of Proposition~\ref{m=ap}.
\item[(iii)] Definition~\ref{defgoodbadpath} \textbf{(i)} implies that $w_{\mu,3}^{P} = 0$. The remainder of the statement is a consequence of Proposition~\ref{m=ap}.
\end{description}
\end{description}
\end{proof}

\subsection{Path extensions I}
\label{subsecpathextensions}

In this subsection, we develop the theory of extensions of $P$-paths. Proposition~\ref{pnotationentries} implies that we can recover the Burau matrix $\beta_4\left(P\right)$ from the weighted numbers of $P$-paths of various types. However, the weighted numbers of $P$-paths of various types contain more information than the Burau matrix $\beta_4\left(P\right)$. 

Lemma~\ref{lemmagoodext} \textbf{(i)} - \textbf{(iii)} imply the fundamental principle that good $P$-paths accumulate $q$-resistance faster than bad $P$-paths. Indeed, we have seen in Example~\ref{exmain} that the strategy for constructing maximal $q$-resistance extensions of a $P$-path is to ensure that the vertices of the extension follow the indices of the corresponding generators as much as possible. Also, there are no restrictions on the admissible extensions of good $P$-paths. 

Let us assume that a maximal $q$-resistance $\left(r,\cdot\right)$-type $P$-path $X$ is a good $\left(r,s\right)$-type $P$-path and has sufficiently high $q$-resistance compared to other $\left(r,\cdot\right)$-type $P$-paths (either bad $P$-paths or $P$-paths that are not $\left(r,s\right)$-type $P$-paths). In the remainder of Section~\ref{pathgraph} and in Section~\ref{mains}, we will show that the $P$-path $X$ extends to maximal $q$-resistance $\left(r,\cdot\right)$-type $P'$-paths, assuming generic conditions on $P'\setminus P$ (corresponding to the hypothesis that $\sigma$ is a normal braid). Furthermore, the weights of the maximal $q$-resistance extensions of $X$ to $\left(r,\cdot\right)$-type $P'$-paths will contribute the highest $q$-degree terms in the $r$th row of $\beta_4\left(P'\right)$.

In Section~\ref{smainstronger}, we will consider the case when a maximal $q$-resistance $P$-path is bad. If a maximal $q$-resistance $\left(r,s\right)$-type $P$-path $Y$ is bad, then consider a maximal $q$-resistance $\left(r,s\right)$-type $P$-path $X$. As both $P$-paths extend along $P'\setminus P$ to $P'$-paths, the good $P$-path $X$ will accumulate $q$-resistance faster than the bad $P$-path $Y$. In particular, the discrepancy in $q$-resistances will reduce as both $P$-paths extend along $P'\setminus P$ to $P'$-paths. If the weight of $X$ has opposite sign to the weight of $Y$ and the discrepancy reduces perfectly to $0$ (i.e., there is a precise relationship between the length of $P'\setminus P$ and the initial discrepancy between the $q$-resistances of $X$ and $Y$), then there can be cancellation of weights of $P'$-paths. However, in order for this situation to occur, the exponents in $P'\setminus P$ must be very small in order for the bad $P$-path $Y$ to only extend to bad $P''$-paths for $P\subseteq P''\subseteq P'$. Indeed, Lemma~\ref{lemmagoodext} \textbf{(i)} - \textbf{(iii)} implies that an extension at the second or later generator is always admissible, so if an exponent is greater than $2$, then the extension will not have a vertex change at the last generator. Furthermore, if an extension of $Y$ to a $P''$-path is good and the length of $P''\setminus P$ is not large enough for the discrepancy between the $q$-resistances of $X$ and $Y$ to reduce perfectly to $0$, then a maximal $q$-resistance $P''$-path is good. In this case, we can apply the reasoning of the previous paragraph to deduce that the maximal $q$-resistance extensions of $Y$ to $\left(r,\cdot\right)$-type $P'$-paths contribute the highest $q$-degree terms in the $r$th row of $\beta_4\left(P'\right)$.

If $P\subseteq P'$, where $P$ is an $s$-subproduct of $\sigma$ and $P'$ is an $s'$-subproduct of $\sigma$ for some $s,s'\in \{1,2,3\}$, then the preliminary work is to determine formulas for the weighted numbers of good and bad $P'$-paths in terms of the weighted numbers of good and bad $P$-paths. We apply Lemma~\ref{lemmagoodext} to do so in the remainder of this section, where $P$ is the maximal proper $s$-subproduct of $P'$ (from which the general case can be inductively established). We derive only formulas for the degrees and leading coefficients of these polynomials in the case where a maximal $q$-resistance $P$-path is good and has sufficiently high $q$-resistance compared to other $P$-paths. We will only need this case to study normal braids. In Subsection~\ref{subsecpathextensionsII}, we will derive more general formulas to study weakly normal braids when a maximal $q$-resistance $P$-path is bad.

Firstly, we parametrize extensions of $P$-paths to $P'$-paths, where $P$ is the maximal proper $2$-subproduct of an $s'$-subproduct $P'$ of $\sigma$, for some $s'\in \{1,3\}$. 

\begin{definition}
\label{defpath21/3}
Let $P = \prod_{i=1}^{p-1} \sigma_1^{a_i}\sigma_3^{b_i}\sigma_2^{c_i}$ be a $2$-subproduct of $\sigma$ and let $P' = \left(\prod_{i=1}^{p-1} \sigma_1^{a_i}\sigma_3^{b_i}\sigma_2^{c_i}\right)\sigma_1^{a_p}\sigma_3^{b_p}$. Let $W$ be an $\left(r,2\right)$-type $P$-path. 
\begin{description}
\item[(a)] Let us assume that either $a_p = 0$ or $b_p = 0$. If $a_p = 0$, then we define $s' = 3$, and if $b_p = 0$, then we define $s' = 1$. If $a_p = 0$, then we define $\phi_p = b_p$, and if $b_p = 0$, then we define $\phi_p = a_p$. If $1\leq \alpha\leq \phi_p$, then we define $W^{\alpha}$ to be the extension of $W$ to an $\left(r,s'\right)$-type $P'$-path by a vertex change at the $\alpha$th generator in $P'\setminus P$. 
\item[(b)] Let us assume that $a_p,b_p>0$. If $1\leq \alpha\leq a_p$, then we define $W^{\alpha,0}$ to be the extension of $W$ to an $\left(r,1\right)$-type $P'$-path by a vertex change at the $\alpha$th $\sigma_1$ in $\sigma_1^{a_p}$. If $1\leq \beta\leq b_p$, then we define $W^{0,\beta}$ to be the extension of $W$ to an $\left(r,3\right)$-type $P'$-path by a vertex change at the $\beta$th $\sigma_3$ in $\sigma_3^{b_p}$.
\end{description}
\end{definition}

We now characterize admissible extensions of $P$-paths to $P'$-paths, and determine a formula for their weights, in the case where either $a_p = 0$ or $b_p = 0$ in Definition~\ref{defpath21/3}.

\begin{proposition}
\label{ppre21/3ext}
Let $P = \prod_{i=1}^{p-1} \sigma_1^{a_i}\sigma_3^{b_i}\sigma_2^{c_i}$ be a $2$-subproduct of $\sigma$ and let $P' = \left(\prod_{i=1}^{p-1} \sigma_1^{a_i}\sigma_3^{b_i}\sigma_2^{c_i}\right)\sigma_1^{a_p}\sigma_3^{b_p}$. Let us assume that either $a_p = 0$ or $b_p = 0$. We adopt the notation of Definition~\ref{defpath21/3} \textbf{(a)}.

If $a_p = 0$, then we define $\psi = 1$ and $s = 1$. If $b_p = 0$, then we define $\psi = 0$ and $s = 3$.
\begin{description}
\item[(i)] An admissible $\left(r,s'\right)$-type $P'$-path is of the form $W^{\alpha}$ for some $\left(r,2\right)$-type $P$-path $W$. The weight of $W^{\alpha}$ is the product of $\left(-1\right)^{1-\psi}q^{\phi_p-\alpha+1-\psi}$ with the weight of $W$. 
\item[(ii)] If $t\in \{2,s\}$, then an admissible $\left(r,t\right)$-type $P'$-path is the extension of an $\left(r,t\right)$-type $P$-path without a vertex change. The weight of the $\left(r,t\right)$-type $P$-path and its extension are equal. 
\item[(iii)] If $X$ is a good $\left(r,2\right)$-type $P$-path, then $X^{1}$ is the maximal $q$-resistance extension of $X$ to an $\left(r,s'\right)$-type $P'$-path. If $Y$ is a bad $\left(r,2\right)$-type $P$-path, then $Y^{2}$ is the maximal $q$-resistance extension of $Y$ to an $\left(r,s'\right)$-type $P'$-path. (If $\phi_p = 1$, then a bad $\left(r,2\right)$-type $P$-path $Y$ does not have an admissible extension to an $\left(r,s'\right)$-type $P'$-path.)
\item[(iv)] The bad $\left(r,s'\right)$-type $P'$-paths are of the form $X^{\phi_p}$ or $Y^{\phi_p}$ for some good $\left(r,2\right)$-type $P$-path $X$ or some bad $\left(r,2\right)$-type $P$-path $Y$.
\end{description}
\end{proposition}
\begin{proof}
\begin{description}
\item[(i)] In Figure~\ref{fig1}, a vertex change $s\to s'$ cannot occur at $\sigma_{s'}$. The first statement is now a consequence of Lemma~\ref{lemmagoodext} \textbf{(iv)}. The second statement is a direct computation based on the pictures of $\sigma_1$ and $\sigma_3$ in Figure~\ref{fig1}.
\item[(ii)] In Figure~\ref{fig1}, a vertex change to $t$ cannot occur at $\sigma_{s'}$ and the weight of the constant edge $t\to t$ is $1$ in the picture of $\sigma_{s'}$.
\item[(iii)] The first statement is a consequence of \textbf{(i)} and Lemma~\ref{lemmagoodext} \textbf{(i)}. A bad $\left(r,2\right)$-type $P$-path $Y$ is an $s'$-switch $\left(r,2\right)$-type $P$-path since either $a_p = 0$ or $b_p = 0$. The second statement is now a consequence of \textbf{(i)} and Lemma~\ref{lemmagoodext} \textbf{(iii)}.
\item[(iv)] The statement is a consequence of \textbf{(i)}.
\end{description}
\end{proof}

Let $P$ be the maximal proper $2$-subproduct of an $s'$-subproduct $P'$ for some $s'\in \{1,3\}$. Let us assume that a maximal $q$-resistance $\left(r,2\right)$-type $P$-path is a good $\left(r,2\right)$-type $P$-path. In this case, we will determine formulas for the weighted numbers of good and bad $P'$-paths in terms of the weighted numbers of good and bad $P$-paths. Firstly, we do this in the context of Proposition~\ref{ppre21/3ext} (where either $a_p = 0$ or $b_p = 0$).

\begin{proposition}
\label{p21/3ext}
Let $P = \prod_{i=1}^{p-1} \sigma_1^{a_i}\sigma_3^{b_i}\sigma_2^{c_i}$ be a $2$-subproduct of $\sigma$ and let $P' = \left(\prod_{i=1}^{p-1} \sigma_1^{a_i}\sigma_3^{b_i}\sigma_2^{c_i}\right)\sigma_1^{a_p}\sigma_3^{b_p}$. We assume that either $a_p = 0$ or $b_p = 0$. We adopt the notation of Proposition~\ref{ppre21/3ext}.

\begin{description}
\item[(i)] Let us assume that $\phi_p>1$. If $d_{\lambda}^{P}\geq d_{\mu}^{P}$, then $d_{\lambda}^{P'} = d_{\lambda}^{P} + \phi_p - \psi > d_{\mu}^{P'}$ and $\lambda_{0}^{P'} = \left(-1\right)^{1-\psi}\lambda_{0}^{P}$. If $d_{\lambda}^{P}>d_{\mu}^{P}$, then $d_{\mu}^{P'} = d_{\lambda}^{P} + 1 - \psi$ and $\mu_{0}^{P'} = \left(-1\right)^{1-\psi}\lambda_{0}^{P}$. 
\item[(ii)] If $\phi_p = 1$, then $w_{\lambda}^{P'} = 0$ and $w_{\mu}^{P'} = \left(-1\right)^{1-\psi}w_{\lambda}^{P}$. 
\item[(iii)] We have $w_{\nu}^{P'} = w_{\nu}^{P}$. 
\item[(iv)] The $\left(r,2\right)$-entry of $\beta_4\left(P'\right)$ is equal to the $\left(r,2\right)$-entry of $\beta_4\left(P\right)$.
\end{description}
\end{proposition}
\begin{proof}
\begin{description}
\item[(i)] If $d_{\lambda}^{P}\geq d_{\mu}^{P}$, then Proposition~\ref{ppre21/3ext} \textbf{(i)} and \textbf{(iii)} imply that a maximal $q$-resistance $\left(r,s'\right)$-type $P'$-path is good and of the form $X^{1}$, where $X$ is a maximal $q$-resistance good $\left(r,2\right)$-type $P$-path. (The $\left(r,s'\right)$-type $P'$-path $X^{1}$ is good since $\phi_p>1$.) 

If $d_{\lambda}^{P}>d_{\mu}^{P}$, then Proposition~\ref{ppre21/3ext} \textbf{(i)} and \textbf{(iv)} imply that a maximal $q$-resistance bad $\left(r,s'\right)$-type $P'$-path is of the form $X^{\phi_p}$, where $X$ is a maximal $q$-resistance good $\left(r,2\right)$-type $P$-path.
\item[(ii)] If $\phi_p = 1$, then Proposition~\ref{ppre21/3ext} \textbf{(i)} and \textbf{(iii)} imply that every admissible $\left(r,s'\right)$-type $P'$-path is of the form $X^{1}$, where $X$ is a good $\left(r,2\right)$-type $P$-path. Furthermore, $X^{1}$ is a bad $\left(r,s'\right)$-type $P'$-path.
\item[(iii)] The statement is a consequence of Proposition~\ref{ppre21/3ext} \textbf{(ii)}. 
\item[(iv)] The statement is a consequence of Proposition~\ref{m=ap} and Proposition~\ref{ppre21/3ext} \textbf{(ii)}.
\end{description}
\end{proof}

We now characterize admissible extensions of $P$-paths to $P'$-paths, and determine a formula for their weights, in the case where $a_p,b_p>0$ in Definition~\ref{defpath21/3}.

\begin{proposition}
\label{ppre21/3bothext}
Let $P = \prod_{i=1}^{p-1} \sigma_1^{a_i}\sigma_3^{b_i}\sigma_2^{c_i}$ be a $2$-subproduct of $\sigma$ and let $P' = \left(\prod_{i=1}^{p-1} \sigma_1^{a_i}\sigma_3^{b_i}\sigma_2^{c_i}\right)\sigma_1^{a_p}\sigma_3^{b_p}$. Let us assume that $a_p,b_p>0$.
\begin{description}
\item[(i)] An $\left(r,1\right)$-type $P'$-path is of the form $W^{\alpha,0}$ for some $\left(r,2\right)$-type $P$-path $W$. An $\left(r,3\right)$-type $P'$-path is of the form $W^{0,\beta}$ for some $\left(r,2\right)$-type $P$-path $W$. 
\item[(ii)] The weight of $W^{\alpha,0}$ is the product of $-q^{a_p-\alpha+1}$ with the weight of $W$. The weight of $W^{0,\beta}$ is the product of $q^{b_p-\beta}$ with the weight of $W$.
\item[(iii)] An admissible $\left(r,2\right)$-type $P'$-path is the extension of an $\left(r,2\right)$-type $P$-path without a vertex change. The weight of the $\left(r,2\right)$-type $P$-path and its extension are equal.
\item[(iv)] If $X$ is a good $\left(r,2\right)$-type $P$-path, then $X^{1,0}$ is the maximal $q$-resistance extension of $X$ to an $\left(r,1\right)$-type $P'$-path and $X^{0,1}$ is the maximal $q$-resistance extension of $X$ to an $\left(r,3\right)$-type $P'$-path. If $a_p = 1$, then $X^{1,0}$ is a bad $\left(r,1\right)$-type $P'$-path, but otherwise, $X^{1,0}$ is a good $\left(r,1\right)$-type $P'$-path. In general, $X^{0,1}$ is a good $\left(r,3\right)$-type $P'$-path.
\end{description}
\end{proposition}
\begin{proof}
\begin{description}
\item[(i)] In Figure~\ref{fig1}, a vertex change to $s_1$ cannot occur at $\sigma_{s_2}$ if $\{s_1,s_2\} = \{1,3\}$. The statement is now a consequence of Lemma~\ref{lemmagoodext} \textbf{(iv)}.
\item[(ii)] The statement is a direct computation based on the pictures of $\sigma_1$ and $\sigma_3$ in Figure~\ref{fig1}.
\item[(iii)] In Figure~\ref{fig1}, a vertex change to $2$ cannot occur at either $\sigma_{1}$ or $\sigma_{3}$, and the weight of the constant edge $2\to 2$ is $1$ in the pictures of $\sigma_1$ and $\sigma_3$.
\item[(iv)] The statement is a consequence of \textbf{(ii)} and Lemma~\ref{lemmagoodext} \textbf{(i)}.
\end{description}
\end{proof}

If a maximal $q$-resistance $\left(r,2\right)$-type $P$-path is a good $\left(r,2\right)$-type $P$-path in the context of Proposition~\ref{ppre21/3bothext}, then we now determine formulas for the weighted numbers of good and bad $P'$-paths in terms of the weighted numbers of good and bad $P$-paths. 

\begin{proposition}
\label{p21/3bothext}
Let $P = \prod_{i=1}^{p-1} \sigma_1^{a_i}\sigma_3^{b_i}\sigma_2^{c_i}$ be a $2$-subproduct of $\sigma$ and let $P' = \left(\prod_{i=1}^{p-1} \sigma_1^{a_i}\sigma_3^{b_i}\sigma_2^{c_i}\right)\sigma_1^{a_p}\sigma_3^{b_p}$. We assume that $a_p,b_p>0$ and $c_{p-1}\geq 2$. 
\begin{description}
\item[(a)] Let us consider the case $d_{\lambda}^{P}>\max\{d_{\mu,1}^{P},d_{\mu,3}^{P}\}$. 
\begin{description}
\item[(i)] Let us assume that $a_p = 1$. In this case, $d_{\lambda,3}^{P'}>d_{\mu,3}^{P'}$, $d_{\lambda,3}^{P'}+1\geq d_{\mu,1}^{P'}$, and $\lambda_{3,0}^{P'} = \lambda_{0}^{P}$.
\item[(ii)] Let us assume that $a_p > 1$. In this case, $d_{\lambda,1}^{P'}>d_{\mu,1}^{P'}$, $\lambda_{1,0}^{P'} = -\lambda_{0}^{P}$, and $\lambda_{3,0}^{P'} = \lambda_{0}^{P}$.
\end{description}
\item[(b)] The $\left(r,2\right)$-entry of $\beta_4\left(P'\right)$ is equal to the $\left(r,2\right)$-entry of $\beta_4\left(P\right)$.
\end{description}
\end{proposition}
\begin{proof}
\begin{description}
\item[(a)] Firstly, the hypothesis that $d_{\lambda}^{P}>\max\{d_{\mu,1}^{P},d_{\mu,3}^{P}\}$ and Proposition~\ref{ppre21/3bothext} \textbf{(i)} and \textbf{(iv)} imply that a maximal $q$-resistance $\left(r,1\right)$-type $P'$-path is of the form $X^{1,0}$ and a maximal $q$-resistance $\left(r,3\right)$-type $P'$-path is of the form $X^{0,1}$, where $X$ is a maximal $q$-resistance good $\left(r,2\right)$-type $P$-path. 

\begin{description}
\item[(i)] The statements $d_{\lambda,3}^{P'}>d_{\mu,3}^{P'}$ and $\lambda_{3,0}^{P'} = \lambda_{0}^{P}$ are now immediate. Proposition~\ref{ppre21/3bothext} \textbf{(ii)} implies that $d_{\lambda,3}^{P'}+1\geq d_{\mu,1}^{P'}$.
\item[(ii)] The statement is now immediate.
\end{description}
\item[(b)] The statement is a consequence of Proposition~\ref{m=ap} and Proposition~\ref{ppre21/3bothext} \textbf{(iii)}.
\end{description}
\end{proof}

We parametrize extensions of $P$-paths to $P'$-paths, where $P$ is the maximal proper $s$-subproduct of a $2$-subproduct $P'$ of $\sigma$, for some $s\in \{1,3\}$. 

\begin{definition}
\label{defpath1/32}
Let $P = \left(\prod_{i=1}^{p-1} \sigma_1^{a_i}\sigma_3^{b_i}\sigma_2^{c_i}\right)\sigma_1^{a_p}\sigma_3^{b_p}$ be an $s$-subproduct of $\sigma$ for $s\in \{1,3\}$ and let $P' = \prod_{i=1}^{p} \sigma_1^{a_i}\sigma_3^{b_i}\sigma_2^{c_i}$. Let $W$ be either an $\left(r,1\right)$-type $P$-path or an $\left(r,3\right)$-type $P$-path. 

If $1\leq \alpha\leq c_p$, then we define $W^{\alpha}$ to be the extension of $W$ to an $\left(r,2\right)$-type $P'$-path by a vertex change at the $\alpha$th generator in $P'\setminus P$.
\end{definition}

We now characterize admissible extensions of $P$-paths to $P'$-paths, and determine a formula for their weights, in the context of Definition~\ref{defpath1/32}.

\begin{proposition}
\label{ppre1/32bothext}
Let $P = \left(\prod_{i=1}^{p-1} \sigma_1^{a_i}\sigma_3^{b_i}\sigma_2^{c_i}\right)\sigma_1^{a_p}\sigma_3^{b_p}$ be an $s$-subproduct of $\sigma$ for $s\in \{1,3\}$ and let $P' = \prod_{i=1}^{p} \sigma_1^{a_i}\sigma_3^{b_i}\sigma_2^{c_i}$.

\begin{description}
\item[(i)] An $\left(r,2\right)$-type $P'$-path is of the form $W^{\alpha}$, where $W$ is either an $\left(r,1\right)$-type $P$-path or an $\left(r,3\right)$-type $P$-path.
\item[(ii)] If $W$ is an $\left(r,1\right)$-type $P$-path, then the weight of $W^{\alpha}$ is the product of $q^{c_p-\alpha}$ with the weight of $W$. If $W$ is an $\left(r,3\right)$-type $P$-path, then the weight of $W^{\alpha}$ is the product of $-q^{c_p-\alpha+1}$ with the weight of $W$.
\item[(iii)] If $s\in \{1,3\}$, then an admissible $\left(r,s\right)$-type $P'$-path is the extension of an $\left(r,s\right)$-type $P$-path without a vertex change. The weight of the $\left(r,s\right)$-type $P$-path and its extension are equal. 
\item[(iv)] If $W$ is a good $\left(r,s\right)$-type $P$-path for some $s\in \{1,3\}$, then $W^{1}$ is the maximal $q$-resistance extension of $W$ to an $\left(r,2\right)$-type $P'$-path. If $W$ is a bad $\left(r,s\right)$-type $P$-path for some $s\in \{1,3\}$, then $W^{2}$ is the maximal $q$-resistance extension of $W$ to an $\left(r,2\right)$-type $P'$-path. (If $c_p = 1$, then a bad $\left(r,s\right)$-type $P$-path $W$ does not have an admissible extension to an $\left(r,2\right)$-type $P'$-path.)
\end{description}
\end{proposition}
\begin{proof}
\begin{description}
\item[(i)] The statement is a consequence of Lemma~\ref{lemmagoodext} \textbf{(v)}.
\item[(ii)] The statement is a direct computation based on the picture of $\sigma_2$ in Figure~\ref{fig1}.
\item[(iii)] In Figure~\ref{fig1}, a vertex change to $s$ cannot occur at $\sigma_2$, and the weight of the constant edge $s\to s$ is $1$ in the picture of $\sigma_2$. 
\item[(iv)] The first statement is a consequence of \textbf{(ii)} and Lemma~\ref{lemmagoodext} \textbf{(i)}. The second statement is a consequence of \textbf{(ii)} and Lemma~\ref{lemmagoodext} \textbf{(ii)}.
\end{description}
\end{proof}

Let $P$ be the maximal proper $s$-subproduct of a $2$-subproduct $P'$ for some $s\in \{1,3\}$. Let us assume that a maximal $q$-resistance $\left(r,s\right)$-type $P$-path is a good $\left(r,s\right)$-type $P$-path and has sufficiently high $q$-resistance compared to other (either bad or different type) $P$-paths. In this case, we will determine formulas for the weighted numbers of good and bad $P'$-paths in terms of the weighted numbers of good and bad $P$-paths. Firstly, we do this where either $a_p = 0$ or $b_p = 0$ in the context of Proposition~\ref{ppre1/32bothext}.

\begin{proposition}
\label{p1/32ext}
Let $P = \left(\prod_{i=1}^{p-1} \sigma_1^{a_i}\sigma_3^{b_i}\sigma_2^{c_i}\right)\sigma_1^{a_p}\sigma_3^{b_p}$ be an $s$-subproduct of $\sigma$ for $s\in \{1,3\}$ and let $P' = \prod_{i=1}^{p} \sigma_1^{a_i}\sigma_3^{b_i}\sigma_2^{c_i}$. We assume that either $a_p = 0$ or $b_p = 0$. If $a_p = 0$, then we define $\psi = 1$, and if $b_p = 0$, then we define $\psi = 0$. 

\begin{description}
\item[(a)] Let us assume that $d_{\lambda}^{P}\geq d_{\mu}^{P}$ and $d_{\lambda}^{P}>d_{\nu}^{P} + \left(-1\right)^{\psi}$. 
\begin{description}
\item[(i)] If $c_p\geq 2$, then $d_{\lambda}^{P'}>\max\{d_{\mu,1}^{P'},d_{\mu,3}^{P'}\}$ and $\lambda_{0}^{P'} = \left(-1\right)^{\psi}\lambda_{0}^{P}$.  
\item[(ii)] If $c_p = 1$ and $a_p = 0 = b_{p+1}$, then $d_{\lambda}^{P'}>\max\{d_{\mu}^{P'},d_{\nu}^{P'}\}$ and $\lambda_{0}^{P'} = -\lambda_{0}^{P}$. If $c_p = 1$ and $b_p = 0 = a_{p+1}$, then $d_{\lambda}^{P'}>\max\{d_{\mu}^{P'},d_{\nu}^{P'}-1\}$ and $\lambda_{0}^{P'} = \lambda_{0}^{P}$. 
\end{description}
\item[(b)] If $s\in \{1,3\}$, then the $\left(r,s\right)$-entry of $\beta_4\left(P'\right)$ is equal to the $\left(r,s\right)$-entry of $\beta_4\left(P\right)$.
\end{description}
\end{proposition}
\begin{proof}
\begin{description}
\item[(a)] If $a_p = 0$, then we define $s_1= 3$ and $s_2 = 1$, and if $b_p = 0$, then we define $s_1 = 1$ and $s_2 = 3$. The hypothesis that $d_{\lambda}^{P}\geq d_{\mu}^{P}$ and Proposition~\ref{ppre1/32bothext} \textbf{(iv)} imply that the maximal $q$-resistance extension of an $\left(r,s_1\right)$-type $P$-path to an $\left(r,2\right)$-type $P'$-path is of the form $X^{1}$, where $X$ is a maximal $q$-resistance good $\left(r,s_1\right)$-type $P$-path. Furthermore, the hypothesis that $d_{\lambda}^{P} > d_{\nu}^{P} + \left(-1\right)^{\psi}$ and Proposition~\ref{ppre1/32bothext} \textbf{(i)} imply that a maximal $q$-resistance $\left(r,2\right)$-type $P'$-path is of the form $X^{1}$, where $X$ is a maximal $q$-resistance good $\left(r,s_1\right)$-type $P$-path. 

We are assuming that either $c_p\geq 2$, or $c_p = 1$ and either $a_p = 0 = b_{p+1}$ or $b_p = 0 = a_{p+1}$, throughout. In this case, $X^{1}$ is a good $\left(r,2\right)$-type $P'$-path. 
 \begin{description}
\item[(i)] The statement is now immediate. 
\item[(ii)] The statement that $d_{\lambda}^{P'}>d_{\mu}^{P'}$ is now immediate. Proposition~\ref{m=ap} and Proposition~\ref{pnotationentries} \textbf{(b)} \textbf{(i)} imply that $w_{\nu}^{P'} = w_{\lambda}^{P} + w_{\mu}^{P}$. If $a_p = 0$, then $X$ is an $\left(r,3\right)$-type $P$-path and Proposition~\ref{ppre1/32bothext} \textbf{(ii)} implies that $d_{\lambda}^{P'}>d_{\nu}^{P'}$ and $\lambda_{0}^{P'} = -\lambda_{0}^{P}$. If $b_p = 0$, then $X$ is an $\left(r,1\right)$-type $P$-path and Proposition~\ref{ppre1/32bothext} \textbf{(ii)} implies that $d_{\lambda}^{P'}>d_{\nu}^{P'}-1$ and $\lambda_{0}^{P'} = \lambda_{0}^{P}$.
\end{description}
\item[(b)] The statement is a consequence of Proposition~\ref{m=ap} and Proposition~\ref{ppre1/32bothext} \textbf{(iii)}. 
\end{description}
\end{proof}

We now establish a version of Proposition~\ref{p1/32ext} in the case where $a_p,b_p>0$ (instead of either $a_p = 0$ or $b_p = 0$).

\begin{proposition}
\label{p1/32bothext}
Let $P = \left(\prod_{i=1}^{p-1} \sigma_1^{a_i}\sigma_3^{b_i}\sigma_2^{c_i}\right)\sigma_1^{a_p}\sigma_3^{b_p}$ be an $s$-subproduct of $\sigma$ for $s\in \{1,3\}$ and let $P' = \prod_{i=1}^{p} \sigma_1^{a_i}\sigma_3^{b_i}\sigma_2^{c_i}$. We assume that $a_p,b_p>0$.
\begin{description}
\item[(a)] Let us consider the case $a_p>1$. 
\begin{description}
\item[(i)] Let us assume that $d_{\lambda,1}^{P}\geq d_{\mu,1}^{P}$ and $c_p\geq 2$. If $d_{\lambda,1}^{P}\neq d_{\lambda,3}^{P} + 1$, then $d_{\lambda}^{P'}>\max\{d_{\mu,1}^{P'},d_{\mu,3}^{P'}\}$. If $d_{\lambda,1}^{P} > d_{\lambda,3}^{P} + 1$, then $\lambda_{0}^{P'} = \lambda_{1,0}^{P}$. If $d_{\lambda,1}^{P} < d_{\lambda,3}^{P} + 1$, then $\lambda_{0}^{P'} = -\lambda_{3,0}^{P}$. If $d_{\lambda,1}^{P} = d_{\lambda,3}^{P} + 1$ and $\lambda_{1,0}^{P} - \lambda_{3,0}^{P}\neq 0$, then $d_{\lambda}^{P'}>\max\{d_{\mu,1}^{P'},d_{\mu,3}^{P'}\}$ and $\lambda_{0}^{P'} = \lambda_{1,0}^{P} - \lambda_{3,0}^{P}$. 
\item[(ii)] If $c_p = 1$, then the weighted number of $\left(r,2\right)$-type $P'$-paths is equal to $w_{\lambda,1}^{P} - qw_{\lambda,3}^{P}$.
\end{description}
\item[(b)] Let us consider the case $a_p = 1$ and $d_{\lambda,3}^{P} + 1\geq d_{\mu,1}^{P}$.
\begin{description}
\item[(i)] Let us assume that $d_{\lambda,3}^{P}\geq d_{\mu,3}^{P}$ and $c_p\geq 2$. In this case, $d_{\lambda}^{P'}>\max\{d_{\mu,1}^{P'},d_{\mu,3}^{P'}\}$ and $\lambda_{0}^{P'} = -\lambda_{3,0}^{P}$. 
\item[(ii)] If $c_p = 1$, then the weighted number of $\left(r,2\right)$-type $P'$-paths is equal to $-qw_{\lambda,3}^{P}$. 
\end{description} 
\item[(c)] If $s\in \{1,3\}$, then the $\left(r,s\right)$-entry of $\beta_4\left(P'\right)$ is equal to the $\left(r,s\right)$-entry of $\beta_4\left(P\right)$. 
\end{description}
\end{proposition}
\begin{proof}
\begin{description}
\item[(a)] Firstly, Proposition~\ref{pnotationentries} \textbf{(b)} \textbf{(iii)} implies that $w_{\mu,3}^{P} = 0$.
\begin{description}
\item[(i)] The hypothesis that $d_{\lambda,1}^{P}\geq d_{\mu,1}^{P}$ and Proposition~\ref{ppre1/32bothext} \textbf{(iv)} imply that a maximal $q$-resistance extension of an $\left(r,1\right)$-type $P$-path to an $\left(r,2\right)$-type $P'$-path is of the form $X^{1}$, where $X$ is a maximal $q$-resistance good $\left(r,1\right)$-type $P$-path. Proposition~\ref{ppre1/32bothext} \textbf{(iv)} implies that a maximal $q$-resistance extension of an $\left(r,3\right)$-type $P$-path to an $\left(r,2\right)$-type $P'$-path is of the form $U^{1}$, where $U$ is a maximal $q$-resistance good $\left(r,3\right)$-type $P$-path. 

If $d_{\lambda,1}^{P}>d_{\lambda,3}^{P} + 1$, then Proposition~\ref{ppre1/32bothext} \textbf{(i)} and \textbf{(ii)} imply that a maximal $q$-resistance $\left(r,2\right)$-type $P'$-path is of the form $X^{1}$, where $X$ is a maximal $q$-resistance good $\left(r,1\right)$-type $P$-path. If $d_{\lambda,1}^{P}<d_{\lambda,3}^{P} + 1$, then Proposition~\ref{ppre1/32bothext} \textbf{(i)} and \textbf{(ii)} imply that a maximal $q$-resistance $\left(r,2\right)$-type $P'$-path is of the form $U^{1}$, where $U$ is a maximal $q$-resistance good $\left(r,3\right)$-type $P$-path. If $d_{\lambda,1}^{P} = d_{\lambda,3}^{P} + 1$ and $\lambda_{1,0}^{P} - \lambda_{3,0}^{P}\neq 0$, then Proposition~\ref{ppre1/32bothext} \textbf{(i)} and \textbf{(ii)} imply that a maximal $q$-resistance $\left(r,2\right)$-type $P'$-path is either of the form $X^{1}$ or of the form $U^{1}$, where $X$ is a maximal $q$-resistance good $\left(r,1\right)$-type $P$-path and $U$ is a maximal $q$-resistance good $\left(r,3\right)$-type $P$-path. 

If $c_p\geq 2$, then $X^{1}$ and $U^{1}$ are good $\left(r,2\right)$-type $P'$-paths.
\item[(ii)] If $W$ is a bad $\left(r,s\right)$-type $P$-path for some $s\in \{1,3\}$, then Proposition~\ref{ppre1/32bothext} \textbf{(iv)} implies that $W^{1}$ is not admissible.
\end{description}
\item[(b)] Firstly, Proposition~\ref{pnotationentries} \textbf{(b)} \textbf{(ii)} implies that $w_{\lambda,1}^{P} = 0$. 
\begin{description}
\item[(i)] The hypothesis that $d_{\lambda,3}^{P}\geq d_{\mu,3}^{P}$ and Proposition~\ref{ppre1/32bothext} \textbf{(iv)} imply that a maximal $q$-resistance extension of an $\left(r,3\right)$-type $P$-path to an $\left(r,2\right)$-type $P'$-path is of the form $U^{1}$, where $U$ is a maximal $q$-resistance good $\left(r,3\right)$-type $P$-path. Proposition~\ref{ppre1/32bothext} \textbf{(iv)} implies that a maximal $q$-resistance extension of an $\left(r,1\right)$-type $P$-path to an $\left(r,2\right)$-type $P'$-path is of the form $Y^{2}$, where $Y$ is a maximal $q$-resistance bad $\left(r,1\right)$-type $P$-path. The hypothesis that $d_{\lambda,3}^{P}+1\geq d_{\mu,1}^{P}$ and Proposition~\ref{ppre1/32bothext} \textbf{(i)} and \textbf{(ii)} imply that a maximal $q$-resistance $\left(r,2\right)$-type $P'$-path is of the form $U^{1}$, where $U$ is a maximal $q$-resistance good $\left(r,3\right)$-type $P$-path. 

If $c_p\geq 2$, then $U^{1}$ is a good $\left(r,2\right)$-type $P'$-path. 
\item[(ii)] If $W$ is a bad $\left(r,s\right)$-type $P$-path for some $s\in \{1,3\}$, then Proposition~\ref{ppre1/32bothext} \textbf{(iv)} implies that $W^{1}$ is not admissible.
\end{description}
\item[(c)] The statement is a consequence of Proposition~\ref{ppre1/32bothext} \textbf{(iii)}.
\end{description}
\end{proof}

\section{Proof of Theorem~\ref{invariantdetect}}
\label{mains}
In this subsection, we will establish Theorem~\ref{invariantdetect}, which states that if $\sigma$ is a normal braid, then the leading coefficients of multiple entries in a row of $\beta_4\left(\sigma\right)$ are non-zero modulo each prime $t\neq 2$.

The strategy to prove Theorem~\ref{invariantdetect} is as follows. Let $P$ be the maximal proper $s$-subproduct of $\sigma$ for some $s\in \{1,2,3\}$. Let us assume that a maximal $q$-resistance $\left(r,\cdot\right)$-type $P$-path is a good $\left(r,s\right)$-type $P$-path $X$ and has sufficiently high $q$-resistance compared to other $\left(r,\cdot\right)$-type $P$-paths (either bad or of different type). In this case, Lemma~\ref{lemmagoodext} \textbf{(i)} implies that the maximal $q$-resistance extensions of $X$ to $\sigma$-paths are maximal $q$-resistance $\left(r,\cdot\right)$-type $\sigma$-paths. Proposition~\ref{m=ap} implies that the weight of $X$ and its maximal $q$-resistance extensions to $\left(r,\cdot\right)$-type $\sigma$-paths contribute multiple highest $q$-degree terms to the $r$th row of $\beta_4\left(\sigma\right)$. If $P$ is an $s$-subproduct of $\sigma$, then we will precisely formulate this hypothesis on $P$-paths in Definition~\ref{defrregular} and refer to it as \textit{strong $r$-regularity}. If $\sigma$ is a normal braid, then we will use the theory of path extensions developed in Subsection~\ref{subsecpathextensions} to show that strong $r$-regularity is an inductive property of $s$-subproducts of $\sigma$.

In Subsection~\ref{subsecconstraint}, we will introduce the product decomposition of the minimal form of $\sigma$ into \textit{blocks} and \textit{roads} in $\sigma$ (Definition~\ref{blockroad}). We will introduce a classification of the blocks in $\sigma$ (Definition~\ref{blockclass}) into $2$-blocks ($\sigma_1^{a_p}\sigma_3^{b_p}\sigma_2\sigma_1^{a_{p+1}}$) and $3$-blocks ($\sigma_2^{c_{p-1}}\sigma_3\sigma_2^{c_p}$). We will reinterpret the condition that $\sigma$ is a normal braid as the condition that the $2$-blocks and $3$-blocks in $\sigma$ are \textit{normal blocks} (Proposition~\ref{pnormalblock}). 

In Subsection~\ref{subsecstrongregularity}, we will show that if $P\subseteq P'$, where $P$ is a strongly $r$-regular $s$-subproduct of $\sigma$ and $P'$ is an $s'$-subproduct of $\sigma$, then $P'$ is strongly $r$-regular if $P'\setminus P$ is a road (Corollary~\ref{lemmaroadregular}) and $P'\setminus P$ is a normal block (Corollary~\ref{lemmablockstronglyregular}). Finally, we will establish the base case of the induction (Lemma~\ref{initial}) and Theorem~\ref{invariantdetect}.

In the rest of the paper, we will assume that \textit{every path under consideration is admissible} and generally omit the adjective ``admissible" preceding ``$\sigma$-path", unless we wish to emphasize the admissibility (e.g., if it is a property that we need to establish). Furthermore, $\sigma=\prod_{i=1}^{n} \sigma_1^{a_i}\sigma_3^{b_i}\sigma_2^{c_i}$ will denote the minimal form of the positive braid $\sigma$ where $a_i+b_i$ is a positive integer for $2\leq i\leq n$ and $c_i$ is a positive integer for $1\leq i\leq n-1$.

\subsection{The product decomposition of a positive braid into blocks and roads}
\label{subsecconstraint}

In this subsection, we introduce a new decomposition of the minimal form of a positive braid.

\begin{definition}
\label{blockroad}
A \textit{block} in $\sigma$ is an expression either of the form $\sigma_1^{a_p}\sigma_3^{b_p}\sigma_2\sigma_1^{a_{p+1}}$ for $a_p,a_{p+1}>0$ (a $2\textit{-block}$) or of the form $\sigma_2^{c_{p-1}}\sigma_3\sigma_2^{c_p}$ for $c_{p-1},c_p>0$ (a $3\textit{-block}$), in the minimal form of $\sigma$. A \textit{subroad} in $\sigma$ is a subproduct of the minimal form of $\sigma$ that does not overlap with a block. A \textit{road} in $\sigma$ is a subproduct of the minimal form of $\sigma$, maximal with respect to the property of being a subroad.

The \textit{block-road decomposition} of $\sigma$ is the unique product decomposition $\sigma = \prod_{i=1}^{k} B_iR_i$, where $B_i$ is a block and $R_i$ is a road for each $1\leq i\leq k$ (with the possibility that $B_1$ and/or $R_k$ may be empty).
\end{definition}

A side of the equation defining a braid relation corresponds to a block. For example, the occurrence of $\sigma_1\sigma_2\sigma_1$ corresponds to a $2$-block and the occurrence of $\sigma_2\sigma_3\sigma_2$ corresponds to a $3$-block. On the other hand, the only Artin relation that can be applied to a subroad is the commutativity relation $\sigma_1\sigma_3 = \sigma_3\sigma_1$ and thus product expansions of subroads are unique (up to rearrangements of $\sigma_1$s and $\sigma_3$s).
  
We will construct (admissible) $\sigma$-paths by inductively extending $P$-paths, where $P$ is a subproduct of blocks and roads in $\sigma$. The block-road decomposition of $\sigma$ is convenient because the manner in which $P$-paths extend to $P'$-paths is according to whether $P'\setminus P$ is a $2$-block, $3$-block, or a road.

We introduce the following classification of blocks in the minimal form of $\sigma$.

\begin{definition}
\label{blockclass}
The \textit{initial block} in $\sigma$ is $\sigma_1^{a_1}\sigma_2\sigma_1^{a_2}\sigma_3^{b_2}$ for $a_1,a_2>0$, if it exists in the minimal form of $\sigma$. A \textit{singular block} in $\sigma$ is either a $2$-block of the form $\sigma_1^{a_{p}}\sigma_3\sigma_2\sigma_1^{a_{p+1}}$, or a $3$-block of the form $\sigma_2^{c_{p-1}}\sigma_3\sigma_2$ that is not the end of the minimal form of $\sigma$ (i.e., for a singular $3$-block, we require $a_{p+1}>0$). A \textit{generic block} in $\sigma$ is a block that is neither an initial block nor a singular block.
\end{definition}

We constrain the placement of an initial block or a singular block in $\sigma$, and also the exponents of generators immediately preceding $2$-blocks and $3$-blocks.

\begin{proposition}
\label{initialblocksingularblockconstraint}
The following statements concerning blocks in $\sigma$ are true:
\begin{description}
\item[(i)] If $B$ is a $2$-block $\sigma_1^{a_p}\sigma_3^{b_p}\sigma_2\sigma_1^{a_{p+1}}$ with $b_p = 0$, then $p=1$. If, in addition, $b_{p+1} = 0$, then $B$ is the initial block.
\item[(ii)] If $B$ is a singular block, then a $\sigma_3$ cannot precede $B$ in the minimal form of $\sigma$. In particular, if a singular block exists, then it is unique and it is the first non-initial block.
\item[(iii)] Let us assume that $B = \sigma_1^{a_p}\sigma_3^{b_p}\sigma_2\sigma_1^{a_{p+1}}$ is a non-initial $2$-block. In general, $b_{p+1} = 0$. If either $p>2$ or $p=2$ and $b_1 > 0$, then $c_{p-1}\geq 2$.
\item[(iv)] Let us assume that $B = \sigma_2^{c_{p-1}}\sigma_3\sigma_2^{c_p}$ is a $3$-block. In general, $b_{p-1} = 0$. If $p>3$, then $c_{p-2}\geq 2$. 
\end{description}
\end{proposition}
\begin{proof}
We use the set of constraints on the minimal form of $\sigma$ in Lemma~\ref{minimalform}. 
\begin{description}
\item[(i)] The statement is a consequence of Lemma~\ref{minimalform} \textbf{(ii)}.
\item[(ii)] The first statement is a consequence of Lemma~\ref{minimalform} \textbf{(iv)}. Furthermore, a non-initial block contains a $\sigma_3$. The second statement is a consequence of this and the first statement.
\item[(iii)] The statement is a consequence of Lemma~\ref{minimalform} \textbf{(ii)}. 
\item[(iv)] The statement is a consequence of Lemma~\ref{minimalform} \textbf{(iii)}.
\end{description}
\end{proof}

We can further constrain the existence of a singular block in a positive braid indivisible by $\Delta$.

\begin{proposition}
\label{pblockDelta}
If $\sigma$ is a positive braid indivisible by $\Delta$, then an initial block and a singular block do not simultaneously exist in $\sigma$.
\end{proposition}
\begin{proof}
Let us assume, for a contradiction, that there is an initial block $B$ and a singular block $B'$ in $\sigma$. We can write the beginning of the initial block $B$ as $\sigma_1^{a_1}\sigma_2\sigma_1^{a_2} = \left(\sigma_1\sigma_2\sigma_1\right) \sigma_2^{a_1-1}\sigma_1^{a_2-1}$ by $a_1-1$ applications of the braid relation $\sigma_1\sigma_2\sigma_1 = \sigma_2\sigma_1\sigma_2$. 

If $B'$ is a singular $2$-block, then a copy of $\sigma_3\sigma_2\sigma_1$ appears in $B'$. If $B'$ is a singular $3$-block, then Lemma~\ref{minimalform} \textbf{(ii)} implies that a copy of $\sigma_3\sigma_2\sigma_1$ appears in the minimal form of $\sigma$, such that $\sigma_3\sigma_2$ is the end of $B'$ and $\sigma_1$ appears immediately after $B'$.

We also have $\sigma_i\left(\sigma_3\sigma_2\sigma_1\right) = \left(\sigma_3\sigma_2\sigma_1\right)\sigma_{i+1}$ for $i\in \{1,2\}$ by the application of a braid relation and a commutativity relation. Furthermore, Proposition~\ref{initialblocksingularblockconstraint} \textbf{(ii)} implies that a $\sigma_3$ does not appear in $\sigma$ before the copy of $\sigma_3\sigma_2\sigma_1$. In particular, we can apply a finite sequence of Artin relations to $\sigma$ such that a copy of $\sigma_3\sigma_2\sigma_1$ appears immediately after the copy of $\sigma_1\sigma_2\sigma_1$ at the beginning of the initial block $B$. In other words, there is a copy of $\Delta = \left(\sigma_1\sigma_2\sigma_1\right)\left(\sigma_3\sigma_2\sigma_1\right)$ at the beginning of $\sigma$. However, this contradicts the hypothesis that $\sigma$ is indivisible by $\Delta$.
\end{proof}

We can paraphrase the definition of a normal braid (Definition~\ref{normal}) in the Introduction more naturally using the language of blocks. Firstly, we define the notion of a normal block.

\begin{definition}
\label{dnormalblock}
A non-initial block $B$ in $\sigma$ is a \textit{normal block} if one of the following conditions is satisfied:
\begin{description}
\item[(i)] If $B = \sigma_1^{a_p}\sigma_3^{b_p}\sigma_2\sigma_1^{a_{p+1}}$ is a $2$-block such that $a_p\geq b_p + 1$, then $a_{p+1}\neq 2$. If $a_p = 1 = b_p$, then $a_{p+1}\neq 2$.
\item[(ii)] If $B = \sigma_2^{c_{p-1}}\sigma_3\sigma_2^{c_p}$ is a $3$-block in $\sigma$, then $B$ is not a singular $3$-block and $c_p\neq 2$. If $a_{p-1}+1\geq c_p$, then $c_{p-1}\neq 1$.
\end{description}
An \textit{abnormal block} is a block that is not normal.
\end{definition}

The condition of normality on a block is generically satisfied.

\begin{proposition}
\label{pnormalblock}
A positive braid $\sigma$ is a normal braid if and only if every non-initial block in $\sigma$ is a normal block.
\end{proposition}
\begin{proof}
The statement is a reformulation of Definition~\ref{normal} using Definition~\ref{blockroad}, Definition~\ref{blockclass}, and Definition~\ref{dnormalblock}.
\end{proof}

\subsection{Strong regularity of subproducts of a positive braid}
\label{subsecstrongregularity}

In this subsection, we define and study the notion of strong regularity of subproducts of a positive braid. We will show that if $\sigma$ is a normal braid, then strong regularity is an inductive property for subproducts of $\sigma$. We will use this to establish Theorem~\ref{invariantdetect}. 

\begin{notation}
\label{notssubproduct}
In the remainder of this paper, we will fix the following notation in the definitions, statements and proofs, in order to avoid repetition. In general, $P$ will denote an $s$-subproduct of $\sigma$. If $s=2$, then we denote $P=\prod_{i=1}^{p-1} \sigma_1^{a_i}\sigma_3^{b_i}\sigma_2^{c_i}$ and if $s\in \{1,3\}$, then we denote $P=\left(\prod_{i=1}^{p-1}\sigma_1^{a_i}\sigma_3^{b_i}\sigma_2^{c_i}\right)\sigma_1^{a_{p}}\sigma_3^{b_{p}}$. We will also freely use the notation introduced in Notation~\ref{notnumgb} (for weighted numbers of $P$-paths) in Section~\ref{pathgraph}. 
\end{notation}

The rough definition of strong regularity is that an $s$-subproduct $P$ of $\sigma$ is strongly $r$-regular if a maximal $q$-resistance $P$-path is a good $\left(r,s\right)$-type $P$-path and has sufficiently high $q$-resistance relative to other $P$-paths (that are either bad or have different type). We formally define this as follows.

\begin{definition}
\label{defrregular}
If $P$ is an $s$-subproduct of $\sigma$, then we write that $P$ is \textit{strongly $r$-regular} if one of the following conditions \textbf{(a)} (if $s=2$) or \textbf{(b)} (if $s\in \{1,3\}$) on the weighted numbers of $P$-paths with initial vertex equal to $r$ is satisfied.
\begin{description}
\item[(a)] Let $P$ be a $2$-subproduct. In \textbf{(i)}, we will state the conditions in the case that $c_{p-1}\geq 2$, and in \textbf{(ii)}, we will state the conditions in the case that $c_{p-1}=1$. (The reason for the partition is that Notation~\ref{notnumgb} is different in each case.) 

\begin{description}
\item[(i)] Let us assume that $c_{p-1}\geq 2$. In this case, $\lambda_{0}^{P}$ is a non-zero integer with only $2$ as a prime factor, and $d_{\lambda}^{P}>\max\{d_{\mu,1}^{P},d_{\mu,3}^{P}\}$.
\item[(ii)]  Let us assume that $c_{p-1} = 1$. Firstly, $\lambda_{0}^{P}$ is a non-zero integer with only $2$ as a prime factor, and if either $a_p = 0$ or $b_p = 0$, then $d_{\lambda}^{P}>d_{\mu}^{P}$. Secondly, if $a_p = 0$, then $d_{\nu}^{P}\leq d_{\lambda}^{P}$, and if $b_p = 0$, then $d_{\nu}^{P} + 1\leq d_{\lambda}^{P}$.
\end{description}
\item[(b)] Let $P$ be an $s$-subproduct for $s\in \{1,3\}$. In \textbf{(i)}, we will state the conditions in the case that either $a_p=0$ or $b_p=0$, and in \textbf{(ii)}, we will state the conditions in the case that $a_p,b_p>0$. (The reason for the partition is that Notation~\ref{notnumgb} is different in each case.)

\begin{description}
\item[(i)] Let us assume that either $a_p=0$ or $b_p=0$. Firstly, $\lambda_{0}^{P}$ is a non-zero integer with only $2$ as a prime factor, and $d_{\lambda}^{P}>d_{\mu}^{P}$. Secondly, if $a_p = 0$, then $d_{\nu}^{P} + 1\leq d_{\lambda}^{P}$, and if $b_p = 0$, then $d_{\nu}^{P} + 3\leq d_{\lambda}^{P}$. 
\item[(ii)] Let us assume that $a_p,b_p>0$ throughout. (Note that if $a_p>1$, then $w^{P}_{\mu,3}=0$, and if $a_p = 1$, then $w^{P}_{\lambda,1} = 0$; see Proposition~\ref{pnotationentries} \textbf{(b)} \textbf{(ii)} - \textbf{(iii)}.) The following conditions are satisfied. We will state the conditions separately in the cases $a_p>1$ and $a_p = 1$. In both cases, one condition is that $\lambda_{3,0}^{P}$ is a non-zero integer with only $2$ as a prime factor.

Firstly, we consider the case $a_p>1$. In this case, $\lambda_{1,0}^{P} + \lambda_{3,0}^{P} = 0$ and $d_{\lambda,1}^{P}>d_{\mu,1}^{P}$.

Secondly, we consider the case $a_p = 1$. In this case, $d_{\lambda,3}^{P} + 1\geq d_{\mu,1}^{P}$ and $d_{\lambda,3}^{P} > d_{\mu,3}^{P}$.
\end{description}
\end{description}
If $B$ is a block in $\sigma$ and if $P$ is the $s$-subproduct immediately preceding $B$, then we will write that $B$ is a \textit{strongly $r$-regular block} if $PB$ (which is an $s'$-subproduct for $s'\neq s$) is a strongly $r$-regular $s'$-subproduct.
\end{definition}

Lemma~\ref{minimalform} \textbf{(ii)} implies that in Definition~\ref{defrregular} \textbf{(a)} \textbf{(ii)} ($c_{p-1} = 1$), the only instance where $a_p,b_p>0$ is possible is if $p = 2$ and $b_1 = 0$, in which case $P = \sigma_1^{a_1}\sigma_2$. In particular, except for this special instance, we can consider Definition~\ref{defrregular} \textbf{(a)} \textbf{(ii)} exclusively in the case where either $a_p = 0$ or $b_p = 0$.

Let $\sigma$ be a normal braid and let $P$ be the maximal proper $s$-subproduct of $\sigma$. The motivation for Definition~\ref{defrregular} is that if $P$ is a strongly $r$-regular $s$-subproduct, then we can establish Theorem~\ref{invariantdetect}. The idea is that if $X$ is a maximal $q$-resistance good $\left(r,s\right)$-type $P$-path, then we can extend $X$ by an immediate vertex change to a maximal $q$-resistance $\sigma$-path. The extension is admissible by Lemma~\ref{lemmagoodext} \textbf{(i)} and we will show that it is also a maximal $q$-resistance $\sigma$-path using the hypothesis that $P$ is strongly $r$-regular. The strong $r$-regularity of $P$ implies that the weights of the maximal $q$-resistance good $\left(r,s\right)$-type $P$-paths and their maximal $q$-resistance extensions to $\sigma$-paths will contribute the non-zero (modulo $t$) leading coefficients in the $r$th row of $\beta_4\left(\sigma\right)$.

We now establish these statements rigorously. In Subsection~\ref{subsubsecstrongregularityroads}, we establish Corollary~\ref{lemmaroadregular}. Let $P\subseteq P'$, where $P$ is a strongly $r$-regular $s$-subproduct of $\sigma$ and $P'$ is an $s'$-subproduct of $\sigma$ such that $P'\setminus P$ is a road. Corollary~\ref{lemmaroadregular} states that if $P'\neq \sigma$, then $P'$ is strongly $r$-regular, and in general (including the case $P'=\sigma$), the leading coefficients of multiple entries in the $r$th row of $\beta_4\left(P'\right)$ are non-zero modulo each prime $t\neq 2$. The proof of this will be based on the theory of path extensions that we developed in Subsection~\ref{subsecpathextensions}. 

In Subsection~\ref{subsubsecstrongregularityblocks}, we establish Corollary~\ref{lemmablockstronglyregular}, which is an analog of Corollary~\ref{lemmaroadregular} in the case where $P'\setminus P$ is a normal block. Corollary~\ref{lemmaroadregular} and Corollary~\ref{lemmablockstronglyregular} are the key ingredients in the proof of Theorem~\ref{invariantdetect} and establish that strong regularity is an inductive property with respect to the block-road decomposition of a normal braid $\sigma$. Finally, in Subsection~\ref{subsubsecconclusion}, we establish Lemma~\ref{initial}, the base case of the induction and conclude the proof of Theorem~\ref{invariantdetect}. Lemma~\ref{initial} states that if $\sigma$ is a normal braid, then there is a proper strongly $r$-regular subproduct of $\sigma$ for some $r\in \{1,2,3\}$ except for a very specific case which can be easily handled separately.

\subsubsection{Strong regularity is an inductive property along roads}
\label{subsubsecstrongregularityroads}

We will split the proof of Collary~\ref{lemmaroadregular} into several steps, for ease of readability. 

\begin{lemma}
\label{lemma2road}
Let $P$ be a $2$-subproduct of $\sigma$ and let $P'$ be the minimal $s'$-subproduct containing $P$ for $s'\in \{1,3\}$. Let $t\neq 2$ be a prime number. Let us assume that $P'\setminus P$ is a subroad. 

If $P$ is strongly $r$-regular, and if $P'\neq \sigma$, then $P'$ is strongly $r$-regular. If $P$ is strongly $r$-regular, then the leading coefficients of multiple entries in the $r$th row of $\beta_4\left(P'\right)$ are non-zero modulo $t$. 
\end{lemma}
\begin{proof}
We denote $P = \prod_{i=1}^{p-1} \sigma_1^{a_i}\sigma_3^{b_i}\sigma_2^{c_i}$ (Notation~\ref{notssubproduct}). If $P$ is strongly $r$-regular (Definition~\ref{defrregular} \textbf{(a)}), then Proposition~\ref{pnotationentries} \textbf{(a)} \textbf{(i)} and \textbf{(ii)} imply that the leading coefficient of the $\left(r,2\right)$-entry of $\beta_4\left(P\right)$ is non-zero modulo $t$. Proposition~\ref{p21/3ext} \textbf{(iv)} and Proposition~\ref{p21/3bothext} \textbf{(b)} imply that the $\left(r,2\right)$-entry of $\beta_4\left(P'\right)$ is equal to the $\left(r,2\right)$-entry of $\beta_4\left(P\right)$. We deduce that the leading coefficient of the $\left(r,2\right)$-entry of $\beta_4\left(P'\right)$ is non-zero modulo $t$.

We will show that the leading coefficient of at least one other entry in the $r$th row of $\beta_4\left(P'\right)$ is non-zero modulo $t$. If $P'\neq \sigma$, then we will also show that $P'$ is strongly $r$-regular. We split into cases according to the values of $a_p$ and $b_p$.

\begin{description}[style=unboxed]
\item[(i)] (We assume that either $a_p = 0$ or $b_p = 0$.) If $a_p = 1$, then Lemma~\ref{minimalform} \textbf{(i)} implies that $P' = \sigma$. If $b_p = 1$, then the hypothesis that $P'\setminus P$ does not overlap with a $3$-block implies that $P' = \sigma$. In particular, if either $a_p = 1$ or $b_p = 1$, then we need only show that the leading coefficients of multiple entries in the $r$th row of $\beta_4\left(P'\right)$ are non-zero modulo $t$. However, the hypothesis that $P$ is strongly $r$-regular (Definition~\ref{defrregular} \textbf{(a)}), Proposition~\ref{pnotationentries} \textbf{(b)} \textbf{(i)}, and Proposition~\ref{p21/3ext} \textbf{(ii)} imply that if $a_p = 0$, then the leading coefficient of the $\left(r,3\right)$-entry of $\beta_4\left(P'\right)$ is non-zero modulo $t$, and if $b_p = 0$, then the leading coefficient of the $\left(r,1\right)$-entry of $\beta_4\left(P'\right)$ is non-zero modulo $t$.

Let us adopt the setup of Proposition~\ref{p21/3ext} and assume that $\phi_p>1$. In this case, the hypothesis that $P$ is strongly $r$-regular (Definition~\ref{defrregular} \textbf{(a)}), Proposition~\ref{pnotationentries} \textbf{(a)} \textbf{(iv)} and Proposition~\ref{p21/3ext} \textbf{(i)} and \textbf{(iii)} imply that $P'$ is strongly $r$-regular (Definition~\ref{defrregular} \textbf{(b)} \textbf{(i)}). Furthermore, Proposition~\ref{pnotationentries} \textbf{(b)} \textbf{(i)} and Proposition~\ref{p21/3ext} \textbf{(i)} imply that if $a_p = 0$, then the leading coefficient of the $\left(r,3\right)$-entry of $\beta_4\left(P'\right)$ is non-zero modulo $t$, and if $b_p = 0$, then the leading coefficient of the $\left(r,1\right)$-entry of $\beta_4\left(P'\right)$ is non-zero modulo $t$.
\item[(ii)] (We assume that $a_p,b_p>0$.) In this case, Lemma~\ref{minimalform} \textbf{(ii)} and the hypothesis imply that either $c_{p-1}\geq 2$, or $p = 2$, $c_1 = 1$ and $a_1 = 0 = b_1$, since $P'\setminus P$ does not overlap with an initial block. If $p = 2$, $c_1 = 1$ and $a_1 = 0 = b_1$, then the hypothesis that $P$ is strongly $r$-regular (Definition~\ref{defrregular} \textbf{(a)} \textbf{(ii)}) and the assumption that $a_2,b_2>0$ imply that $r = 2$ and the unique $P$-path $2,2$ with initial vertex $r = 2$ is a good $\left(r,2\right)$-type $P$-path.

In both cases, the hypothesis that $P$ is strongly $r$-regular (Definition~\ref{defrregular} \textbf{(a)} \textbf{(i)} if $c_{p-1}\geq 2$, and Definition~\ref{defrregular} \textbf{(a)} \textbf{(ii)} if $p = 2$ and $c_1 = 1$), Proposition~\ref{pnotationentries} \textbf{(b)} \textbf{(ii)} and \textbf{(iii)}, and Proposition~\ref{p21/3bothext} \textbf{(a)} imply that the leading coefficient of the $\left(r,3\right)$-entry of $\beta_4\left(P'\right)$ is non-zero modulo $t$. Proposition~\ref{p21/3bothext} \textbf{(a)} also implies that $P'$ is strongly $r$-regular (Definition~\ref{defrregular} \textbf{(b)} \textbf{(ii)}).
\end{description}
Therefore, the statement is established.
\end{proof}

We now establish an analogue of Lemma~\ref{lemma2road}, where $P$ is an $s$-subproduct for $s\in \{1,3\}$, instead of a $2$-subproduct.

\begin{lemma}
\label{lemma13road}
Let $P$ be an $s$-subproduct of $\sigma$ for $s\in \{1,3\}$ and let $P'$ be the minimal $2$-subproduct containing $P$. Let $t\neq 2$ be a prime number. Let us assume that $P'\setminus P$ is a subroad.

If $P$ is strongly $r$-regular, and if $P'\neq \sigma$, then $P'$ is strongly $r$-regular. If $P$ is strongly $r$-regular, then the leading coefficients of multiple entries in the $r$th row of $\beta_4\left(P'\right)$ are non-zero modulo $t$.
\end{lemma}
\begin{proof}
We denote $P = \left(\prod_{i=1}^{p-1} \sigma_1^{a_i}\sigma_3^{b_i}\sigma_2^{c_i}\right)\sigma_1^{a_p}\sigma_3^{b_p}$ (Notation~\ref{notssubproduct}). We split into cases according to the values of $a_p$ and $b_p$. 

\begin{description}
\item[(i)] (We assume that either $a_p = 0$ or $b_p = 0$) The hypothesis that $P$ is strongly $r$-regular (Definition~\ref{defrregular} \textbf{(b)} \textbf{(i)}) and Proposition~\ref{pnotationentries} \textbf{(b)} \textbf{(i)} imply that if $a_p=0$, then the leading coefficient of the $\left(r,3\right)$-entry of $\beta_4\left(P\right)$ is non-zero modulo $t$, and if $b_p = 0$, then the leading coefficient of the $\left(r,1\right)$-entry of $\beta_4\left(P\right)$ is non-zero modulo $t$. Proposition~\ref{p1/32ext} \textbf{(b)} implies that if $a_p = 0$, then the $\left(r,3\right)$-entry of $\beta_4\left(P'\right)$ is equal to the $\left(r,3\right)$-entry of $\beta_4\left(P\right)$, and if $b_p = 0$, then the $\left(r,1\right)$-entry of $\beta_4\left(P'\right)$ is equal to the $\left(r,1\right)$-entry of $\beta_4\left(P\right)$. We deduce that if $a_p = 0$, then the leading coefficient of the $\left(r,3\right)$-entry of $\beta_4\left(P'\right)$ is non-zero modulo $t$, and if $b_p = 0$, then the leading coefficient of the $\left(r,1\right)$-entry of $\beta_4\left(P'\right)$ is non-zero modulo $t$.

If $c_p = 1$, then Lemma~\ref{minimalform} \textbf{(ii)} and the hypothesis that $P'\setminus P$ does not overlap with an initial block imply that either $a_p = 0 = b_{p+1} $ or $b_p = 0 = a_{p+1}$. In general (for all values of $c_p\geq 1$), the hypothesis that $P$ is strongly $r$-regular (Definition~\ref{defrregular} \textbf{(b)} \textbf{(i)}), Proposition~\ref{pnotationentries} \textbf{(a)} \textbf{(i)} and \textbf{(ii)}, and Proposition~\ref{p1/32ext} \textbf{(a)} imply that the leading coefficient of the $\left(r,2\right)$-entry of $\beta_4\left(P'\right)$ is non-zero modulo $t$. We deduce that the leading coefficients of multiple entries in the $r$th row of $\beta_4\left(P'\right)$ are non-zero modulo $t$. Furthermore, Proposition~\ref{p1/32ext} \textbf{(a)} implies that $P'$ is strongly $r$-regular (Definition~\ref{defrregular} \textbf{(a)}).
\item[(ii)] (We assume that $a_p,b_p>0$) The hypothesis that $P$ is strongly $r$-regular (Definition~\ref{defrregular} \textbf{(b)} \textbf{(ii)}) and Proposition~\ref{pnotationentries} \textbf{(b)} \textbf{(ii)} and \textbf{(iii)} imply that the leading coefficient of the $\left(r,3\right)$-entry of $\beta_4\left(P\right)$ is non-zero modulo $t$. Proposition~\ref{p1/32bothext} \textbf{(c)} implies that the $\left(r,3\right)$-entry of $\beta_4\left(P'\right)$ is equal to the $\left(r,3\right)$-entry of $\beta_4\left(P\right)$. We deduce that the leading coefficient of the $\left(r,3\right)$-entry of $\beta_4\left(P'\right)$ is non-zero modulo $t$.

If $c_p = 1$, then $P' = \sigma$ since $P'\setminus P$ does not overlap with a $2$-block. In this case, we need only show that the leading coefficients of multiple entries in the $r$th row of $\beta_4\left(P'\right)$ are non-zero modulo $t$. However, the hypothesis that $P$ is strongly $r$-regular (Definition~\ref{defrregular} \textbf{(b)} \textbf{(ii)}) and Proposition~\ref{p1/32bothext} \textbf{(a)} \textbf{(ii)} and \textbf{(b)} \textbf{(ii)} imply that the leading coefficient of the $\left(r,2\right)$-entry of $\beta_4\left(P'\right)$ is non-zero modulo $t$. 

If $c_p\geq 2$, then the hypothesis that $P$ is strongly $r$-regular (Definition~\ref{defrregular} \textbf{(b)} \textbf{(ii)}), Proposition~\ref{pnotationentries} \textbf{(a)} \textbf{(i)} and Proposition~\ref{p1/32bothext} \textbf{(a)} \textbf{(i)} and \textbf{(b)} \textbf{(i)} imply that the leading coefficient of the $\left(r,2\right)$-entry of $\beta_4\left(P'\right)$ is non-zero modulo $t$. Proposition~\ref{p1/32bothext} \textbf{(a)} \textbf{(i)} and \textbf{(b)} \textbf{(i)} also imply that $P'$ is strongly $r$-regular (Definition~\ref{defrregular} \textbf{(a)} \textbf{(i)}). 
\end{description}
Therefore, the statement is established.
\end{proof}

We combine the results established thus far (Lemma~\ref{lemma2road} and Lemma~\ref{lemma13road}) to conclude that strong $r$-regularity is preserved along roads, and if the $s$-subproduct preceding the final road in $\sigma$ is strongly $r$-regular, then the leading coefficients of multiple entries in the $r$th row of $\beta_4\left(\sigma\right)$ are non-zero modulo each prime $t\neq 2$.

\begin{corollary}
\label{lemmaroadregular}
Let $P$ be an $s$-subproduct of $\sigma$ and let $P'$ be an $s'$-subproduct of $\sigma$ for some $s,s'\in \{1,2,3\}$ such that $P\subseteq P'$. Let $t\neq 2$ be a prime number. Let us assume that $P'\setminus P$ is a subroad. 

If $P$ is strongly $r$-regular, and if $P'\neq \sigma$, then $P'$ is strongly $r$-regular. If $P$ is strongly $r$-regular, then the leading coefficients of multiple entries in the $r$th row of $\beta_4\left(P'\right)$ are non-zero modulo $t$.
\end{corollary}
\begin{proof}
The statement follows from repeated applications of Lemma~\ref{lemma2road} and Lemma~\ref{lemma13road}. Indeed, there is a finite chain $P=P_0\subsetneq P_1\subsetneq\cdots\subsetneq P_k=P'$ with the property that for $s_0=s$ and $1\leq i\leq k$, $P_i$ is the minimal $s_i$-subproduct containing $P_{i-1}$, for some $s_i\in \{1,2,3\}$ such that $\{s_{i-1},s_i\}$ consists of one element in $\{1,3\}$ and one element in $\{2\}$, and $P_i\setminus P_{i-1}$ is a subroad.
\end{proof}

\subsubsection{Strong regularity is an inductive property around blocks}
\label{subsubsecstrongregularityblocks}

The next step is to establish Corollary~\ref{lemmablockstronglyregular}, which is an analog of Corollary~\ref{lemmaroadregular} in the case where $P'\setminus P$ is a normal block. We will split the proof of Corollary~\ref{lemmablockstronglyregular} into several steps according to whether $P'\setminus P$ is a normal $2$-block or a normal $3$-block, for ease of readability.

Firstly, we parametrize extensions of $P$-paths to $P'$-paths, where $P'\setminus P$ is a $2$-block.

\begin{definition}
\label{def2blockextension}
Let $B = \sigma_1^{a_p}\sigma_3^{b_p}\sigma_2\sigma_1^{a_{p+1}}$ be a $2$-block in $\sigma$. Let $P = \prod_{i=1}^{p-1} \sigma_1^{a_i}\sigma_3^{b_i}\sigma_2^{c_i}$ be the $2$-subproduct of $\sigma$ immediately preceding $B$ and let $P' = PB$. Let $W$ be an $\left(r,2\right)$-type $P$-path.
\begin{description}
\item[(i)] If $1\leq \alpha\leq a_p$ and $1\leq \gamma\leq a_{p+1}$, then we define $W^{\alpha,0,\gamma}$ to be the extension of $W$ to an $\left(r,1\right)$-type $P'$-path by a first vertex change at the $\alpha$th $\sigma_1$ in $\sigma_1^{a_p}$, a second vertex change at $\sigma_2$, and a third vertex change at the $\gamma$th $\sigma_1$ in $\sigma_1^{a_{p+1}}$.
\item[(ii)] We define $W^{a_p,0,0}$ to be the extension of $W$ to an $\left(r,1\right)$-type $P'$-path by a single vertex change at the $a_p$th $\sigma_1$ in $\sigma_1^{a_p}$ (and no further vertex change). (We also consider $W^{a_p,0,0}$ to be of the type $W^{\alpha,0,\gamma}$ with $\alpha = a_p$ and $\gamma = 0$.)
\item[(iii)] If $1\leq \beta\leq b_p$, then we define $W^{0,\beta}$ to be the extension of $W$ to an $\left(r,3\right)$-type $P'$-path by a vertex change at the $\beta$th $\sigma_3$ in $\sigma_3^{b_p}$. 
\item[(iv)] If $1\leq \beta\leq b_p$ and $1\leq \gamma\leq a_{p+1}$, then we define $W^{0,\beta,\gamma}$ to be the extension of $W$ to an $\left(r,1\right)$-type $P'$-path by extending $W^{0,\beta}$ by a first vertex change at $\sigma_2$ and a second vertex change at the $\gamma$th $\sigma_1$ in $\sigma_1^{a_{p+1}}$. 
\end{description}
\end{definition}

We now characterize admissible extensions of $P$-paths to $P'$-paths, and determine a formula for their weights, in the case where $P'\setminus P$ is a $2$-block.

\begin{proposition}
\label{lemma2blockextadmissible}
Let $B = \sigma_1^{a_p}\sigma_3^{b_p}\sigma_2\sigma_1^{a_{p+1}}$ be a $2$-block in $\sigma$. Let $P = \prod_{i=1}^{p-1} \sigma_1^{a_i}\sigma_3^{b_i}\sigma_2^{c_i}$ be the $2$-subproduct of $\sigma$ immediately preceding $B$ and let $P' = PB$. Let $W$ be a $P$-path.

\begin{description}
\item[(i)] If $W$ has an extension to an admissible $P'$-path, then $W$ is an $\left(r,2\right)$-type $P$-path and the extension is of one of the forms in Definition~\ref{def2blockextension}.
\item[(ii)] Let $W$ be an $\left(r,2\right)$-type $P$-path. The weight of $W^{\alpha,0,\gamma}$ is the product of $q^{a_p - \alpha + a_{p+1} - \gamma + 2}$ with the weight of $W$. The weight of $W^{0,\beta}$ is the product of $q^{b_p - \beta}$ with the weight of $W$. The weight of $W^{0,\beta,\gamma}$ is the product of $q^{b_{p}-\beta + a_{p+1} - \gamma+2}$ with the weight of $W$. If either $a_p>1$ or $W$ is not a $1$-switch $\left(r,2\right)$-type $P$-path, then the $\left(r,1\right)$-type $P'$-path $W^{a_p,0,0}$ is admissible and the weight of $W^{a_p,0,0}$ is the product of $-q^{a_{p+1}+1}$ with the weight of $W$.
\item[(iii)] The $\left(r,1\right)$-type $P'$-paths $W^{0,b_p,1}$ and $W^{a_p,0,0}$ constitute a cancelling pair if both are admissible, in the sense that there is a one-to-one correspondence between extensions of $W^{0,b_p,1}$ and $W^{a_p,0,0}$ to $\sigma$-paths such that corresponding extensions have cancelling weights.
\item[(iv)] Let $W$ be a good $\left(r,2\right)$-type $P$-path. The maximal $q$-resistance extension of $W$ to an $\left(r,3\right)$-type $P'$-path is $W^{0,1}$. If $a_p>1$, then the maximal $q$-resistance (non-cancelling) extension of $W$ to an $\left(r,1\right)$-type $P'$-path of the form $W^{\alpha,0,\gamma}$ is $W^{1,0,2}$. If $a_p=1$, then there is no (non-cancelling) admissible extension of $W$ to an $\left(r,1\right)$-type $P'$-path of the form $W^{\alpha,0,\gamma}$. 

If $b_p>1$, then the maximal $q$-resistance extension of $W$ to an $\left(r,1\right)$-type $P'$-path of the form $W^{0,\beta,\gamma}$ is $W^{0,1,1}$. If $b_p = 1$, then the maximal $q$-resistance (non-cancelling) extension of $W$ to an $\left(r,1\right)$-type $P'$-path of the form $W^{0,\beta,\gamma}$ is $W^{0,1,2}$.
\end{description}
\end{proposition}
\begin{proof}
Proposition~\ref{initialblocksingularblockconstraint} \textbf{(iii)} implies that $b_{p+1} = 0$.

\begin{description}
\item[(i)] The statement is a consequence of Lemma~\ref{lemmagoodext} \textbf{(iv)} and \textbf{(v)}.
\item[(ii)] The statements on the weights are direct computations based on Figure~\ref{fig1}. The statement on when $W^{a_p,0,0}$ is an admissible $\left(r,1\right)$-type $P'$-path is a variant of Figure~\ref{fig9}. 
\item[(iii)] The previous part \textbf{(ii)} implies that the weights of $X^{0,b_p,1}$ and $X^{a_p,0,0}$ are negatives of each other. If $a_{p+1} = 1$, then Lemma~\ref{minimalform} \textbf{(i)} implies that $P'=\sigma$, in which case the statement is immediate. If $a_{p+1}>1$, then $X^{0,b_p,1}$ and $X^{a_p,0,0}$ are good $\left(r,1\right)$-type $P'$-paths with cancelling weights, if both are admissible. The statement is now a consequence of Lemma~\ref{lemmagoodext} \textbf{(i)}. 
\item[(iv)] The statement is a consequence of \textbf{(ii)}, \textbf{(iii)}, and Lemma~\ref{lemmagoodext} \textbf{(i)} - \textbf{(iii)}.
\end{description}
\end{proof}

Proposition~\ref{lemma2blockextadmissible} is sufficient to establish the analog of Corollary~\ref{lemmaroadregular} if $P'\setminus P$ is a normal $2$-block. However, if $P'\setminus P$ is the first $2$-block in a normal braid $\sigma$ and there is a very specific kind of first road in the block-road decomposition of $\sigma$, then it is possible that $P$ is not strongly $r$-regular. We use part \textbf{(i)} of the following statement to characterize $P'$-paths and ultimately establish the desired conclusion even in this case. We will only use part \textbf{(ii)} of the statement in Section~\ref{smainstronger}, but \textbf{(ii)} is similar to \textbf{(i)} and we include both here. 

\begin{proposition}
\label{lemma2blockextadmissibleextra}
Let us adopt the setup of Proposition~\ref{lemma2blockextadmissible}.
\begin{description}
\item[(i)] Let $W$ be a $1$-switch $\left(r,2\right)$-type $P$-path. The maximal $q$-resistance extension of $W$ to an $\left(r,3\right)$-type $P'$-path is $W^{0,1}$. If $a_p>2$, then the maximal $q$-resistance (non-cancelling) extension of $W$ to an $\left(r,1\right)$-type $P'$-path of the form $W^{\alpha,0,\gamma}$ is $W^{2,0,2}$. If $1\leq a_p\leq 2$, then there is no (non-cancelling) admissible extension of $W$ to an $\left(r,1\right)$-type $P'$-path of the form $W^{\alpha,0,\gamma}$.

If $b_p>1$, then the maximal $q$-resistance extension of $W$ to an $\left(r,1\right)$-type $P'$-path of the form $W^{0,\beta,\gamma}$ is $W^{0,1,1}$. If $a_p>1$ and $b_p = 1$, then the maximal $q$-resistance (non-cancelling) extension of $W$ to an $\left(r,1\right)$-type $P'$-path of the form $W^{0,\beta,\gamma}$ is $W^{0,1,2}$. If $a_p = 1 = b_p$, then there is no admissible extension of $W$ to an $\left(r,1\right)$-type $P'$-path of the form $W^{0,\beta,\gamma}$.
\item[(ii)] Let $W$ be a $3$-switch $\left(r,2\right)$-type $P$-path. If $b_p>1$, then the maximal $q$-resistance extension of $W$ to an $\left(r,3\right)$-type $P'$-path is $W^{0,2}$. If $b_p = 1$, then there is no admissible extension of $W$ to an $\left(r,3\right)$-type $P'$-path. 

If either $a_p>2$ or $b_p>1$, then the maximal $q$-resistance (non-cancelling) extension of $W$ to an $\left(r,1\right)$-type $P'$-path of the form $W^{\alpha,0,\gamma}$ is $W^{1,0,2}$. If $a_p = 2$ and $b_p = 1$, then the maximal $q$-resistance (non-cancelling) extension of $W$ to an $\left(r,1\right)$-type $P'$-path of the form $W^{\alpha,0,\gamma}$ is $W^{1,0,3}$. If $a_p = 1$ and $b_p>1$, then there is no (non-cancelling) admissible extension of $W$ to an $\left(r,1\right)$-type $P'$-path of the form $W^{\alpha,0,\gamma}$. If $a_p = 1 = b_p$, then the only admissible extension of $W$ to an $\left(r,1\right)$-type $P'$-path is $W^{1,0,0}$. 

If $b_p>2$, then the maximal $q$-resistance extension of $W$ to an $\left(r,1\right)$-type $P'$-path of the form $W^{0,\beta,\gamma}$ is $W^{0,2,1}$. If $b_p = 2$, then the maximal $q$-resistance (non-cancelling) extension of $W$ to an $\left(r,1\right)$-type $P'$-path of the form $W^{0,\beta,\gamma}$ is $W^{0,2,2}$. If $b_p = 1$, then there is no admissible extension of $W$ to an $\left(r,3\right)$-type $P'$-path of the form $W^{0,\beta,\gamma}$.
\end{description}
\end{proposition}
\begin{proof}
\begin{description}
\item[(i)] The statement is a consequence of Proposition~\ref{lemma2blockextadmissible} \textbf{(ii)} and \textbf{(iii)}, and Lemma~\ref{lemmagoodext} \textbf{(i)} - \textbf{(iii)}.
\item[(ii)] The statement is a consequence of Proposition~\ref{lemma2blockextadmissible} \textbf{(ii)} and \textbf{(iii)}, and Lemma~\ref{lemmagoodext} \textbf{(i)} - \textbf{(iii)}.
\end{description}
\end{proof}

We now establish an analog of Corollary~\ref{lemmaroadregular} in the case where $P'\setminus P$ is a normal $2$-block.

\begin{lemma}
\label{l2blockstronglyregular}
Let $B = \sigma_1^{a_p}\sigma_3^{b_p}\sigma_2\sigma_1^{a_{p+1}}$ be a normal $2$-block in $\sigma$. Let $P = \prod_{i=1}^{p-1} \sigma_1^{a_i}\sigma_3^{b_i}\sigma_2^{c_i}$ be the $2$-subproduct of $\sigma$ immediately preceding $B$ and let $P' = PB$. Let $t\neq 2$ be a prime number. Let us assume that either $P$ is strongly $r$-regular, or $\mu_{1,0}^{P}$ is a non-zero integer with only $2$ as a prime factor and $w_{\lambda}^{P} = 0 = w_{\mu,3}^{P}$. If $a_p = 1 = b_p$, then we assume that $P$ is strongly $r$-regular. 

The leading coefficients of multiple entries in the $r$th row of $\beta_4\left(P'\right)$ are non-zero modulo $t$. If $P'\neq \sigma$, then $P'$ is strongly $r$-regular. 
\end{lemma}
\begin{proof}
Firstly, we assume that $P$ is strongly $r$-regular. Proposition~\ref{initialblocksingularblockconstraint} \textbf{(iii)} implies that either $c_{p-1}\geq 2$, or $p = 2$ and $b_1 = 0$. If $p = 2$ and $b_1 = 0$, then the hypothesis that $P$ is strongly $r$-regular (Definition~\ref{defrregular} \textbf{(a)} \textbf{(ii)}) implies that $a_1 = 0$ and $r = 2$. In particular, if $c_{p-1} = 1$, then we deduce that $p = 2$, $a_1 = 0 = b_1$, and $r = 2$. In this case, the unique $P$-path $2,2$ with initial vertex $r = 2$ is a good $\left(r,2\right)$-type $P$-path and $\lambda_{0}^{P} = 1$. If $c_{p-1}\geq 2$, then the hypothesis that $P$ is strongly $r$-regular (Definition~\ref{defrregular} \textbf{(a)} \textbf{(i)}) also implies that a maximal $q$-resistance $\left(r,2\right)$-type $P$-path is a good $\left(r,2\right)$-type $P$-path and $\lambda_{0}^{P}$ is non-zero modulo $t$. If $P'\neq \sigma$, then Lemma~\ref{minimalform} \textbf{(i)} implies that $a_{p+1}>1$. 

Proposition~\ref{lemma2blockextadmissible} \textbf{(i)} and \textbf{(iv)} imply that a maximal $q$-resistance $\left(r,3\right)$-type $P'$-path is of the form $X^{0,1}$, where $X$ is a maximal $q$-resistance good $\left(r,2\right)$-type $P$-path. Proposition~\ref{lemma2blockextadmissible} \textbf{(ii)} implies that $d_{\nu}^{P'} = d_{\lambda}^{P} + b_p - 1$ and $\nu_{0}^{P'} = \lambda_{0}^{P}$. In particular, the leading coefficient of the $\left(r,3\right)$-entry of $\beta_4\left(P'\right)$ is non-zero modulo $t$. We split the remainder of the proof into cases according to the relative values of $a_p$ and $b_p$.

\begin{description}
\item[(a)] (We assume that $a_p>b_p+1$) Proposition~\ref{lemma2blockextadmissible} \textbf{(i)}, \textbf{(ii)} and \textbf{(iv)} imply that a maximal $q$-resistance $\left(r,1\right)$-type $P'$-path is of the form $X^{1,0,2}$, where $X$ is a maximal $q$-resistance good $\left(r,2\right)$-type $P$-path. In particular, the leading coefficient of the $\left(r,1\right)$-entry of $\beta_4\left(P'\right)$ is non-zero modulo $t$. If $P'\neq \sigma$, then $a_{p+1}> 2$ since $B$ is a normal $2$-block. In this case, we deduce that $X^{1,0,2}$ is a good $\left(r,1\right)$-type $P'$-path. Proposition~\ref{lemma2blockextadmissible} \textbf{(ii)} implies that $d_{\lambda}^{P'} = d_{\lambda}^{P} + a_p + a_{p+1} - 1>d_{\mu}^{P'}$ and $\lambda_{0}^{P'} = \lambda_{0}^{P}$. If $P'\neq \sigma$, then we deduce that $P'$ is strongly $r$-regular.

\item[(b)] (We assume that $a_p<b_p+1$ and $b_p>1$) Proposition~\ref{lemma2blockextadmissible} \textbf{(i)}, \textbf{(ii)} and \textbf{(iv)} imply that a maximal $q$-resistance $\left(r,1\right)$-type $P$-path is of the form $X^{0,1,1}$, where $X$ is a maximal $q$-resistance good $\left(r,2\right)$-type $P$-path. In particular, the leading coefficient of the $\left(r,1\right)$-entry of $\beta_4\left(P'\right)$ is non-zero modulo $t$. If $P'\neq \sigma$, then $a_{p+1}>1$ and $X^{0,1,1}$ is a good $\left(r,1\right)$-type $P'$-path. Proposition~\ref{lemma2blockextadmissible} \textbf{(ii)} implies that $d_{\lambda}^{P'} = d_{\lambda}^{P} + b_p + a_{p+1} > d_{\mu}^{P'}$ and $\lambda_{0}^{P'} = \lambda_{0}^{P}$. If $P'\neq \sigma$, then we deduce that $P'$ is strongly $r$-regular. 

\item[(c)] (We assume that $a_p<b_p + 1$ and $b_p = 1$) In this case, $a_p = 1 = b_p$. Proposition~\ref{lemma2blockextadmissible} \textbf{(i)} and \textbf{(iv)} imply that a maximal $q$-resistance (non-cancelling) $\left(r,1\right)$-type $P'$-path is of the form $X^{0,1,2}$, where $X$ is a maximal $q$-resistance good $\left(r,2\right)$-type $P$-path. In particular, the leading coefficient of the $\left(r,1\right)$-entry of $\beta_4\left(P'\right)$ is non-zero modulo $t$. If $P'\neq \sigma$, then the hypothesis implies that $a_{p+1}> 2$ and $X^{0,1,2}$ is a good $\left(r,1\right)$-type $P'$-path. Proposition~\ref{lemma2blockextadmissible} \textbf{(ii)} implies that $d_{\lambda}^{P'} = d_{\lambda}^{P} + 1 + a_{p+1}>d_{\mu}^{P'}$ and $\lambda_{0}^{P'} = \lambda_{0}^{P}$. If $P'\neq \sigma$, then we deduce that $P'$ is strongly $r$-regular.

\item[(d)] (We assume that $a_p = b_p + 1$ and $b_p>1$) Proposition~\ref{lemma2blockextadmissible} \textbf{(i)} and \textbf{(iv)} imply that a maximal $q$-resistance $\left(r,1\right)$-type $P'$-path is either of the form $X^{1,0,2}$ or of the form $X^{0,1,1}$, where $X$ is a maximal $q$-resistance good $\left(r,2\right)$-type $P$-path. Proposition~\ref{lemma2blockextadmissible} \textbf{(ii)} implies that the weights of $X^{1,0,2}$ and $X^{0,1,1}$ have the same sign. In particular, the leading coefficient of the $\left(r,1\right)$-entry of $\beta_4\left(P'\right)$ is non-zero modulo $t$. If $P'\neq \sigma$, then the hypothesis implies that $a_{p+1}>2$ and $X^{1,0,2}$ and $X^{0,1,1}$ are good $\left(r,1\right)$-type $P'$-paths. Proposition~\ref{lemma2blockextadmissible} \textbf{(ii)} implies that $d_{\lambda}^{P'} = d_{\lambda}^{P} + b_p + a_{p+1} > d_{\mu}^{P'}$ and $\lambda_{0}^{P'} = 2\lambda_{0}^{P}$. If $P'\neq \sigma$, then we deduce that $P'$ is strongly $r$-regular.
 
\item[(e)] (We assume that $a_p = b_p + 1$ and $b_p = 1$) In this case, $a_p = 2$ and $b_p = 1$. Proposition~\ref{lemma2blockextadmissible} \textbf{(i)}, \textbf{(ii)} and \textbf{(iv)} imply that a maximal $q$-resistance (non-cancelling) $\left(r,1\right)$-type $P'$-path is of the form $X^{1,0,2}$, where $X$ is a maximal $q$-resistance good $\left(r,2\right)$-type $P$-path. In particular, the leading coefficient of the $\left(r,1\right)$-entry of $\beta_4\left(P'\right)$ is non-zero modulo $t$. If $P'\neq \sigma$, then the hypothesis implies that $a_{p+1}>2$ and $X^{1,0,2}$ is a good $\left(r,1\right)$-type $P'$-path. Proposition~\ref{lemma2blockextadmissible} \textbf{(ii)} implies that $d_{\lambda}^{P'} = d_{\lambda}^{P} + a_{p+1} + 1 > d_{\mu}^{P'}$ and $\lambda_{0}^{P'} = \lambda_{0}^{P}$. If $P'\neq \sigma$, then we deduce that $P'$ is strongly $r$-regular.
\end{description}
If $\mu_{1,0}^{P}$ is a non-zero integer with only $2$ as a prime factor and if $w_{\lambda}^{P} = 0 = w_{\mu,3}^{P}$, then the proof of the statement is similar. We replace all applications of Proposition~\ref{lemma2blockextadmissible} \textbf{(iv)} with applications of Proposition~\ref{lemma2blockextadmissibleextra} \textbf{(i)}, and this case is effectively equivalent to the case where $P$ is strongly $r$-regular with the value of $a_p$ reduced by $1$. (Otherwise, a minor modification is that if $a_p = 2$ and $b_p = 1$, then a maximal $q$-resistance (non-cancelling) $\left(r,1\right)$-type $P'$-path is of the form $Y^{0,1,2}$, where $Y$ is a maximal $q$-resistance $1$-switch $\left(r,2\right)$-type $P$-path.) 
\end{proof}

We needed to consider the extra case where $\mu_{1,0}^{P}$ is a non-zero integer with only $2$ as a prime factor and $w_{\lambda}^{P} = 0 = w_{\mu,3}^{P}$ in the statement of Lemma~\ref{l2blockstronglyregular}, because of the special case mentioned prior to the statement of Proposition~\ref{lemma2blockextadmissibleextra}. We will handle this case in the proof of Lemma~\ref{initial}.

We now parametrize extensions of $P$-paths to $P'$-paths, where $P'\setminus P$ is a $3$-block.

\begin{definition}
\label{def3blockextension}
Let $B = \sigma_2^{c_{p-1}}\sigma_3\sigma_2^{c_p}$ be a $3$-block in $\sigma$ and let $P = \left(\prod_{i=1}^{p-1} \sigma_1^{a_i}\sigma_3^{b_i}\sigma_2^{c_i}\right)\sigma_1^{a_{p-1}}$ be the $s$-subproduct of $\sigma$ immediately preceding $B$ for $s\in \{1,3\}$ (Lemma~\ref{initialblocksingularblockconstraint} \textbf{(iv)} implies that $b_{p-1} = 0$). Let $P'=PB$. 

\begin{description}
\item[(a)] Let $W$ be an $\left(r,1\right)$-type $P$-path.
\begin{description}
\item[(i)] If $1\leq \alpha\leq c_{p-1}$, then we define $W^{\alpha,0}$ to be the extension of $W$ to an $\left(r,3\right)$-type $P'$-path by a first vertex change at the $\alpha$th $\sigma_2$ in $\sigma_2^{c_{p-1}}$ and a second vertex change at $\sigma_3$.
\item[(ii)] If $1\leq \alpha\leq c_{p-1}$ and $1\leq \beta\leq c_p$, then we define $W^{\alpha,\beta}$ to be the extension of $W$ to an $\left(r,2\right)$-type $P'$-path by extending $W^{\alpha,0}$ by a vertex change at the $\beta$th $\sigma_2$ in $\sigma_2^{c_p}$.
\item[(iii)] If $1\leq \gamma\leq c_p$, then we define $W^{\gamma}$ to be the extension of $W$ to an $\left(r,2\right)$-type $P'$-path by a single vertex change at the $\gamma$th $\sigma_2$ in $\sigma_2^{c_p}$.
\end{description}
\item[(b)] Let $W$ be an $\left(r,3\right)$-type $P$-path.
\begin{description}
\item[(i)] If $1\leq \alpha < c_{p-1}$, then we define $W^{\alpha,0}$ to be the extension of $W$ to an $\left(r,3\right)$-type $P'$-path by a first vertex change at the $\alpha$th $\sigma_2$ in $\sigma_2^{c_{p-1}}$ and a second vertex change at $\sigma_3$.
\item[(ii)] If $\alpha = c_{p-1}$, then we define $W^{\alpha,0}$ to be the extension of $W$ to an $\left(r,2\right)$-type $P'$-path by a single vertex change at the $c_{p-1}$th $\sigma_2$ in $\sigma_2^{c_{p-1}}$ (and no further vertex change).
\item[(iii)] If $1\leq \alpha<c_{p-1}$ and $1\leq \beta\leq c_p$, then we define $W^{\alpha,\beta}$ to be the extension of $W$ to an $\left(r,2\right)$-type $P'$-path by extending $W^{\alpha,0}$ by a vertex change at the $\beta$th $\sigma_2$ in $\sigma_2^{c_p}$.
\end{description}
\end{description}
\end{definition}

We now characterize admissible extensions of $P$-paths to $P'$-paths, and determine a formula for their weights, in the case where $P'\setminus P$ is a $3$-block.

\begin{proposition} 
\label{lemma3blockextadmissible}
Let $B = \sigma_2^{c_{p-1}}\sigma_3\sigma_2^{c_p}$ be a $3$-block in $\sigma$ and let $P = \left(\prod_{i=1}^{p-2} \sigma_1^{a_i}\sigma_3^{b_i}\sigma_2^{c_i}\right)\sigma_1^{a_{p-1}}$ be the $s$-subproduct in $\sigma$ immediately preceding $B$ for $s\in \{1,3\}$. Let $P' = PB$. 

\begin{description}
\item[(a)] If a $P$-path $W$ has an admissible extension to a $P'$-path, then $W$ is either an $\left(r,1\right)$-type $P$-path or an $\left(r,3\right)$-type $P$-path. An extension of $W$ to a $P'$-path is of one of the forms in Definition~\ref{def3blockextension}.
\item[(b)] Let $W$ be an $\left(r,1\right)$-type $P$-path.
\begin{description}
\item[(i)] The weight of $W^{\alpha,0}$ is the product of $q^{c_{p-1}-\alpha}$ with the weight of $W$. The weight of $W^{\alpha,\beta}$ is the product of $-q^{c_{p-1}-\alpha+c_p-\beta+1}$ with the weight of $W$. The weight of $W^{\gamma}$ is the product of $q^{c_p-\gamma}$ with the weight of $W$.
\item[(ii)] If $X$ is a good $\left(r,1\right)$-type $P$-path, then $X^{1,0}$ is the maximal $q$-resistance extension of $X$ to an $\left(r,3\right)$-type $P'$-path. If $Y$ is a bad $\left(r,1\right)$-type $P$-path, then $Y^{2,0}$ is the maximal $q$-resistance extension of $Y$ to an $\left(r,3\right)$-type $P'$-path. (If $c_{p-1} = 1$, then a bad $\left(r,1\right)$-type $P$-path $Y$ does not have an admissible extension to an $\left(r,3\right)$-type $P'$-path.) 
\item[(iii)] If $X$ is a good $\left(r,1\right)$-type $P$-path, then $X^{1,2}$ is the maximal $q$-resistance extension of $X$ to an $\left(r,2\right)$-type $P'$-path of the form $X^{\alpha,\beta}$. If $Y$ is a bad $\left(r,1\right)$-type $P$-path, then $Y^{2,2}$ is the maximal $q$-resistance extension of $Y$ to an $\left(r,2\right)$-type $P'$-path of the form $Y^{\alpha,\beta}$. (If $c_{p-1} = 1$, then a bad $\left(r,1\right)$-type $P$-path $Y$ does not have an admissible extension to an $\left(r,2\right)$-type $P'$-path of the form $Y^{\alpha,\beta}$.)
\item[(iv)] The maximal $q$-resistance extension of $W$ to an $\left(r,2\right)$-type $P'$-path of the form $W^{\gamma}$ is $W^{1}$. 
\item[(v)] If $1\leq \gamma\leq c_p - 2$, then the $\left(r,2\right)$-type $P'$-paths $W^{c_{p-1},\gamma+1}$ and $W^{\gamma}$ constitute a cancelling pair if both are admissible, in the sense that there is a one-to-one correspondence between extensions of $W^{c_{p-1},\gamma+1}$ and $W^{\gamma}$ to $\sigma$-paths such that corresponding extensions have cancelling weights.
\end{description}
\item[(c)] Let $W$ be an $\left(r,3\right)$-type $P$-path.
\begin{description}
\item[(i)] The weight of $W^{\alpha,0}$ is the product of $-q^{c_{p-1}-\alpha+1}$ with the weight of $W$. The $\left(r,2\right)$-type $P'$-path $W^{c_{p-1},0}$ is admissible and the weight of $W^{c_{p-1},0}$ is the product of $-q^{c_p+1}$ with the weight of $W$. The weight of $W^{\alpha,\beta}$ is the product of $q^{c_{p-1}-\alpha+c_p-\beta+2}$ with the weight of $W$.
\item[(ii)] If $c_{p-1}>1$, then the maximal $q$-resistance extension of $W$ to an $\left(r,3\right)$-type $P'$-path is $W^{1,0}$. If $c_{p-1} = 1$, then there is no admissible extension of $W$ to an $\left(r,3\right)$-type $P'$-path. 
\item[(iii)] If $c_{p-1}>2$ and $c_p>1$, then the maximal $q$-resistance extension of $W$ to an $\left(r,2\right)$-type $P'$-path is $W^{1,2}$. If $c_{p-1}=2$ and $c_p > 1$, then a maximal $q$-resistance extension of $W$ to an $\left(r,2\right)$-type $P'$-path is either of the form $W^{1,2}$ or of the form $W^{2,0}$ (which have cancelling weights). If either $c_{p-1} = 1$ or $c_p = 1$, then the only admissible extension of $W$ to an $\left(r,2\right)$-type $P'$-path is $W^{c_{p-1},0}$. 
\end{description}
\end{description}
\end{proposition}
\begin{proof}
\begin{description}
\item[(a)] The first statement is a consequence of Lemma~\ref{lemmagoodext} \textbf{(v)}. Furthermore, Lemma~\ref{lemmagoodext} \textbf{(iv)} implies that an analogue of the extension $W^{\gamma}$ of an $\left(r,1\right)$-type $P$-path $W$, for an $\left(r,3\right)$-type $P$-path $W$, is not admissible. Lemma~\ref{lemmagoodext} \textbf{(iii)} implies that that an $\left(r,3\right)$-type $P$-path does not extend to an $\left(r,2\right)$-type $P'$-path by a vertex change at the $c_{p-1}$th $\sigma_2$ in $\sigma_2^{c_{p-1}}$ and a further vertex change (a further vertex change, if it exists, would have to occur at $\sigma_3$). The second statement is a consequence of the previous two sentences and Lemma~\ref{lemmagoodext} \textbf{(v)}.  
\item[(b)] 
\begin{description}
\item[(i)] The statement is a direct computation based on the pictures of $\sigma_2$ and $\sigma_3$ in Figure~\ref{fig1}. 
\item[(ii)] The statement is a consequence of \textbf{(b)} \textbf{(i)} and Lemma~\ref{lemmagoodext} \textbf{(i)} and \textbf{(ii)}.
\item[(iii)] The statement is a consequence of \textbf{(b)} \textbf{(i)} and Lemma~\ref{lemmagoodext} \textbf{(i)} and \textbf{(ii)}.
\item[(iv)] The statement is a consequence of \textbf{(b)} \textbf{(i)} and Lemma~\ref{lemmagoodext} \textbf{(i)} and \textbf{(iv)}.
\item[(v)] The part \textbf{(b)} \textbf{(i)} implies that $W^{c_{p-1},\gamma+1}$ and $W^{\gamma}$ are good $\left(r,2\right)$-type $P'$-paths with cancelling weights, if both are admissible. The statement is now a consequence of Lemma~\ref{lemmagoodext} \textbf{(i)}.
\end{description}
\item[(c)] 
\begin{description}
\item[(i)] The statement is a direct computation based on the pictures of $\sigma_2$ and $\sigma_3$ in Figure~\ref{fig1}. The fact that $W^{c_{p-1},0}$ is an admissible $\left(r,2\right)$-type $P'$-path is a variant of Figure~\ref{fig9} (translated one unit to the right).
\item[(ii)] The statement is a consequence of \textbf{(c)} \textbf{(i)} and Lemma~\ref{lemmagoodext} \textbf{(i)} and \textbf{(iii)}. 
\item[(iii)] The statement is a consequence of \textbf{(c)} \textbf{(i)} and Lemma~\ref{lemmagoodext} \textbf{(i)} and \textbf{(iii)}.
\end{description}
\end{description}
\end{proof}

We now establish an analog of Corollary~\ref{l2blockstronglyregular} in the case where $P'\setminus P$ is a normal $3$-block.

\begin{lemma}
\label{l3blockstronglyregular}
Let $B = \sigma_2^{c_{p-1}}\sigma_3\sigma_2^{c_p}$ be a normal $3$-block in $\sigma$. Let $P = \left(\prod_{i=1}^{p-2} \sigma_1^{a_i}\sigma_3^{b_i}\sigma_2^{c_i}\right)\sigma_1^{a_{p-1}}$ be the $s$-subproduct of $\sigma$ immediately preceding $B$ for $s\in \{1,3\}$. Let $P_0 = \prod_{i=1}^{p-2} \sigma_1^{a_i}\sigma_3^{b_i}\sigma_2^{c_i}$ be the maximal proper $2$-subproduct of $P$ and let $P' = PB$. Let $t\neq 2$ be a prime number. 

\begin{description}
\item[(i)] Let us assume that $P$ is strongly $r$-regular. If $c_{p-1}>1$, then $P'$ is strongly $r$-regular and the leading coefficients of multiple entries in the $r$th row of $\beta_4\left(P'\right)$ are non-zero modulo $t$.
\item[(ii)] Let us assume that $\mu_0^{P}$ is a non-zero integer with only $2$ as a prime factor, $d_{\mu}^{P} + a_{p-1}\geq d_{\lambda}^{P} + 1$, and $d_{\nu}^{P}+3\leq d_{\mu}^{P}$. If $P_0$ is strongly $r$-regular, then these conditions on $P$-paths are satisfied. If $c_{p-1} = 1$, then $P'$ is strongly $r$-regular and the leading coefficients of multiple entries in the $r$th row of $\beta_4\left(P'\right)$ are non-zero modulo $t$.
\end{description}
\end{lemma}
\begin{proof}
Firstly, if $P$ is strongly $r$-regular (Definition~\ref{defrregular} \textbf{(b)} \textbf{(i)}), then Proposition~\ref{pnotationentries} \textbf{(b)} \textbf{(i)} implies that the leading coefficient of the $\left(r,1\right)$-entry of $\beta_4\left(P\right)$ is non-zero modulo $t$. Lemma~\ref{lemmagoodext} \textbf{(iv)} implies that the $\left(r,1\right)$-entry of $\beta_4\left(P'\right)$ is equal to the $\left(r,1\right)$-entry of $\beta_4\left(P\right)$. In particular, if $P$ is strongly $r$-regular, then the leading coefficient of the $\left(r,1\right)$-entry of $\beta_4\left(P'\right)$ is non-zero modulo $t$.

We now establish each statement separately.
\begin{description}
\item[(i)] Firstly, we have established that the leading coefficient of the $\left(r,1\right)$-entry of $\beta_4\left(P'\right)$ is non-zero modulo $t$. The hypothesis that $B$ is a normal $3$-block implies that $c_p>2$. If $X$ is a maximal $q$-resistance good $\left(r,1\right)$-type $P$-path, then Proposition~\ref{lemma3blockextadmissible} \textbf{(b)} \textbf{(i)}, \textbf{(iii)} and \textbf{(iv)} and the hypothesis that $c_{p-1}>1$ imply that the maximal $q$-resistance extension of $X$ to an $\left(r,2\right)$-type $P'$-path is $X^{1,2}$. Furthermore, $X^{1,2}$ is a good $\left(r,2\right)$-type $P'$-path since $c_p>2$. 

Let $Z$ be a maximal $q$-resistance $\left(r,3\right)$-type $P$-path. If $c_{p-1}>2$, then Proposition~\ref{lemma3blockextadmissible} \textbf{(c)} \textbf{(iii)} implies that the maximal $q$-resistance extension of $Z$ to an $\left(r,2\right)$-type $P'$-path is $Z^{1,2}$. If $c_{p-1} = 2$, then Proposition~\ref{lemma3blockextadmissible} \textbf{(c)} \textbf{(iii)} implies that a maximal $q$-resistance extension of $Z$ to an $\left(r,2\right)$-type $P'$-path is either of the form $Z^{1,2}$ or of the form $Z^{2,0}$. In both cases, Proposition~\ref{lemma3blockextadmissible} \textbf{(b)} \textbf{(i)} and \textbf{(c)} \textbf{(i)} and the hypothesis that $P$ is strongly $r$-regular (the inequality $d_{\nu}^{P} + 3\leq d_{\lambda}^{P}$ in Definition~\ref{defrregular} \textbf{(b)} \textbf{(i)}) imply that the $q$-resistance of a maximal $q$-resistance extension of $Z$ to an $\left(r,2\right)$-type $P'$-path is strictly less than the $q$-resistance of $X^{1,2}$. 

Proposition~\ref{lemma3blockextadmissible} \textbf{(a)} now implies that a maximal $q$-resistance $\left(r,2\right)$-type $P'$-path is of the form $X^{1,2}$, where $X$ is a maximal $q$-resistance good $\left(r,1\right)$-type $P$-path. In particular, Proposition~\ref{lemma3blockextadmissible} \textbf{(b)} \textbf{(i)} implies that $d_{\lambda}^{P'} = d_{\lambda}^{P} + c_{p-1} + c_p - 2 >\max\{d_{\mu,1}^{P'},d_{\mu,3}^{P'}\}$ and $\lambda_{0}^{P'} = -\lambda_{0}^{P}$. We deduce that $P'$ is strongly $r$-regular and the leading coefficient of the $\left(r,2\right)$-entry of $\beta_4\left(P'\right)$ is non-zero modulo $t$.

\item[(ii)] Firstly, if $P_0$ is strongly $r$-regular, then Proposition~\ref{p21/3ext} \textbf{(i)} and \textbf{(iii)} imply that the conditions on $P$-paths in the first sentence are satisfied. Furthermore, this implies that $P$ is strongly $r$-regular, and we have established that the leading coefficient of the $\left(r,1\right)$-entry of $\beta_4\left(P'\right)$ is non-zero modulo $t$. 

 Let us assume that the conditions on $P$-paths in the first sentence are satisfied. If $Y$ is a maximal $q$-resistance bad $\left(r,1\right)$-type $P$-path, then Proposition~\ref{lemma3blockextadmissible} \textbf{(b)} \textbf{(iii)} and \textbf{(iv)} imply that the maximal $q$-resistance extension of $Y$ to an $\left(r,2\right)$-type $P'$-path is $Y^{1}$. If $X$ is a maximal $q$-resistance good $\left(r,1\right)$-type $P$-path, then Proposition~\ref{lemma3blockextadmissible} \textbf{(b)} \textbf{(v)} implies that the maximal $q$-resistance extensions of $X$ to (non-cancelling) $\left(r,2\right)$-type $P'$-paths are $X^{1,c_p}$ and $X^{c_p-1}$. (Proposition~\ref{lemma3blockextadmissible} \textbf{(b)} \textbf{(i)} implies that $X^{1,c_p}$ and $X^{c_p-1}$ have cancelling weights. However, $X^{1,c_p}$ is a $3$-switch $\left(r,2\right)$-type $P'$-path and $X^{c_p-1}$ is a good $\left(r,2\right)$-type $P'$-path. In particular, they may not constitute a cancelling pair, in the sense that there may not be a one-to-one correspondence between the set of extensions of $X^{1,c_p}$ to a $\sigma$-path and the set of extensions of $X^{c_p-1}$ to a $\sigma$-path.) 

The hypothesis that $B$ is a normal block implies that $c_p>a_{p-1}+1$. Proposition~\ref{lemma3blockextadmissible} \textbf{(b)} \textbf{(i)} implies that the $q$-resistance of $Y^{1}$ is strictly greater than the $q$-resistance of $X^{1,c_p}$ and $X^{c_p -1}$. We deduce that a maximal $q$-resistance extension of an $\left(r,1\right)$-type $P$-path to an $\left(r,2\right)$-type $P'$-path is of the form $Y^{1}$, where $Y$ is a maximal $q$-resistance bad $\left(r,1\right)$-type $P$-path. 

If $Z$ is a maximal $q$-resistance $\left(r,3\right)$-type $P$-path, then Proposition~\ref{lemma3blockextadmissible} \textbf{(c)} \textbf{(iii)} implies that the only extension of $Z$ to an $\left(r,2\right)$-type $P'$-path is $Z^{1,0}$. Proposition~\ref{lemma3blockextadmissible} \textbf{(b)} \textbf{(i)} and \textbf{(c)} \textbf{(i)} and the hypothesis $d_{\nu}^{P} +3\leq d_{\mu}^{P}$ imply that the $q$-resistance of the maximal $q$-resistance extension of $Z$ to an $\left(r,2\right)$-type $P'$-path is strictly less than the $q$-resistance of $Y^{1}$. 

Proposition~\ref{lemma3blockextadmissible} \textbf{(a)} now implies that a maximal $q$-resistance $\left(r,2\right)$-type $P'$-path is of the form $Y^{1}$, where $Y$ is a maximal $q$-resistance bad $\left(r,1\right)$-type $P$-path. Furthermore, $Y^{1}$ is a good $\left(r,2\right)$-type $P'$-path since $c_p>1$. In particular, Proposition~\ref{lemma3blockextadmissible} \textbf{(b)} \textbf{(i)} implies that $d_{\lambda}^{P'} = d_{\mu}^{P} + c_p - 1>\max\{d_{\mu,1}^{P'},d_{\mu,3}^{P'}\}$ and $\lambda_{0}^{P'} = \mu_{0}^{P}$. We deduce that $P'$ is strongly $r$-regular and the leading coefficient of the $\left(r,2\right)$-entry of $\beta_4\left(P'\right)$ is non-zero modulo $t$.
\end{description}
Therefore, the statement is established.
\end{proof}

We combine Lemma~\ref{l2blockstronglyregular} and Lemma~\ref{l3blockstronglyregular}. In particular, we establish an analog of Corollary~\ref{lemmaroadregular} in the case where $P'\setminus P$ is a normal block.

\begin{corollary}
\label{lemmablockstronglyregular}
Let $B$ be a normal block in $\sigma$. Let $P$ be the $s$-subproduct of $\sigma$ immediately preceding $B$ and let $P'=PB$. Let $t\neq 2$ be a prime number. 

If $B$ is a $2$-block, then let us assume that either $P$ is strongly $r$-regular, or $\mu_{1,0}^{P}$ is a non-zero integer with only $2$ as a prime factor and $w_{\lambda}^{P} = 0 = w_{\mu,3}^{P}$. If $B = \sigma_1\sigma_3\sigma_2\sigma_1^{a_{p+1}}$, then we assume that $P$ is strongly $r$-regular. 

If $B$ is a $3$-block, then let us assume that either $P$ is strongly $r$-regular, or $\mu_{0}^{P}$ is a non-zero integer with only $2$ as a prime factor, $d_{\mu}^{P} + a_{p-1}\geq d_{\lambda}^{P} + 1$, and $d_{\nu}^{P}+3\leq d_{\mu}^{P}$. 

The leading coefficients of multiple entries in the $r$th row of $\beta_4\left(P'\right)$ are non-zero modulo $t$. If $P'\neq \sigma$, then $B$ is a strongly $r$-regular block (i.e., $P'$ is strongly $r$-regular).
\end{corollary}
\begin{proof}
The statement is equivalent to the combination of Lemma~\ref{l2blockstronglyregular} and Lemma~\ref{l3blockstronglyregular}. 
\end{proof}

\subsubsection{Conclusion of the proof}
\label{subsubsecconclusion}

Finally, we address the base case of the induction and establish that the first non-initial block (if it exists) in a positive braid $\sigma$ indivisible by $\Delta$ is strongly $r$-regular for some $r\in \{1,2,3\}$, or (if there is no non-initial block in $\sigma$) the leading coefficients of multiple entries in a row of $\beta_4\left(\sigma\right)$ are non-zero modulo $t$, for each prime $t\neq 2$. The following terminology will be convenient. 

\begin{definition}
If there is a first non-initial block $B$ in $\sigma$, then let $R$ be the $s$-subproduct of $\sigma$ immediately preceding $B$. If there is no non-initial block in $\sigma$, then let $R = \sigma$.

Let $r\in \{1,2,3\}$ and let $R_0$ be a proper $s_0$-subproduct of $\sigma$. We write that $R_0$ is an \textit{initial $r$-inductive subproduct} of $\sigma$ if the initial block in $\sigma$ (if it exists) is contained in $R_0$ and one of the following conditions is satisfied. If $B$ is a $2$-block, then $R_0\subseteq R$ and $R_0$ is strongly $r$-regular. If $B$ is a $3$-block, then either $R_0\subsetneq R$ and $R_0$ is strongly $r$-regular, or $R_0 = R$ and the conditions on $R$-paths in the first sentence of Lemma~\ref{l3blockstronglyregular} \textbf{(ii)} are satisfied.
\end{definition}

If $R$ is the $s$-subproduct of $\sigma$ immediately preceding the first non-initial block in $\sigma$, then we now establish the main result characterizing $R$-paths. The statement is true for all positive braids $\sigma$ indivisible by $\Delta$ (without the condition that $\sigma$ is a normal braid).

\begin{lemma}
\label{initial}
Let $\sigma$ be a positive braid indivisible by $\Delta$ and let $t\neq 2$ be a prime number. If there is a first non-initial block $B$ in $\sigma$, then let $R$ be the $s$-subproduct of $\sigma$ immediately preceding $B$. If there is no non-initial block in $\sigma$, then let $R = \sigma$.

If $R\neq \sigma$, then one of the following statements is true: 

\begin{description}
\item[(i)] An initial $r$-inductive subproduct of $\sigma$ exists for some $r\in \{1,2,3\}$. 
\item[(ii)] We have $\mu_{1,0}^{R} = 1$ and $w_{\lambda}^{R} = 0 = w_{\mu,3}^{R}$ for some $r\in \{1,3\}$ (where $r$ is the initial vertex of the $R$-paths under consideration).
\end{description}
If \textbf{(ii)} is true, then there is an initial block in $\sigma$ and $B$ is a $2$-block.

If $R = \sigma$, then the leading coefficients of multiple entries in the $r$th row of $\beta_4\left(\sigma\right)$ are non-zero modulo $t$ for some $r\in \{1,2,3\}$.
\end{lemma}
\begin{proof}
Let us assume for a contradiction that the statement is false. Firstly, if $R = \sigma$ and there is an initial $r$-inductive subproduct of $\sigma$, then Corollary~\ref{lemmaroadregular} implies that the leading coefficients of multiple entries in the $r$th row of $\beta_4\left(\sigma\right)$ are non-zero modulo $t$. In particular, the existence of an initial $r$-inductive subproduct of $\sigma$ will be a contradiction in all cases. 

Let $\sigma = \prod_{i=1}^{n} \sigma_1^{a_i}\sigma_3^{b_i}\sigma_2^{c_i}$ be the minimal form of $\sigma$. If $n\leq 2$, then it is straightforward to verify that the leading coefficients of multiple entries in the $r$th row of $\beta_4\left(\sigma\right)$ are non-zero modulo $t$ using the theory of $\sigma$-paths, or even direct matrix multiplication. In particular, if $R = \sigma$, then we assume that $n>2$.

Firstly, we assume that $\sigma_1^{a_1}\sigma_3^{b_1}\sigma_2^{c_1}$ does not overlap with a $3$-block, in which case $P_1 = \sigma_1^{a_1}\sigma_3^{b_1}\sigma_2^{c_1}\subseteq R$. If $c_1\geq 2$, then $P_1$ is an initial $1$-inductive subproduct of $\sigma$, which is a contradiction. If $c_1 = 1$ and $a_2 = 0$, then Lemma~\ref{minimalform} \textbf{(ii)} implies that $b_1 = 0$, in which case also $P_1$ is an initial $1$-inductive subproduct of $\sigma$, which is a contradiction.

If $\sigma_1^{a_1}\sigma_3^{b_1}\sigma_2^{c_1}$ overlaps with a $3$-block, then $a_2 = 0$ and $b_2 = 1$. Furthermore, Proposition~\ref{initialblocksingularblockconstraint} \textbf{(iv)} implies that $b_1 = 0$. In this case, $\sigma_1^{a_1}$ is an initial $2$-inductive subproduct of $\sigma$, which is a contradiction.

We deduce that $c_1 = 1$ and $a_2>0$. If $b_1>0$, then Lemma~\ref{minimalform} \textbf{(ii)} implies that $b_2 = 0$. In this case, $\sigma_1^{a_1}\sigma_3^{b_1}\sigma_2^{c_1}\sigma_1^{a_2}$ is an initial $3$-inductive subproduct of $\sigma$, which is a contradiction. We deduce that $b_1 = 0$, and the beginning of the minimal form of $\sigma$ is $\sigma_1^{a_1}\sigma_2\sigma_1^{a_2}\sigma_3^{b_2}$. 

If $c_2 = 1$, then Lemma~\ref{minimalform} \textbf{(ii)} implies that $b_3 = 0$ and $\sigma_1^{a_1}\sigma_2\sigma_1\sigma_3\sigma_2\sigma_1$ is at the beginning of $\sigma$. However, in this case, Proposition~\ref{pblockDelta} implies that $\sigma$ is divisible by $\Delta$, which is a contradiction. We deduce that $c_2\geq 2$. If $b_2>0$, then $a_2=1=b_2$, since otherwise $\sigma_1^{a_1}\sigma_2\sigma_1^{a_2}\sigma_3^{b_2}\sigma_2^{c_2}$ would be an initial $1$-inductive subproduct of $\sigma$. In particular, the beginning of the minimal form of $\sigma$ is either $\sigma_1^{a_1}\sigma_2\sigma_1\sigma_3$ or $\sigma_1^{a_1}\sigma_2\sigma_1^{a_2}$.

Let us first consider the case where the beginning of the minimal form of $\sigma$ is $\sigma_1^{a_1}\sigma_2\sigma_1\sigma_3$. If $P_2 = \prod_{i=1}^{2} \sigma_1^{a_i}\sigma_3^{b_i}\sigma_2^{c_i}$, then $c_2=2$ and the maximal $q$-resistance $\left(1,2\right)$-type $P_2$-path must have a $3\to 2$ vertex change at the last $\sigma_2$ in $\sigma_2^{c_2}$. Otherwise, $P_2$ would be an initial $1$-inductive subproduct of $\sigma$. Similarly, the maximal $q$-resistance $\left(3,2\right)$-type $P_2$-path must have a $1\to 2$ vertex change at the last $\sigma_2$ in $\sigma_2^{c_2}$. Otherwise, $P_2$ would be an initial $3$-inductive subproduct of $\sigma$.

We repeat the reasoning to deduce that unless $R=\sigma_1^{a_1}\sigma_2\sigma_1\sigma_3\sigma_2^2\sigma_1\sigma_3\sigma_2^2\cdots$ (an alternating product of $\sigma_1\sigma_3$s and $\sigma_2^2$s following $\sigma_1^{a_1}\sigma_2$, but where the exponents at only the end $\sigma_1^{a_p}\sigma_3^{b_p}\sigma_2^{c_p}$ may be arbitrary in the case $R = \sigma$), there is an initial $r$-inductive subproduct of $\sigma$ for some $r\in \{1,3\}$. If $R = \sigma$, then we now show that the leading coefficients of multiple entries in either the first row or the third row of $\beta_4\left(\sigma\right)$ are non-zero modulo $t$ by examination of the possibilities for the end of the minimal form of $\sigma$.

Indeed, if the minimal form of $\sigma$ ends in $\sigma_1^{a_p}\sigma_3^{b_p}\sigma_2^{c_p}$, where  $a_p+b_p>2$ if $c_p = 1$ and $a_p+b_p>1$ if $c_p = 0$, then the leading coefficients of multiple entries in each of the first and the third rows of $\beta_4\left(\sigma\right)$ are non-zero modulo $t$. If the minimal form of $\sigma$ ends in either $\sigma_1\sigma_3\sigma_2$ (with $a_p + b_p = 2$ and $c_p = 1$), or $\sigma_1$ or $\sigma_3$ (with $a_p + b_p = 1$ and $c_p = 0$), then the leading coefficients of multiple entries in precisely one of the first row or the third row of $\beta_4\left(\sigma\right)$ are non-zero modulo $t$.

If $R\neq \sigma$, then Proposition~\ref{initialblocksingularblockconstraint} \textbf{(iv)} implies that the first non-initial block $B$ in $\sigma$ is a $2$-block. If the number of $\sigma_2^2$s in $R$ is even, then we choose $r = 1$, and if the number of $\sigma_2^2s$ in $R$ is odd, then we choose $r = 3$. In this case, $\mu_{1,0}^{R} = 1$ and $w_{\lambda}^{R} = 0 = w_{\mu,3}^{R}$. 

We conclude that the beginning of the minimal form of $\sigma$ is $\sigma_1^{a_1}\sigma_2\sigma_1^{a_2}$. In this case, we observe that $\sigma_1^{a_1}\sigma_2\sigma_1^{a_2}$ is an initial $3$-inductive subproduct of $\sigma$. We have considered all cases and the statement is established.
\end{proof}

We can now establish Theorem~\ref{invariantdetect}, and its consequence Theorem~\ref{main} (the main result).

\begin{proof}[Proof of Theorem~\ref{invariantdetect}]
Lemma~\ref{initial} implies the statement if there is no non-initial block in $\sigma$. If there is a first non-initial block $B$ in $\sigma$, then let $P'$ be the minimal $s'$-subproduct of $\sigma$ containing $B$. If \textbf{(i)} in Lemma~\ref{initial} is true, then Corollary~\ref{lemmaroadregular} and Corollary~\ref{lemmablockstronglyregular} imply that $P'$ is strongly $r$-regular for some $r\in \{1,2,3\}$. If \textbf{(ii)} in Lemma~\ref{initial} is true, then Proposition~\ref{pblockDelta} implies that $B\neq \sigma_1\sigma_3\sigma_2\sigma_1^{a_{p+1}}$ and the hypothesis of Corollary~\ref{lemmablockstronglyregular} is satisfied. In particular, Corollary~\ref{lemmablockstronglyregular} also implies that $P'$ is strongly $r$-regular for some $r\in \{1,3\}$. 

We deduce that $P'$ is strongly $r$-regular for some $r\in \{1,2,3\}$ in all cases. Furthermore, there is a finite chain $P'=P_0\subsetneq P_1\subsetneq\cdots\subsetneq P_k=\sigma$ such that for $1\leq i\leq k$, $P_i\setminus P_{i-1}$ is a road for $i$ odd and $P_i\setminus P_{i-1}$ is a normal block for $i$ even. Therefore, the statement follows from repeated applications of Corollary~\ref{lemmaroadregular} and Corollary~\ref{lemmablockstronglyregular}.
\end{proof}

\section{A stronger version of the main result}
\label{smainstronger}

In this section, we establish Theorem~\ref{mainstronger}, which is a stronger version of Theorem~\ref{main}. Theorem~\ref{main} states that if $g = \Delta^{k}\sigma$ is the Garside normal form of $g\in B_4$ and if $\sigma$ is a normal braid, then $g\not\in \text{ker}\left(\beta_4\right)$. Theorem~\ref{mainstronger} extends Theorem~\ref{main} to the case where $\sigma$ is a weakly normal braid. 

We now roughly define the notion of a weakly normal braid in order to briefly summarize the statement of Theorem~\ref{mainstronger} before the precise definition in Subsection~\ref{subsecweaklynormalbraidbadsubroad}. We also highlight that weakly normal braids are generic even among positive braids that are not normal. The rough definition of a \textit{bad subroad} is a highly constrained subroad $R$ such that the exponents of the generators in $R$ are very small (zero, one or two, with an exception for alternating exponents in isolated $\sigma_2$ subproducts in $R$). The \textit{width of a block} $B$ is a numerical quantity defined in terms of the exponents in $B$, and the \textit{weight of a bad subroad} $R$ is a numerical quantity closely related to the length of $R$. The rough definition of a \textit{weakly normal braid} is a positive braid $\sigma$ such that the maximal bad subroad following each abnormal block $B$ does not have length \textit{exactly} equal to the width of $B$.

If $\sigma$ is a normal braid, then Proposition~\ref{pnormalblock} implies that there is no abnormal block in $\sigma$, and $\sigma$ is a weakly normal braid. Furthermore, if we simply increase the width of an abnormal block $B$ in $\sigma$ and keep the rest of $\sigma$ unchanged, then the probability of existence of a long bad subroad after $B$ very rapidly becomes vanishingly small. In particular, weakly normal braids are generic, even among positive braids that are not normal. 

The proof of Theorem~\ref{mainstronger} is based on a more general version of Theorem~\ref{invariantdetect}. If $\sigma$ is a weakly normal braid, then we will construct $\sigma$-paths with weights that contribute the highest $q$-degree terms in multiple entries in a row of $\beta_4\left(\sigma\right)$.

The organization of this section is as follows. In Subsection~\ref{subsecweaklynormalbraidbadsubroad}, we will precisely define bad subroads, the width of a block, the weight of a bad subroad, and weakly normal braids. We will also precisely state Theorem~\ref{mainstronger}. In Subsection~\ref{subsecpathextensionsII}, we will extend the theory of path extensions developed in Subsection~\ref{subsecpathextensions}. The theory of path extensions developed in Subsection~\ref{subsecpathextensions} considered extensions of $P$-paths to $P'$-paths where $P$ is strongly $r$-regular. In Subsection~\ref{subsecpathextensionsII}, we will develop the theory of path extensions of $P$-paths to $P'$-paths where $P$ is not strongly $r$-regular. Finally, in Subsection~\ref{subsecproofstrongerversion}, we will prove Theorem~\ref{mainstronger}. 

We now highlight the general ideas of the proof of Theorem~\ref{mainstronger} for a specific family of positive braids before the formal proof. We also illustrate the notion of a bad subroad, width of a block, weight of a bad subroad, and the genericity of the condition on a positive braid that it is a weakly normal braid. The main consideration for positive braids that are not normal is that global cancellation of weights of paths can occur (which does not occur for normal positive braids).

\begin{example}
\label{exglobalcancellation}
Let $\sigma = \sigma_1^n\sigma_2\left(\sigma_1^2\sigma_2^2\cdots \sigma_1^2\sigma_2^2\right)\sigma_1$, where there are $m$ $\sigma_1^2$s in the parenthesis. We write $P_1' = \sigma_1^n$ and $P_1 = \sigma_1^n\sigma_2$. If $2\leq i\leq m+1$, then we write $P_i' = \sigma_1^n\sigma_2\left(\sigma_1^2\sigma_2^2\cdots \sigma_1^2\right)$ and $P_i = \sigma_1^n\sigma_2\left(\sigma_1^2\sigma_2^2\cdots \sigma_1^2\sigma_2^2\right)$, where there are $i-1$ $\sigma_1^2$s in each parenthesis. We observe that $\sigma = P_m\sigma_1$ and $P_2' = \sigma_1^{n}\sigma_2\sigma_1^2$ is the initial block in $\sigma$. If $n>1$, then we will see that $P_2'$ is not strongly $2$-regular. However, we will use the theory of good and bad paths to show that the leading coefficients of multiple entries in the $2$nd row of $\beta_4\left(\sigma\right)$ are non-zero modulo each prime $t$, unless the specific equality $m = \frac{n}{2} + 1$ between the length of $\sigma$ and the length of the initial block $P_2'$ holds. Proposition~\ref{pnormalblock} implies that $\sigma$ is a normal braid since there is no non-initial block in $\sigma$. However, we now describe the behavior of $\left(2,\cdot\right)$-type $\sigma$-paths and the initial block $P_2'$, and this behavior will model the behavior of $\sigma$-paths and abnormal blocks in general (where $\sigma$ is not a normal braid).
 
If $X_1$ is a $\left(2,\cdot\right)$-type $P_1$-path, then Lemma~\ref{lemmagoodext} \textbf{(v)} implies that $X_1$ has a $2\to 1$ vertex change at some $\sigma_1$ (in $P_1$). If $1\leq \alpha\leq n-1$, then let $X_{\alpha,1}$ denote the $\left(2,2\right)$-type $P_1$-path with a first vertex change at the $\alpha$th $\sigma_1$ in $\sigma_1^{n}$ and a second vertex change at $\sigma_2$ (in $P_1$). The weight of $X_{\alpha,1}$ is $-q^{n-\alpha+1}$. A $\left(2,1\right)$-type $P_1'$-path with a $2\to 1$ vertex change at the last $\sigma_1$ in $\sigma_1^{n}$ is bad. Lemma~\ref{lemmagoodext} \textbf{(ii)} implies that if a $\left(2,\cdot\right)$-type $P_1$-path has a vertex change at the last $\sigma_1$ in $\sigma_1^{n}$, then it does not have a $1\to 2$ vertex change at $\sigma_2$ (in $P_1$), and it is a good $\left(2,1\right)$-type $P_1$-path (this is illustrated in Figure~\ref{fig9} in Subsection~\ref{subsecadspath}). We denote this good $\left(2,1\right)$-type $P_1$-path by $Y_1$. The weight of $Y_1$ is $-q$.

We observe that $X_{\alpha,1}$ is a bad $\left(2,2\right)$-type $P_1$-path. Lemma~\ref{lemmagoodext} \textbf{(iii)} implies that there is a unique extension of $X_{\alpha,1}$ to an $\left(2,1\right)$-type $P_2'$-path defined by a $2\to 1$ vertex change at the second $\sigma_1$ in $\sigma_1^2$ (at the end of $P_2'$). The resulting $\left(2,1\right)$-type $P_2'$-path is bad and we denote it by $X_{\alpha,2}'$. The weight of $X_{\alpha,2}'$ is $q^{n-\alpha+2}$. On the other hand, there is a unique extension of the good $\left(2,1\right)$-type $P_1$-path $Y_1$ to a $\left(2,\cdot\right)$-type $P_2'$-path (with no further vertex change). The resulting $P_2'$-path is a good $\left(2,1\right)$-type $P_2'$-path which we denote by $Y_2'$. The weight of $Y_2'$ is $-q^3$. 

We deduce that there are $n-1$ bad $\left(2,1\right)$-type $P_2'$-paths $X_{\alpha,2}'$ for $1\leq \alpha\leq n-1$ and a unique good $\left(2,1\right)$-type $P_2'$-path $Y_2'$. The weights of $X_{\alpha,2}'$ and $Y_2'$ have opposite sign. If $\alpha < n-1$, then the $q$-resistance of $X_{\alpha,2}'$ is larger than the $q$-resistance of $Y_2'$. However, if the length of $\sigma$ is sufficiently large relative to $n-\alpha$, and if we consider maximal $q$-resistance extensions of both $P_2'$-paths to $\sigma$-paths, then we will show that the $q$-resistance of $Y_2'$ will eventually ``catch-up" to the $q$-resistance of $X_{\alpha,2}'$. The fundamental principle (based on Lemma~\ref{lemmagoodext}) is that good paths accumulate $q$-resistance faster than bad paths. In particular, we will show that there is a unique extension of $X_{\alpha,2}$ to a $\sigma$-path, and if the length of $\sigma$ is sufficiently large (relative to $n - \alpha$), then its weight will cancel and not contribute to an entry of $\beta_4\left(\sigma\right)$. The phenomenon can be interpreted as a \textit{global cancellation of weights of $\sigma$-paths created by the (initial) block $P_2'$}. The \textit{width of the (initial) block} $P_2'$ is the discrepancy $n-2$ in $q$-resistance between the maximal $q$-resistance good $\left(2,1\right)$-type $P_2'$-path $Y_2'$ and the maximal $q$-resistance bad $\left(2,1\right)$-type $P_2'$-path $X_{1,2}'$. The phenomenon of global cancellation of weights of $\sigma$-paths is more generally created by an abnormal block in a positive braid $\sigma$ and is the main difficulty in removing the condition that $\sigma$ is a normal braid in Theorem~\ref{main} (the main result). However, we can still generically control global cancellation and we now illustrate that in the case of the initial block $P_2'$.

Firstly, repeated applications of Lemma~\ref{lemmagoodext} \textbf{(ii)} and \textbf{(iii)} show that there is a unique extension of $X_{\alpha,2}'$ to a $\left(2,1\right)$-type $P_i'$-path $X_{\alpha,i}'$ and a unique extension of $X_{\alpha,2}'$ to a $\left(2,2\right)$-type $P_i$-path $X_{\alpha,i}$. The extensions are defined by $1\to 2$ vertex changes at every second $\sigma_2$ and $2\to 1$ vertex changes at every second $\sigma_1$. The $P_i'$-path $X_{\alpha,i}'$ and the $P_i$-path $X_{\alpha,i}$ are bad. The weights of $X_{\alpha,i}'$ and $X_{\alpha,i}$ are $\left(-1\right)^{i}q^{n-\alpha+i}$. Secondly, repeated applications of Lemma~\ref{lemmagoodext} \textbf{(i)} show that the maximal $q$-resistance extension of $Y_{2}'$ to a $\left(2,1\right)$-type $P_i'$-path $Y_{i}'$ and to a $\left(2,2\right)$-type $P_i$-path $Y_i$ are by vertex changes at every possible opportunity (at the first $\sigma_1$ or $\sigma_2$ in each square $\sigma_1^2$ or $\sigma_2^2$, respectively). The $P_i'$-path $Y_i'$ and the $P_i$-path $Y_i$ are good. The weight of $Y_{i}'$ is $\left(-1\right)^{i-1}q^{3(i-1)}$ and the weight of $Y_{i}$ is $\left(-1\right)^{i-1}q^{3(i-1) + 1}$.

The maximal $q$-resistance bad $\left(2,2\right)$-type $P_{i}$-path of the form $X_{\alpha,i}$ is $X_{1,i}$ and the maximal $q$-resistance good $\left(2,2\right)$-type $P_{i}$-path is $Y_{i}$. If $i>\frac{n+1}{2}$, then the $q$-resistance of $Y_{i}$ is greater than the $q$-resistance of $X_{1,i}$. In particular, if $m>\frac{n+1}{2}$, then the maximal $q$-resistance $\left(2,2\right)$-type $P_m$-path $Y_{m}$ is good. Lemma~\ref{lemmagoodext} \textbf{(i)} implies that $Y_{m}$ extends to the maximal $q$-resistance $\left(2,1\right)$-type $\sigma$-path by a vertex change at the last $\sigma_1$ in $\sigma$ (recall $\sigma = P_m\sigma_1$). If $m>\frac{n+1}{2}$, then Proposition~\ref{m=ap} implies that the leading coefficients of multiple entries in the $2$nd row of $\beta_4\left(\sigma\right)$ are non-zero (in fact, $\pm 1$). (We can also directly apply Corollary~\ref{lemmaroadregular} since $P_m$ is strongly $2$-regular in this case.)

If $m<\frac{n+1}{2}$, then the maximal $q$-resistance $\left(2,2\right)$-type $P_m$-path $X_{1,m}$ is bad. Proposition~\ref{m=ap} implies that the weight of $X_{1,m}$ contributes the highest $q$-degree term to the $\left(2,2\right)$-entry of $\beta_4\left(\sigma\right)$. However, Lemma~\ref{lemmagoodext} \textbf{(iii)} implies that no $\left(2,2\right)$-type $P_m$-path of the form $X_{\alpha,m}$ extends to a $\sigma$-path, since $X_{\alpha,m}$ is bad and the final generator in $\sigma$ is an isolated $\sigma_1$. We deduce that the maximal $q$-resistance $\left(2,1\right)$-type $\sigma$-path is the extension of the maximal $q$-resistance good $\left(2,2\right)$-type $P_m$-path $Y_{m}$ by a $2\to 1$ vertex change at $\sigma_1$ (at the end of $\sigma$). Proposition~\ref{m=ap} implies that the weight of this $\left(2,1\right)$-type $\sigma$-path contributes the highest $q$-degree term to the $\left(2,1\right)$-entry of $\beta_4\left(\sigma\right)$. We deduce that the leading coefficients of multiple entries in the $2$nd row of $\beta_4\left(\sigma\right)$ are non-zero (in fact, $\pm 1$).

If $m = \frac{n+1}{2}$, then global cancellation occurs, in the sense that the set of $\left(2,2\right)$-type $\sigma$-paths is a disjoint union of cancelling pairs of $\left(2,2\right)$-type $\sigma$-paths. Let us elaborate. If $2\leq i\leq m$, then define $\alpha_i' = n - 2i + 3$ and $\alpha_i = n-2i + 2$. The weights of the $\left(2,1\right)$-type $P_i'$-paths $X_{\alpha_i',i}'$ and $Y_{i}'$ cancel. We can extend $Y_i'$ to a $\left(2,2\right)$-type $\sigma$-path by the same vertex changes as the unique extension of $X_{\alpha_i',i}'$ to a $\left(2,2\right)$-type $\sigma$-path. The unique extension of $X_{\alpha_i',i}'$ to a $\left(2,2\right)$-type $\sigma$-path constitutes a cancelling pair with this extension of $Y_i'$ to a $\left(2,2\right)$-type $\sigma$-path. Similarly, we can extend $Y_i$ to a $\left(2,2\right)$-type $\sigma$-path by the same vertex changes as the unique extension of $X_{\alpha_i,i}$ to a $\left(2,2\right)$-type $\sigma$-path. The unique extension of $X_{\alpha_i,i}$ to a $\left(2,2\right)$-type $\sigma$-path constitutes a cancelling pair with this extension of $Y_i$ to a $\left(2,2\right)$-type $\sigma$-path.

If $m = \frac{n+1}{2}$, then Proposition~\ref{m=ap} implies that the $\left(2,2\right)$-entry of $\beta_4\left(\sigma\right)$ is zero. However, the leading coefficient of the $\left(2,1\right)$-entry of $\beta_4\left(\sigma\right)$ is non-zero (in fact, $\left(-1\right)^{m+1}$) since the maximal $q$-resistance good $\left(2,2\right)$-type $P_m$-path $Y_m$ extends to the maximal $q$-resistance $\left(2,1\right)$-type $\sigma$-path. 

In this context, $R = \sigma_2^2\sigma_1^2\cdots $ is an example of a bad subroad, where we can write $\sigma = P_2'R\sigma_1$ and $P_m = P_2'R$. The main characteristic of the bad subroad $R$ is that a bad $\left(2,1\right)$-type $P_2'$-path extends uniquely to a $\left(2,2\right)$-type $P_m$-path and the resulting $\left(2,2\right)$-type $P_m$-path is bad. In particular, the $q$-resistance of the maximal $q$-resistance good $\left(2,1\right)$-type $P_2'$-path $Y_2'$ can ``catch-up" to the $q$-resistance of the maximal $q$-resistance bad $\left(2,1\right)$-type $P_2'$-path $X_{1,2}'$. If we consider maximal $q$-resistance extensions of $X_{1,2}'$ and $Y_2'$ to $\left(2,2\right)$-type $P_m$-paths, then the \textit{weight of the bad subroad} $R$ is the reduction in the discrepancy between the $q$-resistances of $X_{1,2}'$ and $Y_2'$. The equality $m = \frac{n+1}{2}$ is equivalent to the statement that the weight of the bad subroad $R$ is equal to the width of the (initial) block $P_2'$. We have shown that as long as this equality is not satisfied, global cancellation of weights of $\left(2,2\right)$-type $\sigma$-paths does not occur and the leading coefficients of multiple entries in the $2$nd row of $\beta_4\left(\sigma\right)$ are non-zero. More generally, although ``catch-up" is created by an abnormal block $B$ in an arbitrary positive braid $\sigma$, unless there is a bad subroad of a precise length (in relation to the width of $B$) following $B$, global cancellation does not occur.

Furthermore, if $m = \frac{n+1}{2}$, then we observe that if even a \textit{single} square of $\sigma_1$ or $\sigma_2$ in $P_{m}'$ is a cube or higher exponent (i.e., the length of the maximal bad subroad following the abnormal block $P_2'$ is strictly less than the width of the block $P_2'$), then the leading coefficients of multiple entries in the $2$nd row of $\beta_4\left(\sigma\right)$ are non-zero. Let us assume that the first power of a generator that is not a square is a power of $\sigma_1$ at the end of $P_{i}'$ for $i\leq m$. The maximal $q$-resistance $\left(2,2\right)$-type $P_{i-1}$-path is $X_{1,i-1}$ and the $q$-resistance of $X_{1,i-1}$ is at least two more than the $q$-resistance of $Y_{i-1}$. The maximal $q$-resistance extension of $X_{1,i-1}$ to a $\left(2,1\right)$-type $P_i'$-path is good and it is the maximal $q$-resistance $\left(2,\cdot\right)$-type $P_i'$-path. We deduce that $P_i'$ is strongly $2$-regular. The intuition is that if we consider maximal $q$-resistance extensions of both $\left(2,1\right)$-type $P_2'$-paths $X_{1,2}'$ and $Y_2'$, the $q$-resistance of $Y_2'$ does not catch up to the $q$-resistance of $X_{1,2}'$ before $X_{1,2}'$ extends to a good $P_i'$-path. Similarly, if the first power of a generator that is not a square is a power of $\sigma_2$ at the end of $P_i$ for $i<m$, then $P_{i}$ is strongly $2$-regular. In particular, if at least one square in the bad subroad $\sigma_2^2\sigma_1^2\cdots $ is a cube or higher exponent (not including the $\sigma_2^2$ at the end), then Corollary~\ref{lemmaroadregular} implies that the leading coefficients of multiple entries in the $2$nd row of $\beta_4\left(\sigma\right)$ are non-zero.

Although it is possible for there to be a single non-zero entry in the second row of $\beta_4\left(\sigma\right)$ (if $m = \frac{n+1}{2}$), we observe that $\beta_4\left(\sigma\right)\neq \beta_4\left(\Delta^{k}\right)$ for each integer $k$ irrespective of the values of $m$ and $n$. Indeed, $P_1$ is strongly $3$-regular since the maximal $q$-resistance $\left(3,\cdot \right)$-type $P_1$-path is defined by a $3\to 2$ vertex change at $\sigma_2$ (in $P_1$) and it is a $\left(3,2\right)$-type $P_1$-path. Corollary~\ref{lemmaroadregular} implies that the leading coefficients of multiple entries in the $3$rd row of $\beta_4\left(\sigma\right)$ are non-zero.

We will establish Theorem~\ref{mainstronger} (the main result of this section and the stronger version of Theorem~\ref{main}) by generalizing these techniques and studying the phenomenon of global cancellation and ``catch-up" created by an abnormal block and a bad subroad. If $m\neq \frac{n+1}{2}$ or if a single square of either $\sigma_1$ or $\sigma_2$ in $P_m'$ is a cube or higher exponent (the weight of the maximal bad subroad following the (initial) block $P_2'$ is not equal to the width of the (initial) block $P_2'$), then global cancellation does not occur. More generally, if the weight of the maximal bad subroad following each abnormal block $B$ in an arbitrary positive braid $\sigma$ does not equal the width of $B$, then $\sigma$ is a weakly normal braid and we will use very similar techniques to demonstrate Theorem~\ref{mainstronger}. 
\end{example}

\subsection{Weakly normal braids}
\label{subsecweaklynormalbraidbadsubroad}

In this subsection, we will precisely define weakly normal braids and precisely state the stronger version of the main result (Theorem~\ref{mainstronger}). The first step is to precisely define bad subroads. We begin by introducing the general type of subroads that we will consider.

\begin{definition}
\label{defss'subroad}
Let $R$ be a subroad of $\sigma$. We write that $R$ is an \textit{$\left(s,s'\right)$-subroad} of $\sigma$ if $R = P'\setminus P$, where $P$ is an $s''$-subproduct of $\sigma$ and $P'$ is an $s'$-subproduct of $\sigma$, for some $s''\in \{1,2,3\}$ such that $\{s,s''\}$ contains one element in $\{1,3\}$ and one element in $\{2\}$. (A $\left(1,s'\right)$-subroad is the same as a $\left(3,s'\right)$-subroad, and an $\left(s,1\right)$-subroad is the same as an $\left(s,3\right)$-subroad, similar to Definition~\ref{defssubproduct}.)
\end{definition}

An $\left(s,s'\right)$-subroad is a subroad that begins with $\sigma_s$ (where either $s\in \{1,3\}$ or $s = 2$) and ends with $\sigma_{s'}$ (where either $s'\in \{1,3\}$ or $s' = 2$), and is maximal with respect to this property. We now define elementary subroads, which are the basic building blocks of all $\left(s,s'\right)$-subroads.

\begin{definition}
\label{defelementarysubroads}
Let $R$ be an $\left(s,s'\right)$-subroad of $\sigma$. An \textit{elementary subroad of $R$} is a subroad of $R$ that is either a non-isolated $\sigma_2$ subroad of $R$ (defined in \textbf{(a)}) or an isolated $\sigma_2$ subroad of $R$ (defined in \textbf{(b)}).
\begin{description}
\item[(a)] A \textit{non-isolated $\sigma_2$ subroad of $R$} is a subroad of $R$ that is either an $i$-constant subroad of $R$ for some $i\in \{1,3\}$ (defined in \textbf{(i)}) or an alternating subroad of $R$ (defined in \textbf{(ii)}). We require that a non-isolated $\sigma_2$ subroad of $R$ is an $\left(s_0,s_0'\right)$-subroad for some $s_0'\in \{1,3\}$ unless it is at the end of a non-isolated $\sigma_2$ subproduct of $\sigma$ or it is at the end of $R$.

\begin{description}
\item[(i)] An $i$-\textit{constant subroad of $R$} is an $\left(s_0,s_0'\right)$-subroad of $R$ contained in a non-isolated $\sigma_2$ subproduct of $\sigma$ such that it is maximal with respect to being the product of alternating copies of $\sigma_2^{c_j}$ and $\sigma_i^{\phi_j}$ for a fixed $i\in \{1,3\}$ (where $j$ varies over the product) and there is at least one copy of $\sigma_i^{\phi_j}$. 
\item[(ii)] An \textit{alternating subroad of $R$} is an $\left(s_0,s_0'\right)$-subroad of $R$ contained in a non-isolated $\sigma_2$ subproduct of $\sigma$ such that it is maximal with respect to being the product of alternating copies of $\sigma_2^{c_j}$ and $\sigma_1^{a_j}\sigma_3^{b_j}$ with $a_j,b_j>0$ (where $j$ varies over the product). 
\end{description}
\item[(b)] An \textit{isolated $\sigma_2$ subroad of $R$} is a subroad of $R$ that is either a $1$-initial isolated $\sigma_2$ subroad of $R$ (defined in \textbf{(i)}), a $3$-initial isolated $\sigma_2$ subroad of $R$ (defined in \textbf{(ii)}), or an end isolated $\sigma_2$ subroad of $R$ (defined in \textbf{(iii)}). Let us adopt the setup of Definition~\ref{isp}. 
\begin{description}
\item[(i)] A \textit{$1$-initial isolated $\sigma_2$ subroad of $R$} is an $\left(s_0,s_0'\right)$-subroad of $R$ that is maximal with respect to being either contained in a type II isolated $\sigma_2$ subproduct of $\sigma$, properly contained in a type I isolated $\sigma_2$ subproduct of $\sigma$ with $a_p>0$, or contained in a type III isolated $\sigma_2$ subproduct of $\sigma$.
\item[(ii)] A \textit{$3$-initial isolated $\sigma_2$ subroad of $R$} is an $\left(s_0,s_0'\right)$-subroad of $R$ that is maximal with respect to being contained in a type I isolated $\sigma_2$ subproduct of $\sigma$ with $a_p = 0$.
\item[(iii)] An \textit{end isolated $\sigma_2$ subroad of $R$} is an $\left(s_0,s_0'\right)$-subroad of $R$ that is a type I isolated $\sigma_2$ subproduct of $\sigma$ with $a_p>0$.
\end{description}
\end{description}
\end{definition}

We include Remark~\ref{remarkendisolated} to note the simple characterization of end isolated $\sigma_2$ subroads.

\begin{remark}
\label{remarkendisolated}
Let $\sigma = \prod_{i=1}^{n} \sigma_1^{a_i}\sigma_3^{b_i}\sigma_2^{c_i}$ be the minimal form of $\sigma$. A type I isolated $\sigma_2$ subproduct of $\sigma$ with $a_p>0$ contains a $2$-block unless it is at the end of the minimal form of $\sigma$ and it is equal to $\sigma_1^{a_n}\sigma_3^{b_n}\sigma_2$ (with $c_n = 1$). In particular, an end isolated $\sigma_2$ subroad of $R$ is equal to $\sigma_1^{a_n}\sigma_3^{b_n}\sigma_2$, which justifies the terminology.
\end{remark}

Let $P$ be an $s''$-subproduct of $\sigma$ and let $R$ be the $\left(s,s'\right)$-subroad immediately following $P$. We will find it convenient to extend $P$-paths to $PR$-paths with respect to the decomposition of $R$ into elementary subroads. 

\begin{proposition}
If $R$ is an $\left(s,s'\right)$-subroad, then $R$ can be uniquely expressed as a product $R = \prod_{j=1}^{k} R_j$, where each $R_j$ is an elementary subroad of $R$. 
\end{proposition}
\begin{proof}
We consider the intersection of $R$ with the product decomposition of $\sigma$ into isolated $\sigma_2$ subproducts and non-isolated $\sigma_2$ subproducts. 

Let us adopt the setup of Definition~\ref{isp}. The intersection of $R$ with either a type II isolated $\sigma_2$ subproduct of $\sigma$, a type I isolated $\sigma_2$ subproduct of $\sigma$ with $a_p>0$ that is not contained in $R$, or a type III isolated $\sigma_2$ subproduct of $\sigma$, is a $1$-initial isolated $\sigma_2$ subroad of $R$. The intersection of $R$ with a type I isolated $\sigma_2$ subproduct of $\sigma$ with $a_p = 0$ is a $3$-initial isolated $\sigma_2$ subroad of $R$. Finally, the intersection of $R$ with a type I isolated $\sigma_2$ subproduct of $\sigma$ with $a_p>0$ that is contained in $R$ is an end isolated $\sigma_2$ subroad of $R$.

The intersection of $R$ with a non-isolated $\sigma_2$ product of $\sigma$ can be uniquely expressed as the product of non-isolated $\sigma_2$ subroads. Indeed, the requirement at the beginning of Definition~\ref{defelementarysubroads} \textbf{(a)} implies that distinct non-isolated $\sigma_2$ subroads do not overlap.
\end{proof}

\subsubsection{The theory of bad subroads}

The next step is to precisely define bad subroads. The rough definition (the formal definition follows) is that an $\left(s,s'\right)$-subroad $R$ is bad if the exponents of generators in $R$ are all equal to $0$, $1$ or $2$ (with an exception for alternating exponents in isolated $\sigma_2$ subproducts of $R$), and there are also further strong constraints on the order in which generators appear. Firstly, we define elementary bad subroads, which are the basic building blocks of all bad subroads.

\begin{definition}
\label{defelementarybadsubroad}
An \textit{elementary bad subroad} is either a non-isolated $\sigma_2$ bad subroad (defined in \textbf{(a)}) or an \textit{isolated $\sigma_2$ bad subroad} (defined in \textbf{(b)}). 
\begin{description}
\item[(a)] A \textit{non-isolated $\sigma_2$ bad subroad} is either an $i$-constant bad subroad (defined in \textbf{(i)}) or an alternating bad subroad (defined in \textbf{(ii)}). 
\begin{description}
\item[(i)] An \textit{$i$-constant bad subroad} is an $i$-constant subroad such that each exponent of a generator is at most $2$. 
\item[(ii)] An \textit{alternating bad subroad} is an alternating subroad such that each exponent of $\sigma_2$ is equal to $2$ and each exponent of $\sigma_1$ and $\sigma_3$ is equal to $1$. 
\end{description}
\item[(b)] An \textit{isolated $\sigma_2$ bad subroad} is either an $i$-initial isolated $\sigma_2$ bad subroad (defined in \textbf{(i)}) or an end isolated $\sigma_2$ bad subroad (defined in \textbf{(ii)}).
\begin{description}
\item[(i)] Let $R$ be an $i$-initial isolated $\sigma_2$ subroad. Let $R$ be an $\left(s,s'\right)$-subroad and let us define $i'$ such that $\{i,i'\} = \{1,3\}$. We write that $R$ is an \textit{$i$-initial isolated $\sigma_2$ bad subroad} if the following conditions are satisfied.  If $s = 2$, then the index of the first generator in $R$ that is not equal to $2$ is $i'$. If $s\in \{1,3\}$, then the index of the first generator in $R$ is $i$. Finally, each (non-zero) exponent of $\sigma_i$ is at most $2$. 
\item[(ii)] An \textit{end isolated $\sigma_2$ bad subroad} is an end isolated $\sigma_2$ subroad such that the exponents of $\sigma_1$ and $\sigma_3$ are both equal to $1$. 
\end{description}
\end{description}
\end{definition}

We include Remark~\ref{remarkexponentone} to note that the ``at most two" condition on the exponents of generators in an elementary bad subroad $R$ is simply ``equal to $2$", except possibly at the beginning or the end of the minimal form of $\sigma$. The statements in Remark~\ref{remarkexponentone} are all a consequence of Lemma~\ref{minimalform} \textbf{(i)} and \textbf{(ii)} and the hypothesis that $R$ is a subroad (i.e., $R$ does not overlap with a block).

\begin{remark}
\label{remarkexponentone}
If it exists, an end isolated $\sigma_2$ bad subroad is unique and is equal to $\sigma_1\sigma_3\sigma_2$. The exponents of the generators in an $i$-constant subroad $R$, except possibly the first and last exponents, are all greater than $1$. If $R$ is at the beginning (resp. end) of the minimal form of $\sigma$ and the first (resp. last) generator is not $\sigma_2$, then the first (resp. last) exponent can equal $1$ (and this is the only case where an exponent in $R$ can equal $1$). In particular, in an $i$-constant bad subroad, all but possibly the first and last exponents are equal to $2$. Similarly, the exponents of $\sigma_2$ in an alternating subroad are all greater than $1$. Finally, in an $i$-initial isolated $\sigma_2$ subroad $R$, the (non-zero) exponents of $\sigma_i$, except possibly the last exponent, are all greater than $1$. If $R$ is at the end of the minimal form of $\sigma$, then the last exponent can equal $1$ (and this is the only case where an exponent of $\sigma_i$ in $R$ can equal $1$). In particular, in an $i$-initial isolated $\sigma_2$ bad subroad, all but possibly the last exponent of $\sigma_i$ are equal to $2$.
\end{remark}

In light of Remark~\ref{remarkexponentone}, we can view elementary bad subroads as elementary subroads where the exponents of the generators are minimized (in the case of isolated $\sigma_2$ subroads, we consider alternating exponents). A generic elementary subroad is not an elementary bad subroad. 

The fundamental characteristic of an elementary bad subroad $P'\setminus P$ is a uniqueness of extension property of bad $P$-paths to $P'$-paths. A special case of this is exhibited in Example~\ref{exglobalcancellation}, where the elementary bad subroad is a $1$-constant bad subroad. We now precisely state and establish this characteristic of elementary bad subroads in full generality.

\begin{proposition}
\label{pelementarybadsubroadimpliesuniqueextension}
Let us adopt the notation of Definition~\ref{defss'subroad} and let $R = P'\setminus P$ be an elementary bad $\left(s,s'\right)$-subroad. Let $W$ be a bad $\left(r,s''\right)$-type $P$-path. If $s''\in \{1,3\}$, then define $i = s''$. If $s'' = 2$, then define $i$ such that $W$ is an $i$-switch bad $P$-path. (Note that $i\in \{1,3\}$.) 

Let $W'$ be an extension of $W$ to an $\left(r,s'''\right)$-type $P'$-path for some $s'''\in \{1,2,3\}$. We uniquely characterize $W'$ according to the type of the elementary bad subroad $R$ and the value of $s'''$, as follows. (We note that in each of the following cases, the final vertex change of the extension may be omitted to ensure that the final vertex of the extension is $s'''$.) Let us define $i'$ such that $\{i,i'\} = \{1,3\}$. 
\begin{description}
\item[(a)] Let $R$ be an $i$-constant bad subroad. The $P'$-path $W'$ is uniquely defined by extending the $P$-path $W$ by $i\to 2$ vertex changes at every second $\sigma_2$ in $R$ and $2\to i$ vertex changes at every second $\sigma_i$ in $R$. 
\item[(b)] Let $R$ be an alternating bad subroad. The $P'$-path $W'$ is uniquely defined by extending the $P$-path $W$ by the following vertex changes. We have a $2\to i'$ vertex change at every odd-numbered $\sigma_1\sigma_3$ in $R$ followed by an $i'\to 2$ vertex change at the second $\sigma_2$ that follows every odd-numbered $\sigma_1\sigma_3$ in $R$. We also have a $2\to i$ vertex change at every even-numbered $\sigma_1\sigma_3$ in $R$ followed by an $i\to 2$ vertex change at the second $\sigma_2$ that follows every even-numbered $\sigma_1\sigma_3$ in $R$. 
\item[(c)] Let $R$ be an $i$-initial isolated $\sigma_2$ bad subroad. The $P'$-path $W'$ is uniquely defined by extending the $P$-path $W$ by $i\to 2$ vertex changes at every second $\sigma_2$ in $R$ and $2\to i$ vertex changes at every second $\sigma_i$ in $R$.
\item[(d)] Let $R$ be an end isolated $\sigma_2$ bad subroad (in this case, $s'' = 2$). The $P'$-path $W'$ is uniquely defined by extending the $\left(r,2\right)$-type $P$-path $W$ by a $2\to i'$ vertex change at the $\sigma_{i'}$ in $R = \sigma_1\sigma_3\sigma_2$.
\end{description}
\end{proposition}
\begin{proof}
The statements are consequences of Lemma~\ref{lemmagoodext}. 
\end{proof}

We now characterize the possible values of $s'''$ for which there exists the (unique) extension of $W$ to an $\left(r,s'''\right)$-type $P'$-path $W'$ in the context of Proposition~\ref{pelementarybadsubroadimpliesuniqueextension}.

\begin{corollary}
\label{celementarybadsubroadimpliesuniqueextension}
Let us adopt the setup of Proposition~\ref{pelementarybadsubroadimpliesuniqueextension}. 

\begin{description}
\item[(a)] Let $R$ be an $i$-constant bad subroad. Firstly, there is no extension of $W$ to an $\left(r,i'\right)$-type $P'$-path.
\begin{description}
\item[(i)] If the last exponent in $R$ is greater than $1$, then there is a unique extension of $W$ to an $\left(r,s'\right)$-type $P'$-path $W'$. The $P'$-path $W'$ is bad.
\item[(ii)] If the last exponent in $R$ is equal to $1$ (in which case, the last generator in $R$ is $\sigma_i$), then there is a unique extension of $W$ to a $P'$-path $W'$ and $W'$ is an $\left(r,2\right)$-type $P'$-path.
\end{description}
\item[(b)] If $R$ is an alternating bad subroad, then there is a unique extension of $W$ to an $\left(r,s'\right)$-type $P'$-path $W'$. Let $N$ be the number of $\sigma_1\sigma_3$s in $R$. If $N$ is even, then define $t=i$. If $N$ is odd, then define $t$ such that $\{t,i\} = \{1,3\}$. If $s'\in \{1,3\}$, then $W'$ is a bad $\left(r,t\right)$-type $P'$-path, and if $s' = 2$, then $W$ is a $t$-switch $\left(r,2\right)$-type $P'$-path.
\item[(c)] Let $R$ be an $i$-initial isolated $\sigma_2$ bad subroad. Let $M$ be the number of $\sigma_2$s in $R$. Firstly, there is no extension of $W$ to an $\left(r,i'\right)$-type $P'$-path.
\begin{description}
\item[(i)] Let $s' = 2$. If $M$ is even, then there is a unique extension of $W$ to an $\left(r,s'\right)$-type $P'$-path and it is a bad $\left(r,s'\right)$-type $P'$-path. If $M$ is odd, then there is a unique extension of $W$ to a $P'$-path and it is an $\left(r,i\right)$-type $P'$-path. 
\item[(ii)] Let $s'\in \{1,3\}$. If $M$ is even, then there is a unique extension of $W$ to an $\left(r,s'\right)$-type $P'$-path and it is a bad $\left(r,i\right)$-type $P'$-path. If $M$ is odd, then there is a unique extension of $W$ to a $P'$-path and it is a good $\left(r,i\right)$-type $P'$-path (in this case, the last generator in $R$ is $\sigma_{i'}$).
\end{description}
\item[(d)] If $R$ is an end isolated $\sigma_2$ bad subroad, then there is a unique extension of $W$ to a $P'$-path and it is an $\left(r,i'\right)$-type $P'$-path. 
\end{description}
\end{corollary}
\begin{proof}
The statements are consequences of the respective statements in Proposition~\ref{pelementarybadsubroadimpliesuniqueextension}.
\end{proof}

We now define bad subroads in full generality in terms of the decomposition of an $\left(s,s'\right)$-subroad into elementary subroads. 

\begin{definition}
\label{defbadsubroad}
Let $R$ be an $\left(s,s'\right)$-subroad and let $R = \prod_{j=1}^{k} R_j$ be the product decomposition of $R$ into elementary subroads. We write that $R$ is a \textit{bad subroad} if the product decomposition satisfies the following properties:
\begin{description}
\item[(i)] The elementary subroad $R_j$ is an elementary bad subroad (in the sense of Definition~\ref{defelementarybadsubroad}) for each $1\leq j\leq k$.
\item[(ii)] An $i$-constant subroad is not followed by an $i'$-constant subroad for $i'\neq i$. 
\item[(iii)] Let $1\leq j\leq k$ be such that an $i$-constant subroad $R_{j}$ is followed by an isolated $\sigma_2$ subroad $R_{j+1}$. If $R_{j+1}$ is an $i'$-initial isolated $\sigma_2$ subroad, then $i'=i$.
\item[(iv)]  Let $1\leq j\leq k$ be such that $R_{j}$ is an alternating subroad (in particular, neither $R_{j-1}$ nor $R_{j+1}$ is an alternating subroad). Let $i$ be the index of the last generator in $R_{j-1}$, and let $i'$ be the index of the first generator in $R_{j+1}$ that is not equal to $2$ (in particular, $\{i,i'\}\subseteq \{1,3\}$). Let $N$ be the number of $\sigma_1\sigma_3$s in $R_{j}$. If $N$ is even, then $i' = i$, and if $N$ is odd, then $i'\neq i$.
\item[(v)] If an isolated $\sigma_2$ subroad is followed by a non-isolated $\sigma_2$ subroad, then the number of $\sigma_2$s in the isolated $\sigma_2$ subroad is even. 
\item[(vi)] Let $1\leq j\leq k$ be such that an isolated $\sigma_2$ subroad $R_{j}$ is followed by a non-isolated $\sigma_2$ subroad $R_{j+1}$. Let us assume that $\{i,i'\} = \{1,3\}$. If the index of the last generator in $R_{j}$ is $i$, then $R_{j+1}$ is not an $i'$-constant subroad. 
\end{description}
\end{definition}

We introduce the following categorization of bad subroads.

\begin{definition}
\label{defiswitchbadsubroad}
Let $R$ be a bad subroad and let $R = \prod_{j=1}^{k} R_j$ be the product decomposition of $R$ into elementary subroads. If $R_1$ is an alternating subroad, then define $N$ to be the number of $\sigma_1\sigma_3$s in $R_1$ and define $j_1 = 2$. If $R_1$ is not an alternating subroad, then define $N = 0$ and define $j_1 = 1$. If $N$ is even, then define $i_0= i$, and if $N$ is odd, then define $i_0$ such that $\{i,i_0\} = \{1,3\}$. We write that $R$ is an \textit{$i$-switch bad subroad} if $R_{j_1}$ is either an $i_0$-constant subroad, an $i_0$-initial isolated $\sigma_2$ subroad, or an end isolated $\sigma_2$ subroad.
\end{definition}

If $R$ is a bad subroad, then $R$ is an $i$-switch bad subroad for some $i\in \{1,3\}$. In Definition~\ref{defbadsubroad}, condition \textbf{(iv)} of a bad subroad is vacuous if $j = 1$. The condition that a bad subroad $R$ is an $i$-switch bad subroad is designed to place a restriction on the beginning of $R$ relative to $i$ that extends condition \textbf{(iv)} and implies the following uniqueness of extension property of paths. If $R = P'\setminus P$ is an $i$-switch bad subroad, and if $W$ and $i$ are as in Proposition~\ref{pelementarybadsubroadimpliesuniqueextension}, then there is at most one extension of the bad $P$-path $W$ to a $P'$-path with a prescribed final vertex. We now characterize the extension.

\begin{proposition}
\label{pbadsubroad}
Let us adopt the notation of Definition~\ref{defss'subroad} and let $R = P'\setminus P$ be an $i$-switch bad subroad. Let us write $R = \prod_{j=1}^{k} R_j$ for the product decomposition of $R$ into elementary subroads. Let $W$ be a bad $\left(r,s''\right)$-type $P$-path. If $s''\in \{1,3\}$, then define $i = s''$. If $s'' = 2$, then define $i$ such that $W$ is an $i$-switch bad $P$-path. (Note that $i\in \{1,3\}$.) 

Let $W'$ be an extension of $W$ to an $\left(r,s'''\right)$-type $P'$-path for some $s'''\in \{1,2,3\}$. If $0\leq j'\leq k$, then let $P_{j'} = P\prod_{j=1}^{j'}R_j$ (in particular, $P_0 = P$ and $P_k = P'$). We uniquely characterize $W'$ according to the value of $s'''$ inductively as follows.

If $0\leq j'\leq k-1$, then let us assume that we have defined an extension of $W$ to an $\left(r,s_{j'}\right)$-type $P_{j'}$-path $W_{j'}$ for some $s_{j'}\in \{1,2,3\}$ such that $P_{j'}$ is an $s_{j'}$-subproduct. We define the $\left(r,s_{j'+1}\right)$-type $P_{j'+1}$-path $W_{j'+1}$ by uniquely extending $W_{j'}$ according to the type of the elementary bad subroad $R_{j'+1}$ as in the statement of Proposition~\ref{pelementarybadsubroadimpliesuniqueextension}. We have $W' = W_k$.
\end{proposition}
\begin{proof}
Let $0\leq j'\leq k-1$. We will establish the statement by induction on $j'$ with the following hypotheses. Firstly, we hypothesize that there is a unique extension of $W$ to an $\left(r,s_{j'}\right)$-type $P_{j'}$-path $W_{j'}$ for some $s_{j'}\in \{1,2,3\}$ such that $P_{j'}$ is an $s_{j'}$-subproduct. Secondly, we hypothesize that if there is an extension of $W$ to an $\left(r,t_{j'}\right)$-type $P_{j'}$-path for $t_{j'}\neq s_{j'}$, then there is no (admissible) extension of an $\left(r,t_{j'}\right)$-type $P_{j'}$-path to a $P_{j'+1}$-path. If $j' = 0$, then the two hypotheses are certainly satisfied. Let us assume that the two hypotheses are satisfied for $j'$. If $0\leq j'\leq k-2$, then we will show that the two hypotheses are satisfied for $j'+1$. 

Firstly, Corollary~\ref{celementarybadsubroadimpliesuniqueextension} implies that there is a unique extension of $W_{j'}$ to an $\left(r,s_{j'+1}\right)$-type $P_{j'+1}$-path. Let us assume that there is an extension of $W_{j'}$ to an $\left(r,t_{j'+1}\right)$-type $P_{j'+1}$-path for some $t_{j'+1}\neq s_{j'+1}$. We will show that there is no extension of an $\left(r,t_{j'+1}\right)$-type $P_{j'+1}$-path to a $P_{j'+2}$-path. We consider cases according to the type of the elementary bad subroad $R_{j'+1}$. Note that the hypothesis of either Corollary~\ref{celementarybadsubroadimpliesuniqueextension} \textbf{(a)} \textbf{(i)}, \textbf{(b)} or \textbf{(c)} \textbf{(ii)} is satisfied for the elementary bad subroad $R_{j'+1}$ since $j'\leq k-2$.

If $R_{j'+1}$ is an $i$-constant bad subroad, then Corollary~\ref{celementarybadsubroadimpliesuniqueextension} \textbf{(a)} \textbf{(i)} implies that $t_{j'+1}\in \{i,2\}$. The hypothesis that $R$ is a bad subroad implies that either $R_{j'+2}$ is an alternating bad subroad and $t_{j'+1} = 2$, $R_{j'+2}$ is an $i$-initial isolated $\sigma_2$ bad subroad and $t_{j'+1} = i$, or $R_{j'+2}$ is an end isolated $\sigma_2$ bad subroad and $t_{j'+1}=i$. Lemma~\ref{lemmagoodext} \textbf{(v)} (in the first case) and Lemma~\ref{lemmagoodext} \textbf{(iv)} (in the second and third cases) imply that an $\left(r,t_{j'+1}\right)$-type $P_{j'+1}$-path does not have an extension to a $P_{j'+2}$-path.

If $R_{j'+1}$ is an alternating bad subroad, then Corollary~\ref{celementarybadsubroadimpliesuniqueextension} \textbf{(b)} implies that $t_{j'+1}\in \{i,2\}$ for some $i\in \{1,3\}$. The hypothesis that $R$ is a bad subroad implies that either $R_{j'+2}$ is an $i$-constant bad subroad and $t_{j'+1}=2$, $R_{j'+2}$ is an $i$-initial isolated $\sigma_2$ bad subroad and $t_{j'+1} = i$, or $R_{j'+2}$ is an end isolated $\sigma_2$ bad subroad and $t_{j'+1} = i$. Lemma~\ref{lemmagoodext} \textbf{(v)} (in the first case) and Lemma~\ref{lemmagoodext} \textbf{(iv)} (in the second and third cases) imply that there is no extension of an $\left(r,t_{j'+1}\right)$-type $P_{j'+1}$-path to a $P_{j'+2}$-path.

If $R_{j'+1}$ is an isolated $\sigma_2$ bad subroad, then Corollary~\ref{celementarybadsubroadimpliesuniqueextension} \textbf{(c)} \textbf{(ii)} and the hypothesis that $R$ is a bad subroad imply that $t_{j'+1} = 2$. Lemma~\ref{lemmagoodext} \textbf{(v)} implies that there is no extension of an $\left(r,t_{j'+1}\right)$-type $P_{j'+1}$-path to a $P_{j'+2}$-path. 

We have considered all cases and established that if $0\leq j'\leq k-1$, then the two hypotheses are satisfied for $j'$. Finally, Corollary~\ref{celementarybadsubroadimpliesuniqueextension} and the fact that the two hypotheses are satisfied for $j' = k-1$ imply the conclusion. 
\end{proof}
 
We state the fundamental uniqueness property for path extensions with respect to an $i$-switch bad subroad.

\begin{corollary}
\label{cuniquebadsubroad}
Let us adopt the notation of Definition~\ref{defss'subroad} and let $R = P'\setminus P$ be an $\left(s,s'\right)$-subroad such that the last exponent in $R$ is greater than $1$. Let $W$ be a bad $\left(r,s''\right)$-type $P$-path. If $s''\in \{1,3\}$, then define $i = s''$. If $s'' = 2$, then define $i$ such that $W$ is an $i$-switch bad $P$-path. If $R$ is an $i$-switch bad subroad, then there is a unique extension of $W$ to an $\left(r,s'\right)$-type $P'$-path.
\end{corollary}
\begin{proof}
The statement is a consequence of Proposition~\ref{pbadsubroad}.
\end{proof}

\begin{remark}
In fact, the converse of Corollary~\ref{cuniquebadsubroad} is true. If there is a unique extension of $W$ to an $\left(r,s'\right)$-type $P'$-path, then $R$ is an $i$-switch bad subroad. We do not need the statement in the sequel but we note that it is straightforward to prove. Firstly, we can establish the statement if $R$ is an elementary bad subroad, and secondly, we can establish that the conditions on the product decomposition of $R$ into elementary subroads in Definition~\ref{defbadsubroad} and Definition~\ref{defiswitchbadsubroad} are satisfied. 
\end{remark}

The following terminology for $i$-switch bad subroads will be convenient.

\begin{definition}
Let $R$ be an $i$-switch bad subroad such that the last exponent in $R$ is greater than $1$. Let us adopt the setup of Corollary~\ref{cuniquebadsubroad}. If the unique extension of $W$ to an $\left(r,s'\right)$-type $P'$-path is either an $\left(r,i'\right)$-type $P'$-path, or an $i'$-switch $\left(r,2\right)$-type $P'$-path, then we write that $R$ is an \textit{$i'$-terminal bad subroad}.
\end{definition}

If $R$ is an $i'$-terminal bad subroad, then $i'\in \{1,3\}$. Furthermore, a bad subroad $R$ such that the last exponent in $R$ is greater than $1$ is an $i'$-terminal bad subroad for some $i'\in \{1,3\}$. We can explicitly determine the value of $i'$ such that an $i$-switch bad subroad is an $i'$-terminal bad subroad.

\begin{proposition}
\label{pterminalbadsubroad}
Let $R$ be an $i'$-terminal $i$-switch bad subroad. Let $N$ be the number of $\sigma_1\sigma_3$s in $R$. If $N$ is even, then $i' = i$, and if $N$ is odd, then $i'$ is such that $\{i,i'\} = \{1,3\}$.
\end{proposition}
\begin{proof}
The statement is a consequence of Proposition~\ref{pbadsubroad}.
\end{proof}

We now define the weight of a bad subroad, which we considered in a special case in Example~\ref{exglobalcancellation}.

\begin{definition}
\label{defwidth}
Let $R$ be an $i_0$-switch bad subroad and let $R = \prod_{j=1}^{k} R_j$ be the product decomposition of $R$ into elementary subroads. If $1\leq j\leq k$, then we refer to $R_j$ as a \textit{factor of $R$}. 
\begin{description}
\item[(i)] Let $\omega'\left(R\right)$ be the total number of squares of generators in non-isolated $\sigma_2$ subroads that are factors of $R$. 
\item[(ii)] Let $i\in \{1,3\}$ and define $i'$ to be such that $\{i,i'\} = \{1,3\}$. Let $\omega_{i}''\left(R\right)$ be the total number of $\sigma_{i}^2$s in $i$-initial isolated $\sigma_2$ subroads that are factors of $R$. Let $\omega_i'''\left(R\right)$ be the sum of the exponents of $\sigma_{i'}$s in $i$-initial isolated $\sigma_2$ subroads that are factors of $R$. We define $\omega''\left(R\right) = \omega_1''\left(R\right) + \omega_3''\left(R\right)$ and $\omega'''\left(R\right) = \omega_1'''\left(R\right) + \omega_3'''\left(R\right)$. 
\end{description}
The \textit{weight of $R$} is $\omega\left(R\right) = \omega'\left(R\right) + \omega''\left(R\right) + \omega'''\left(R\right)$.
\end{definition}

Of course, the weight of a bad subroad $R$ can easily be computed by inspection, since it is defined by a simple formula in terms of a product expansion of $R$.

Let $R = P'\setminus P$ be an $i$-switch bad subroad. Let us consider a bad $P$-path $X$ that is either an $i$-switch $\left(r,2\right)$-type $P$-path or an $\left(r,i\right)$-type $P$-path. Let us consider a good $P$-path $Y$. If we consider maximal $q$-resistance extensions of $X$ and $Y$ to $P'$-paths, then we will establish in Lemma~\ref{lweightcompute} (in Subsection~\ref{subsecproofstrongerversion}) that the weight of $R$ measures the amount by which the $q$-resistance of $Y$ catches up to the $q$-resistance of $X$. 

A special case of Lemma~\ref{lweightcompute} is exhibited in Example~\ref{exglobalcancellation}. We recall from Example~\ref{exglobalcancellation} that if the weight of $P'\setminus P$ is not precisely equal to the discrepancy between the $q$-resistance of $X$ and the $q$-resistance of $Y$, then $P'$ is strongly $r$-regular and global cancellation of weights of paths does not occur. We will establish this formally in Subsection~\ref{subsecproofstrongerversion}.

Finally, we define the singular weight of an isolated $\sigma_2$ subroad following a block, where all the (non-zero) exponents of $\sigma_1$ and $\sigma_3$ are $2$. We will only consider the singular weight for such isolated $\sigma_2$ subroads following special blocks (which will include singular blocks).

\begin{definition}
\label{defsingularroadweight}
Let $B$ be a block. The \textit{singular subroad following $B$} is the maximal subroad $R'$ following $B$ and contained in an isolated $\sigma_2$ subproduct such that each non-zero exponent of $\sigma_1$ and $\sigma_3$ in $R'$ is $2$ and $R'$ does not end with $\sigma_2$.

We define the \textit{singular weight} $\omega_{\text{sing}}\left(R'\right)$ of $R'$ according to the parity of the number of $\sigma_2$s in $R'$ and whether $B$ is a $2$-block or a $3$-block. Let $M$ be one more than the number of $\sigma_2$s in $R'$.

\begin{description}
\item[(a)] If $B$ is a $2$-block, then \[ \omega_{\text{sing}}\left(R'\right) =        
\begin{cases} \frac{3M-2}{2}  & \text{if }M\text{ is even} \\
       \frac{3M-5}{2} & \text{if }M\text{ is odd}.
   \end{cases}
\]
\item[(b)] If $B$ is a $3$-block, then \[ \omega_{\text{sing}}\left(R'\right) =        
\begin{cases} \frac{3M}{2}  & \text{if }M\text{ is even} \\
       \frac{3M-1}{2} & \text{if }M\text{ is odd}.
   \end{cases}
\]
\end{description}
\end{definition}

We can easily compute the singular weight of the singular subroad following a block by inspection. 

\begin{remark}
Let $B$ be a block and let $R'$ be the singular subroad following $B$. We note that $R'$ is either followed by an isolated $\sigma_2$ (if $R'$ is not the maximal isolated $\sigma_2$ subproduct following $B$), or followed by a non-isolated $\sigma_2$ (if $R'$ is the maximal isolated $\sigma_2$ subproduct following $B$), unless $\sigma$ ends with $R'$.
\end{remark}

\subsubsection{The statement of the stronger version of the main result}

We now precisely state the definition of weakly normal braids and Theorem~\ref{mainstronger}. Firstly, we define the width of a block, which we considered in a special case in Example~\ref{exglobalcancellation}. If $P$ is the minimal $s$-subproduct of $\sigma$ containing a block $B$, then we will use the width of $B$ to measure the discrepancy between the $q$-resistances of a maximal $q$-resistance good $P$-path and a maximal $q$-resistance bad $P$-path. If $B$ is the first abnormal block in $\sigma$, then the width of $B$ is a precise measurement of this discrepancy (in general, the width of $B$ is a perturbation of this discrepancy).

\begin{definition}
\label{defblockwidth}
Let $B$ be a block in $\sigma$.
\begin{description}
\item[(i)] If $B = \sigma_1^{a_p}\sigma_3^{b_p}\sigma_2\sigma_1^{2}$ is a generic $2$-block, then the \textit{width of $B$} is $a_p - b_p - 1$.
\item[(ii)] If $B = \sigma_1^{a_p}\sigma_3\sigma_2\sigma_1^{2}$ is a singular $2$-block, then the \textit{width of $B$} is $a_p$.
\item[(iii)] If $B = \sigma_2^{c_{p-1}}\sigma_3\sigma_2^{c_p}$ is a $3$-block with $c_{p-1}>1$ and $c_p = 2$, then the \textit{width of $B$} is $c_{p-1} - 1$. 
\item[(iv)] If $B = \sigma_2^{c_{p-1}}\sigma_3\sigma_2^{c_p}$ is a $3$-block with $c_{p-1} = 1$ and $c_p>1$, then the \textit{width of $B$} is $a_{p-1}-c_p$. 
\item[(v)] If $B = \sigma_2^{c_{p-1}}\sigma_3\sigma_2$ is a singular $3$-block, then the \textit{width of $B$} is $c_{p-1} $. 
\end{description}
\end{definition}

The next step is to define weakly normal braids. A positive braid $\sigma$ is normal if there is no abnormal block in $\sigma$. Similarly, we will define a positive braid $\sigma$ to be weakly normal if there is no \textit{terminal abnormal string} in $\sigma$, which is a much more general condition, since terminal abnormal strings are typically very long and highly constrained. Furthermore, the condition that a positive braid is weakly normal is generic even among positive braids that are not normal. 

Firstly, we define the notion of a \textit{bad string}.

\begin{definition}
\label{defbadstring}
Let $S = \prod_{j=1}^{k} B_jR_j$ be a product, where $B_j$ is a block, $R_j$ is an $\left(s_j,s_j'\right)$-road for $1\leq j\leq k-1$, and $R_k$ is an $\left(s_k,s_k'\right)$-subroad. 

If $B_j$ is a $2$-block, then we define $i_j = 1$. If $B_j = \sigma_2^{c_{p_j-1}}\sigma_3\sigma_2^{c_{p_j}}$ is a generic $3$-block such that either $c_{p_j-1}=1$, or $j>1$ and $c_{p_j-1}=2$, then we define $i_j = 1$. Otherwise, if $B_j$ is a $3$-block, then we define $i_j = 3$.

If $B_{j+1}$ is a $3$-block and $i_{j+1} = 1$, then we define $R_{j,0}$ to be the maximal $\left(s_j,2\right)$-subroad such that $R_{j,0}\subseteq R_j$. Otherwise, we define $R_{j,0} = R_j$. 

The product $S = \prod_{j=1}^{k} B_jR_j$ is a \textit{bad string} if $R_{j,0}$ is an $i_j$-switch bad subroad for $1\leq j\leq k$, and $R_{j,0}$ is a $1$-terminal bad subroad when $R_{j,0}\neq R_j$. 
\end{definition}

We introduce terminology for bad strings.

\begin{definition}
\label{deftermbadstring}
Let $S = \prod_{j=1}^{k} B_jR_j$ be a bad string and let us adopt the notation of Definition~\ref{defbadstring}. If $B_{j+1}$ is a $2$-block (in which case $R_j = R_{j,0}$), then we write that $B_j$ is an \textit{$i_j'$-switch $2$-block} if $R_{j,0}$ is an $i_j'$-terminal bad subroad. 

Let $P$ be the minimal $s$-subproduct of $\sigma$ such that $S\subseteq P$. We write that $P$ is an \textit{$i_k'$-switch subproduct} if $R_k$ is an $i_k'$-terminal bad subroad. If $P$ ends with an isolated $\sigma_2$ subproduct with an odd number of $\sigma_2$s, then we write that $P$ is a \textit{$\nu$-switch subproduct}. If $P$ either does not end with an isolated $\sigma_2$ subproduct or ends with an isolated $\sigma_2$ subproduct with an even number of $\sigma_2$s, then we write that $P$ is a \textit{$\mu$-switch subproduct}.
\end{definition}

We now define \textit{abnormal strings}.

\begin{definition}
\label{defabnormalstring}
Let us adopt the notation of Definition~\ref{defbadstring}. We inductively define a special family $S_1,\dots,S_l$ of bad strings in $\sigma$ as follows. Let us assume that we have defined $S_{i'}$ for $i'<i$. Let $B_1$ be the first abnormal block in $\sigma$ after $S_{i-1}$ (if $i=1$, then $B_1$ is the first abnormal block in $\sigma$). We define $S_i = \prod_{j=1}^{k} B_jR_j$ to be the maximal bad string beginning with $B_1$ such that precisely one of the following conditions is satisfied for each $1\leq j\leq k$ and $\omega\left(R_{j,0}\setminus R_j'\right) + 1 < \Omega\left(B_j\right)$ for each $1\leq j\leq k-1$, where we also inductively define $R_j'$ and the \textit{adjusted width} $\Omega\left(B_j\right)$ of a block $B_j$ in $S_i$.

\begin{description}
\item[(a)] We have $j = 1$ and $B_1$ is a generic block. If $B_1 = \sigma_2\sigma_3\sigma_2^{c_{p_1}}$ is a $3$-block, then $a_{p_1-1}>c_{p_1} + 1$. We define $R_1' = \emptyset$ and we define the \textit{adjusted width} $\Omega\left(B_1\right)$ to be the width of $B_1$ (see Definition~\ref{defblockwidth}).
\item[(b)] We have $j = 1$ and $B_1$ is a singular block. We define $R_1'$ to be the singular subroad following $B_1$. In this case, one of the following conditions is satisfied and we define the \textit{adjusted width} $\Omega\left(B_1\right)$ in each case. 

\begin{description}
\item[(i)] The singular subroad $R_1'$ is followed by an isolated $\sigma_2$. The \textit{adjusted width} $\Omega\left(B_1\right)$ is one more than the width of $B_1$.
\item[(ii)] The singular subroad $R_1'$ is followed by a non-isolated $\sigma_2$. The \textit{adjusted width} $\Omega\left(B_1\right)$ is the width of $B_1$. 
\end{description}
\item[(c)] We have $j>1$ and $B_j = \sigma_1^{a_{p_j}}\sigma_3^{b_{p_j}}\sigma_2\sigma_1^2$ is a $2$-block. If $B_j$ is a $3$-switch $2$-block, then $b_{p_j}>2$. We define $R_j' = \emptyset$. In this case, one of the following conditions is satisfied and we define the \textit{adjusted width} $\Omega\left(B_j\right)$ in each case.

\begin{description}
\item[(i)] The block $B_j$ is a $1$-switch $2$-block and $a_{p_j}\geq b_{p_j}+2$. The \textit{adjusted width} $\Omega\left(B_j\right) = a_{p_j}-b_{p_j}-2$ (one less than the width of $B_j$). 
\item[(ii)] The block $B_j$ is a $3$-switch $2$-block and $a_{p_j}\geq b_{p_j}$. The \textit{adjusted width} $\Omega\left(B_j\right) = a_{p_j}-b_{p_j}$ (one more than the width of $B_j$).
\end{description}
\item[(d)] We have $j>1$ and $B_j = \sigma_1^{a_{p_j}}\sigma_3^{b_{p_j}}\sigma_2\sigma_1^2$ is a $3$-switch $2$-block such that $b_{p_j} = 2$. We define $R_j'$ to be the singular subroad following $B_j$. Let $M_j$ be one more than the number of $\sigma_2$s in $R_j'$. Let $\Xi_{j-1} = \Omega\left(B_{j-1}\right) - \omega\left(R_{j-1}\setminus R_{j-1}'\right)$. If $j < k$, then one of the following conditions is satisfied and we define the \textit{adjusted width} $\Omega\left(B_j\right)$ in each case.

\begin{description} 
\item[(i)]  The singular subroad $R_j'$ is followed by an isolated $\sigma_2$, $M_j$ is odd, and $\Xi_{j-1} - 5 > \omega_{\text{sing}}\left(R_j'\right)$. The \textit{adjusted width} $\Omega\left(B_j\right) = a_{p_j}+1$. 
\item[(ii)]  The singular subroad $R_j'$ is followed by an isolated $\sigma_2$, $M_j$ is odd, and $\Xi_{j-1} - 5 < \omega_{\text{sing}}\left(R_j'\right)$. The \textit{adjusted width} $\Omega\left(B_j\right) = a_{p_j} + \Xi_{j-1} - 4 - \omega_{\text{sing}}\left(R_j'\right)$.
\item[(iii)] The singular subroad $R_j'$ is followed by a non-isolated $\sigma_2$, $M_j$ is odd, and $\Xi_{j-1} - 4 > \omega_{\text{sing}}\left(R_j'\right)$. The \textit{adjusted width} $\Omega\left(B_j\right) = a_{p_j}$. 
\item[(iv)] The singular subroad $R_j'$ is followed by a non-isolated $\sigma_2$, $M_j$ is odd, and $\Xi_{j-1}-4<\omega_{\text{sing}}\left(R_j'\right)$. The \textit{adjusted width} $\Omega\left(B_j\right) = a_{p_j} + \Xi_{j-1} -4 - \omega_{\text{sing}}\left(R_j'\right)$. 
\end{description}
\item[(e)] We have $j>1$ and $B_{j} = \sigma_2^{c_{p_j-1}}\sigma_3\sigma_2^{c_{p_j}}$ is a $3$-block. We define $R_j' = \emptyset$. Let $\Xi_{j-1} = \Omega\left(B_{j-1}\right) - \omega\left(R_{j-1,0}\setminus R_{j-1}'\right)$. In this case, one of the following conditions is satisfied and we define the \textit{adjusted width} $\Omega\left(B_j\right)$ in each case. 

\begin{description}
\item[(i)] We have $c_{p_j-1}>2$ and $c_{p_j}=2$. The \textit{adjusted width} $\Omega\left(B_j\right) = c_{p_j-1}-2$ (one less than the width of $B_j$). 
\item[(ii)] We have $j= k$, $c_{p_k - 1} = 1$, and $a_{p_k-1}>c_{p_k} + 2$. The \textit{adjusted width} $\Omega\left(B_k\right) = a_{p_k-1}-c_{p_k} - 1$ (one less than the width of $B_k$). 
\item[(iii)] We have $j = k$, $c_{p_k-1} = 2$, $c_{p_k}>2$, and $\Xi_{k-1}>c_{p_k}+1$. The \textit{adjusted width} $\Omega\left(B_k\right) = \Xi_{k-1} - c_{p_k}$. 
\item[(iv)] We have $j = k$ and $c_{p_k-1} = 2 = c_{p_k}$. The \textit{adjusted width} $\Omega\left(B_k\right) = \Xi_{k-1}-2$. 
\end{description}
\end{description}
An \textit{abnormal string} is a bad string of the form $S_i$ for some $1\leq i\leq l$.
\end{definition}

The abnormal strings in $\sigma$ correspond to potential global cancellation of weights of $\sigma$-paths (see also Example~\ref{exglobalcancellation}). Let us consider the context of Definition~\ref{defabnormalstring}. If $B_j$ is a block in an abnormal string $S$ and if $P_j$ is the minimal $s_j$-subproduct of $\sigma$ containing $B_j$ and $R_j'$, then we will prove in Subsection~\ref{subsecproofstrongerversion} that $\Omega\left(B_j\right)$ (the adjusted width of $B_j$) measures the discrepancy in $q$-resistance between a maximal $q$-resistance bad $P_j$-path and a maximal $q$-resistance good $P_j$-path. In this case, $P_j$ is not strongly $r$-regular for each $1\leq j\leq k$ and this discrepancy is positive. Furthermore, the conditions that $R_{j,0}$ is an $i_j$-switch bad subroad and $\omega\left(R_{j,0}\setminus R_j'\right) + 1 < \Omega\left(B_j\right)$ for $1\leq j\leq k-1$ imply that catchup does not occur along the road $R_j$ for $1\leq j\leq k-1$. 

In analogy with Example~\ref{exglobalcancellation}, catchup and global cancellation can occur at the end of an abnormal string if there is a precise relationship between $\Omega\left(B_k\right)$ (the adjusted width of $B_k$) and $\omega\left(R_k\setminus R_k'\right)$ (the weight of $R_k\setminus R_k'$). The next step is to precisely define these abnormal strings, which we will refer to as terminal abnormal strings. 

Firstly, we need to define the \textit{sign of an abnormal string} (global cancellation can only occur at the end of an abnormal string if the sign is negative). 

\begin{definition}
\label{defsignabnormalstring}
Let $S = \prod_{j=1}^{k} B_jR_j$ be an abnormal string and let us adopt the notation of Definition~\ref{defabnormalstring}. Let $P$ be the minimal $s$-subproduct of $\sigma$ such that $S\subseteq P$. Let $j'\leq k$ be maximal such that $B_{j'}$ is either a generic $2$-block of the form $\sigma_1^{a_{p_{j'}}}\sigma_3^{b_{p_{j'}}}\sigma_2\sigma_1^2$ with $b_{p_{j'}}>2$ if it is a $3$-switch $2$-block, or a generic $3$-block of the form $\sigma_2^{c_{p_{j'}-1}}\sigma_3\sigma_2^{c_{p_{j'}}}$ with $c_{p_{j'}-1}\neq 2$ (if there is no such $j'$, then we define $j' = 0$). The block $B_{j'}$ is \textit{sign-changing} if $i_{j'} = 3$, where $i_{j'}$ is as in Definition~\ref{defbadstring}. If $j'< j\leq k$, then a block $B_j$ is \textit{sign-changing} if it satisfies one of the following properties:

\begin{description}
\item[(i)] The block $B_j$ is a singular block (only possible at most once).
\item[(ii)] The block $B_j$ is a $3$-switch $2$-block of the form $\sigma_1^{a_{p_j}}\sigma_3^2\sigma_2\sigma_1^2$ such that $M_j$ is odd. If $R_j'$ is followed by an isolated $\sigma_2$, then $\Xi_{j-1} - 5 > \omega_{\text{sing}}\left(R_j'\right)$. If $R_j'$ is followed by a non-isolated $\sigma_2$, then $\Xi_{j-1}-4>\omega_{\text{sing}}\left(R_j'\right)$. 
\item[(iii)] We have $j = k$ and the block $B_j$ is a $3$-block of the form $\sigma_2^{2}\sigma_3\sigma_2^{c_{p_j}}$. 
\end{description}

Let $M'$ be the number of sign-changing blocks $B_j$ for $j'\leq j\leq k$. Let $M_1$ be the number of isolated $\sigma_2$s in the union $\bigcup_{j=j'}^{k} R_j\setminus R_j'$. Let $M_2$ be the number of isolated $\sigma_2$s in the union of $R_j'$ over all $j$ such that $B_j$ is a $3$-switch $2$-block of the form $\sigma_1^{a_{p_j}}\sigma_3^2\sigma_2\sigma_1^2$ that does not satisfy \textbf{(ii)}. Finally, let $M = M_1 + M_2$. 

If $P$ is a $\mu$-switch subproduct, then we write that $S$ is \textit{positive} if $\frac{M}{2} + M'$ is even and $S$ is \textit{negative} if $\frac{M}{2} + M'$ is odd. If $P$ is a $\nu$-switch subproduct, then we write that $S$ is \textit{positive} if $\frac{M-3}{2} + M'$ is even and $S$ is \textit{negative} if $\frac{M-3}{2} + M'$ is odd.
\end{definition}

The following terminology for the final block in an abnormal string will be convenient.

\begin{definition}
\label{deffinalblockcancelling}
Let $S = \prod_{j=1}^{k} B_jR_j$ be an abnormal string. We define $B_k$ to be a \textit{long-cancelling block} if the following conditions are satisfied:
\begin{description}
\item[(i)] If $B_k = \sigma_1^{a_{p_k}}\sigma_3^{b_{p_k}}\sigma_2\sigma_1^2$ is a $3$-switch $2$-block with $b_{p_k} = 2$, then $B_k$ satisfies one of the conditions \textbf{(d)} \textbf{(i)} - \textbf{(iv)} in Definition~\ref{defabnormalstring} (and the adjusted width $\Omega\left(B_k\right)$ is accordingly defined).
\item[(ii)] If $B_k$ is a $3$-block, then $i_k = 3$, where $i_k$ is as in Definition~\ref{defbadstring}.
\end{description}
\end{definition}

In the context of Definition~\ref{deffinalblockcancelling}, if $B_k$ is a $2$-block but not a $3$-switch $2$-block $\sigma_1^{a_{p_k}}\sigma_3^2\sigma_2\sigma_1^2$, then $B_k$ is vacuously a long-cancelling block. We now define terminal abnormal strings which constitute the sole cause of global cancellation of weights of $\sigma$-paths. 

\begin{definition}
\label{defterminalabnormalstring}
Let $S = \prod_{j=1}^{k} B_jR_j$ be an abnormal string and let us adopt the notation of Definition~\ref{defabnormalstring}. Let $P$ be the minimal $s$-subproduct of $\sigma$ such that $S\subseteq P$ and let $P'$ be the minimal $s'$-subproduct of $\sigma$ such that $P\subsetneq P'$. We write that $S$ is a \textit{terminal abnormal string} if one of the following conditions is satisfied. 

\begin{description}
\item[(a)] Let us consider the case where $B_k$ is a long-cancelling block (see Definition~\ref{deffinalblockcancelling}), $S$ is negative (see Definition~\ref{defsignabnormalstring}), and $P$ is a $2$-subproduct of $\sigma$ such that $P'\setminus P = \sigma_1^{a_p}\sigma_3^{b_p}$ is a subroad. 

\begin{description}
\item[(i)] Let us assume that either $a_p = 0$ or $b_p = 0$. If $a_p = 0$, then we define $i=3$ and $i' = 1$. If $b_p = 0$, then we define $i = 1$, and $i' = 3$. If $P$ is an $i$-switch subproduct (see Definition~\ref{deftermbadstring}), then $\Omega\left(B_k\right) = \omega\left(R_k\setminus R_k'\right) + 1$. If $P$ is an $i'$-switch subproduct, then $\Omega\left(B_k\right) = \omega\left(R_k\setminus R_k'\right)$.
\item[(ii)] Let us assume that $a_p>1$ and $b_p>0$. If $P$ is a $1$-switch subproduct, then either $\Omega\left(B_k\right) = \omega\left(R_k\setminus R_k'\right)$ and $a_p\leq b_p - 1$, or $\Omega\left(B_k\right) = \omega\left(R_k\setminus R_k'\right) + 1$ and $b_p\leq a_p - 1$. If $P$ is a $3$-switch subproduct, then either $\Omega\left(B_k\right) = \omega\left(R_k\setminus R_k'\right)$ and $b_p\leq a_p - 1$, or $\Omega\left(B_k\right) = \omega\left(R_k\setminus R_k'\right) + 1$ and $a_p\leq b_p - 2$.
\item[(iii)] Let us assume that $a_p = 1$ and $b_p>0$. In this case, $\Omega\left(B_k\right)-\omega\left(R_k\setminus R_k'\right)\in \{0,1\}$. 
\end{description}
\item[(b)] Let us consider the case where $B_k$ is a long-cancelling block, $S$ is negative, and $P$ is an $s$-subproduct of $\sigma$ for $s\in \{1,3\}$ such that $P'\setminus P = \sigma_2^{c_p}$ is a subroad.
\begin{description}
\item[(i)] Let us assume that either $a_p = 0$ or $b_p = 0$. If $P$ is a $\mu$-switch subproduct (see Definition~\ref{deftermbadstring}), then $\Omega\left(B_k\right) = \omega\left(R_k\setminus R_k'\right) + 1$.
\item[(ii)] Let us assume that either $a_p = 0$ or $b_p = 0$. If $P$ is a $\nu$-switch subproduct (see Definition~\ref{deftermbadstring}), then $\Omega\left(B_k\right) = \omega\left(R_k\setminus R_k'\right)$.
\item[(iii)] Let us assume that $a_p,b_p>0$. In this case, $\Omega\left(B_k\right) = \omega\left(R_k\setminus R_k'\right) + 1$.
\end{description}
\item[(c)] Let us consider the case where $B_k$ is a long-cancelling block, $S$ is negative, and $P$ is immediately followed by a $2$-block $B = \sigma_1^{a_{p_j}}\sigma_3^{b_{p_j}}\sigma_2\sigma_1^2$. In this case, $\Omega\left(B_k\right) - \omega\left(R_k\setminus R_{k}'\right)\in \{0,1\}$. 
\item[(d)] Let us consider the case where $B_k$ is a long-cancelling block, $S$ is negative, and $P$ is immediately followed by a $3$-block $B = \sigma_2^{c_{p-1}}\sigma_3\sigma_2^{c_p}$. Let $\Xi_{k-1} = \Omega\left(B_{k-1}\right) - \omega\left(R_{k-1,0}\setminus R_{k-1}'\right)$ and $\Xi_0 = 0$. In this case, one of the following conditions is satisfied.

\begin{description}
\item[(i)] We have $c_{p-1}>2$ and $\Omega\left(B_k\right) - \omega\left(R_k\setminus R_k'\right)\in \{0,1\}$. 
\item[(ii)] We have $c_{p-1} = 1$ and $\Omega\left(B_k\right) - \omega\left(R_k\setminus R_k'\right) = 0$.
\item[(iii)] We have $c_{p-1} = 1$, $\Omega\left(B_k\right) - \omega\left(R_k\setminus R_k'\right)>0$, and $a_{p-1}\in \{c_p,c_p+1,c_p+2\}$. 
\item[(iv)] We have $c_{p-1} = 2$, $c_p>2$, and $\Xi_{k-1}\in \{c_p-1,c_p,c_p+1\}$. 
\item[(v)] We have $c_{p-1} = 2 = c_p$, and $\Omega\left(B_k\right) - \omega\left(R_k\setminus R_k'\right)\in \{0,1\}$. 
\end{description}
\item[(e)] Let us consider a generic $3$-block $B = \sigma_2^{c_{p-1}}\sigma_3\sigma_2^{c_p}$ such that $c_{p-1} = 1$ and $a_{p-1}\in \{c_p-1,c_p,c_p+1\}$. In this case, one of the following conditions is satisfied.
\begin{description}
\item[(i)] A $3$-block of this form exists between the end of $S$ and the beginning of the abnormal string following $S$. 
\item[(ii)] If $S$ is the first abnormal string, then a $3$-block of this form exists before $S$.
\end{description}
\item[(f)] Let us consider the case where $B_k = \sigma_1^{a_{p_k}}\sigma_3^{b_{p_k}}\sigma_2\sigma_1^{2}$ is a $3$-switch $2$-block with $b_{p_k} = 2$. Let $R_k'$ be the singular subroad following $B_k$ and let $M_k$ be one more than the number of $\sigma_2$s in $R_k'$. Let $\Xi_{k-1} = \Omega\left(B_{k-1}\right) - \omega\left(R_{k-1,0}\setminus R_{k-1}'\right)$. In this case, one of the following conditions is satisfied.
\begin{description}
\item[(i)] The singular subroad $R_k'$ is followed by an isolated $\sigma_2$, $M_k$ is odd, and $\Xi_{k-1} - 5 = \omega_{\text{sing}}\left(R_k'\right)$. Furthermore, either $S$ is positive and $M_k\equiv 1\pmod 4$, or $S$ is negative and $M_k\equiv 3\pmod 4$.
\item[(ii)] The singular subroad $R_k'$ is followed by a non-isolated $\sigma_2$, $M_k$ is odd, and $\Xi_{k-1}-4 = \omega_{\text{sing}}\left(R_k'\right)$. Furthermore, either $S$ is positive and $M_k\equiv 1\pmod 4$, or $S$ is negative and $M_k\equiv 3\pmod 4$.
\item[(iii)] The singular subroad $R_k'$ is followed by an isolated $\sigma_2$, $M_k$ is even, and $a_{p_k} + \Xi_{k-1} - 4 = \omega_{\text{sing}}\left(R_k'\right)$. Furthermore, either $S$ is positive and $M_k\equiv 2\pmod 4$, or $S$ is negative and $M_k\equiv 0\pmod 4$.
\item[(iv)] The singular subroad $R_k'$ is followed by a non-isolated $\sigma_2$, $M_k$ is even, and $a_{p_k} + \Xi_{k-1} - 3 = \omega_{\text{sing}}\left(R_k'\right)$. Furthermore, either $S$ is positive and $M_k\equiv 2\pmod 4$, or $S$ is negative and $M_k\equiv 0\pmod 4$.
\end{description}
\item[(g)] Let us consider the case where $B_k$ is a $3$-block and $i_k = 1$. Let $R_k'$ be the singular subroad following $B_k$ and let $M_k$ be one more than the number of $\sigma_2$s in $R_k'$. In this case, one of the following conditions is satisfied.

\begin{description}
\item[(i)] The singular subroad $R_k'$ is followed by an isolated $\sigma_2$ and $\Omega\left(B_k\right) - \omega_{\text{sing}}\left(R_k'\right) = 1$. Furthermore, either $S$ is positive and $M_k\equiv 2,3\pmod 4$, or $S$ is negative and $M_k\equiv 0,1\pmod 4$. 
\item[(ii)] The singular subroad $R_k'$ is followed by a non-isolated $\sigma_2$, $M_k$ is odd, and $\Omega\left(B_k\right) - \omega_{\text{sing}}\left(R_k'\right)= 0$. Furthermore, either $S$ is positive and $M_k\equiv 2,3\pmod 4$, or $S$ is negative and $M_k\equiv 0,1\pmod 4$. 
\item[(iii)] The singular subroad $R_k'$ is followed by a non-isolated $\sigma_2$, $M_k$ is even, and $\Omega\left(B_k\right) - \omega_{\text{sing}}\left(R_k'\right)\geq 0$. 
\item[(iv)] The singular subroad $R_k'$ is immediately followed by a $3$-block (in which case, $R_k'=\emptyset$ and $R_k$ is a single power of $\sigma_1$), or $R_k = \sigma_1^{a_p}\sigma_3^{b_p}$ with $a_p -2 = b_p>0$ (in which case $R_k'=\emptyset$) and $S$ is negative. 
\end{description}
\end{description}
\end{definition}

If $\sigma$ is a normal braid, then we observe that there are no terminal abnormal strings in $\sigma$, simply because there are no abnormal blocks in $\sigma$. We also observe that a terminal abnormal string is highly constrained. For example, the conditions in Definition~\ref{defterminalabnormalstring} for an abnormal string to be terminal require (in particular) the existence of an abnormal block $B$ followed by a bad subroad $R$ such that there is an exact linear relationship between the weight (or singular weight) of $R$ and the adjusted width of $B$. If we increase the adjusted width of $B$ (keeping the rest of $\sigma$ constant), then the existence of a long bad subroad $R$ to imply such an exact relationship is extremely unlikely. Furthermore, a minor perturbation of even one exponent in $\sigma$ would destroy a terminal abnormal string (compare Example~\ref{exglobalcancellation}). We now define the notion of a weakly normal braid.

\begin{definition}
\label{defweaklynormalbraid}
The positive braid $\sigma$ is a \textit{weakly normal braid} if there are no terminal abnormal strings in $\sigma$.
\end{definition}

We can determine if a positive braid is weakly normal by inspection of its minimal form. Indeed, the weight of a bad subroad $R$ (or the singular weight of a singular subroad $R'$) is a straightforward arithmetic expression in terms of a product expansion of $R$ (or $R'$), and the adjusted width of a block is a straightforward arithmetic expression in terms of the exponents in the block. The conditions in Definition~\ref{defterminalabnormalstring} can be readily checked to determine the existence of terminal abnormal strings for a specific positive braid $\sigma$.

\begin{theorem}
\label{strongkernelconstraint}
Let $t\neq 2, 3$ be a prime number. If $\sigma$ is a weakly normal braid indivisible by $\Delta$, then the leading coefficients of multiple entries in at least one row of $\beta_4\left(\sigma\right)$ are non-zero modulo $t$. 
\end{theorem}

We outline the strategy of the proof of Theorem~\ref{strongkernelconstraint} and we continue the discussion following Definition~\ref{defabnormalstring} on global cancellation of weights of $\sigma$-paths. If global cancellation does not occur (i.e., each abnormal string is not a terminal abnormal string), then a reset occurs at the end of each abnormal string and the minimal $s'$-subproduct of $\sigma$ properly containing each abnormal string is strongly $r$-regular. The process then iterates at the next abnormal block after the abnormal string. If there are no terminal abnormal strings in $\sigma$ (i.e., $\sigma$ is a weakly normal braid), then we can ultimately establish that the leading coefficients of multiple entries in the $r$th row of $\beta_4\left(\sigma\right)$ are non-zero modulo each prime $t\neq 2,3$. 

We now state and prove the stronger version of the main result contingent on Theorem~\ref{strongkernelconstraint}.

\begin{theorem}
\label{mainstronger}
Let $g\in B_4$ be a non-identity braid and write $g=\Delta^{k}\sigma$ for the Garside normal form of $g$. Let $t\neq 2,3$ be a prime number and let $a,b$ be coprime integers such that $b\neq 0$ has a prime factor distinct from $2$ and $3$.

If either $k\geq 0$ or $\sigma$ is a weakly normal braid, then $g\not\in \text{ker}\left(\left(\beta_4\right)_{t}\right)$, and in particular, $g\not\in \text{ker}\left(\beta_4\right)$. If $\sigma$ is a weakly normal braid, then $g\not\in \text{ker}\left(\left.\left(\beta_4\right)\right|_{q=\frac{a}{b}}\right)$.
\end{theorem}
\begin{proof}
The statement is a consequence of Lemma~\ref{reduction} and Theorem~\ref{strongkernelconstraint}. The proof is similar to that of Theorem~\ref{main} immediately following Theorem~\ref{invariantdetect}. 
\end{proof}

The proof we present for Theorem~\ref{strongkernelconstraint} does not work if $t\in \{2,3\}$. We note that interestingly the Burau representation $\beta_4$ is known not to be faithful modulo two and three~\cite{longpatonsmallprime}. 

\subsection{Path extensions II}
\label{subsecpathextensionsII}

In this subsection, we extend the theory of path extensions that we developed in Subsection~\ref{subsecpathextensions} to the case where $P$ is not necessarily strongly $r$-regular in the statements of the results. 

We establish the following addendum to Proposition~\ref{p21/3ext}. 

\begin{proposition}
\label{p21/3extadd}
Let $P = \prod_{i=1}^{p-1} \sigma_1^{a_i}\sigma_3^{b_i}\sigma_2^{c_i}$ be a $2$-subproduct of $\sigma$ and let $P' = \left(\prod_{i=1}^{p-1} \sigma_1^{a_i}\sigma_3^{b_i}\sigma_2^{c_i}\right)\sigma_1^{a_p}\sigma_3^{b_p}$. We assume that either $a_p = 0$ or $b_p = 0$. We adopt the notation of Proposition~\ref{ppre21/3ext}.

Let us assume that $\phi_p>1$. If $d_{\lambda}^{P} + 1 = d_{\mu}^{P}$, then let us also assume that $\lambda_{0}^{P} + \mu_{0}^{P}\neq 0$. 
\begin{description}
\item[(i)] If $d_{\lambda}^{P} + 1 < d_{\mu}^{P}$ and $\phi_p>2$, then $d_{\lambda}^{P'} = d_{\mu}^{P} + \phi_p - 1 - \psi >d_{\mu}^{P'}$ and $\lambda_{0}^{P'} = \left(-1\right)^{1-\psi}\mu_{0}^{P}$. If $d_{\lambda}^{P} + 1 = d_{\mu}^{P}$ and $\phi_p > 2$, then $d_{\lambda}^{P'} = d_{\lambda}^{P} + \phi_p - \psi >d_{\mu}^{P'}$ and $\lambda_{0}^{P'} = \left(-1\right)^{1-\psi}\left(\lambda_{0}^{P} + \mu_{0}^{P}\right)$. 
\item[(ii)] If $d_{\lambda}^{P'}\leq d_{\mu}^{P'}$, then $\phi_p=2$. Furthermore, in this case, $d_{\lambda}^{P'} = d_{\lambda}^{P} + 2 - \psi\leq d_{\mu}^{P} + 1 - \psi = d_{\mu}^{P'}$, $\lambda_{0}^{P'} = \left(-1\right)^{1-\psi}\lambda_{0}^{P}$, and $\mu_{0}^{P'} = \left(-1\right)^{1 - \psi}\mu_{0}^{P}$. 
\end{description}
\end{proposition}
\begin{proof}
We establish each statement separately.
\begin{description}
\item[(i)] If $d_{\lambda}^{P} + 1 < d_{\mu}^{P}$, then Proposition~\ref{ppre21/3ext} \textbf{(i)} and \textbf{(iii)} imply that a maximal $q$-resistance $\left(r,s'\right)$-type $P'$-path is of the form $Y^{2}$, where $Y$ is a maximal $q$-resistance bad $\left(r,2\right)$-type $P$-path. If $\phi_p>2$, then $Y^{2}$ is a good $\left(r,s'\right)$-type $P'$-path.

If $d_{\lambda}^{P} + 1 = d_{\mu}^{P}$, then Proposition~\ref{ppre21/3ext} \textbf{(i)} and \textbf{(iii)} imply that a maximal $q$-resistance $\left(r,s'\right)$-type $P'$-path is either of the form $X^{1}$ or of the form $Y^{2}$, where $X$ is a maximal $q$-resistance good $\left(r,2\right)$-type $P$-path and $Y$ is a maximal $q$-resistance bad $\left(r,2\right)$-type $P$-path. The hypothesis that $\lambda_{0}^{P} + \mu_{0}^{P} \neq 0$ implies the statement.
\item[(ii)] If $d_{\lambda}^{P'}\leq d_{\mu}^{P'}$, then Proposition~\ref{p21/3ext} \textbf{(i)} and the previous part \textbf{(i)} imply that $\phi_p = 2$ and $d_{\lambda}^{P}<d_{\mu}^{P}$. In this case, Proposition~\ref{ppre21/3ext} \textbf{(i)}, \textbf{(iii)} and \textbf{(iv)} imply that a maximal $q$-resistance good $\left(r,s'\right)$-type $P'$-path is of the form $X^{1}$, and a maximal $q$-resistance bad $\left(r,s'\right)$-type $P'$-path is of the form $Y^{2}$, where $X$ is a maximal $q$-resistance good $\left(r,2\right)$-type $P$-path and $Y$ is a maximal $q$-resistance bad $\left(r,2\right)$-type $P$-path. 
\end{description}
\end{proof}

We establish the following addendum to Proposition~\ref{ppre21/3bothext}.

\begin{proposition}
\label{ppre21/3bothextadd}
Let $P = \prod_{i=1}^{p-1} \sigma_1^{a_i}\sigma_3^{b_i}\sigma_2^{c_i}$ be a $2$-subproduct of $\sigma$ and let $P' = \left(\prod_{i=1}^{p-1} \sigma_1^{a_i}\sigma_3^{b_i}\sigma_2^{c_i}\right)\sigma_1^{a_p}\sigma_3^{b_p}$. Let us assume that $a_p,b_p>0$.
\begin{description}
\item[(i)] If $Y$ is a $1$-switch $\left(r,2\right)$-type $P$-path, then $Y^{2,0}$ is the maximal $q$-resistance extension of $Y$ to an $\left(r,1\right)$-type $P'$-path and $Y^{0,1}$ is the maximal $q$-resistance extension of $Y$ to an $\left(r,3\right)$-type $P'$-path. (If $a_p = 1$, then a $1$-switch $\left(r,2\right)$-type $P$-path does not have an admissible extension to an $\left(r,1\right)$-type $P'$-path.)
\item[(ii)] If $V$ is a $3$-switch $\left(r,2\right)$-type $P$-path, then $V^{1,0}$ is the maximal $q$-resistance extension of $V$ to an $\left(r,1\right)$-type $P'$-path and $V^{0,2}$ is the maximal $q$-resistance extension of $V$ to an $\left(r,3\right)$-type $P'$-path. (If $b_p = 1$, then a $3$-switch $\left(r,2\right)$-type $P$-path does not have an admissible extension to an $\left(r,3\right)$-type $P'$-path.)
\end{description}
\end{proposition}
\begin{proof}
The statements are a consequence of Lemma~\ref{lemmagoodext} \textbf{(iii)} and Proposition~\ref{ppre21/3bothext} \textbf{(ii)}. 
\end{proof}

We establish the following addendum to Proposition~\ref{p21/3bothext}.

\begin{proposition}
\label{p21/3bothextadd}
Let $P = \prod_{i=1}^{p-1} \sigma_1^{a_i}\sigma_3^{b_i}\sigma_2^{c_i}$ be a $2$-subproduct of $\sigma$ and let $P' = \left(\prod_{i=1}^{p-1} \sigma_1^{a_i}\sigma_3^{b_i}\sigma_2^{c_i}\right)\sigma_1^{a_p}\sigma_3^{b_p}$. We assume that $a_p,b_p>0$, $c_{p-1}\geq 2$, and $d_{\lambda}^{P}>\min\{d_{\mu,1}^{P},d_{\mu,3}^{P}\}$. 
\begin{description}
\item[(a)] Let us consider the case $a_p = 1$. Let us also assume that if $d_{\lambda}^{P} = d_{\mu,1}^{P}$, then $\lambda_{0}^{P} + \mu_{1,0}^{P}\neq 0$, and if $d_{\lambda}^{P} + 1 = d_{\mu,3}^{P}$, then $\lambda_{0}^{P} + \mu_{3,0}^{P}\neq 0$.

Firstly, if $b_p > 1$, then $d_{\lambda,3}^{P'}>d_{\mu,3}^{P'}$ and $d_{\lambda,3}^{P'}+1\geq d_{\mu,1}^{P'}$. If $b_p = 1$, then either $d_{\lambda,3}^{P'}>d_{\mu,3}^{P'}$ or $d_{\lambda,3}^{P'}+1\geq d_{\mu,1}^{P'}$. Furthermore, the following statements are true:
\begin{description}
\item[(i)] If $d_{\lambda}^{P}<d_{\mu,1}^{P}$ and $b_p>1$, then $\lambda_{3,0}^{P'} = \mu_{1,0}^{P}$. If $d_{\lambda}^{P} = d_{\mu,1}^{P}$ and $b_p>1$, then $\lambda_{3,0}^{P'} = \lambda_{0}^{P} + \mu_{1,0}^{P}$.
\item[(ii)] If $d_{\lambda}^{P} + 1 < d_{\mu,3}^{P}$ and $b_p>1$, then $\lambda_{3,0}^{P'} = \mu_{3,0}^{P}$. If $d_{\lambda}^{P} + 1 = d_{\mu,3}^{P}$ and $b_p > 1$, then $\lambda_{3,0}^{P'} = \lambda_{0}^{P} + \mu_{3,0}^{P}$. 
\item[(iii)] If $d_{\lambda,3}^{P'}\leq d_{\mu,3}^{P'}$, then $b_p = 1$ and $d_{\lambda}^{P}\leq d_{\mu,1}^{P}$. In this case, $d_{\mu,1}^{P'} = d_{\lambda}^{P} + 1$, $d_{\lambda,3}^{P'} = d_{\lambda}^{P}$, $d_{\mu,3}^{P'} = d_{\mu,1}^{P}$, $\mu_{1,0}^{P'} = -\lambda_{0}^{P}$, $\lambda_{3,0}^{P'} = \lambda_{0}^{P}$, and $\mu_{3,0}^{P'} = \mu_{1,0}^{P}$. 
\item[(iv)] If $d_{\lambda,3}^{P'}+2\leq d_{\mu,1}^{P'}$, then $b_p = 1$ and $d_{\lambda}^{P}< d_{\mu,3}^{P}$. In this case, $d_{\mu,1}^{P'} = d_{\mu,3}^{P} + 1$, $d_{\lambda,3}^{P'} = d_{\lambda}^{P}$, $d_{\mu,3}^{P'} = d_{\mu,1}^{P}$, $\mu_{1,0}^{P'} = -\mu_{3,0}^{P}$, $\lambda_{3,0}^{P'} = \lambda_{0}^{P}$, and $\mu_{3,0}^{P'} = \mu_{1,0}^{P}$. 
\end{description}
\item[(b)] Let us consider the case $a_p>1$. Let us also assume that if $d_{\lambda}^{P} = d_{\mu,3}^{P}$, then $\lambda_{0}^{P} + \mu_{3,0}^{P}\neq 0$, and if $d_{\lambda}^{P} +1 = d_{\mu,1}^{P}$, then $\lambda_{0}^{P} + \mu_{1,0}^{P}\neq 0$. 

Firstly, we have that either $d_{\lambda,1}^{P'}> d_{\mu,1}^{P'}$ or $d_{\lambda,3}^{P'}+1\geq d_{\mu,1}^{P'}$. Furthermore, the following statements are true:
\begin{description}
\item[(i)] Let us assume that $d_{\lambda}^{P} +1 < d_{\mu,3}^{P}$. In this case, $d_{\lambda,1}^{P'}>d_{\mu,1}^{P'}$ and $\lambda_{1,0}^{P'} = -\mu_{3,0}^{P}$. If $b_p>1$, then $\lambda_{3,0}^{P'} = \mu_{3,0}^{P}$. If $b_p = 1$, then $d_{\lambda,1}^{P} > d_{\lambda,3}^{P} + 1$ and $\lambda_{3,0}^{P'} = \lambda_{0}^{P}$.
\item[(ii)] Let us assume that $d_{\lambda}^{P} + 1 = d_{\mu,3}^{P}$ and $\lambda_{0}^{P} + \mu_{3,0}^{P}\neq 0$. In this case, $d_{\lambda,1}^{P'}>d_{\mu,1}^{P'}$ and $\lambda_{1,0}^{P'} = -\mu_{3,0}^{P}$. If $b_p>1$, then $\lambda_{3,0}^{P'} = \lambda_{0}^{P} + \mu_{3,0}^{P}$. If $b_p = 1$, then $\lambda_{3,0}^{P'} = \lambda_{0}^{P}$.
\item[(iii)] Let us assume that $d_{\lambda}^{P} = d_{\mu,3}^{P}$. In this case, $d_{\lambda,1}^{P'}>d_{\mu,1}^{P'}$, $\lambda_{1,0}^{P'} = -\left(\lambda_{0}^{P} + \mu_{3,0}^{P}\right)$, and $\lambda_{3,0}^{P'} = \lambda_{0}^{P}$.
\item[(iv)]  Let us assume that $d_{\lambda}^{P} + 1 < d_{\mu,1}^{P}$. In this case, $\lambda_{3,0}^{P'} = \mu_{1,0}^{P}$. If $a_p > 2$, then $d_{\lambda,1}^{P'}>d_{\mu,1}^{P'}$ and $\lambda_{1,0}^{P'} = -\mu_{1,0}^{P}$. If $a_p = 2$, then $d_{\lambda,1}^{P'} < d_{\mu,1}^{P'}$, $d_{\lambda,3}^{P'} + 1\geq d_{\mu,1}^{P'}$, and $\lambda_{1,0}^{P'} = -\lambda_{0}^{P}$. 
\item[(v)] Let us assume that $d_{\lambda}^{P} + 1 = d_{\mu,1}^{P}$. In this case, $\lambda_{3,0}^{P'} = \mu_{1,0}^{P}$. If $a_p > 2$, then $d_{\lambda,1}^{P'}>d_{\mu,1}^{P'}$ and $\lambda_{1,0}^{P'} = -\left(\lambda_{0}^{P} + \mu_{1,0}^{P}\right)$. If $a_p = 2$, then $d_{\lambda,3}^{P'} + 1\geq d_{\mu,1}^{P'}$ and $\lambda_{1,0}^{P'} = -\lambda_{0}^{P}$. 
\item[(vi)] Let us assume that $d_{\lambda}^{P} = d_{\mu,1}^{P}$ and $\lambda_{0}^{P} + \mu_{1,0}^{P}\neq 0$. In this case, $d_{\lambda,1}^{P'}>d_{\mu,1}^{P'}$, $\lambda_{1,0}^{P'} = -\lambda_{0}^{P}$, and $\lambda_{3,0}^{P'} = \lambda_{0}^{P} + \mu_{1,0}^{P}$. 
\end{description}
\end{description}
\end{proposition}
\begin{proof}
We establish each statement separately.
\begin{description}
\item[(a)] If $b_p>1$, then we will show in \textbf{(i)} - \textbf{(ii)} that $d_{\lambda,3}^{P'} > d_{\mu,3}^{P'}$ and $d_{\lambda,3}^{P'} + 1\geq d_{\mu,1}^{P'}$. In \textbf{(i)}, we will show this if $d_{\lambda}^{P}\leq d_{\mu,1}^{P}$, and in \textbf{(ii)}, we will show this if $d_{\lambda}^{P}\leq d_{\mu,3}^{P}$. If $d_{\lambda}^{P}>\max\{d_{\mu,1}^{P},d_{\mu,3}^{P}\}$, then Proposition~\ref{p21/3bothext} \textbf{(a)} \textbf{(i)} implies that $d_{\lambda,3}^{P'}>d_{\mu,3}^{P'}$ and $d_{\lambda,3}^{P'}+1\geq d_{\mu,1}^{P'}$. If $b_p = 1$, then the proof that either $d_{\lambda,3}^{P'}>d_{\mu,3}^{P'}$ or $d_{\lambda,3}^{P'} + 1\geq d_{\mu,1}^{P'}$ is similar and based on \textbf{(iii)} - \textbf{(iv)}.

\begin{description}
\item[(i)] We assume that $d_{\lambda}^{P}\leq d_{\mu,1}^{P}$ and $b_p>1$. The hypothesis $d_{\lambda}^{P}>\min\{d_{\mu,1}^{P},d_{\mu,3}^{P}\}$ implies that $d_{\mu,3}^{P}<d_{\lambda}^{P}\leq d_{\mu,1}^{P}$.

If $d_{\lambda}^{P}<d_{\mu,1}^{P}$, then Proposition~\ref{ppre21/3bothext} \textbf{(i)} and Proposition~\ref{ppre21/3bothextadd} \textbf{(i)} imply that a maximal $q$-resistance $\left(r,3\right)$-type $P'$-path is of the form $Y^{0,1}$, where $Y$ is a maximal $q$-resistance $1$-switch $\left(r,2\right)$-type $P$-path. The $\left(r,3\right)$-type $P'$-path $Y^{0,1}$ is good since $b_p>1$. Proposition~\ref{ppre21/3bothext} \textbf{(i)} implies that a maximal $q$-resistance bad $\left(r,1\right)$-type $P'$-path is of the form $X^{1,0}$, where $X$ is a maximal $q$-resistance good $\left(r,2\right)$-type $P$-path. Proposition~\ref{ppre21/3bothext} \textbf{(ii)} implies that $d_{\lambda,3}^{P'}>d_{\mu,3}^{P'}$, $d_{\lambda,3}^{P'} + 1\geq d_{\mu,1}^{P'}$, and $\lambda_{3,0}^{P'} = \mu_{1,0}^{P}$. 

If $d_{\lambda}^{P} = d_{\mu,1}^{P}$, then Proposition~\ref{ppre21/3bothext} \textbf{(i)} and \textbf{(iv)} and Proposition~\ref{ppre21/3bothextadd} \textbf{(i)} imply that a maximal $q$-resistance $\left(r,3\right)$-type $P'$-path is either of the form $X^{0,1}$ or $Y^{0,1}$, where $X$ is a maximal $q$-resistance good $\left(r,2\right)$-type $P$-path and $Y$ is a maximal $q$-resistance $1$-switch $\left(r,2\right)$-type $P$-path. Proposition~\ref{ppre21/3bothext} \textbf{(i)} and \textbf{(iv)} and Proposition~\ref{ppre21/3bothextadd} \textbf{(i)} also imply that a maximal $q$-resistance bad $\left(r,1\right)$-type $P'$-path is of the form $X^{1,0}$, where $X$ is a maximal $q$-resistance good $\left(r,2\right)$-type $P$-path. Proposition~\ref{ppre21/3bothext} \textbf{(ii)} and the hypothesis that $\lambda_{0}^{P} + \mu_{1,0}^{P}\neq 0$ imply that $d_{\lambda,3}^{P'}>d_{\mu,3}^{P'}$, $d_{\lambda,3}^{P'}+1\geq d_{\mu,1}^{P'}$, and $\lambda_{3,0}^{P'} = \lambda_{0}^{P} + \mu_{1,0}^{P}$. 
\item[(ii)] We assume that $d_{\lambda}^{P} < d_{\mu,3}^{P}$ and $b_p>1$. The hypothesis $d_{\lambda}^{P}>\min\{d_{\mu,1}^{P},d_{\mu,3}^{P}\}$ implies that $d_{\mu,1}^{P}<d_{\lambda}^{P}< d_{\mu,3}^{P}$. 

Proposition~\ref{ppre21/3bothext} \textbf{(i)} and Proposition~\ref{ppre21/3bothextadd} \textbf{(ii)} imply that a maximal $q$-resistance bad $\left(r,1\right)$-type $P'$-path is of the form $V^{1,0}$, where $V$ is a maximal $q$-resistance $3$-switch $\left(r,2\right)$-type $P$-path. 

If $d_{\lambda}^{P} + 1 < d_{\mu,3}^{P}$, then Proposition~\ref{ppre21/3bothext} \textbf{(i)}, \textbf{(ii)} and \textbf{(iv)} and Proposition~\ref{ppre21/3bothextadd} \textbf{(ii)} imply that a maximal $q$-resistance $\left(r,3\right)$-type $P'$-path is of the form $V^{0,2}$, where $V$ is a maximal $q$-resistance $3$-switch $\left(r,2\right)$-type $P$-path. The $\left(r,3\right)$-type $P'$-path $V^{0,2}$ is good since $V$ is not a $1$-switch $\left(r,2\right)$-type $P$-path. Proposition~\ref{ppre21/3bothext} \textbf{(ii)} implies that $d_{\lambda,3}^{P'}>d_{\mu,3}^{P'}$, $d_{\lambda,3}^{P'} + 1\geq d_{\mu,1}^{P'}$, and $\lambda_{3,0}^{P'} = \mu_{3,0}^{P}$. 

If $d_{\lambda}^{P} +1 = d_{\mu,3}^{P}$, then Proposition~\ref{ppre21/3bothext} \textbf{(i)}, \textbf{(ii)} and \textbf{(iv)} and Proposition~\ref{ppre21/3bothextadd} \textbf{(ii)} imply that a maximal $q$-resistance $\left(r,3\right)$-type $P'$-path is either of the form $X^{0,1}$ or $V^{0,2}$, where $X$ is a maximal $q$-resistance good $\left(r,2\right)$-type $P$-path and $V$ is a maximal $q$-resistance $3$-switch $\left(r,2\right)$-type $P$-path. Proposition~\ref{ppre21/3bothext} \textbf{(ii)} and the hypothesis that $\lambda_{0}^{P} + \mu_{3,0}^{P}\neq 0$ imply that $d_{\lambda,3}^{P'}>d_{\mu,3}^{P'}$, $d_{\lambda,3}^{P'} + 1\geq d_{\mu,1}^{P'}$, and $\lambda_{3,0}^{P'} = \lambda_{0}^{P} + \mu_{3,0}^{P}$. 

Finally, in order to verify that $d_{\lambda,3}^{P'}>d_{\mu,3}^{P'}$ and $d_{\lambda,3}^{P'} + 1\geq d_{\mu,1}^{P'}$ in all cases, we consider the case $d_{\lambda}^{P} = d_{\mu,3}^{P}$. In this case, Proposition~\ref{ppre21/3bothext} \textbf{(i)}, \textbf{(ii)} and \textbf{(iv)} and Proposition~\ref{ppre21/3bothextadd} \textbf{(ii)} imply that a maximal $q$-resistance $\left(r,3\right)$-type $P'$-path is of the form $X^{0,1}$, where $X$ is a maximal $q$-resistance good $\left(r,2\right)$-type $P$-path. Proposition~\ref{ppre21/3bothext} \textbf{(ii)} implies that $d_{\lambda,3}^{P'}>d_{\mu,3}^{P'}$ and $d_{\lambda,3}^{P'} + 1\geq d_{\mu,1}^{P'}$ in all cases.
\item[(iii)] We have established in all cases that if $b_p>1$, then $d_{\lambda,3}^{P'}>d_{\mu,3}^{P'}$ thus far in the proof. The contrapositive is that if $d_{\lambda,3}^{P'}\leq d_{\mu,3}^{P'}$, then $b_p = 1$. 

If $b_p = 1$, then Proposition~\ref{ppre21/3bothext} \textbf{(i)} and \textbf{(iv)} imply that a maximal $q$-resistance good $\left(r,3\right)$-type $P'$-path is of the form $X^{0,1}$, where $X$ is a maximal $q$-resistance good $\left(r,2\right)$-type $P$-path. If $b_p = 1$, then Proposition~\ref{ppre21/3bothext} \textbf{(i)} and Proposition~\ref{ppre21/3bothextadd} \textbf{(i)} imply that a maximal $q$-resistance bad $\left(r,3\right)$-type $P'$-path is of the form $Y^{0,1}$, where $Y$ is a maximal $q$-resistance $1$-switch $\left(r,2\right)$-type $P$-path. If $d_{\lambda,3}^{P'}\leq d_{\mu,3}^{P'}$, then we deduce that $d_{\lambda}^{P}\leq d_{\mu,1}^{P}$. 

In this case, the hypothesis that $d_{\lambda}^{P}>\min\{d_{\mu,1}^{P},d_{\mu,3}^{P}\}$ implies that $d_{\mu,3}^{P}<d_{\lambda}^{P}\leq d_{\mu,1}^{P}$. In particular, Proposition~\ref{ppre21/3bothext} \textbf{(i)} and \textbf{(iv)} and Proposition~\ref{ppre21/3bothextadd} \textbf{(i)} imply that a maximal $q$-resistance bad $\left(r,1\right)$-type $P'$-path is of the form $X^{1,0}$, where $X$ is a maximal $q$-resistance good $\left(r,2\right)$-type $P$-path. Proposition~\ref{ppre21/3bothext} \textbf{(ii)} implies the remainder of the statement.
\item[(iv)] The proof is similar to the proof of \textbf{(iii)}.
\end{description}
\item[(b)] We will establish that either $d_{\lambda,1}^{P'}>d_{\mu,1}^{P'}$ or $d_{\lambda,3}^{P'} + 1\geq d_{\mu,1}^{P'}$. Proposition~\ref{ppre21/3bothext} \textbf{(i)} and \textbf{(iv)} and Proposition~\ref{ppre21/3bothextadd} imply that a maximal $q$-resistance $\left(r,1\right)$-type $P$-path is either of the form $X^{1,0}$, $Y^{2,0}$, or $V^{1,0}$, where $X$ is a maximal $q$-resistance good $\left(r,2\right)$-type $P$-path, $Y$ is a maximal $q$-resistance $1$-switch $\left(r,2\right)$-type $P$-path, and $V$ is a maximal $q$-resistance $3$-switch $\left(r,2\right)$-type $P$-path. Proposition~\ref{ppre21/3bothext} \textbf{(ii)} implies that the $q$-resistance of $X^{1,0}$ is equal to the $q$-resistance of $Y^{2,0}$ if and only if $d_{\lambda}^{P} + 1 = d_{\mu,1}^{P}$, and the $q$-resistance of $X^{1,0}$ is equal to the $q$-resistance of $V^{1,0}$ if and only if $d_{\lambda}^{P} = d_{\mu,3}^{P}$. The $\left(r,1\right)$-type $P'$-paths $X^{1,0}$ and $V^{1,0}$ are good since $a_p>1$, and the $\left(r,1\right)$-type $P'$-path $Y^{2,0}$ is good if and only if $a_p>2$. 

If $d_{\lambda,1}^{P'}\leq d_{\mu,1}^{P'}$, then the hypothesis implies that $a_p = 2$ and $d_{\mu,1}^{P}>\max\{d_{\lambda}^{P},d_{\mu,3}^{P}\}$. Proposition~\ref{ppre21/3bothext} \textbf{(i)} and \textbf{(ii)} and Proposition~\ref{ppre21/3bothextadd} \textbf{(i)} imply that $d_{\lambda,3}^{P'}+1\geq d_{\mu,1}^{P'}$. 

The proofs of statements \textbf{(i)} - \textbf{(vi)} are similar to the proofs of statements \textbf{(a)} \textbf{(i)} - \textbf{(ii)} (based on Proposition~\ref{ppre21/3bothext} and Proposition~\ref{ppre21/3bothextadd}).
\end{description}
\end{proof}

We establish the following addendum to Proposition~\ref{p1/32ext}.

\begin{proposition}
\label{p1/32extadd}
Let $P = \left(\prod_{i=1}^{p-1} \sigma_1^{a_i}\sigma_3^{b_i}\sigma_2^{c_i}\right)\sigma_1^{a_p}\sigma_3^{b_p}$ be an $s$-subproduct of $\sigma$ for $s\in \{1,3\}$ and let $P' = \prod_{i=1}^{p} \sigma_1^{a_i}\sigma_3^{b_i}\sigma_2^{c_i}$. We assume that either $a_p = 0$ or $b_p = 0$. If $a_p = 0$, then we define $\psi = 1$, and if $b_p = 0$, then we define $\psi = 0$. 

\begin{description}
\item[(a)] Let us assume that $d_{\nu}^{P} + \left(-1\right)^{\psi} < d_{\lambda}^{P}< d_{\mu}^{P}$. Let us also assume that if $d_{\lambda}^{P} + 1 = d_{\mu}^{P}$, then $\lambda_{0}^{P} + \mu_{0}^{P}\neq 0$. 

Firstly, if $c_p\geq 2$, then $d_{\lambda}^{P'}>\min\{d_{\mu,1}^{P'},d_{\mu,3}^{P'}\}$. Furthermore, the following statements are true:
\begin{description}
\item[(i)] If $d_{\lambda}^{P} + 1 < d_{\mu}^{P}$ and $c_p>2$, then $d_{\lambda}^{P'} = d_{\mu}^{P} + c_p -2 + \psi >\max\{d_{\mu,1}^{P'},d_{\mu,3}^{P'}\}$ and $\lambda_{0}^{P'} = \left(-1\right)^{\psi}\mu_{0}^{P}$. If $d_{\lambda}^{P} + 1 = d_{\mu}^{P}$ and $c_p>2$, then $d_{\lambda}^{P'} = d_{\lambda}^{P} + c_p - 1 + \psi > d_{\mu}^{P'}$ and $\lambda_{0}^{P'} = \left(-1\right)^{\psi}\left(\lambda_{0}^{P} + \mu_{0}^{P}\right)$. 
\item[(ii)] If $c_p\geq 2$ and $d_{\lambda}^{P'}\leq \max\{d_{\mu,1}^{P'},d_{\mu,3}^{P'}\}$, then $c_p = 2$. Furthermore, in this case, $d_{\lambda}^{P'} = d_{\lambda}^{P} + 1 + \psi$ and $\lambda_{0}^{P'} = \left(-1\right)^{\psi}\lambda_{0}^{P}$. If $a_p = 0$, then $d_{\lambda}^{P'}\leq d_{\mu}^{P} + 1 = d_{\mu,3}^{P'}$, and if $b_p = 0$, then $d_{\lambda}^{P'}\leq d_{\mu}^{P} = d_{\mu,1}^{P'}$.  
\item[(iii)] Let us assume that $c_p=1$, and either $a_p = 0 = b_{p+1}$ or $b_p = 0 = a_{p+1}$. In this case, $d_{\lambda}^{P'} = d_{\lambda}^{P} + \psi > d_{\nu}^{P} + 1 - \psi = d_{\mu}^{P'}$, $d_{\nu}^{P'} = d_{\mu}^{P}$, $\lambda_{0}^{P'} = \left(-1\right)^{\psi}\lambda_{0}^{P}$, and $\nu_{0}^{P'} = \mu_{0}^{P}$. 
\end{description} 
\item[(b)] Let us assume that $d_{\nu}^{P} + \left(-1\right)^{\psi}>d_{\lambda}^{P}>d_{\mu}^{P}$.
\begin{description}
\item[(i)] If $c_p\geq 2$, then $d_{\lambda}^{P'}>\max\{d_{\mu,1}^{P'},d_{\mu,3}^{P'}\}$ and $\lambda_{0}^{P'} = \left(-1\right)^{1 - \psi}\nu_{0}^{P}$. 
\item[(ii)] Let us assume that $c_p = 1$, and either $a_p = 0 = b_{p+1}$ or $b_p = 0 = a_{p+1}$. In this case, $d_{\lambda}^{P'} = d_{\lambda}^{P} + \psi$, $d_{\mu}^{P'} = d_{\nu}^{P} + 1 - \psi$, $d_{\nu}^{P'} = d_{\lambda}^{P}$, $\lambda_{0}^{P'} = \left(-1\right)^{\psi}\lambda_{0}^{P}$, $\mu_{0}^{P'} = \left(-1\right)^{1-\psi}\nu_{0}^{P}$, and $\nu_{0}^{P'} = \lambda_{0}^{P}$.
\end{description}
\item[(c)] Let us assume that $d_{\nu}^{P} + \left(-1\right)^{\psi} = d_{\lambda}^{P} > d_{\mu}^{P}$. Let us also assume that $\lambda_{0}^{P} - \nu_{0}^{P}\neq 0$. 

If $c_p\geq 2$, then $d_{\lambda}^{P'}>\max\{d_{\mu,1}^{P'},d_{\mu,3}^{P'}\}$ and $\lambda_{0}^{P'} = \left(-1\right)^{\psi}\lambda_{0}^{P} + \left(-1\right)^{1-\psi}\nu_{0}^{P}$.
\item[(d)] 
\begin{description}
\item[(i)] Let us assume that $d_{\nu}^{P} + \left(-1\right)^{\psi} > d_{\mu}^{P} - 1$ and $d_{\mu}^{P}>d_{\lambda}^{P}$. 

If $c_p\geq 2$, then $d_{\lambda}^{P'}>\max\{d_{\mu,1}^{P'},d_{\mu,3}^{P'}\}$ and $\lambda_{0}^{P'} = \left(-1\right)^{1-\psi}\nu_{0}^{P}$. 
\item[(ii)] Let us assume that $d_{\mu}^{P} - 1>d_{\nu}^{P} + \left(-1\right)^{\psi}>d_{\lambda}^{P}$. 

If $c_p = 2$, then $\lambda_{0}^{P'} = \left(-1\right)^{1-\psi}\nu_{0}^{P}$ and $d_{\lambda}^{P'}>\min\{d_{\mu,1}^{P'},d_{\mu,3}^{P'}\}$. If $a_p = 0$ and $c_p = 2$, then $d_{\lambda}^{P'} = d_{\nu}^{P} +1 \leq d_{\mu}^{P} + 1 = d_{\mu,3}^{P'}$ and $\mu_{3,0}^{P'} = -\mu_{0}^{P}$. If $b_p = 0$ and $c_p = 2$, then $d_{\lambda}^{P'} = d_{\nu}^{P} + 2\leq d_{\mu}^{P} = d_{\mu,1}^{P'}$ and $\mu_{1,0}^{P'} = \mu_{0}^{P}$. If $c_p>2$, then $d_{\lambda}^{P'}>\max\{d_{\mu,1}^{P'},d_{\mu,3}^{P'}\}$ and $\lambda_{0}^{P'} = \left(-1\right)^{\psi}\mu_{0}^{P}$. 
\end{description}
\end{description}
\end{proposition}
\begin{proof}
If $a_p = 0$, then we define $s = 3$ and $s' = 1$, and if $b_p = 0$, then we define $s = 1$ and $s' = 3$. We now establish each statement separately. Let $X$ be a maximal $q$-resistance good $\left(r,s\right)$-type $P$-path, let $Y$ be a maximal $q$-resistance bad $\left(r,s\right)$-type $P$-path, and let $Z$ be a maximal $q$-resistance $\left(r,s'\right)$-type $P$-path.

\begin{description}
\item[(a)] The proof is similar to the proof of Proposition~\ref{p21/3extadd}. The main difference is that there are $\left(r,2\right)$-type $P'$-paths of the form $Z^{\alpha}$, where $Z$ is an $\left(r,s'\right)$-type $P$-path. However, Proposition~\ref{ppre1/32bothext} \textbf{(ii)} and the hypothesis $d_{\nu}^{P} + \left(-1\right)^{\psi}<d_{\lambda}^{P}$ imply that the $q$-resistance of $Z^{1}$ is strictly less than the $q$-resistance of $X^{1}$ and the $q$-resistance of $Y^{2}$. If $c_p\geq 2$, then we deduce that $d_{\lambda}^{P'}>\min\{d_{\mu,1}^{P'},d_{\mu,3}^{P'}\}$. 
\begin{description}
\item[(i) - (ii)] The proofs of the two statements are similar to the proof of Proposition~\ref{p21/3extadd} (based on Proposition~\ref{ppre1/32bothext}).
\item[(iii)] Proposition~\ref{ppre1/32bothext} \textbf{(i)} and \textbf{(iv)} imply that a maximal $q$-resistance good $\left(r,2\right)$-type $P'$-path is of the form $X^{1}$, where $X$ is a maximal $q$-resistance good $\left(r,s\right)$-type $P$-path. Proposition~\ref{ppre1/32bothext} \textbf{(i)} and \textbf{(iv)} imply that a maximal $q$-resistance bad $\left(r,2\right)$-type $P'$-path is of the form $Z^{1}$, where $Z$ is a maximal $q$-resistance $\left(r,s'\right)$-type $P$-path. The statement now follows from Proposition~\ref{ppre1/32bothext} \textbf{(ii)} and \textbf{(iii)}.
\end{description}
\item[(b)] 
\begin{description}
\item[(i)] Proposition~\ref{ppre1/32bothext} \textbf{(ii)} and the hypothesis imply that the $q$-resistance of $Z^{1}$ is strictly greater than the $q$-resistance of $X^{1}$ and the $q$-resistance of $Y^{2}$. Furthermore, the $\left(r,2\right)$-type $P'$-path $Z^{1}$ is good since $c_p\geq 2$.
\item[(ii)] The proof is similar to the proof of \textbf{(a)} \textbf{(iii)}.
\end{description}
\item[(c)] Proposition~\ref{ppre1/32bothext} \textbf{(ii)} and the hypothesis imply that a maximal $q$-resistance $\left(r,2\right)$-type $P'$-path is either of the form $Z^{1}$ or of the form $X^{1}$, where $X$ is a maximal $q$-resistance good $\left(r,s\right)$-type $P$-path and $Z$ is a maximal $q$-resistance $\left(r,s'\right)$-type $P$-path. The $\left(r,2\right)$-type $P'$-paths $Z^{1}$ and $X^{1}$ are good since $c_p\geq 2$.
\item[(d)]
\begin{description}
\item[(i)] The proof is the same as the proof of \textbf{(b)} \textbf{(i)}.
\item[(ii)] Proposition~\ref{ppre1/32bothext} \textbf{(ii)} and the hypothesis imply that a maximal $q$-resistance $\left(r,2\right)$-type $P'$-path is of the form $Y^{2}$, where $Y$ is a maximal $q$-resistance bad $\left(r,s\right)$-type $P$-path. If $c_p>2$, then $Y^{2}$ is good. If $c_p = 2$, then Proposition~\ref{ppre1/32bothext} \textbf{(ii)} implies that a maximal $q$-resistance good $\left(r,2\right)$-type $P'$-path is of the form $Z^{1}$, where $Z$ is a maximal $q$-resistance $\left(r,s'\right)$-type $P$-path. 
\end{description}
\end{description}
\end{proof}

We establish the following addendum to Proposition~\ref{p1/32bothext}.

\begin{proposition}
\label{p1/32bothextadd}
Let $P = \left(\prod_{i=1}^{p-1} \sigma_1^{a_i}\sigma_3^{b_i}\sigma_2^{c_i}\right)\sigma_1^{a_p}\sigma_3^{b_p}$ be an $s$-subproduct of $\sigma$ for $s\in \{1,3\}$ and let $P' = \prod_{i=1}^{p} \sigma_1^{a_i}\sigma_3^{b_i}\sigma_2^{c_i}$. We assume that $a_p,b_p>0$.
\begin{description}
\item[(a)] Let us consider the case $a_p>1$. Let us assume that $d_{\lambda,1}^{P}< d_{\mu,1}^{P}$ and $c_p\geq 2$. If $d_{\lambda,3}^{P} + 1\geq d_{\mu,1}^{P}$, then $d_{\lambda}^{P'}>\max\{d_{\mu,1}^{P'},d_{\mu,3}^{P'}\}$ and $\lambda_{0}^{P'} = -\lambda_{3,0}^{P}$.
\item[(b)] Let us consider the case $a_p = 1$ and $d_{\lambda,3}^{P} + 1\geq d_{\mu,1}^{P}$.
\begin{description}
\item[(i)] Let us assume that $d_{\lambda,3}^{P}+1 = d_{\mu,3}^{P}$. If $c_p>2$ and $\lambda_{3,0}^{P} + \mu_{3,0}^{P}\neq 0$, then $d_{\lambda}^{P'}>\max\{d_{\mu,1}^{P'},d_{\mu,3}^{P'}\}$ and $\lambda_{0}^{P'} = -\left(\lambda_{3,0}^{P} + \mu_{3,0}^{P}\right)$. If $c_p = 2$, then $d_{\mu,1}^{P'}<d_{\lambda}^{P'} = d_{\lambda,3}^{P} + 2 = d_{\mu,3}^{P} + 1 = d_{\mu,3}^{P'}$, $\lambda_{0}^{P'} = -\lambda_{3,0}^{P}$, and $\mu_{3,0}^{P'} = -\mu_{3,0}^{P}$.
\item[(ii)] Let us assume that $d_{\lambda,3}^{P} + 1 < d_{\mu,3}^{P}$. If $c_p>2$, then $d_{\lambda}^{P'}>\max\{d_{\mu,1}^{P'},d_{\mu,3}^{P'}\}$ and $\lambda_{0}^{P'} = -\mu_{3,0}^{P}$. If $c_p = 2$, then $d_{\mu,1}^{P'}< d_{\lambda}^{P'} = d_{\lambda,3}^{P} + 2 < d_{\mu,3}^{P} + 1 = d_{\mu,3}^{P'}$, $\lambda_{0}^{P'} = -\lambda_{3,0}^{P}$, and $\mu_{3,0}^{P'} = -\mu_{3,0}^{P}$.
\end{description} 
\item[(c)] Let us consider the case $a_p = 1$ and $d_{\lambda,3}^{P} > d_{\mu,3}^{P}$. 
\begin{description}
\item[(i)] Let us assume that $d_{\lambda,3}^{P} + 2 = d_{\mu,1}^{P}$. If $c_p>2$ and $\mu_{1,0}^{P} - \lambda_{3,0}^{P}\neq 0$, then $d_{\lambda}^{P'}>\max\{d_{\mu,1}^{P'},d_{\mu,3}^{P'}\}$ and $\lambda_{0}^{P'} = \mu_{1,0}^{P} - \lambda_{3,0}^{P}$. If $c_p = 2$, then $d_{\mu,3}^{P'}<d_{\lambda}^{P'} = d_{\lambda,3}^{P} + 2 = d_{\mu,1}^{P} = d_{\mu,1}^{P'}$, $\lambda_{0}^{P'} = -\lambda_{3,0}^{P}$, and $\mu_{1,0}^{P'} = \mu_{1,0}^{P}$.
\item[(ii)] Let us assume that $d_{\lambda,3}^{P} + 2 < d_{\mu,1}^{P}$. If $c_p>2$, then $d_{\lambda}^{P'} > \max\{d_{\mu,1}^{P'},d_{\mu,3}^{P'}\}$ and $\lambda_{0}^{P'} = \mu_{1,0}^{P}$. If $c_p = 2$, then $d_{\mu,3}^{P'} < d_{\lambda}^{P'} = d_{\lambda}^{P} + 2 < d_{\mu,1}^{P} = d_{\mu,1}^{P'}$, $\lambda_{0}^{P'} = -\lambda_{3,0}^{P}$, and $\mu_{1,0}^{P'} = \mu_{1,0}^{P}$.
\end{description}
\end{description}
\end{proposition}
\begin{proof}
\begin{description}
\item[(a)] Proposition~\ref{ppre1/32bothext} \textbf{(i)}, \textbf{(ii)}, and \textbf{(iv)} and the hypothesis $d_{\lambda,3}^{P}+1\geq d_{\mu,1}^{P}\geq d_{\lambda,1}^{P}$ imply that a maximal $q$-resistance $\left(r,2\right)$-type $P'$-path is of the form $U^{1}$, where $U$ is a maximal $q$-resistance good $\left(r,3\right)$-type $P$-path. Furthermore, $U^{1}$ is a good $\left(r,2\right)$-type $P'$-path since $c_p\geq 2$.
\item[(b)] The proof is similar to the proof of Proposition~\ref{p1/32extadd} \textbf{(a)} \textbf{(i)} - \textbf{(ii)}. Let $Y$ be a maximal $q$-resistance bad $\left(r,1\right)$-type $P$-path, let $U$ be a maximal $q$-resistance good $\left(r,3\right)$-type $P$-path, and let $V$ be a maximal $q$-resistance bad $\left(r,3\right)$-type $P$-path. The hypothesis $d_{\lambda,3}^{P} + 1\geq d_{\mu,1}^{P}$ and Proposition~\ref{ppre1/32bothext} \textbf{(ii)} imply that the $q$-resistance of $Y^{2}$ is strictly less than the $q$-resistance of $U^{1}$ and the $q$-resistance of $V^{2}$. 
\begin{description}
\item[(i) - (ii)] The proofs of the two statements are similar to the proof of Proposition~\ref{p1/32extadd} \textbf{(a)} \textbf{(i)} - \textbf{(ii)} (based on Proposition~\ref{ppre1/32bothext}). 
\end{description}
\item[(c)] The proof is similar to the proof of \textbf{(b)}.
\end{description}
\end{proof}

\subsection{Regularity of subproducts of a positive braid}
\label{subsecproofstrongerversion}

In this subsection, we will establish Theorem~\ref{strongkernelconstraint}. Firstly, we implement a minor modification of the definition of strong regularity (Definition~\ref{defrregular}). We will establish Theorem~\ref{strongkernelconstraint} for $t\neq 2,3$, whereas we established Theorem~\ref{invariantdetect} (the analogue of Theorem~\ref{strongkernelconstraint} in the case where $\sigma$ is a normal braid rather than a weakly normal braid) for $t\neq 2$. The modification accounts for the fact that we now require $t\neq 3$ in addition to $t\neq 2$.

\begin{definition}
\label{defrregularnew}
We modify Definition~\ref{defrregular} as follows. We replace ``$\lambda_{0}^{P}$ is a non-zero integer with only $2$ as a prime factor" with ``$\lambda_{0}^{P}$ is a non-zero integer with at most $2$ and $3$ as prime factors" in \textbf{(a)} \textbf{(i)}, \textbf{(a)} \textbf{(ii)}, and \textbf{(b)} \textbf{(i)}. 

We replace ``$\lambda_{3,0}^{P}$ is a non-zero integer with only $2$ as a prime factor" with ``$\lambda_{3,0}^{P}$ is a non-zero integer with at most $2$ and $3$ as prime factors" in \textbf{(b)} \textbf{(ii)}. We also replace ``$\lambda_{1,0}^{P} + \lambda_{3,0}^{P} = 0$" with ``$\alpha\lambda_{1,0}^{P} + \beta\lambda_{3,0}^{P} = 0$ for some $\alpha,\beta\in \{1,2\}$" in \textbf{(b)} \textbf{(ii)}. 

If $P$ is an $s$-subproduct of $\sigma$, then in the remainder of this paper we write that $P$ is \textit{strongly $r$-regular} if $P$ satisfies this modification of Definition~\ref{defrregular}. 
\end{definition}

Of course, if $P$ is strongly $r$-regular in the sense of Definition~\ref{defrregular}, then $P$ is strongly $r$-regular in the sense of Definition~\ref{defrregularnew}. Furthermore, Corollary~\ref{lemmaroadregular} is still true if ``strongly $r$-regular" is true in the sense of Definition~\ref{defrregularnew} instead of Definition~\ref{defrregular} with the identical proof. We will apply this version of Corollary~\ref{lemmaroadregular} in the sequel.

We now define the property of regularity of subproducts of a positive braid, which is more general than the property of strong regularity (in the sense of Definition~\ref{defrregularnew}). We will show that if $\sigma$ is a weakly normal braid, then regularity is an inductive property for proper subproducts of $\sigma$. We recall that we use Notation~\ref{notssubproduct} in the remainder of this paper.

\begin{definition}
\label{defweakrregular}
If $P$ is an $s$-subproduct of $\sigma$, then we write that $P$ is \textit{$r$-regular} if one of the following conditions \textbf{(a)} (if $s = 2$) or \textbf{(b)} (if $s\in \{1,3\}$) on the weighted numbers of $P$-paths with initial vertex equal to $r$ is satisfied. If one of the following conditions is satisfied, then we also define the \textit{sign of $P$} in some cases and the \textit{discrepancy of $P$} in each case. We denote the discrepancy of $P$ by $\Xi\left(P\right)$. 

\begin{description}
\item[(a)] Let $P$ be a $2$-subproduct. In \textbf{(i)}, we will state the conditions in the case that $c_{p-1}\geq 2$, and in \textbf{(ii)}, we will state the conditions in the case that $c_{p-1} = 1$.

\begin{description}
\item[(i)] Let us assume that $c_{p-1}\geq 2$. In this case, $\lambda_{0}^{P}$ is a non-zero integer with at most $2$ and $3$ as prime factors, and $d_{\lambda}^{P}>\min\{d_{\mu,1}^{P},d_{\mu,3}^{P}\}$. If $d_{\lambda}^{P}\leq d_{\mu,i}^{P}$ for $i\in \{1,3\}$, then $\lambda_{0}^{P}\in \{\pm \mu_{i,0}^{P}\}$ and the \textit{sign of $P$} is the sign of the ratio $\frac{\mu_{i,0}^{P}}{\lambda_{0}^{P}}$. 

The \textit{discrepancy of $P$} is $\Xi\left(P\right) = \max\{d_{\mu,1}^{P},d_{\mu,3}^{P}\} - d_{\lambda}^{P}$. 
\item[(ii)] Let us assume that $c_{p-1} = 1$. If $a_p = 0$, then we define $\psi' = 0$, and if $b_p = 0$, then we define $\psi' = 1$. Firstly, $\lambda_{0}^{P}$ is a non-zero integer with at most $2$ and $3$ as prime factors. Secondly, either $d_{\lambda}^{P}>d_{\mu}^{P}$ or $d_{\nu}^{P} + \psi'\leq d_{\lambda}^{P}$. If $d_{\lambda}^{P}\leq d_{\mu}^{P}$, then $\lambda_{0}^{P}\in \{\pm \mu_{0}^{P}\}$ and the \textit{sign of $P$} is the sign of the ratio $\frac{\mu_{0}^{P}}{\lambda_{0}^{P}}$. If $d_{\nu}^{P} + \psi'>d_{\lambda}^{P}$, then $\lambda_{0}^{P}\in \{\pm \nu_{0}^{P}\}$ and the \textit{sign of $P$} is the sign of the ratio $\left(-1\right)^{1-\psi'}\frac{\nu_{0}^{P}}{\lambda_{0}^{P}}$. 

If $d_{\nu}^{P} + \psi'\leq d_{\lambda}^{P}$, then the \textit{discrepancy of $P$} is $\Xi\left(P\right) = d_{\mu}^{P} - d_{\lambda}^{P}$. If $d_{\nu}^{P}+\psi'>d_{\lambda}^{P}$, then the \textit{discrepancy of $P$} is $\Xi\left(P\right) = d_{\nu}^{P}+\psi'-d_{\lambda}^{P}$. 
\end{description}
\item[(b)] Let $P$ be an $s$-subproduct for $s\in \{1,3\}$. In \textbf{(i)}, we will state the conditions in the case that either $a_p = 0$ or $b_p = 0$, and in \textbf{(ii)}, we will state the conditions in the case that $a_p,b_p>0$. 

\begin{description}
\item[(i)] Let us assume that either $a_p = 0$ or $b_p = 0$. If $a_p = 0$, then we define $\psi' = 0$, and if $b_p = 0$, then we define $\psi' = 1$. Firstly, $\lambda_{0}^{P}$ is a non-zero integer with at most $2$ and $3$ as prime factors. Secondly, either $d_{\lambda}^{P}>d_{\mu}^{P}$ or $d_{\nu}^{P} + 2\psi'+1\leq d_{\lambda}^{P}$. If $d_{\lambda}^{P}\leq d_{\mu}^{P}$, then $\lambda_{0}^{P}\in \{\pm \mu_{0}^{P}\}$ and the \textit{sign of $P$} is the sign of $\frac{\mu_{0}^{P}}{\lambda_{0}^{P}}$. If $d_{\nu}^{P}+2\psi'+1>d_{\lambda}^{P}$, then $\lambda_{0}^{P}\in \{\pm \nu_{0}^{P}\}$ and the \textit{sign of $P$} is the sign of $-\frac{\nu_{0}^{P}}{\lambda_{0}^{P}}$. 

If $d_{\nu}^{P} + 2\psi' + 1\leq d_{\lambda}^{P}$, then the \textit{discrepancy of $P$} is $\Xi\left(P\right) = d_{\mu}^{P} - d_{\lambda}^{P}$. If $d_{\nu}^{P} + 2\psi' + 1 > d_{\lambda}^{P}$, then the \textit{discrepancy of $P$} is $\Xi\left(P\right) = d_{\nu}^{P} + 2\psi' + 1 - d_{\lambda}^{P}$. 
\item[(ii)] Let us assume that $a_p,b_p>0$ throughout. (Note that if $a_p>1$, then $w_{\mu,3}^{P} = 0$, and if $a_p = 1$, then $w_{\lambda,1}^{P} = 0$; see Proposition~\ref{pnotationentries} \textbf{(b)} \textbf{(ii)} - \textbf{(iii)}.) The following conditions are satisfied. We will state the conditions separately in the cases $a_p>1$ and $a_p = 1$. In both cases, one condition is that either $\lambda_{1,0}^{P}$ or $\lambda_{3,0}^{P}$ is a non-zero integer with at most $2$ and $3$ as prime factors.

Firstly, we consider the case $a_p>1$. If $d_{\lambda,1}^{P}<d_{\mu,1}^{P}$, then $d_{\lambda,3}^{P}+1\geq d_{\mu,1}^{P}$ and $\lambda_{3,0}^{P}$ is a non-zero integer with at most $2$ and $3$ as prime factors. If $d_{\lambda,1}^{P}\geq d_{\mu,1}^{P}$, then either $d_{\lambda,1}^{P}> d_{\lambda,3}^{P} + 1$ and $\lambda_{1,0}^{P}$ is a non-zero integer with at most $2$ and $3$ as prime factors, $d_{\lambda,1}^{P} <  d_{\lambda,3}^{P} + 1$ and $\lambda_{3,0}^{P}$ is a non-zero integer with at most $2$ and $3$ as prime factors, or $\alpha\lambda_{1,0}^{P} + \beta\lambda_{3,0}^{P} = 0$ for some $\alpha,\beta\in \{1,2\}$. 

If $\alpha\lambda_{1,0}^{P} + \beta\lambda_{3,0}^{P}\neq 0$ for every $\alpha,\beta\in \{1,2\}$, then the \textit{discrepancy of $P$} is $\Xi\left(P\right) = 0$. If $\alpha\lambda_{1,0}^{P} + \beta\lambda_{3,0}^{P} = 0$ for some $\alpha,\beta\in \{1,2\}$, then the \textit{discrepancy of $P$} is $\Xi\left(P\right) = d_{\mu,1}^{P} - d_{\lambda,1}^{P}$. 
 
Secondly, we consider the case $a_p = 1$. In this case, $d_{\lambda,3}^{P}> \min\{d_{\mu,1}^{P}-2,d_{\mu,3}^{P}\}$. If $d_{\lambda,3}^{P}\leq d_{\mu,1}^{P} - 2$, then $\lambda_{3,0}^{P}\in \{\pm \mu_{1,0}^{P}\}$ and the \textit{sign of $P$} is the sign of $-\frac{\mu_{1,0}^{P}}{\lambda_{3,0}^{P}}$. If $d_{\lambda,3}^{P}\leq d_{\mu,3}^{P}$, then $\lambda_{0}^{P}\in \{\pm \mu_{3,0}^{P}\}$ and the \textit{sign of $P$} is $\frac{\mu_{3,0}^{P}}{\lambda_{3,0}^{P}}$.

The \textit{discrepancy of $P$} is $\Xi\left(P\right) = \max\{d_{\mu,1}^{P}-2,d_{\mu,3}^{P}\} - d_{\lambda,3}^{P}$. 
\end{description}
\end{description}
If the sign of $P$ is positive (resp. negative), then we write that $P$ is \textit{positive} (resp. \textit{negative}). If $B$ is a block in $\sigma$ and if $P$ is the $s$-subproduct immediately preceding $B$, then we will write that \textit{$B$ is an $r$-regular block} if $PB$ (which is an $s'$-subproduct for $s'\neq s$) is an $r$-regular $s'$-subproduct. In this case, the \textit{discrepancy of $B$} is defined to be $\Xi\left(P'\right)$ and we denote it by $\Xi\left(B\right)$.
\end{definition}

\begin{remark}
If $P$ is an $r$-regular subproduct of $\sigma$, then $P$ is strongly $r$-regular if and only if the discrepancy $\Xi\left(P\right) < 0$ (compare Definition~\ref{defrregularnew} with reference to Definition~\ref{defrregular}). 
\end{remark}

We have already studied strongly $r$-regular subproducts of $\sigma$ in Section~\ref{subsecstrongregularity} and we primarily focus here on $r$-regular subproducts $P$ with nonnegative discrepancy $\Xi\left(P\right)\geq 0$. Firstly, we observe that an $r$-regular $s$-subproduct $P$ for $s\in \{1,3\}$ with $a_p>1$ and $b_p>0$ is very close to being strongly $r$-regular.

\begin{proposition}
\label{pregularap>1}
Let $P$ be an $r$-regular $s$-subproduct of $\sigma$ for $s\in \{1,3\}$ and let $P'$ be the minimal $2$-subproduct of $\sigma$ such that $P\subseteq P'$. Let us assume that $a_p>1$ and $b_p>0$. If $t\neq 2,3$ is a prime number, then the leading coefficients of multiple entries in the $r$th row of $\beta_4\left(P'\right)$ are non-zero modulo $t$. If $P'\neq \sigma$, then $P'$ is strongly $r$-regular.
\end{proposition}
\begin{proof}
The statement is a consequence of Proposition~\ref{p1/32bothext} \textbf{(a)} \textbf{(i)} and Proposition~\ref{p1/32bothextadd} \textbf{(a)}. 
\end{proof}

We introduce the following categorization of $r$-regular subproducts.

\begin{definition}
Let $P$ be an $r$-regular $s$-subproduct of $\sigma$ such that $\Xi\left(P\right)\geq 0$. 
\begin{description}
\item[(a)] Let us consider the case $s = 2$. 
\begin{description}
\item[(i)] Let us assume that $c_{p-1}\geq 2$. If $d_{\lambda}^{P}\leq d_{\mu,i}^{P}$ for $i\in \{1,3\}$, then we write that $P$ is an \textit{$i$-switch $\mu$ $r$-regular subproduct}. 
\item[(ii)] Let us assume that $c_{p-1} = 1$. If $a_p = 0$, then we define $i = 3$ and $i' = 1$, and if $b_p = 0$, then we define $i = 1$ and $i' = 3$. If $d_{\lambda}^{P}\leq d_{\mu}^{P}$, then we write that $P$ is an \textit{$i$-switch $\mu$ $r$-regular subproduct}. If $d_{\nu}^{P} + \psi'>d_{\lambda}^{P}$ (where $\psi'$ is as in Definition~\ref{defweakrregular} \textbf{(a)} \textbf{(ii)}), then we write that $P$ is an \textit{$i'$-switch $\nu$ $r$-regular subproduct}.
\end{description}
\item[(b)] Let us consider the case $s\in \{1,3\}$.
\begin{description}
\item[(i)] Let us assume that either $a_p = 0$ or $b_p = 0$. If $a_p = 0$, then we define $i = 3$ and $i' = 1$, and if $b_p = 0$, then we define $i = 1$ and $i' = 3$. If $d_{\lambda}^{P}\leq d_{\mu}^{P}$, then we write that $P$ is an \textit{$i$-switch $\mu$ $r$-regular subproduct}. If $d_{\nu}^{P} + 2\psi'+1> d_{\lambda}^{P}$ (where $\psi'$ is as in Definition~\ref{defweakrregular} \textbf{(b)} \textbf{(i)}), then we write that $P$ is an \textit{$i'$-switch $\nu$ $r$-regular subproduct}. 
\item[(ii)] Let us assume that $a_p = 1 = b_p$. If $d_{\lambda,3}^{P}\leq d_{\mu,1}^{P}-2$, then we write that $P$ is a \textit{$1$-switch $\mu$ $r$-regular subproduct}. If $d_{\lambda,3}^{P}\leq d_{\mu,3}^{P}$, then we write that $P$ is \textit{$3$-switch $\mu$ $r$-regular subproduct}. 
\end{description}
\end{description}
\end{definition}

\subsubsection{Regularity is generically an inductive property along roads}

Let $P\subseteq P'$, where $P$ is an $s$-subproduct of $\sigma$ and $P'$ is an $s'$-subproduct of $\sigma$. If $P$ is $r$-regular, then we will determine conditions on $P'\setminus P$ that imply that $P'$ is $r$-regular, and determine the change in the discrepancy $\Xi\left(P'\right) - \Xi\left(P\right)$ under these conditions. If $\sigma$ is a weakly normal braid, then these conditions are always satisfied.

Firstly, we define an adjustment to the weight of a bad subroad. If $P'\setminus P$ is a bad subroad, then we will use this in order to establish a formula relating the change in discrepancy to $\omega\left(P'\setminus P\right)$ (the weight of the bad subroad $P'\setminus P$).

\begin{definition}
Let $R$ be an $i_0$-switch bad subroad and let $R = \prod_{j=1}^{k} R_j$ be the product decomposition of $R$ into elementary subroads. We define the \textit{increment of $R$} as follows and denote it by $\iota\left(R\right)$.
\begin{description}
\item[(a)] Let us consider the case where $R_k$ is a non-isolated $\sigma_2$ subroad.
\begin{description}
\item[(i)] If either $R_k$ does not end with $\sigma_1\sigma_3$ or $R_k$ is a $3$-terminal bad subroad, then we define $\iota\left(R\right) = 0$.
\item[(ii)] If $R_k$ ends with $\sigma_1\sigma_3$ and $R_k$ is a $1$-terminal bad subroad, then we define $\iota\left(R\right) = 1$.
\end{description}
\item[(b)] Let us consider the case where $R_k$ is an $i$-initial isolated $\sigma_2$ subroad for $i\in \{1,3\}$, and let $i'$ be such that $\{i,i'\} = \{1,3\}$. 
\begin{description}
\item[(i)] If $R_k$ is an $i$-initial isolated $\sigma_2$ subroad of $R$ that does not end with $\sigma_{i'}$, then we define $\iota\left(R\right) = 0$.
\item[(ii)] If $R_k$ is an $i$-initial isolated $\sigma_2$ subroad of $R$ that ends with $\sigma_{i'}$, then we define $\iota\left(R\right) = -2$.
\end{description}
\end{description}
\end{definition}

We now establish fundamental statements concerning the property of $r$-regularity and bad subroads.

\begin{lemma}
\label{lweightcompute}
Let $P$ be an $i$-switch $\mu$ $r$-regular $s$-subproduct of $\sigma$ such that $\Xi\left(P\right)>0$, and if $i = 1$, then $P$ does not end with $\sigma_1\sigma_3$. Let $P'$ be an $s'$-subproduct of $\sigma$ such that $P\subseteq P'$. Let us assume that $P'\setminus P$ is an $i$-switch bad subroad and $P'\neq \sigma$. Let $M$ be the number of $\sigma_2$s in isolated $\sigma_2$ subroads of $P'\setminus P$ and let $t\neq 2,3$ be a prime number. The following statements are true:

\begin{description}
\item[(a)] Let us consider the case $\omega\left(P'\setminus P\right) + \iota\left(P'\setminus P\right) \leq \Xi\left(P\right)$. Firstly, $P'$ is $r$-regular.  
\begin{description}
\item[(i)] We have $\Xi\left(P'\right) = \Xi\left(P\right) - \omega\left(P'\setminus P\right) - \iota\left(P'\setminus P\right)$.
\item[(ii)] If $P'\setminus P$ is an $i'$-terminal bad subroad, then $P'$ is $i'$-switch $r$-regular.
\item[(iii)] If $M$ is even, then $P'$ is $\mu$-switch $r$-regular. if $M$ is odd, then $P'$ is $\nu$-switch $r$-regular.
\item[(iv)] If $M\equiv 0,3 \pmod 4$, then the sign of $P'$ is equal to the sign of $P$. If $M\equiv 1,2\pmod 4$, then the sign of $P'$ is the opposite of the sign of $P$.
\end{description}
\item[(b)] If $\omega\left(P'\setminus P\right) + \iota\left(P'\setminus P\right)>\Xi\left(P\right)$, then $P'$ is strongly $r$-regular. 
\end{description}
\end{lemma}
\begin{proof}
The proof that $P$ is $r$-regular (if $\omega\left(P'\setminus P\right) + \iota\left(P'\setminus P\right)\leq \Xi\left(P\right)$) and strongly $r$-regular (if $\omega\left(P'\setminus P\right) + \iota\left(P'\setminus P\right) > \Xi\left(P\right)$) is based on the theory of path extensions in Subsection~\ref{subsecpathextensions} and Subsection~\ref{subsecpathextensionsII}. The proof is analogous to the proof of Corollary~\ref{lemmaroadregular} in Section~\ref{mains}. 

Let us establish the statements \textbf{(a)} \textbf{(i)} - \textbf{(iv)}. Let $R = P'\setminus P$. If $R$ is an $i$-switch bad subroad, then there is a finite chain $P = P_0\subsetneq P_1\subsetneq \cdots \subsetneq P_k = P'$ with the property that $P_j\setminus P_{j-1}$ is an $i_{j-1}$-switch elementary bad subroad for $1\leq j\leq k$, where $P_j$ is an $s_j$-subproduct of $\sigma$ and $i_0 = i$. We use induction on $j$ to establish that the conclusion is true with $P'$ replaced by $P_j$ and that $P_j$ is an $i_j$-switch $r$-regular subproduct. The conclusion is certainly true with $P'$ replaced by $P_0 = P$ and $P_0 = P$ is an $i_0$-switch $r$-regular subproduct. Let us assume that the conclusion is true with $P'$ replaced by $P_j$ and that $P_j$ is an $i_j$-switch $r$-regular subproduct. If $1\leq j\leq k-1$, then we will show that the conclusion is true with $P'$ replaced by $P_{j+1}$ and that $P_{j+1}$ is an $i_{j+1}$-switch $r$-regular subproduct. If $1\leq j\leq k-1$, then note that $P_j$ is not a $\nu$ $r$-regular subproduct by Definition~\ref{defbadsubroad}.

Let $X_j$ be a maximal $q$-resistance good $\left(r,s_j\right)$-type $P_j$-path. If $s_j\in \{1,3\}$, then let $W_j$ be a maximal $q$-resistance bad $\left(r,i_j\right)$-type $P_j$-path, and if $s_j = 2$, then let $W_j$ be a maximal $q$-resistance $i_j$-switch bad $\left(r,2\right)$-type $P_j$-path. If either $i_j = 3$ or $P_j$ does not end with $\sigma_1\sigma_3$, then $\Xi\left(P_j\right)$ is equal to the difference between the $q$-resistance of $W_j$ and the $q$-resistance of $X_j$. If $i_j = 1$ and $P_j$ ends with $\sigma_1\sigma_3$, then $\Xi\left(P_j\right)$ is equal to two less than this difference. Let $X_j'$ be a maximal $q$-resistance extension of $X_j$ to a $P_{j+1}$-path and let $W_{j}'$ be a maximal $q$-resistance extension of $W_j$ to a $P_{j+1}$-path. Proposition~\ref{pelementarybadsubroadimpliesuniqueextension} characterizes $W_j'$ according to the type of the $i_j$-switch elementary bad subroad $R_{j} = P_{j+1}\setminus P_{j}$. We now characterize $X_j'$, compute the difference between the $q$-resistance of $W_j'$ and the $q$-resistance of $X_j'$, and determine $\Xi\left(P_{j+1}\right)$ based on this. We split into cases according to the type of the $i_j$-switch elementary bad subroad $R_j$.

Let $R_j$ be an $i_j$-constant bad subroad. Lemma~\ref{lemmagoodext} \textbf{(i)} implies that the $P_{j+1}$-path $X_{j}'$ is defined by extending $X_j$ by vertex changes at every opportunity (by $i_j\to 2$ vertex changes at the first $\sigma_2$ in every power of $\sigma_2$ and $2\to i_j$ vertex changes at the first $\sigma_{i_j}$ in every power of $\sigma_{i_j}$). If $i_j = 1$ and $P_{j}$ ends with $\sigma_1\sigma_3$, then the difference $\Xi\left(P_j\right) - \Xi\left(P_{j+1}\right)$ is $\omega\left(R_j\right) - 1$ (one less than the number of squares in $R_j$). If either $i_j = 3$ or $P_j$ does not end with $\sigma_1\sigma_3$, then the difference $\Xi\left(P_j\right) - \Xi\left(P_{j+1}\right)$ is $\omega\left(R_j\right)$ (the number of squares in $R_j$). In both cases, the conclusion is true with $P'$ replaced by $P_{j+1}$ and $P_{j+1}$ is an $i_{j+1}$-switch $r$-regular subproduct (with $i_{j+1} = i_j$).

Let $R_j$ be an alternating bad subroad. Lemma~\ref{lemmagoodext} \textbf{(i)} implies that the $P_{j+1}$-path $X_j'$ is defined by extending $X_j$ by $2\to 3$ vertex changes at every $\sigma_3$ and $3\to 2$ vertex changes at the first $\sigma_2$ in every power of $\sigma_2$. The difference $\Xi\left(P_{j}\right) - \Xi\left(P_{j+1}\right)$ is $\omega\left(R_j\right) + \iota\left(R_j\right)$ (the number of squares in $R_j$ if either $P_{j+1}$ does not end with $\sigma_1\sigma_3$ or $R_j$ is $3$-terminal, or one more than the number of squares in $R_j$ if $P_{j+1}$ ends with $\sigma_1\sigma_3$ and $R_j$ is $1$-terminal). In this case, we deduce that the conclusion is true with $P'$ replaced by $P_{j+1}$ and $P_{j+1}$ is an $i_{j+1}$-switch $r$-regular subproduct (with $i_{j+1} = i_j$ if there is an even number of $\sigma_1\sigma_3$s in $R_j$ and $i_j\neq i_{j+1}\in \{1,3\}$ if there is an odd number of $\sigma_1\sigma_3$s in $R_j$).

If $R_j$ is an $i_j$-initial isolated $\sigma_2$ subroad, then Lemma~\ref{lemmagoodext} \textbf{(i)} implies that the $P_{j+1}$-path $X_j'$ is defined by extending $X_j$ by vertex changes at every opportunity (if $i_j'$ is such that $\{i_j,i_j'\} = \{1,3\}$, then $i_j\to 2$ vertex changes at every odd-numbered $\sigma_2$ in $R_j$, $i_j'\to 2$ vertex changes at every even-numbered $\sigma_2$ in $R_j$, $2\to 1$ vertex changes at the first $\sigma_1$ in every power of $\sigma_1$, and $2\to 3$ vertex changes at the first $\sigma_3$ in every power of $\sigma_3$). If $P_{j+1}$ does not end with $\sigma_{i_j'}$, then the difference $\Xi\left(P_j\right) - \Xi\left(P_{j+1}\right)$ is $\omega\left(R_j\right)$. If $P_{j+1}$ ends with $\sigma_{i_j'}$, then $j+1 = k$ and the difference $\Xi\left(P_j\right)-\Xi\left(P_{j+1}\right)$ is $\omega\left(R_j\right) + \iota\left(R\right)$ (where $\iota\left(R\right) = -2$). In both cases, the conclusion is true with $P'$ replaced by $P_{j+1}$ and $P_{j+1}$ is an $i_{j+1}$-switch $r$-regular subproduct (with $i_{j+1} = i_j$).

We have established the statement in all cases by induction on $j$.
\end{proof}

We now establish a complementary statement to Lemma~\ref{lweightcompute}. Let $P\subseteq P'$, where $P$ is an $s$-subproduct of $\sigma$ and $P'$ is an $s'$-subproduct of $\sigma$. If $P$ is an $i$-switch $r$-regular subproduct and $P'\setminus P$ is not an $i$-switch bad subroad, then we will establish that $P'$ is strongly $r$-regular and the leading coefficients of multiple entries in the $r$th row of $\beta_4\left(P'\right)$ are non-zero modulo each prime $t\neq 2,3$ in generic situations. In analogy with Example~\ref{exglobalcancellation}, the generic situations correpond to the case where catchup does not occur along the maximal $i$-switch bad subroad $R$ following $P$ (i.e., there is not an exact linear relationship between the discrepancy of $P$ and the weight of $R$). If $\sigma$ is a weakly normal braid, then this is always true. 

\begin{lemma}
\label{luniqueextensionimpliesbadsubroad}
Let $P$ be an $i$-switch $r$-regular $s$-subproduct of $\sigma$ such that $\Xi\left(P\right)\geq 0$. Let $P'$ be the minimal $s'$-subproduct of $\sigma$ that properly contains $P$. Let $t\neq 2,3$ be a prime number. Let us consider one of the following conditions \textbf{(a)} (if $s = 2$) or \textbf{(b)} (if $s\in \{1,3\}$) on $P$. 
\begin{description}
\item[(a)] Let $P$ be a $2$-subproduct. 
\begin{description}
\item[(i)] Let us consider the case where $c_{p-1}\geq 2$. In this case, either $a_p = 0$ or $b_p = 0$. If either $a_p = 0$ and $i = 3$, or $b_p = 0$ and $i = 1$, then $\Xi\left(P\right)\neq 1$. If either $a_p = 0$ and $i = 1$, or $b_p = 0$ and $i = 3$, then $\Xi\left(P\right)\neq 0$. 
\item[(ii)] Let us consider the case where $c_{p-1} = 1$. If $P$ is $\mu$ $r$-regular, then $\Xi\left(P\right)\neq 1$.
\item[(iii)] Let us consider the case where $c_{p-1}\geq 2$. In this case $a_p,b_p>0$ and $\Xi\left(P\right)\not\in \{0,1\}$.
\end{description}
\item[(b)] Let $P$ be an $s$-subproduct for $s\in \{1,3\}$.
\begin{description}
\item[(i)] Let us consider the case where either $a_p = 0$ or $b_p = 0$. In this case, $P$ is $\mu$ $r$-regular and $\Xi\left(P\right)\neq 1$.
\item[(ii)] Let us consider the case where either $a_p = 0$ or $b_p = 0$. In this case, $P$ is $\nu$ $r$-regular and $\Xi\left(P\right)\neq 2$.
\item[(iii)] Let us consider the case where $a_p>0$ and $b_p>0$. If $a_p=1$ and $P$ is $1$-switch $r$-regular, then $\Xi\left(P\right)\neq 0$. If $a_p = 1$ and $P$ is $3$-switch $r$-regular, then $\Xi\left(P\right)\neq 1$.
\end{description}
\end{description}
If one of the above conditions is satisfied, and we also require $\Xi\left(P\right)\neq 0$ in \textbf{(a)} \textbf{(i)} - \textbf{(ii)} and \textbf{(b)} \textbf{(i)}, then the leading coefficients of multiple entries in the $r$th row of $\beta_4\left(P'\right)$ are non-zero modulo $t$. 

If one of the above conditions except \textbf{(a)} \textbf{(iii)} is satisfied, $P'\neq \sigma$, and $P'\setminus P$ is not an $i$-switch bad subroad, then $P'$ is strongly $r$-regular. If \textbf{(a)} \textbf{(iii)} is satisfied, then $P'$ is $r$-regular. 
\end{lemma}
\begin{proof}
The proof that the leading coefficients of multiple entries in the $r$th row of $\beta_4\left(P'\right)$ are non-zero modulo $t$ is based on the theory of path extensions in Subsection~\ref{subsecpathextensions} and Subsection~\ref{subsecpathextensionsII}. The proof is analogous to the proof of Corollary~\ref{lemmaroadregular} in Section~\ref{mains}.

Let us establish the statements concerning the strong regularity (or regularity in \textbf{(a)} \textbf{(iii)}) of $P'$. 

\begin{description}
\item[(a)] 
\begin{description}
\item[(i) - (ii)] In both \textbf{(i)} and \textbf{(ii)}, either $a_{p} = 0$ or $b_{p} = 0$ since $P'\setminus P$ is a subroad. If $a_{p} = 0$, then the hypothesis on $P'\setminus P$ implies that either $b_{p}>2$ or $i= 1$. If $b_{p} = 0$, then the hypothesis on $P'\setminus P$ implies that either $a_{p}>2$ or $i= 3$. If one of the conditions is satisfied, then Proposition~\ref{pnotationentries} \textbf{(a)} \textbf{(iii)} - \textbf{(iv)} imply that either $d_{\lambda}^{P}\geq d_{\mu}^{P}$ or $d_{\lambda}^{P} + 1 < d_{\mu}^{P}$. In particular, Proposition~\ref{p21/3ext} \textbf{(i)} and \textbf{(iii)} and Proposition~\ref{p21/3extadd} \textbf{(i)} imply that $P'$ is strongly $r$-regular.  If $c_{p-1} = 1$, then the hypothesis on $P'\setminus P$ implies that $P$ is $\mu$ $r$-regular. If $P$ is $\mu$ $r$-regular and $\Xi\left(P\right) = 0$, then Proposition~\ref{p21/3ext} \textbf{(i)} implies that $P'$ is strongly $r$-regular.
\item[(iii)] The hypothesis on $P'\setminus P$ implies that either $a_{p}>1$ or $b_{p}>1$. If $\Xi\left(P\right)>1$ and $a_{p} = 1$, then Proposition~\ref{p21/3bothextadd} \textbf{(a)} \textbf{(i)} - \textbf{(ii)} imply that $P'$ is strongly $r$-regular. If $\Xi\left(P\right)>1$ and $a_{p}>1$, then Proposition~\ref{p21/3bothextadd} \textbf{(b)} \textbf{(i)} and \textbf{(iv)} imply that $P'$ is $r$-regular. 
\end{description}
\item[(b)] 
\begin{description}
\item[(i)] The hypothesis on $P'\setminus P$ implies that $c_{p}>2$. If $\Xi\left(P\right)>1$, then Proposition~\ref{p1/32extadd} \textbf{(a)} \textbf{(i)} implies that $P'$ is strongly $r$-regular. If $\Xi\left(P\right) = 0$, then Proposition~\ref{p1/32ext} \textbf{(a)} \textbf{(i)} implies that $P'$ is strongly $r$-regular.
\item[(ii)] The hypothesis on $P'\setminus P$ implies that $c_{p}\geq 2$. If $\Xi\left(P\right)> 2$, then Proposition~\ref{p1/32extadd} \textbf{(b)} \textbf{(i)} implies that $P'$ is strongly $r$-regular. If $0\leq \Xi\left(P\right) < 2$, then Proposition~\ref{p1/32ext} \textbf{(a)} \textbf{(i)} implies that $P'$ is strongly $r$-regular.
\item[(iii)] The hypothesis that $P'\setminus P$ is a subroad implies that $c_{p}>2$. Let us consider the case where $a_p = 1$ and $P$ is $3$-switch $r$-regular. If $\Xi\left(P\right) = 0$, then Proposition~\ref{p1/32bothext} \textbf{(b)} \textbf{(i)} implies that $P'$ is strongly $r$-regular. If $\Xi\left(P\right) > 1$, then Proposition~\ref{p1/32bothextadd} \textbf{(b)} \textbf{(ii)} implies that $P'$ is strongly $r$-regular. 

Let us consider the case where $a_p = 1$ and $P$ is $1$-switch $r$-regular. If $\Xi\left(P\right)>0$, then Proposition~\ref{p1/32bothextadd} \textbf{(c)} \textbf{(ii)} implies that $P'$ is strongly $r$-regular.

If $a_p>1$, then Proposition~\ref{pregularap>1} implies that $P'$ is strongly $r$-regular.
\end{description}
\end{description}
\end{proof}

We now consider the small values of $\Xi\left(P\right)$ that are excluded in the conditions of Lemma~\ref{luniqueextensionimpliesbadsubroad}. We split into two statements corresponding to whether $P$ is a $2$-subproduct of $\sigma$ or an $s$-subproduct of $\sigma$ for some $s\in \{1,3\}$. Firstly, we consider the case where $P$ is a $2$-subproduct of $\sigma$.

\begin{lemma}
\label{l2subdisc01}
Let $P$ be an $r$-regular $2$-subproduct of $\sigma$ such that $\Xi\left(P\right)\in \{0,1\}$. Let $P'$ be the minimal $s'$-subproduct of $\sigma$ for $s'\in \{1,3\}$ such that $P\subseteq P'$. Let $t\neq 2,3$ be a prime number. Let us consider one of the following conditions on $P$. 
\begin{description}
\item[(a)] Let us consider the case where either $a_p = 0$ or $b_p = 0$. If $a_p = 0$, then we define $\phi_p = b_p$, $i = 3$, and $i' = 1$. If $b_p = 0$, then we define $\phi_p = a_p$, $i = 1$, and $i' = 3$.
\begin{description}
\item[(i)] The subproduct $P$ is positive $i$-switch $\mu$ $r$-regular and $\Xi\left(P\right) = 1$.
\item[(ii)] The subproduct $P$ is positive $i'$-switch $\mu$ $r$-regular and $\Xi\left(P\right) = 0$.
\end{description}
\item[(b)] Let us consider the case where $a_p>1$ and $b_p>0$. 
\begin{description}
\item[(i)] The subproduct $P$ is positive $r$-regular.
\item[(ii)] The subproduct $P$ is negative $1$-switch $r$-regular. Furthermore, either $\Xi\left(P\right) = 0$ and $a_p>b_p - 1$, or $\Xi\left(P\right) = 1$ and $b_p>a_p - 1$.
\item[(iii)] The subproduct $P$ is negative $3$-switch $r$-regular. Furthermore, either $\Xi\left(P\right) = 0$ and $b_p>a_p - 1$, or $\Xi\left(P\right) = 1$ and $a_p>b_p - 2$.  
\end{description}
\item[(c)] Let us consider the case where $a_p = 1$ and $b_p > 1$. The subproduct $P$ is positive $i$-switch $r$-regular.
\end{description}
If one of the above conditions is satisfied, then the leading coefficients of multiple entries in the $r$th row of $\beta_4\left(P'\right)$ are non-zero modulo $t$. If one of the above conditions is satisfied, and we also require $\phi_p>2$ in \textbf{(a)} \textbf{(i)}, and $P'\neq \sigma$, then $P'$ is strongly $r$-regular.
\end{lemma}
\begin{proof}
The proof that the leading coefficients of multiple entries in the $r$th row of $\beta_4\left(P'\right)$ are non-zero modulo $t$ is based on the theory of path extensions in Subsection~\ref{subsecpathextensions} and Subsection~\ref{subsecpathextensionsII}. The proof is analogous to the proof of Corollary~\ref{lemmaroadregular} in Section~\ref{mains}. 

Let us establish the statements concerning the strong $r$-regularity of $P'$.
\begin{description}
\item[(a)] 
\begin{description}
\item[(i)] Proposition~\ref{p21/3extadd} \textbf{(i)} implies that $P'$ is strongly $r$-regular.
\item[(ii)] Proposition~\ref{pnotationentries} \textbf{(a)} \textbf{(iii)} - \textbf{(iv)} imply that $d_{\lambda}^{P}>d_{\mu}^{P}$. Proposition~\ref{p21/3ext} \textbf{(i)} implies that $P'$ is strongly $r$-regular.
\end{description}
\item[(b)] 
\begin{description}
\item[(i)] Let $P$ be an $i$-switch $r$-regular subproduct. If $i = 1$ and $\Xi\left(P\right) = 0$, then Proposition~\ref{p21/3bothextadd} \textbf{(b)} \textbf{(vi)} implies that $P'$ is strongly $r$-regular. If $i = 1$ and $\Xi\left(P\right) = 1$, then Proposition~\ref{p21/3bothextadd} \textbf{(b)} \textbf{(v)} implies that $P'$ is strongly $r$-regular. If $i = 3$ and $\Xi\left(P\right) = 0$, then Proposition~\ref{p21/3bothextadd} \textbf{(b)} \textbf{(iii)} implies that $P'$ is strongly $r$-regular. If $i = 3$ and $\Xi\left(P\right) = 1$, then Proposition~\ref{p21/3bothextadd} \textbf{(b)} \textbf{(ii)} implies that $P'$ is strongly $r$-regular.
\item[(ii)] We apply Proposition~\ref{ppre21/3bothext} \textbf{(i)} - \textbf{(ii)} and \textbf{(iv)} and Proposition~\ref{ppre21/3bothextadd} \textbf{(i)} to conclude the following.

If $\Xi\left(P\right) = 0$, then $d_{\lambda,1}^{P'} = d_{\lambda}^{P} + a_p$, $d_{\mu,1}^{P'}<d_{\lambda}^{P} + 1$, $d_{\lambda,3}^{P'} < d_{\lambda}^{P} + b_p - 1$, and $\lambda_{1,0}^{P'} = -\lambda_{0}^{P}$. If $a_p>b_p-1$, then we deduce that $P'$ is $r$-regular.

If $\Xi\left(P\right) = 1$, then $d_{\mu,1}^{P'} = d_{\mu,1}^{P} + 1$, $d_{\lambda,3}^{P'} = d_{\mu,1}^{P} + b_p - 1$, and $\lambda_{3,0}^{P'} = \mu_{1,0}^{P}$. If $a_p>2$, then $d_{\lambda,1}^{P'}<d_{\lambda}^{P} + a_p$, and if $a_p = 2$, then $d_{\lambda,1}^{P'} = d_{\lambda}^{P} + a_p$. If $b_p>a_p-1$, then we deduce that $P'$ is $r$-regular. 
\item[(iii)] We apply Proposition~\ref{ppre21/3bothext} \textbf{(i)} - \textbf{(ii)} and \textbf{(iv)} and Proposition~\ref{ppre21/3bothextadd} \textbf{(ii)} to conclude the following.

If $\Xi\left(P\right) = 0$, then $d_{\lambda,1}^{P'} < d_{\lambda}^{P} + a_p$, $d_{\mu,1}^{P'} < d_{\lambda}^{P} + 1$, $d_{\lambda,3}^{P'} = d_{\lambda}^{P} + b_p - 1$, and $\lambda_{3,0}^{P} = \lambda_{0}^{P}$. If $b_p>a_p - 1$, then we deduce that $P'$ is $r$-regular.

If $\Xi\left(P\right) = 1$, then $d_{\lambda,1}^{P'} = d_{\mu,3}^{P} + a_p$ and $d_{\mu,1}^{P'} = d_{\mu,3}^{P} + 1$. If $b_p>1$, then $d_{\lambda,3}^{P'}<d_{\lambda}^{P} + b_p - 1$, and if $b_p = 1$, then $d_{\lambda,3}^{P'} = d_{\lambda}^{P}$. If $a_p>b_p - 2$, then we deduce that $P'$ is $r$-regular.
\end{description}
\item[(c)] If $i = 1$ and $\Xi\left(P\right)\in \{0,1\}$, then Proposition~\ref{p21/3bothextadd} \textbf{(a)} \textbf{(i)} implies that $P'$ is strongly $r$-regular. If $i = 3$ and $\Xi\left(P\right)\in \{0,1\}$, then Proposition~\ref{p21/3bothextadd} \textbf{(a)} \textbf{(ii)} implies that $P'$ is strongly $r$-regular.
\end{description}
\end{proof}

We now establish an analogue of Lemma~\ref{l2subdisc01} in the case where $P$ is an $s$-subproduct of $\sigma$ for $s\in \{1,3\}$.

\begin{lemma}
\label{l1/3subdisc01}
Let $P$ be an $r$-regular $s$-subproduct of $\sigma$ for $s\in \{1,3\}$. Let $P'$ be the minimal $2$-subproduct of $\sigma$ such that $P\subseteq P'$. Let $t\neq 2,3$ be a prime number. Let us consider one of the following conditions on $P$.
\begin{description}
\item[(a)] Let us consider the case where either $a_p=0$ or $b_p = 0$.
\begin{description}
\item[(i)] The subproduct $P$ is positive $\mu$ $r$-regular and $\Xi\left(P\right) = 1$. 
\item[(ii)] The subproduct $P$ is positive $\nu$ $r$-regular and $\Xi\left(P\right) = 2$.
\end{description}
\item[(b)] Let us consider the case where either $a_p = 1$ and $b_p>0$.
\begin{description}
\item[(i)] The subproduct $P$ is positive $1$-switch $r$-regular and $\Xi\left(P\right) = 0$.
\item[(ii)] The subproduct $P$ is positive $3$-switch $r$-regular and $\Xi\left(P\right) = 1$. 
\end{description}
\end{description}
If one of the above conditions is satisfied, then the leading coefficients of multiple entries in the $r$th row of $\beta_4\left(P'\right)$ are non-zero modulo $t$. If one of the above conditions is satisfied, and we also require $c_p>1$ in \textbf{(a)} \textbf{(ii)} and $c_p>2$ in \textbf{(a)} \textbf{(i)} and \textbf{(b)} \textbf{(i)} - \textbf{(ii)}, then $P'$ is strongly $r$-regular.
\end{lemma}
\begin{proof}
The proof that the leading coefficients of multiple entries in the $r$th row of $\beta_4\left(P'\right)$ are non-zero modulo $t$ is based on the theory of path extensions in Subsection~\ref{subsecpathextensions} and Subsection~\ref{subsecpathextensionsII}. The proof is analogous to the proof of Corollary~\ref{lemmaroadregular} in Section~\ref{mains}.

Let us establish the statements concerning the strong $r$-regularity of $P'$.
\begin{description}
\item[(a)]  
\begin{description}
\item[(i)] Proposition~\ref{p1/32extadd} \textbf{(a)} \textbf{(i)} implies that $P'$ is strongly $r$-regular.
\item[(ii)] Proposition~\ref{p1/32extadd} \textbf{(c)} implies that $P'$ is strongly $r$-regular.
\end{description}
\item[(b)] 
\begin{description}
\item[(i)] Proposition~\ref{p1/32bothextadd} \textbf{(c)} \textbf{(i)} implies that $P'$ is strongly $r$-regular.
\item[(ii)] Proposition~\ref{p1/32bothextadd} \textbf{(b)} \textbf{(i)} implies that $P'$ is strongly $r$-regular.
\end{description}
\end{description}
\end{proof}

\subsubsection{Regularity is generically an inductive property around blocks}

If $P$ is the $s$-subproduct immediately preceding a block $B$, and if $P' = PB$, then the next step is to characterize extensions of $P$-paths to $P'$-paths, when $P$ is not necessarily strongly regular or the block $B$ is not necessarily a normal block. 

Firstly, we consider the case of generic $2$-blocks and we establish the following addendum to Lemma~\ref{l2blockstronglyregular}.

\begin{lemma}
\label{lemma2blockweaklyregular}
Let $B = \sigma_1^{a_p}\sigma_3^{b_p}\sigma_2\sigma_1^{a_{p+1}}$ be a generic $2$-block and let $P = \prod_{i=1}^{p-1} \sigma_1^{a_i}\sigma_3^{b_i}\sigma_2^{c_i}$ be the $2$-subproduct immediately preceding $B$. Let $P' = PB$. We assume that $d_{\lambda}^{P}>\min\{d_{\mu,1}^{P},d_{\mu,3}^{P}\}$. 

\begin{description}
\item[(a)] Let us consider the case $a_{p+1} = 2$ and $a_p\geq b_p + 1$. 

Let us assume that $d_{\lambda}^{P}>\max\{d_{\mu,1}^{P},d_{\mu,3}^{P}\}$. In this case, $d_{\lambda}^{P'} = d_{\lambda}^{P} + b_p + 2\leq  d_{\lambda}^{P} + a_p + 1 = d_{\mu}^{P'}$, $d_{\nu}^{P'}+3\leq d_{\lambda}^{P'}$, and $\lambda_{0}^{P'} = \lambda_{0}^{P} = \mu_{0}^{P'}$.
\item[(b)] 
\begin{description}
\item[(i)] Let us consider the case $a_{p+1} = 2$ and $a_p\geq b_p + 2$. 

Let us assume that $d_{\lambda}^{P} + 1 <d_{\mu,1}^{P}$. In this case, $d_{\lambda}^{P'} = d_{\mu,1}^{P} + b_p + 2\leq d_{\mu,1}^{P} + a_p = d_{\mu}^{P'}$, $d_{\nu}^{P'}+3\leq d_{\lambda}^{P'}$, and $\lambda_{0}^{P'} = \mu_{1,0}^{P} = \mu_{0}^{P'}$.
\item[(ii)] Let us consider the case $a_{p+1} = 2$ and $a_p < b_p + 2$.

Let us assume that $d_{\lambda}^{P} +1 < d_{\mu,1}^{P}$. In this case, $d_{\lambda}^{P'}>\max\{d_{\mu}^{P'},d_{\nu}^{P'} + 2\}$ and $\lambda_{0}^{P'} = \mu_{1,0}^{P}$. 
\end{description}
\item[(c)] 
\begin{description}
\item[(i)] Let us consider the case $a_{p+1} = 2$ and $a_p\geq b_p>2$. 

Let us assume that $d_{\lambda}^{P} + 1< d_{\mu,3}^{P}$. In this case, $d_{\lambda}^{P'} = d_{\mu,3}^{P} + b_p+1\leq d_{\mu,3}^{P} + a_p + 1 = d_{\mu}^{P'}$, $d_{\nu}^{P'}+3\leq d_{\lambda}^{P'}$, and $\lambda_{0}^{P'} = \mu_{3,0}^{P} = \mu_{0}^{P'}$.
\item[(ii)] Let us consider the case $a_{p+1} = 2$ and $a_p < b_p > 2$. 

Let us assume that $d_{\lambda}^{P} + 1 < d_{\mu,3}^{P}$. In this case, $d_{\lambda}^{P'} > \max\{d_{\mu}^{P'},d_{\nu}^{P'} + 2\}$ and $\lambda_{0}^{P'} = \mu_{3,0}^{P}$. 
\item[(iii)] Let us consider the case $a_{p+1} = 2$ and $b_p = 2$.

Let us assume that $d_{\lambda}^{P} + 1 < d_{\mu,3}^{P}$. In this case, $d_{\lambda}^{P'} = d_{\lambda}^{P} + 4$, $d_{\mu}^{P'} = d_{\mu,3}^{P} + a_p+1$, $d_{\nu}^{P'} = d_{\mu,3}^{P}$, $\lambda_{0}^{P'} = \lambda_{0}^{P}$, and $\mu_{0}^{P'} = \mu_{3,0}^{P} = \nu_{0}^{P'}$.
\end{description}
\item[(d)] Let us consider the case $a_{p+1} > 2$. 
\begin{description}
\item[(i)] Let us assume that $d_{\lambda}^{P} + 1 < d_{\mu,1}^{P}$. In this case, $d_{\lambda}^{P'}>\max\{d_{\mu}^{P'},d_{\nu}^{P'} + 2\}$. If $a_p\neq b_p + 2$, then $\lambda_{0}^{P'} = \mu_{1,0}^{P}$. If $a_p = b_p + 2$, then $\lambda_{0}^{P'} = 2\mu_{1,0}^{P}$. 
\item[(ii)] Let us assume that $d_{\lambda}^{P} + 1 < d_{\mu,3}^{P}$. In this case, $d_{\lambda}^{P'} > \max\{d_{\mu}^{P'},d_{\nu}^{P'}+2\}$. If either $a_p\neq b_p$ or $b_p = 2$, then $\lambda_{0}^{P'} = \mu_{3,0}^{P}$. If $a_p = b_p$ and $b_p>2$, then $\lambda_{0}^{P'} = 2\mu_{3,0}^{P}$.
\end{description}
\end{description}
Finally, in each of the above cases, if the leading coefficient of the $\left(r,2\right)$-entry of $\beta_4\left(P\right)$ is non-zero modulo a prime $t\neq 2,3$, then the leading coefficients of multiple entries in the $r$th row of $\beta_4\left(P'\right)$ are non-zero modulo $t$. 
\end{lemma}
\begin{proof}
Firstly, we apply Proposition~\ref{lemma2blockextadmissible} \textbf{(i)} in each case.
\begin{description}
\item[(a)] The statement is a consequence of Proposition~\ref{lemma2blockextadmissible} \textbf{(ii)} and \textbf{(iv)}. 
\item[(b)] The statements are consequences of Proposition~\ref{lemma2blockextadmissible} \textbf{(ii)} and Proposition~\ref{lemma2blockextadmissibleextra} \textbf{(i)}.
\item[(c)] The statements are consequences of Proposition~\ref{lemma2blockextadmissible} \textbf{(ii)} and Proposition~\ref{lemma2blockextadmissibleextra} \textbf{(ii)}.
\item[(d)]
\begin{description}
\item[(i)] If $a_p>b_p + 2$, then Proposition~\ref{lemma2blockextadmissible} \textbf{(ii)} and Proposition~\ref{lemma2blockextadmissibleextra} \textbf{(i)} imply that a maximal $q$-resistance $\left(r,1\right)$-type $P'$-path is of the form $Y^{2,0,2}$, where $Y$ is a maximal $q$-resistance $1$-switch $\left(r,2\right)$-type $P$-path. If $a_p<b_p + 2$, then Proposition~\ref{lemma2blockextadmissible} \textbf{(ii)} and Proposition~\ref{lemma2blockextadmissibleextra} \textbf{(i)} imply that a maximal $q$-resistance $\left(r,1\right)$-type $P'$-path is of the form $Y^{0,1,1}$, where $Y$ is a maximal $q$-resistance $1$-switch $\left(r,2\right)$-type $P$-path. If $a_p = b_p + 2$, then Proposition~\ref{lemma2blockextadmissible} \textbf{(ii)} and Proposition~\ref{lemma2blockextadmissibleextra} \textbf{(i)} imply that a maximal $q$-resistance $\left(r,1\right)$-type $P'$-path is either of the form $Y^{2,0,2}$ or of the form $Y^{0,1,1}$, where $Y$ is a maximal $q$-resistance $1$-switch $\left(r,2\right)$-type $P$-path. The hypothesis that $a_{p+1}>2$ implies that a maximal $q$-resistance $\left(r,1\right)$-type $P'$-path is good in each case.
\item[(ii)] If $b_p>2$, then the proof is similar to the proof of \textbf{(d)} \textbf{(i)}. The main difference is that each instance of $Y$ is replaced by $V$, where $V$ is a maximal $q$-resistance $3$-switch $\left(r,2\right)$-type $P$-path. We also apply Proposition~\ref{lemma2blockextadmissible} \textbf{(ii)} and Proposition~\ref{lemma2blockextadmissibleextra} \textbf{(ii)}. 

If $b_p = 2$ and $a_p>1$, then Proposition~\ref{lemma2blockextadmissible} \textbf{(ii)} and \textbf{(iii)} imply that a maximal $q$-resistance $\left(r,1\right)$-type $P'$-path is of the form $V^{1,0,2}$, where $V$ is a maximal $q$-resistance $3$-switch $\left(r,2\right)$-type $P$-path. If $b_p = 2$ and $a_p = 1$, then Proposition~\ref{lemma2blockextadmissible} \textbf{(ii)} and \textbf{(iii)} imply that a maximal $q$-resistance $\left(r,1\right)$-type $P'$-path is of the form $V^{0,2,2}$, where $V$ is a maximal $q$-resistance $3$-switch $\left(r,2\right)$-type $P$-path. The hypothesis that $a_{p+1}>2$ implies that a maximal $q$-resistance $\left(r,1\right)$-type $P'$-path is good in each case.
\end{description}
\end{description}
Finally, the reasoning also implies that the leading coefficients of the $\left(r,1\right)$-entry and the $\left(r,3\right)$-entry of $\beta_4\left(P'\right)$ are non-zero modulo $t$ in each case.
\end{proof}

We observe that regularity is an inductive property around generic $2$-blocks in all parts of Lemma~\ref{lemma2blockweaklyregular}, except possibly in \textbf{(c)} \textbf{(iii)}. We will address this case later. We now consider the case of generic $3$-blocks and we establish the following addendum to Lemma~\ref{l3blockstronglyregular}.

\begin{lemma}
\label{lemma3blockweaklyregular}
Let $B = \sigma_2^{c_{p-1}}\sigma_3\sigma_2^{c_p}$ be a generic $3$-block in $\sigma$ and let $P = \left(\prod_{i=1}^{p-2} \sigma_1^{a_i}\sigma_3^{b_i}\sigma_2^{c_i}\right)\sigma_1^{a_{p-1}}$ be the $s$-subproduct immediately preceding $B$ for $s\in \{1,3\}$. Let $P' = PB$. We assume that $d_{\nu}^{P} + 3\leq d_{\lambda}^{P}$.

\begin{description}
\item[(a)] Let us consider the case $c_{p-1} > 1$ and $c_p=2$.

Let us assume that $d_{\lambda}^{P}>d_{\mu}^{P}$. In this case, $d_{\lambda}^{P'} = d_{\lambda}^{P} + 1 < d_{\lambda}^{P} + c_{p-1} = d_{\mu,3}^{P'}$, $d_{\mu,1}^{P'} = d_{\lambda}^{P}$, $\lambda_{0}^{P'} = \lambda_{0}^{P}$, and $\mu_{3,0}^{P'} = -\lambda_{0}^{P}$.
\item[(b)] Let us consider the case $c_{p-1} > 2$ and $c_p = 2$. 

Let us assume that $d_{\lambda}^{P} + 1 < d_{\mu}^{P}$. In this case, $d_{\lambda}^{P'} = d_{\mu}^{P} + 1 < d_{\mu}^{P} + c_{p-1} - 1 = d_{\mu,3}^{P'}$, $d_{\mu,1}^{P'} = d_{\mu}^{P}$, $\lambda_{0}^{P'} = \mu_{0}^{P}$, and $\mu_{3,0}^{P'} = -\mu_{0}^{P}$.
\item[(c)] Let us consider the case $c_{p-1} > 2$ and $c_p>2$.

Let us assume that $d_{\lambda}^{P} + 1 < d_{\mu}^{P}$. In this case, $d_{\lambda}^{P'}>\max\{d_{\mu,1}^{P'},d_{\mu,3}^{P'}\}$ and $\lambda_{0}^{P'} = -\mu_{0}^{P}$.
\item[(d)] Let us consider the case $c_{p-1} = 1$ and $c_p > 1$.
\begin{description}
\item[(i)] Let us assume that $d_{\lambda}^{P} + 1< d_{\mu}^{P} + c_p - 1$. In this case, $d_{\lambda}^{P'}>\max\{d_{\mu,1}^{P'},d_{\mu,3}^{P'}\}$ and $\lambda_{0}^{P'} = \mu_{0}^{P}$.
\item[(ii)] Let us assume that $d_{\lambda}^{P} + 1 > d_{\mu}^{P} + c_p - 1$. In this case, $d_{\lambda}^{P'} = d_{\lambda}^{P} + 1$, $d_{\mu,3}^{P'} = d_{\lambda}^{P} + 1$, $d_{\mu,1}^{P'} = d_{\lambda}^{P}$, $\lambda_{0}^{P'} = \lambda_{0}^{P}$, $\mu_{3,0}^{P'} = -\lambda_{0}^{P}$, and $\mu_{1,0}^{P'} = \lambda_{0}^{P}$. Furthermore, $d_{\lambda,1}^{P'} = d_{\mu}^{P} + c_p - 1\leq d_{\lambda}^{P} = d_{\mu,1}^{P'}$ and $\lambda_{1,0}^{P'} = \mu_{0}^{P}$ (see Notation~\ref{notnumgb}). 
\end{description}
\item[(e)] Let us consider the case $c_{p-1} = 2$ and $c_p>2$.
\begin{description}
\item[(i)] Let us assume that $d_{\lambda}^{P} + c_p > d_{\mu}^{P} + 1$. In this case, $d_{\lambda}^{P'}>\max\{d_{\mu,1}^{P'},d_{\mu,3}^{P'}\}$ and $\lambda_{0}^{P'} = \lambda_{0}^{P}$.
\item[(ii)] Let us assume that $d_{\lambda}^{P} + c_p < d_{\mu}^{P} + 1$. In this case, $d_{\lambda}^{P'} = d_{\mu}^{P} + 1$, $d_{\mu,3}^{P'} = d_{\mu}^{P} + 1$, $d_{\mu,1}^{P'} = d_{\mu}^{P}$, $\lambda_{0}^{P'} = \mu_{0}^{P}$, $\mu_{3,0}^{P'} = -\mu_{0}^{P}$, and $\mu_{1,0}^{P'} = \mu_{0}^{P}$. Furthermore, $d_{\lambda,1}^{P'} = d_{\lambda}^{P} + c_p\leq d_{\mu}^{P} = d_{\mu,1}^{P'}$ and $\lambda_{1,0}^{P'} = -\lambda_{0}^{P}$.
\end{description}
\item[(f)] Let us consider the case $c_{p-1} = 2$ and $c_p = 2$.

Let us assume that $d_{\lambda}^{P} + 1 < d_{\mu}^{P}$. In this case, $d_{\lambda}^{P'} = d_{\mu}^{P} + 1$, $d_{\mu,3}^{P'} = d_{\mu}^{P} + 1$, $d_{\mu,1}^{P'} = d_{\mu}^{P}$, $\lambda_{0}^{P'} = \mu_{0}^{P}$, $\mu_{3,0}^{P'} = -\mu_{0}^{P}$, and $\mu_{1,0}^{P'} = \mu_{0}^{P}$. Furthermore, $d_{\lambda,1}^{P'} = d_{\lambda}^{P} + 2\leq d_{\mu}^{P} = d_{\mu,1}^{P'}$ and $\lambda_{1,0}^{P'} = -\lambda_{0}^{P}$.
\end{description}
Finally, in each of the above cases, if the leading coefficient of the $\left(r,1\right)$-entry of $\beta_4\left(P\right)$ is non-zero modulo a prime $t\neq 2,3$, then the leading coefficients of multiple entries in the $r$th row of $\beta_4\left(P'\right)$ are non-zero modulo $t$. 
\end{lemma}
\begin{proof}
Firstly, Proposition~\ref{lemma3blockextadmissible} and the hypothesis $d_{\nu}^{P} + 3\leq d_{\lambda}^{P}$ implies in each case that a maximal $q$-resistance good $P'$-path, a maximal $q$-resistance $3$-switch $P'$-path, and a maximal $q$-resistance $1$-switch $P'$-path are of the form $W^{\alpha,\beta}$ or $W^{\gamma}$, where $W$ is an $\left(r,1\right)$-type $P$-path. 

\begin{description}
\item[(a)] Proposition~\ref{lemma3blockextadmissible} \textbf{(b)} implies that a maximal $q$-resistance good $\left(r,2\right)$-type $P'$-path is of the form $X^{1}$, a maximal $q$-resistance $3$-switch $\left(r,2\right)$-type $P'$-path is of the form $X^{1,2}$, and a maximal $q$-resistance $1$-switch $\left(r,2\right)$-type $P'$-path is of the form $X^{2}$, where $X$ is a maximal $q$-resistance good $\left(r,1\right)$-type $P$-path. 
\item[(b)] Proposition~\ref{lemma3blockextadmissible} \textbf{(b)} implies that a maximal $q$-resistance good $\left(r,2\right)$-type $P'$-path is of the form $Y^{1}$, a maximal $q$-resistance $3$-switch $\left(r,2\right)$-type $P'$-path is of the form $Y^{2,2}$, and a maximal $q$-resistance $1$-switch $\left(r,2\right)$-type $P'$-path is of the form $Y^{2}$, where $Y$ is a maximal $q$-resistance bad $\left(r,1\right)$-type $P$-path.
\item[(c)] Proposition~\ref{lemma3blockextadmissible} \textbf{(b)} implies that a maximal $q$-resistance good $\left(r,2\right)$-type $P'$-path is of the form $Y^{2,2}$, a maximal $q$-resistance $3$-switch $\left(r,2\right)$-type $P'$-path is of the form $Y^{2,c_p}$, and a maximal $q$-resistance $1$-switch $\left(r,2\right)$-type $P'$-path is of the form $Y^{c_p}$, where $Y$ is a maximal $q$-resistance bad $\left(r,1\right)$-type $P$-path.
\item[(d)] 
\begin{description}
\item[(i)] Proposition~\ref{lemma3blockextadmissible} \textbf{(b)} implies that a (non-cancelling) maximal $q$-resistance $\left(r,2\right)$-type $P'$-path is of the form $Y^{1}$, where $Y$ is a maximal $q$-resistance bad $\left(r,1\right)$-type $P$-path. The $P'$-path $Y^{1}$ is good since $c_p>1$.
\item[(ii)] Proposition~\ref{lemma3blockextadmissible} \textbf{(b)} implies that the good (non-cancelling) $\left(r,2\right)$-type $P'$-paths are either of the form $W^{c_p - 1}$, where $W$ is a good $\left(r,1\right)$-type $P$-path, of the form $W^{\gamma}$ for $1\leq \gamma < c_p$, where $W$ is a bad $\left(r,1\right)$-type $P$-path, or of the form $W^{1,0}$, where $W$ is an $\left(r,3\right)$-type $P$-path. Proposition~\ref{lemma3blockextadmissible} \textbf{(b)} also implies that the $3$-switch $\left(r,2\right)$-type $P'$-paths are of the form $W^{1,c_p}$, where $W$ is a good $\left(r,1\right)$-type $P$-path. We deduce $w_{\lambda}^{P'} = qw_{\lambda}^{P} + \left(\sum_{i=1}^{c_p-1} q^i\right)w_{\mu}^{P} -q^{c_p+1}w_{\nu}^{P}$ and $w_{\mu,3}^{P'} = -qw_{\lambda}^{P}$ and the equalities concerning the degrees and leading coefficients of these polynomials follow. Furthermore, $w_{\lambda,1}^{P'} = w_{\lambda}^{P'} + w_{\mu,3}^{P'}$ (Notation~\ref{notnumgb}) and the equalities concerning the degree and leading coefficient of this polynomial follow too.

Finally, Proposition~\ref{lemma3blockextadmissible} \textbf{(b)} implies that a maximal $q$-resistance $1$-switch $\left(r,2\right)$-type $P'$-path is of the form $X^{c_p}$, where $X$ is a maximal $q$-resistance good $\left(r,1\right)$-type $P$-path. 
\end{description}
\item[(e)] 
\begin{description}
\item[(i)] Proposition~\ref{lemma3blockextadmissible} \textbf{(b)} implies that a (non-cancelling) maximal $q$-resistance $\left(r,2\right)$-type $P'$-path is of the form $X^{1,2}$, where $X$ is a maximal $q$-resistance good $\left(r,1\right)$-type $P$-path. The $P'$-path $X^{1,2}$ is good since $c_p>2$.
\item[(ii)] Proposition~\ref{lemma3blockextadmissible} \textbf{(b)} implies that the good (non-cancelling) $\left(r,2\right)$-type $P'$-paths are either of the form $W^{1,\beta}$ for $2\leq \beta < c_p$, where $W$ is either a good $\left(r,1\right)$-type $P$-path or an $\left(r,3\right)$-type $P$-path, of the form $W^{c_p-1}$, where $W$ is an $\left(r,1\right)$-type $P$-path, or of the form $W^{2,0}$, where $W$ is an $\left(r,3\right)$-type $P$-path. Proposition~\ref{lemma3blockextadmissible} \textbf{(b)} also implies that the $3$-switch $\left(r,2\right)$-type $P'$-paths are either of the form $W^{1,c_p}$, where $W$ is either a good $\left(r,1\right)$-type $P$-path or an $\left(r,3\right)$-type $P$-path, or of the form $W^{2,c_p}$, where $W$ is an $\left(r,1\right)$-type $P$-path. We deduce $w_{\lambda}^{P'} = qw_{\mu}^{P} + \left(q - \sum_{i=3}^{c_p} q^i\right)w_{\lambda}^{P} + \left(\sum_{i=4}^{c_p} q^i\right)w_{\nu}^{P}$ and $w_{\mu,3}^{P'} = -qw_{\mu}^{P} + \left(-q-q^2\right)w_{\lambda}^{P} + q^{3}w_{\nu}^{P}$ and the equalities concerning the degrees and leading coefficients of these polynomials follow. Furthermore, $w_{\lambda,1}^{P'} = w_{\lambda}^{P'} + w_{\mu,3}^{P'}$ (Notation~\ref{notnumgb}) and the equalities concerning the degree and leading coefficient of this polynomial follow too. 

Finally, Proposition~\ref{lemma3blockextadmissible} \textbf{(b)} implies that a maximal $q$-resistance $1$-switch $\left(r,2\right)$-type $P'$-path is of the form $Y^{c_p}$, where $Y$ is a maximal $q$-resistance bad $\left(r,1\right)$-type $P$-path. 
\end{description}
\item[(f)] Proposition~\ref{lemma3blockextadmissible} \textbf{(b)} implies that a maximal $q$-resistance good $\left(r,2\right)$-type $P'$-path is of the form $Y^{1}$, a maximal $q$-resistance $3$-switch $\left(r,2\right)$-type $P'$-path is of the form $Y^{2,2}$, and a maximal $q$-resistance $1$-switch $\left(r,2\right)$-type $P'$-path is of the form $Y^{2}$, where $Y$ is a maximal $q$-resistance bad $\left(r,1\right)$-type $P$-path. Proposition~\ref{lemma3blockextadmissible} \textbf{(c)} implies that the weights of the $\left(r,2\right)$-type $P'$-paths $W^{2,0}$ and $W^{1,2}$ have opposite sign, where $W$ is an $\left(r,3\right)$-type $P$-path. Proposition~\ref{lemma3blockextadmissible} also implies that the good or $3$-switch $\left(r,2\right)$-type $P'$-paths are either of the form $W^{2,2}$ or $W^{1}$ (which have weights of opposite sign), where $W$ is an $\left(r,1\right)$-type $P$-path, of the form $W^{1,2}$, where $W$ is a good $\left(r,1\right)$-type $P$-path, or of the form $W^{2,0}$ or $W^{1,2}$ (which have weights of opposite sign), where $W$ is an $\left(r,3\right)$-type $P$-path. We deduce $w_{\lambda,1}^{P'} = w_{\lambda}^{P'} + w_{\mu,3}^{P'} = -q^2w_{\lambda}^{P}$. The equalities concerning the degree and leading coefficient of this polynomial follow.
\end{description}
Finally, Lemma~\ref{lemmagoodext} \textbf{(iv)} implies that the $\left(r,1\right)$-entry of $\beta_4\left(P'\right)$ is equal to the $\left(r,1\right)$-entry of $\beta_4\left(P\right)$. Proposition~\ref{pnotationentries} \textbf{(a)} \textbf{(i)} implies that the leading coefficient of the $\left(r,2\right)$-entry of $\beta_4\left(P'\right)$ is non-zero modulo $t$ in \textbf{(a)}, \textbf{(b)}, \textbf{(c)}, \textbf{(d)} \textbf{(i)}, and \textbf{(e)} \textbf{(i)}. Proposition~\ref{pnotationentries} \textbf{(a)} \textbf{(iv)} implies that the leading coefficient of the $\left(r,3\right)$-entry of $\beta_4\left(P'\right)$ is non-zero modulo $t$ in \textbf{(d)} \textbf{(ii)}, \textbf{(e)} \textbf{(ii)}, and \textbf{(f)}. 
\end{proof}

We observe that regularity is an inductive property around generic $3$-blocks in all parts of Lemma~\ref{lemma3blockweaklyregular}. However, in Lemma~\ref{lemma3blockweaklyregular} \textbf{(d)} \textbf{(ii)}, \textbf{(e)} \textbf{(ii)} and \textbf{(f)}, $P'$ is very close to not being $r$-regular in the sense that $P'$ is negative $3$-switch $r$-regular and $\Xi\left(P'\right) = 0$. In this case, neither the condition \textbf{(a)} \textbf{(i)} in Lemma~\ref{luniqueextensionimpliesbadsubroad} nor the condition \textbf{(a)} \textbf{(ii)} in Lemma~\ref{l2subdisc01} are satisfied if $b_{p+1} = 0$. In particular, if $P''$ is the minimal $s''$-subproduct containing $P'$ for $s''\in \{1,3\}$, then it is not necessarily true that $P''$ is $r$-regular. 

However, the next step is to show that there is generically a strongly $r$-regular subproduct containing $P'$ in the context of Lemma~\ref{lemma2blockweaklyregular} \textbf{(c)} \textbf{(iii)} and Lemma~\ref{lemma3blockweaklyregular} \textbf{(d)} \textbf{(ii)}, \textbf{(e)} \textbf{(ii)}, and \textbf{(f)}. Firstly, we introduce terminology.

\begin{definition}
\label{defpseudoregularblock}
Let $B$ be a block in $\sigma$ and let $P$ be the minimal $s$-subproduct of $\sigma$ containing $B$. 
\begin{description}
\item[(a)] If $B = \sigma_1^{a_p}\sigma_3^{b_p}\sigma_2\sigma_1^2$ is a $2$-block, then we write that $P$ is a \textit{pseudo $r$-regular subproduct} if $d_{\lambda}^{P}\leq d_{\mu}^{P}$, $d_{\mu}^{P} = d_{\nu}^{P} + a_p + 1$, and $\mu_{0}^{P} = \nu_{0}^{P}$ is a non-zero integer with at most $2$ and $3$ as prime factors. In this case, the \textit{pseudo-discrepancy of $P$} is $\Xi_{\text{psuedo}}\left(P\right) = d_{\mu}^{P} - d_{\lambda}^{P}$. We write that $P$ is \textit{positive} if $\lambda_{0}^{P} = \mu_{0}^{P}$ and we write that $P$ is \textit{negative} if $\lambda_{0}^{P} = -\mu_{0}^{P}$. 
\item[(b)] If $B$ is a $3$-block, then we write that $P$ is a \textit{pseudo $r$-regular subproduct} if $d_{\lambda,1}^{P}\leq d_{\mu,1}^{P}$, $d_{\mu,3}^{P} = d_{\mu,1}^{P} + 1$, and $\mu_{3,0}^{P} = -\mu_{1,0}^{P}$ is a non-zero integer with at most $2$ and $3$ as prime factors. In this case, the \textit{pseudo-discrepancy of $P$} is $\Xi_{\text{pseudo}}\left(P\right) = d_{\mu,1}^{P} - d_{\lambda,1}^{P}$. We write that $P$ is \textit{positive} if $\lambda_{1,0}^{P} = \mu_{1,0}^{P}$ and we write that $P$ is \textit{negative} if $\lambda_{1,0}^{P} = -\mu_{1,0}^{P}$.
\end{description}
If $P$ is a pseudo $r$-regular subproduct, then we also write that \textit{$B$ is a pseudo $r$-regular block} and the \textit{pseudo-discrepancy of $B$} is $\Xi_{\text{pseudo}}\left(B\right) = \Xi_{\text{pseudo}}\left(P\right)$. 
\end{definition}

In the context of Lemma~\ref{lemma2blockweaklyregular} \textbf{(c)} \textbf{(iii)}, $B$ is either a pseudo $r$-regular $2$-block or a strongly $r$-regular $2$-block. In the context of Lemma~\ref{lemma3blockweaklyregular} \textbf{(d)} \textbf{(ii)}, \textbf{(e)} \textbf{(ii)}, and \textbf{(f)}, $B$ is a pseudo $r$-regular $3$-block.

We now study pseudo $r$-regular blocks. Firstly, we establish a constraint on the minimal $s'$-subproduct of $\sigma$ for $s'\in \{1,3\}$ containing a pseudo $r$-regular $2$-subproduct $P$. 

\begin{proposition}
\label{pgood3match}
Let $P$ be a pseudo $r$-regular $2$-subproduct of $\sigma$ such that $\Xi_{\text{pseudo}}\left(P\right)> 1$. Let $P'$ be the minimal $s'$-subproduct of $\sigma$ containing $P$ for $s'\in \{1,3\}$. Let $t\neq 2,3$ be a prime number.

The leading coefficients of multiple entries in the $r$th row of $\beta_4\left(P'\right)$ are non-zero modulo $t$. If $a_p>1$ and $b_p>0$, and either $P$ is positive or $a_p-2\neq b_p$, then $P'$ is $r$-regular. If $0\leq a_p\leq 1$ and $b_p>0$, then $P'$ is strongly $r$-regular. 
\end{proposition}
\begin{proof}
Firstly, we consider the case $a_p>1$ and $b_p>0$. Proposition~\ref{ppre21/3bothext} \textbf{(i)} - \textbf{(ii)} and \textbf{(iv)} and Proposition~\ref{ppre21/3bothextadd} imply the following. Firstly, $d_{\lambda,3}^{P'} = d_{\lambda}^{P} + b_p - 1$, $d_{\mu,1}^{P'} = d_{\mu,1}^{P} + 1$, and $\lambda_{3,0}^{P'} = \lambda_{0}^{P}$. If $a_p = 2$, then $d_{\lambda,1}^{P'} = d_{\lambda,1}^{P} + 2\leq d_{\mu,1}^{P'}$. If $a_p>2$, then $d_{\lambda,1}^{P'} = d_{\mu,1}^{P} + a_p - 1>d_{\mu,1}^{P'}$ and $\lambda_{1,0}^{P'} = -\mu_{1,0}^{P}$. In particular, if $a_p>1$ and $b_p>0$, and either $P$ is positive or $a_p-2\neq b_p$, then $P'$ is $r$-regular.

If $a_p = 0$ (and $b_p>0$), then Lemma~\ref{luniqueextensionimpliesbadsubroad} \textbf{(a)} \textbf{(i)} implies that $P'$ is strongly $r$-regular. If $a_p = 1$ and $b_p>0$, then Proposition~\ref{ppre21/3bothext} \textbf{(i)} - \textbf{(ii)} and \textbf{(iv)} and Proposition~\ref{ppre21/3bothextadd} imply that $d_{\lambda,3}^{P'} = d_{\lambda}^{P} + b_p - 1>d_{\mu,3}^{P'} = d_{\mu,1}^{P}$, $d_{\mu,1}^{P'} = d_{\lambda,1}^{P} + 1$, and $\lambda_{3,0}^{P'} = \lambda_{0}^{P}$. In particular, if $a_p = 1$ and $b_p>0$, then $P'$ is strongly $r$-regular.
\end{proof}

In particular, in light of Proposition~\ref{pregularap>1}, we need only consider the case $b_{p+1} = 0$ when studying a pseudo $r$-regular $3$-block $\sigma_2^{c_{p-1}}\sigma_3\sigma_2^{c_p}$. The singular subroad following a pseudo $r$-regular block is particularly relevant (Definition~\ref{defsingularroadweight}). 

\begin{definition}
Let $B$ be a block. Let $P$ be the minimal $s$-subproduct of $\sigma$ containing $B$ and let $R'$ be the singular subroad following $B$. We define the \textit{singular subproduct $P'$ containing $B$} as follows. 
\begin{description}
\item[(i)] If $R'$ is followed by an isolated $\sigma_2$, then $P'$ is the minimal $s'$-subproduct of $\sigma$ for $s'\in \{1,3\}$ such that $R'$ is properly contained in $P'$. (If no such $s'$-subproduct exists, then $P'$ is not defined.)
\item[(ii)] If $R'$ is followed by a non-isolated $\sigma_2$, then $P'$ is the minimal $2$-subproduct of $\sigma$ such that $R'$ is properly contained in $P'$.
\end{description}
If $PR' = \sigma$, then $P'$ is not defined. 
\end{definition}

We now establish that the singular weight of the singular subroad following a pseudo $r$-regular block measures the change in discrepancy between a maximal $q$-resistance good path and a maximal $q$-resistance bad path, if both paths are extended along the singular subroad.

\begin{proposition}
\label{ppseudoregularpathextension}
Let $B$ be a pseudo $r$-regular block and let $P$ be the minimal $s$-subproduct of $\sigma$ containing $B$. Let $R'$ be the singular subroad following $B$. Let $P_0'$ be an $s_0'$-subproduct of $\sigma$ containing $B$ such that $P_0'\setminus P\subseteq R'$. 

If $s_0' = 2$, then let $M$ be the number of $\sigma_2$s in $P_0'\setminus P$, and if $s_0'\in \{1,3\}$, then let $M$ be one more than the number of $\sigma_2$s in $P_0'\setminus P$. Let us assume that $\omega_{\text{sing}}\left(P_0'\setminus P\right)\leq \Xi_{\text{pseudo}}\left(B\right)$. 

\begin{description}
\item[(a)] Let us consider the case where $B = \sigma_1^{a_p}\sigma_3^{b_p}\sigma_2\sigma_1^2$ is a $2$-block. 

\begin{description}
\item[(i)] Let $M$ be even.

If $s_0' = 2$, then $d_{\mu}^{P_0'} - d_{\lambda}^{P_0'} = d_{\mu}^{P} - d_{\lambda}^{P} - \omega_{\text{sing}}\left(P_0'\setminus P\right)$, and $d_{\nu}^{P_0'}\leq \max\{d_{\mu}^{P_0'} -a_p,d_{\lambda}^{P_0'} - 1\}$ with equality if $d_{\mu}^{P_0'} - a_p\neq d_{\lambda}^{P_0'} - 1$.

If $s_0'\in \{1,3\}$, then $d_{\nu}^{P_0'} - d_{\lambda}^{P_0'} = d_{\mu}^{P} - d_{\lambda}^{P} - \omega_{\text{sing}}\left(P_0'\setminus P\right) + 1$ and $d_{\nu}^{P_0'} = d_{\mu}^{P_0'} + a_p$. 
\item[(ii)] Let $M$ be odd and let us assume that $\omega_{\text{sing}}\left(P_0'\setminus P\right) + a_p + 1\leq \Xi_{\text{pseudo}}\left(B\right)$. 

If $s_0' = 2$, then $d_{\mu}^{P_0'} - d_{\lambda}^{P_0'} = d_{\nu}^{P} - d_{\lambda}^{P} - \omega_{\text{sing}}\left(P_0'\setminus P\right)$ and $d_{\nu}^{P_0'} = d_{\mu}^{P_0'} + a_p$.

If $s_0'\in \{1,3\}$, then $d_{\nu}^{P_0'} - d_{\lambda}^{P_0'} = d_{\nu}^{P} - d_{\lambda}^{P} - \omega_{\text{sing}}\left(P_0'\setminus P\right) - 1$ and $d_{\nu}^{P_0'} = d_{\mu}^{P_0'}-a_p-1$.
\end{description}
\item[(b)]  Let us consider the case where $B = \sigma_2^{c_{p-1}}\sigma_3\sigma_2^{c_p}$ is a $3$-block and $b_{p+1} = 0$. 
\begin{description}
\item[(i)] Let $M$ be even. 

If $s_0' = 2$, then $d_{\mu}^{P_0'} - d_{\lambda}^{P_0'} = d_{\mu}^{P} - d_{\lambda,1}^{P} - \omega_{\text{sing}}\left(P_0'\setminus P\right)$ and $d_{\nu}^{P_0'} = d_{\mu}^{P_0'}$.

If $s_0'\in \{1,3\}$, then $d_{\nu}^{P_0'} - d_{\lambda}^{P_0'} = d_{\mu}^{P} - d_{\lambda,1}^{P} - \omega_{\text{sing}}\left(P_0'\setminus P\right) + 1$ and $d_{\nu}^{P_0'} = d_{\mu}^{P_0'}$.

\item[(ii)] Let $M$ be odd. 

If $s_0' = 2$, then $d_{\mu}^{P_0'} - d_{\lambda}^{P_0'} = d_{\nu}^{P} - d_{\lambda,1}^{P} - \omega_{\text{sing}}\left(P_0'\setminus P\right)$, and $d_{\nu}^{P_0'}\leq \max\{d_{\mu}^{P_0'},d_{\lambda}^{P_0'}\}$ with equality if $d_{\mu}^{P_0'}\neq d_{\lambda}^{P_0'}$.

If $s_0'\in \{1,3\}$, then $d_{\nu}^{P_0'} - d_{\lambda}^{P_0'} = d_{\nu}^{P} - d_{\lambda,1}^{P} - \omega_{\text{sing}}\left(P_0'\setminus P\right) -1$ and $d_{\nu}^{P_0'} = d_{\mu}^{P_0'} - 1$.
\end{description}
\end{description}
\end{proposition}
\begin{proof}
If $s_0' = 2$, then define $i'$ to be the index of the last generator in $P_0'$ that is not equal to $2$ and define $i$ to be such that $\{i,i'\} = \{1,3\}$. If $s_0'\in \{1,3\}$, then define $i$ to be the index of the last generator in $P_0'$ and define $i'$ to be such that $\{i,i'\} = \{1,3\}$. If $i = 1$, then define $\psi' = 1$, and if $i = 3$, then define $\psi' = 0$.

If $B$ is a $2$-block, then let $X$ be a maximal $q$-resistance good $\left(r,1\right)$-type $P$-path. If $B$ is a $3$-block, then let $X$ be a maximal $q$-resistance good $\left(r,2\right)$-type $P$-path. If $s_0' = 2$, then Lemma~\ref{lemmagoodext} \textbf{(i)} implies that the maximal $q$-resistance extensions of $X$ to an $\left(r,i'\right)$-type $P_0'$-path $X_0$ and to an $\left(r,2\right)$-type $P_0'$-path $X_0'$ are defined by extending $X$ by vertex changes at every opportunity. Similarly, if $s_0'\in \{1,3\}$, then Lemma~\ref{lemmagoodext} \textbf{(i)} implies that the maximal $q$-resistance extensions of $X$ to an $\left(r,i'\right)$-type $P_0'$-path $X_0$ and to an $\left(r,i\right)$-type $P_0'$-path $X_0'$ are defined by extending $X$ by vertex changes at every opportunity.

If $s_0' = 2$, then the weight of $X_0$ is the product of $\left(-1\right)^{M}q^{2M-\psi'}$ with the weight of $X$, and the weight of $X_{0}'$ is the product of $\left(-1\right)^{M}q^{2M}$ with the weight of $X$. If $s_0'\in \{1,3\}$, then the weight of $X_0$ is the product of $\left(-1\right)^{M - \psi' - 1}q^{2M-\psi'-2}$ with the weight of $X$, and the weight of $X_0'$ is the product of $\left(-1\right)^{M+\psi' - 1}q^{2M+\psi' - 1}$ with the weight of $X$.

\begin{description}
\item[(a)] Let $Y$ be a maximal $q$-resistance $\left(r,1\right)$-type $P$-path and let $Z$ be a maximal $q$-resistance $\left(r,3\right)$-type $P$-path. Let $Y_0'$ be the maximal $q$-resistance extension of $Y$ to an $\left(r,\cdot\right)$-type $P_0'$-path and let $Z_0'$ be the maximal $q$-resistance extension of $Z$ to an $\left(r,\cdot\right)$-type $P_0'$-path.

\begin{description}
\item[(i)] Corollary~\ref{celementarybadsubroadimpliesuniqueextension} \textbf{(c)} implies that $Z_0'$ is an $\left(r,3\right)$-type $P_0'$-path. If $s_0' = 2$, then Corollary~\ref{celementarybadsubroadimpliesuniqueextension} \textbf{(c)} implies that $Y_0'$ is an $\left(r,2\right)$-type $P_0'$-path. If $s_0'\in \{1,3\}$, then Corollary~\ref{celementarybadsubroadimpliesuniqueextension} \textbf{(c)} implies that $Y_0'$ is an $\left(r,1\right)$-type $P_0'$-path.

Proposition~\ref{pelementarybadsubroadimpliesuniqueextension} \textbf{(c)} implies that the weight of $Y_0'$ is the product of $\left(-q\right)^{\frac{M}{2}-1}$ with the weight of $Y$, and the weight of $Z_0'$ is the product of $\left(-q\right)^{\frac{M}{2}}$ with the weight of $Z$.

\item[(ii)] Corollary~\ref{celementarybadsubroadimpliesuniqueextension} \textbf{(c)} implies that $Y_0'$ is an $\left(r,1\right)$-type $P_0'$-path. If $s_0' = 2$, then Corollary~\ref{celementarybadsubroadimpliesuniqueextension} \textbf{(c)} implies that $Z_0'$ is an $\left(r,2\right)$-type $P_0'$-path. If $s_0'\in \{1,3\}$, then Corollary~\ref{celementarybadsubroadimpliesuniqueextension} \textbf{(c)} implies that $Z_0'$ is an $\left(r,3\right)$-type $P_0'$-path.

Proposition~\ref{pelementarybadsubroadimpliesuniqueextension} \textbf{(c)} implies that the weight of $Y_0'$ is the product of $\left(-q\right)^{\frac{M-1}{2}}$ with the weight of $Y$. If $s_0' = 2$, then Proposition~\ref{pelementarybadsubroadimpliesuniqueextension} \textbf{(c)} implies that the weight of $Z_0'$ is the product of $\left(-q\right)^{\frac{M+1}{2}}$ with the weight of $Z$. If $s_0'\in \{1,3\}$, then Proposition~\ref{pelementarybadsubroadimpliesuniqueextension} \textbf{(c)} implies that the weight of $Z_0'$ is the product of $\left(-q\right)^{\frac{M+1}{2} - 1}$ with the weight of $Z$. 
\end{description}
\item[(b)] If $B$ is a $3$-block, then let $Y$ be a maximal $q$-resistance bad ($1$-switch) $\left(r,2\right)$-type $P$-path and let $Z$ be a maximal $q$-resistance $\left(r,3\right)$-type $P$-path. Let $Y_0'$ be the maximal $q$-resistance extension of $Y$ to an $\left(r,\cdot\right)$-type $P_0'$-path and let $Z_0'$ be the maximal $q$-resistance extension of $Z$ to an $\left(r,\cdot\right)$-type $P_0'$-path.
\begin{description}
\item[(i)] Corollary~\ref{celementarybadsubroadimpliesuniqueextension} \textbf{(c)} implies that $Z_0'$ is an $\left(r,3\right)$-type $P_0'$-path. If $s_0' = 2$, then Corollary~\ref{celementarybadsubroadimpliesuniqueextension} \textbf{(c)} implies that $Y_0'$ is an $\left(r,2\right)$-type $P_0'$-path. If $s_0'\in \{1,3\}$, then Corollary~\ref{celementarybadsubroadimpliesuniqueextension} \textbf{(c)} implies that $Y_0'$ is an $\left(r,1\right)$-type $P_0'$-path.

Proposition~\ref{pelementarybadsubroadimpliesuniqueextension} \textbf{(c)} implies that the weight of $Y_0'$ is the product of $\left(-q\right)^{\frac{M}{2}}$ with the weight of $Y$, and the weight of $Z_0'$ is the product of $\left(-q\right)^{\frac{M}{2}}$ with the weight of $Z$.
\item[(ii)] Corollary~\ref{celementarybadsubroadimpliesuniqueextension} \textbf{(c)} implies that $Y_0'$ is an $\left(r,1\right)$-type $P_0'$-path. If $s_0' = 2$, then Corollary~\ref{celementarybadsubroadimpliesuniqueextension} \textbf{(c)} implies that $Z_0'$ is an $\left(r,2\right)$-type $P_0'$-path. If $s_0'\in \{1,3\}$, then Corollary~\ref{celementarybadsubroadimpliesuniqueextension} \textbf{(c)} implies that $Z_0'$ is an $\left(r,3\right)$-type $P_0'$-path. 

Proposition~\ref{pelementarybadsubroadimpliesuniqueextension} \textbf{(c)} implies that the weight of $Y_0'$ is the product of $\left(-q\right)^{\frac{M+1}{2}}$ with the weight of $Y$. If $s_0' = 2$, then Proposition~\ref{pelementarybadsubroadimpliesuniqueextension} \textbf{(c)} implies that the weight of $Z_0'$ is the product of $\left(-q\right)^{\frac{M+1}{2}}$ with the weight of $Z$. If $s_0'\in \{1,3\}$, then Proposition~\ref{pelementarybadsubroadimpliesuniqueextension} \textbf{(c)} implies that the weight of $Z_0'$ is the product of $\left(-q\right)^{\frac{M+1}{2} - 1}$ with the weight of $Z$.
\end{description}
\end{description}
Therefore, the statement is established.
\end{proof}

If $B$ is a pseudo $r$-regular block, then we establish that there is generically a strongly $r$-regular subproduct $P'$ such that $B\subseteq P'$.

\begin{corollary}
\label{cpseudoregularregular}
Let $B$ be a pseudo $r$-regular block in $\sigma$ and let $P$ be the minimal $s$-subproduct of $\sigma$ containing $B$. Let $R'$ be the singular subroad following $B$ and let $P'$ be the singular subproduct containing $B$. Let $M$ be one more than the number of $\sigma_2$s in $R'$.

If $R'$ is followed by an isolated $\sigma_2$, then we define $\phi_{p'}$ such that $P'\setminus R' = \sigma_2\sigma_i^{\phi_{p'}}$. If $R'$ is followed by a non-isolated $\sigma_2$, then we define $\phi_{p'}$ such that $P'\setminus R' = \sigma_2^{\phi_{p'}}$.
 
\begin{description}
\item[(a)] Let $B$ be a $2$-block.
\begin{description}
\item[(i)] Let us assume that $R'$ is followed by an isolated $\sigma_2$, $M$ is odd, and $\omega_{\text{sing}}\left(R'\right) > \Xi_{\text{pseudo}}\left(B\right) - a_p - 2$. 

If $\Xi_{\text{pseudo}}\left(B\right) +2 - \frac{3\left(M-1\right)}{2} - \phi_{p'} <0$, then $P'$ is strongly $r$-regular. 

If $\Xi_{\text{pseudo}}\left(B\right) +2 - \frac{3\left(M-1\right)}{2} - \phi_{p'}\geq 0$, then $P'$ is $\nu$ $r$-regular and $\Xi\left(P'\right) = \Xi_{\text{pseudo}}\left(B\right) +2 - \frac{3\left(M-1\right)}{2} - \phi_{p'}$. Furthermore, in this case, if $M\equiv 1\pmod 4$, then the sign of $P'$ (as a $\nu$ $r$-regular subproduct) is the opposite of the sign of $P$ (as a pseudo $r$-regular subproduct), and if $M\equiv 3\pmod 4$, then the sign of $P'$ is equal to the sign of $P$. 

\item[(ii)] Let us assume that $R'$ is followed by an isolated $\sigma_2$, $M$ is odd, and $\omega_{\text{sing}}\left(R'\right) < \Xi_{\text{pseudo}}\left(B\right) - a_p - 2$. 

If $a_p - \phi_{p'}+3 < 0$, then $P'$ is strongly $r$-regular. 

If $a_p-\phi_{p'}+3\geq 0$, then $P'$ is positive $\nu$ $r$-regular and $\Xi\left(P'\right) = a_p - \phi_{p'} + 3$.

\item[(iii)] Let us assume that $R'$ is followed by an isolated $\sigma_2$ and $M$ is even. 

If $\omega_{\text{sing}}\left(R'\right)\neq \Xi_{\text{pseudo}}\left(B\right) - 1$, then $P'$ is strongly $r$-regular. 

If either $P$ is positive and $M\equiv 0\pmod 4$, or $P$ is negative and $M\equiv 2 \pmod 4$, then $P'$ is strongly $r$-regular.

\item[(iv)] Let us assume that $R'$ is followed by a non-isolated $\sigma_2$ and $M$ is even. 

If $\omega_{\text{sing}}\left(R'\right)\neq \Xi_{\text{pseudo}}\left(B\right)$, then $P'$ is strongly $r$-regular. 

If either $P$ is positive and $M\equiv 0\pmod 4$, or $P$ is negative and $M\equiv 2\pmod 4$, then $P'$ is strongly $r$-regular.
\item[(v)] Let us assume that $R'$ is followed by a non-isolated $\sigma_2$, $M$ is odd, and $\omega_{\text{sing}}\left(R'\right)< \Xi_{\text{pseudo}}\left(B\right) - a_p - 1$. 

If $\phi_{p'}>2$ and $a_p>1$, then $P'$ is strongly $r$-regular. 

If $\phi_{p'}=2$, then $P'$ is negative $r$-regular with $\Xi\left(P'\right) = a_p - 1$. 
\item[(vi)] Let us assume that $R'$ is followed by a non-isolated $\sigma_2$, $M$ is odd, and $\omega_{\text{sing}}\left(R'\right)> \Xi_{\text{pseudo}}\left(B\right) - a_p - 1$. 

If either $\phi_{p'}>2$ and $\omega_{\text{sing}}\left(R'\right) \neq \Xi_{\text{pseudo}}\left(B\right) -2$, or $\phi_{p'} = 2$ and $\omega_{\text{sing}}\left(R'\right) > \Xi_{\text{pseudo}}\left(B\right) -2$, then $P'$ is strongly $r$-regular. 

If $\phi_{p'} = 2$ and $\omega_{\text{sing}}\left(R'\right)\leq \Xi_{\text{pseudo}}\left(B\right) -2$, then $P'$ is $r$-regular with $\Xi\left(P'\right) = \Xi_{\text{pseudo}}\left(B\right) - \omega_{\text{sing}}\left(R'\right) -2$. 
\end{description}
\item[(b)] Let us consider the case where $B$ is a $3$-block.
\begin{description}
\item[(i)] Let us assume that $R'$ is followed by an isolated $\sigma_2$. 

If $\omega_{\text{sing}}\left(R'\right)\geq \Xi_{\text{pseudo}}\left(B\right)$, then $P'$ is strongly $r$-regular. 

If $\omega_{\text{sing}}\left(R'\right) + 1 < \Xi_{\text{pseudo}}\left(B\right)$ and $\phi_{p'}>3$, then $P'$ is strongly $r$-regular. 

If $\omega_{\text{sing}}\left(R'\right) + 1 < \Xi_{\text{pseudo}}\left(B\right)$ and $\phi_{p'} = 3$, then $P'$ is $\nu$ $r$-regular and $\Xi\left(P'\right) = 1$. 
\item[(ii)] Let us assume that $R'$ is followed by an isolated $\sigma_2$ and $\omega_{\text{sing}}\left(R'\right) + 1 = \Xi_{\text{pseudo}}\left(B\right)$. 

If either $P$ is positive and $M\equiv 0,1\pmod 4$, or $P$ is negative and $M\equiv 2,3 \pmod 4$, then $P'$ is strongly $r$-regular. 
\item[(iii)] Let us assume that $R'$ is followed by a non-isolated $\sigma_2$. 

If $M$ is odd and $\omega_{\text{sing}}\left(R'\right)\neq \Xi_{\text{pseudo}}\left(B\right)$, then $P'$ is strongly $r$-regular. 

If $\omega_{\text{sing}}\left(R'\right)>\Xi_{\text{pseudo}}\left(B\right)$, then $P'$ is strongly $r$-regular. 
\end{description}
\end{description}
Finally, in each of the above cases, if $P'$ is not defined, then the leading coefficients of multiple entries in the $r$th row of $\beta\left(\sigma\right)$ are non-zero modulo each prime $t\neq 2,3$.
\end{corollary}
\begin{proof}
\begin{description}
\item[(a)]
\begin{description}
\item[(i)] The statement is a consequence of Proposition~\ref{ppseudoregularpathextension} \textbf{(a)} \textbf{(ii)} and Proposition~\ref{p21/3ext} \textbf{(i)} and \textbf{(iii)}.
\item[(ii)] The statement is a consequence of Proposition~\ref{ppseudoregularpathextension} \textbf{(a)} \textbf{(ii)}, Proposition~\ref{p21/3ext} \textbf{(iii)}, and Proposition~\ref{p21/3extadd} \textbf{(i)}.
\item[(iii)] The statement is a consequence of Proposition~\ref{ppseudoregularpathextension} \textbf{(a)} \textbf{(i)}, Proposition~\ref{p21/3ext} \textbf{(i)} and \textbf{(iii)}, and Proposition~\ref{p21/3extadd} \textbf{(i)}.
\item[(iv)] The statement is a consequence of Proposition~\ref{ppseudoregularpathextension} \textbf{(a)} \textbf{(i)}, Proposition~\ref{p1/32ext} \textbf{(a)} \textbf{(i)}, and Proposition~\ref{p1/32extadd} \textbf{(b)} \textbf{(i)}, \textbf{(c)} and \textbf{(d)} \textbf{(i)}.
\item[(v)] The statement is a consequence of Proposition~\ref{ppseudoregularpathextension} \textbf{(a)} \textbf{(ii)} and Proposition~\ref{p1/32extadd} \textbf{(d)} \textbf{(ii)}.
\item[(vi)] The statement is a consequence of Proposition~\ref{ppseudoregularpathextension} \textbf{(a)} \textbf{(ii)} and Proposition~\ref{p1/32extadd} \textbf{(a)} \textbf{(i)}.
\end{description}
\item[(b)]
\begin{description}
\item[(i)] The statement is a consequence of Proposition~\ref{ppseudoregularpathextension} \textbf{(b)}, Proposition~\ref{p21/3ext} \textbf{(i)} and \textbf{(iii)}, and Proposition~\ref{p21/3extadd} \textbf{(i)}.
\item[(ii)] The statement is a consequence of Proposition~\ref{ppseudoregularpathextension} \textbf{(b)}, Proposition~\ref{p21/3ext} \textbf{(iii)}, and Proposition~\ref{p21/3extadd} \textbf{(i)}.
\item[(iii)] The statement is a consequence of Proposition~\ref{ppseudoregularpathextension} \textbf{(b)}, Proposition~\ref{p1/32ext} \textbf{(i)}, and Proposition~\ref{p1/32extadd} \textbf{(d)} \textbf{(i)}.
\end{description}
\end{description}
\end{proof}

We now compartmentalize Lemma~\ref{lemma2blockweaklyregular} into a single statement on the property of $r$-regularity around $2$-blocks.

\begin{corollary}
\label{cor2blockregular}
Let $B = \sigma_1^{a_p}\sigma_3^{b_p}\sigma_2\sigma_1^{a_{p+1}}$ be a generic $2$-block and let $P = \prod_{i=1}^{p-1} \sigma_1^{a_i}\sigma_3^{b_i}\sigma_2^{c_i}$ be the $2$-subproduct immediately preceding $B$. Let $P' = PB$.
\begin{description}
\item[(a)] Let us assume that $P$ is strongly $r$-regular. If $a_p\geq b_p + 1$ and $a_{p+1} = 2$, then $P'$ is positive $\mu$-switch $r$-regular and $\Xi\left(P'\right) = a_p - b_p - 1$. If either $a_p < b_p + 1$ or $a_{p+1}>2$, then $P'$ is strongly $r$-regular.
\item[(b)] Let us assume that $P$ is $1$-switch $r$-regular and $\Xi\left(P\right)>1$. If $a_p\geq b_p + 2$ and $a_{p+1} = 2$, then $P'$ is positive $\mu$-switch $r$-regular and $\Xi\left(P'\right) = a_p - b_p - 2$. If either $a_p<b_p + 2$ or $a_{p+1}>2$, then $P'$ is strongly $r$-regular.
\item[(c)] Let us assume that $P$ is $3$-switch $r$-regular, $\Xi\left(P\right)>1$, and $b_p>2$. If $a_p\geq b_p$ and $a_{p+1} = 2$, then $P'$ is positive $\mu$-switch $r$-regular and $\Xi\left(P'\right) = a_p - b_p$. If either $a_p<b_p$ or $a_{p+1}>2$, then $P'$ is strongly $r$-regular.
\item[(d)] Let us assume that $P$ is $3$-switch $r$-regular, $\Xi\left(P\right)>1$, and $b_p = 2$. If $a_{p+1} = 2$, then $P'$ is pseudo $r$-regular, $\Xi_{\text{pseudo}}\left(P'\right) = \Xi\left(P\right) + a_p - 3$, and the sign of $P'$ is equal to the sign of $P$. If $a_{p+1}>2$, then $P'$ is strongly $r$-regular.
\end{description}
Finally, in each of the above cases, the leading coefficients of multiple entries in the $r$th row of $\beta_4\left(P'\right)$ are non-zero modulo each prime $t\neq 2,3$.
\end{corollary}
\begin{proof}
\begin{description}
\item[(a)] If $a_p\geq b_p + 1$ and $a_{p+1} = 2$, then the statement is a consequence of Lemma~\ref{lemma2blockweaklyregular} \textbf{(a)}. If either $a_p<b_p + 1$ or $a_{p+1}>2$, then $B$ is a normal $2$-block and the statement is a consequence of Lemma~\ref{l2blockstronglyregular}.
\item[(b)] If $a_p\geq b_p + 2$ and $a_{p+1} = 2$, then the statement is a consequence of Lemma~\ref{lemma2blockweaklyregular} \textbf{(b)} \textbf{(i)}. If $a_p<b_p + 2$ and $a_{p+1} = 2$, then the statement is a consequence of Lemma~\ref{lemma2blockweaklyregular} \textbf{(b)} \textbf{(ii)}. If $a_{p+1}>2$, then the statement is a consequence of Lemma~\ref{lemma2blockweaklyregular} \textbf{(d)} \textbf{(i)}.
\item[(c)] If $a_p\geq b_p$ and $a_{p+1} = 2$, then the statement is a consequence of Lemma~\ref{lemma2blockweaklyregular} \textbf{(c)} \textbf{(i)}. If $a_p<b_p$ and $a_{p+1} = 2$, then the statement is a consequence of Lemma~\ref{lemma2blockweaklyregular} \textbf{(c)} \textbf{(ii)}. If $a_{p+1}>2$, then the statement is a consequence of Lemma~\ref{lemma2blockweaklyregular} \textbf{(d)} \textbf{(ii)}.
\item[(d)] If $a_{p+1} = 2$, then the statement is a consequence of Lemma~\ref{lemma2blockweaklyregular} \textbf{(c)} \textbf{(iii)}. If $a_{p+1}>2$, then the statement is a consequence of Lemma~\ref{lemma2blockweaklyregular} \textbf{(d)} \textbf{(ii)}.
\end{description}
\end{proof}

We now compartmentalize Lemma~\ref{lemma3blockweaklyregular} into a single statement on the property of $r$-regularity around $3$-blocks.

\begin{corollary}
\label{cor3blockregular}
Let $B = \sigma_2^{c_{p-1}}\sigma_3\sigma_2^{c_p}$ be a generic $3$-block in $\sigma$ and let $P = \left(\prod_{i=1}^{p-2} \sigma_1^{a_i}\sigma_3^{b_i}\sigma_2^{c_i}\right)\sigma_1^{a_{p-1}}$ be the $s$-subproduct immediately preceding $B$ for $s\in \{1,3\}$. Let $P' = PB$ and let $P_0 = \prod_{i=1}^{p-2} \sigma_1^{a_i}\sigma_3^{b_i}\sigma_2^{c_i}$ be the maximal $2$-subproduct contained in $P$.

\begin{description}
\item[(a)] Let us assume that $P$ is strongly $r$-regular and $c_{p-1}>1$. 
\begin{description}
\item[(i)] If $c_p = 2$, then $P'$ is negative $3$-switch $r$-regular and $\Xi\left(P'\right) = c_{p-1} - 1$.
\item[(ii)] If $c_p>2$, then $P'$ is strongly $r$-regular.
\end{description}
\item[(b)] Let us assume that $P$ is $r$-regular, $\Xi\left(P\right)>1$, and $c_{p-1}>2$. 
\begin{description}
\item[(i)] If $c_p = 2$, then $P'$ is negative $3$-switch $r$-regular and $\Xi\left(P'\right) = c_{p-1} - 2$.
\item[(ii)] If $c_p>2$, then $P'$ is strongly $r$-regular.
\end{description}
\item[(c)] Let us assume that $c_{p-1} = 1$. 
\begin{description}
\item[(i)] If $P$ is $r$-regular and $\Xi\left(P\right)>0$, then $P'$ is strongly $r$-regular.
\item[(ii)] Let us consider the case where $P_0$ is strongly $r$-regular. If $a_{p-1}+1<c_p$, then $P'$ is strongly $r$-regular. If $a_{p-1}+1>c_p$, then $P'$ is positive pseudo $r$-regular and $\Xi_{\text{pseudo}}\left(P'\right) = a_{p-1} - c_p$. 
\item[(iii)] Let us consider the case where $P_0$ is $1$-switch $r$-regular and $\Xi\left(P_0\right)>1$. If $a_{p-1}<c_p$, then $P'$ is strongly $r$-regular. If $a_{p-1}>c_p$, then $P'$ is pseudo $r$-regular, $\Xi_{\text{pseudo}}\left(P'\right) = a_{p-1}-c_p-1$, and the sign of $P'$ is equal to the sign of $P$.
\end{description}
\item[(d)] Let us assume that $P$ is $r$-regular, $c_{p-1} = 2$, and $c_p>2$. 
\begin{description}
\item[(i)] If $\Xi\left(P\right)<c_p-1$, then $P'$ is strongly $r$-regular. 
\item[(ii)] If $\Xi\left(P\right)>c_p-1$, then $P'$ is pseudo $r$-regular, $\Xi_{\text{psuedo}}\left(P'\right) = \Xi\left(P\right) - c_p$, and the sign of $P'$ is the opposite of the sign of $P$.
\end{description}
\item[(e)] Let us assume that $c_{p-1} = 2$ and $c_p = 2$. If $P$ is $r$-regular and $\Xi\left(P\right)>1$, then $P'$ is pseudo $r$-regular, $\Xi_{\text{pseudo}}\left(P'\right) = \Xi\left(P\right) - 2$, and the sign of $P'$ is the opposite of the sign of $P$.
\end{description}
Finally, in each of the above cases, the leading coefficients of multiple entries in the $r$th row of $\beta_4\left(P'\right)$ are non-zero modulo each prime $t\neq 2,3$.
\end{corollary}
\begin{proof}
\begin{description}
\item[(a)] 
\begin{description}
\item[(i)] The statement is a consequence of Lemma~\ref{lemma3blockweaklyregular} \textbf{(a)}.
\item[(ii)] The statement is a consequence of Lemma~\ref{l3blockstronglyregular} since $B$ is a normal $3$-block.
\end{description}
\item[(b)] 
\begin{description}
\item[(i)] The statement is a consequence of Lemma~\ref{lemma3blockweaklyregular} \textbf{(b)}.
\item[(ii)] The statement is a consequence of Lemma~\ref{lemma3blockweaklyregular} \textbf{(c)}.
\end{description}
\item[(c)] 
\begin{description}
\item[(i)] The statement is a consequence of Lemma~\ref{lemma3blockweaklyregular} \textbf{(d)} \textbf{(i)}.
\item[(ii)] Proposition~\ref{p21/3ext} \textbf{(i)} implies that $d_{\lambda}^{P} = d_{\lambda}^{P_0} + a_{p-1}$, $d_{\mu}^{P} = d_{\lambda}^{P_0} + 1$, and $\lambda_{0}^{P} = -\lambda_{0}^{P_0} = \mu_{0}^{P}$. If $a_{p-1} + 1 < c_p$, then the statement is a consequence of Lemma~\ref{lemma3blockweaklyregular} \textbf{(d)} \textbf{(i)}. If $a_{p-1}+1>c_p$, then the statement is a consequence of Lemma~\ref{lemma3blockweaklyregular} \textbf{(d)} \textbf{(ii)}.
\item[(iii)] If $a_{p-1} = 2$, then Lemma~\ref{lweightcompute} \textbf{(a)} implies that $P$ is $r$-regular and $\Xi\left(P\right) > 0$, in which case \textbf{(c)} \textbf{(i)} implies that $P'$ is strongly $r$-regular. Let us assume that $a_{p-1}>2$. Proposition~\ref{p21/3extadd} \textbf{(i)} implies that $d_{\lambda}^{P} = d_{\mu}^{P_0} + a_{p-1} - 1$, $d_{\mu}^{P} = d_{\mu}^{P_0} + 1$, and $\lambda_{0}^{P} = -\mu_{0}^{P_0} = \mu_{0}^{P}$. If $a_{p-1}<c_p$, then the statement is a consequence of Lemma~\ref{lemma3blockweaklyregular} \textbf{(d)} \textbf{(i)}. If $a_{p-1}>c_p$, then the statement is a consequence of Lemma~\ref{lemma3blockweaklyregular} \textbf{(d)} \textbf{(ii)}.
\end{description}
\item[(d)] 
\begin{description}
\item[(i)] The statement is a consequence of Lemma~\ref{lemma3blockweaklyregular} \textbf{(e)} \textbf{(i)}.
\item[(ii)] The statement is a consequence of Lemma~\ref{lemma3blockweaklyregular} \textbf{(e)} \textbf{(ii)}. 
\end{description}
\item[(e)] The statement is a consequence of Lemma~\ref{lemma3blockweaklyregular} \textbf{(f)}.
\end{description}
\end{proof}

The final step is to address the case of singular blocks. Firstly, we exploit the strong constraints on a singular block in the minimal form to establish that the subproduct immediately preceding a singular block is strongly $r$-regular for some $r\in \{1,2\}$.

\begin{proposition}
\label{psingularblockregular}
Let $B$ be a singular block in $\sigma$. If $R$ is the $s$-subproduct immediately preceding $B$, then $R$ is strongly $r$-regular for some $r\in \{1,2\}$. Furthermore, if $B$ is a $2$-block, then $w_{\mu,3}^{R} = 0$, and if $B$ is a $3$-block, then $w_{\nu}^{R} = 0$.
\end{proposition}
\begin{proof}
Proposition~\ref{pblockDelta} implies that there is no initial block in $\sigma$, and in particular, $R$ is a road. Proposition~\ref{initialblocksingularblockconstraint} \textbf{(ii)} implies that the first generator in $\sigma$ is either $\sigma_1$ or $\sigma_2$. If the former, then $\sigma_1^{a_1}$ is a strongly $1$-regular subproduct of $\sigma$. If the latter, then $\sigma_2^{c_1}$ is a strongly $2$-regular subproduct of $\sigma$. Corollary~\ref{lemmaroadregular} implies in both cases that $R$ is strongly $r$-regular for some $r\in \{1,2\}$. Finally, Proposition~\ref{initialblocksingularblockconstraint} \textbf{(ii)} also implies the second statement.
\end{proof}

We now establish regularity results in the case where a singular block exists. The situation is very similar to the case of pseudo $r$-regular blocks, and we study path extensions along the singular subroad following a singular block.

\begin{lemma}
\label{lsingularblockregular}
Let $B$ be a singular block in $\sigma$. Let $R$ be the road immediately preceding $B$ and let $P$ be the minimal $s$-subproduct of $\sigma$ containing $B$. Let $P'$ be the singular subproduct containing $B$ and let $R'' = P'\setminus P$.

If $B = \sigma_1^{a_p}\sigma_3\sigma_2\sigma_1^2$ is a $2$-block, then we define $i = 1$ and $\iota_p = a_p$. If $B = \sigma_2^{c_{p-1}}\sigma_3\sigma_2$ is a $3$-block, then we define $i = 3$ and $\iota_p = c_{p-1}$. We define $i'$ to be such that $\{i,i'\} = \{1,3\}$. The following statements are true for some $r\in \{1,2\}$.

\begin{description}
\item[(a)] If $R''$ is not an $i$-switch bad subroad, then $P'$ is strongly $r$-regular.
\item[(b)] Let us assume that $R''$ is an $i$-switch bad subroad and let $\phi_{p'}$ be the last exponent in $R''$.
\begin{description}
\item[(i)] If $R''$ ends with $\sigma_2^{\phi_{p'}}$, then $P'$ is negative $r$-regular and $\Xi\left(P'\right) = \iota_p - 1$.
\item[(ii)] If $R''$ ends with $\sigma_{i'}^{\phi_{p'}}$, then $P'$ is positive $\nu$ $r$-regular and $\Xi\left(P'\right) = \iota_p + 3 - \phi_{p'}$.
\end{description}
\end{description}
Finally, if $P'$ is not defined, then the leading coefficients of multiple entries in the $r$th row of $\beta_4\left(\sigma\right)$ are non-zero modulo each prime $t\neq 2$. 
\end{lemma}
\begin{proof}
Firstly, Proposition~\ref{psingularblockregular} implies that $R$ is strongly $r$-regular for some $r\in \{1,2\}$. If $B = \sigma_1^{a_p}\sigma_3\sigma_2\sigma_1^{2}$ is a singular $2$-block, then Proposition~\ref{lemma2blockextadmissible} implies that $w_{\lambda}^{P} = 0$, $d_{\mu}^{P} = d_{\lambda}^{R} + a_p + 1$, $d_{\nu}^{P} = d_{\lambda}^{R}$, and $\mu_{0}^{P} = \lambda_{0}^{R} = \nu_{0}^{P}$. If $B = \sigma_2^{c_{p-1}}\sigma_3\sigma_2$ is a singular $3$-block, then Proposition~\ref{lemma3blockextadmissible} \textbf{(a)} and \textbf{(b)} imply that $w_{\lambda}^{P} = 0$, $d_{\mu}^{P} = d_{\lambda}^{R}$, $d_{\nu}^{P} = d_{\lambda}^{R} + c_{p-1} - 1$, and $\mu_{0}^{P} = \lambda_{0}^{R} = \nu_{0}^{P}$. 

We also observe that $R'$ is simultaneously a $1$-switch elementary bad subroad and a $3$-switch elementary bad subroad. If $B$ is a singular $2$-block, then we define $\iota_{p}' = a_p$, and if $B$ is a singular $3$-block, then we define $\iota_p' = c_{p-1} -1$.  Let $P_0'$ be an $s_0'$-subproduct of $\sigma$ such that $P_0'\setminus P\subseteq R'$. Firstly, Proposition~\ref{pelementarybadsubroadimpliesuniqueextension} \textbf{(c)} implies that $w_{\lambda}^{P_0'} = 0$. 

If $P_0'$ ends with $\sigma_i^2\sigma_2$, then Proposition~\ref{pelementarybadsubroadimpliesuniqueextension} \textbf{(c)} implies that $d_{\nu}^{P_0'} - d_{\mu}^{P_0'} = \iota_p'$, and $\nu_{0}^{P_0'} = \pm \lambda_{0}^{R}$. If $P_0'$ ends with $\sigma_{i'}^2\sigma_2$, then Proposition~\ref{pelementarybadsubroadimpliesuniqueextension} \textbf{(c)} implies that $d_{\mu}^{P_0'}-d_{\nu}^{P_0'} = \iota_p'$, and $\mu_{0}^{P_0'} = \pm \lambda_{0}^{R}$. 

If $P_0'$ ends with $\sigma_i^2$, then Proposition~\ref{pelementarybadsubroadimpliesuniqueextension} \textbf{(c)} implies that $d_{\mu}^{P_0'} - d_{\nu}^{P_0'} = \iota_p'+1$ if $B$ is a $2$-block and $d_{\mu}^{P_0'}-d_{\nu}^{P_0'} = \iota_p'$ if $B$ is a $3$-block, and $\mu_{0}^{P_0'} = \pm \lambda_{0}^{R}$ in either case. If $P_0'$ ends with $\sigma_{i'}^2$, then Proposition~\ref{pelementarybadsubroadimpliesuniqueextension} \textbf{(c)} implies that $d_{\nu}^{P_0'} - d_{\mu}^{P_0'} = \iota_p'$ if $B$ is a $2$-block and $d_{\nu}^{P_0'} - d_{\mu}^{P_0'} = \iota_p'-1$ if $B$ is a $3$-block, and $\nu_{0}^{P_0'} = \pm \lambda_{0}^{R}$ in either case.

We now apply these statements.

\begin{description}
\item[(a)] In this case, either $R''$ ends with $\sigma_{i}^{\phi_{p'}}$, $R''$ ends with $\sigma_{i'}^2\sigma_2^{\phi_{p'}}$ with $\phi_{p'}\geq 2$, or $R''$ ends with $\sigma_{i}^2\sigma_2^{\phi_{p'}}$ with $\phi_{p'}>2$.

If $R''$ ends with $\sigma_{i}^{\phi_{p'}}$ (in which case $R'$ ends with $\sigma_i^2$), then Proposition~\ref{p21/3ext} \textbf{(iii)} and Proposition~\ref{p21/3extadd} \textbf{(i)} imply that $P'$ is strongly $r$-regular. If $R''$ ends with $\sigma_{i'}^2\sigma_2^{\phi_{p'}}$ with $\phi_{p'}\geq 2$ (in which case $R'$ ends with $\sigma_{i'}^2$), then Proposition~\ref{p1/32extadd} \textbf{(d)} \textbf{(i)} implies that $P'$ is strongly $r$-regular. If $R''$ ends with $\sigma_i^2\sigma_2^{\phi_{p'}}$ with $\phi_{p'}>2$ (in which case $R'$ ends with $\sigma_i^2$), then Proposition~\ref{p1/32extadd} \textbf{(d)} \textbf{(ii)} implies that $P'$ is strongly $r$-regular.
\item[(b)] 
\begin{description}
\item[(i)] If $R''$ ends with $\sigma_2^{\phi_{p'}}$, then $R'$ ends with $\sigma_i^2$ and $\phi_{p'} = 2$.  The statement is a consequence of Proposition~\ref{p1/32extadd} \textbf{(d)} \textbf{(ii)}.
\item[(ii)] If $R''$ ends with $\sigma_{i'}^{\phi_{p'}}$, then $R'$ ends with $\sigma_i^2$ and $\phi_{p'}>2$. The statement is a consequence of Proposition~\ref{p21/3ext} \textbf{(iii)} and Proposition~\ref{p21/3extadd} \textbf{(i)}.
\end{description}
\end{description}
If $P'$ is not defined, then Proposition~\ref{pelementarybadsubroadimpliesuniqueextension} \textbf{(c)} implies that the leading coefficients of multiple entries in the $r$th row of $\beta_4\left(\sigma\right)$ are non-zero modulo each prime $t\neq 2$. 
\end{proof}

\subsubsection{Conclusion of the proof}

Finally, we observe that the stronger version of the main result is a consequence of the theory developed in this subsection.

\begin{proof}[Proof of Theorem~\ref{mainstronger}]
The statement is a consequence of Proposition~\ref{pregularap>1}, Lemma~\ref{lweightcompute}, Lemma~\ref{luniqueextensionimpliesbadsubroad}, Lemma~\ref{l2subdisc01}, Lemma~\ref{l1/3subdisc01}, Proposition~\ref{pgood3match}, Corollary~\ref{cpseudoregularregular}, Corollary~\ref{cor2blockregular}, Corollary~\ref{cor3blockregular}, and Lemma~\ref{lsingularblockregular}.
\end{proof}

\bibliography{References}

\begin{thebibliography}{10}

\bibitem{bigelow1999burau}
Stephen Bigelow.
\newblock The {B}urau representation is not faithful for {$n=5$}.
\newblock {\em Geom. Topol.}, 3:397--404, 1999.

\bibitem{bigelow2002does}
Stephen Bigelow.
\newblock Does the {J}ones polynomial detect the unknot?
\newblock {\em J. Knot Theory Ramifications}, 11(4):493--505, 2002.
\newblock Knots 2000 Korea, Vol. 2 (Yongpyong).

\bibitem{birmanbraidslinks}
Joan~S. Birman.
\newblock {\em Braids, links, and mapping class groups}.
\newblock Annals of Mathematics Studies, No. 82. Princeton University Press,
  Princeton, N.J.; University of Tokyo Press, Tokyo, 1974.

\bibitem{burau1935zopfgruppen}
Werner Burau.
\newblock \"{U}ber {Z}opfgruppen und gleichsinnig verdrillte {V}erkettungen.
\newblock {\em Abh. Math. Sem. Univ. Hamburg}, 11(1):179--186, 1935.

\bibitem{button2016minimal}
JO~Button.
\newblock Minimal dimension faithful linear representations of common finitely
  presented groups.
\newblock {\em arXiv preprint, arXiv:1610.03712}, 2016.

\bibitem{cooperlongmod2}
D.~Cooper and D.~D. Long.
\newblock A presentation for the image of {${\rm Burau}(4)\otimes
  \mathbb{Z}_2$}.
\newblock {\em Invent. Math.}, 127(3):535--570, 1997.

\bibitem{longpatonsmallprime}
D.~Cooper and D.~D. Long.
\newblock On the {B}urau representation modulo a small prime.
\newblock In {\em The {E}pstein birthday schrift}, volume~1 of {\em Geom.
  Topol. Monogr.}, pages 127--138. Geom. Topol. Publ., Coventry, 1998.

\bibitem{datta2020burauthesis}
Amitesh Datta.
\newblock {\em On the {B}urau {R}epresentation of the {B}raid {G}roup {$B_4$}}.
\newblock ProQuest LLC, Ann Arbor, MI, 2020.
\newblock Thesis (Ph.D.)--Princeton University,
  \url{https://www.proquest.com/docview/2378078915}.

\bibitem{garside1969braid}
F.~A. Garside.
\newblock The braid group and other groups.
\newblock {\em Quart. J. Math. Oxford Ser. (2)}, 20:235--254, 1969.

\bibitem{itojones}
Tetsuya Ito.
\newblock A kernel of a braid group representation yields a knot with trivial
  knot polynomials.
\newblock {\em Math. Z.}, 280(1-2):347--353, 2015.

\bibitem{jones1987hecke}
V.~F.~R. Jones.
\newblock Hecke algebra representations of braid groups and link polynomials.
\newblock {\em Ann. of Math. (2)}, 126(2):335--388, 1987.

\bibitem{long1993burau}
D.~D. Long and M.~Paton.
\newblock The {B}urau representation is not faithful for {$n\geq 6$}.
\newblock {\em Topology}, 32(2):439--447, 1993.

\bibitem{magnus1969theorem}
Wilhelm Magnus and Ada Peluso.
\newblock On a theorem of {V}. {I}. {A}rnol$\cprime$d.
\newblock {\em Comm. Pure Appl. Math.}, 22:683--692, 1969.

\bibitem{moody1991burau}
John~Atwell Moody.
\newblock The {B}urau representation of the braid group {$B_n$} is unfaithful
  for large {$n$}.
\newblock {\em Bull. Amer. Math. Soc. (N.S.)}, 25(2):379--384, 1991.

\end{thebibliography}
\bibliographystyle{plain}
\end{document}